\documentclass[a4paper, 11pt, reqno]{amsart}

\usepackage{microtype}
\usepackage{fullpage}

\usepackage{amssymb}
\usepackage{mathtools}
\usepackage{enumitem}
\usepackage{bm}

\renewcommand{\Re}{\operatorname{Re}}

\newcommand{\BB}{\mathbb{B}}
\newcommand{\CC}{\mathbb{C}}
\newcommand{\dd}{\,\mathrm{d}}
\newcommand{\ee}{\mathrm{e}}
\newcommand{\ii}{\mathrm{i}}
\newcommand{\NN}{\mathbb{N}}
\newcommand{\pd}{\partial}
\newcommand{\RR}{\mathbb{R}}

\theoremstyle{definition}
\newtheorem{definition}{Definition}[section]

\theoremstyle{plain}
\newtheorem{lemma}{Lemma}[section]

\theoremstyle{plain}
\newtheorem{theorem}{Theorem}[section]

\theoremstyle{plain}
\newtheorem{proposition}{Proposition}[section]

\theoremstyle{remark}
\newtheorem{remark}{Remark}[section]

\numberwithin{equation}{section}

\usepackage{hyperref}
    \hypersetup{
               linktoc = page,
            colorlinks = true,
             linkcolor = blue,
             citecolor = green,
    }

\usepackage{cleveref}
    \Crefname{equation}{Eq.}{Eqs.}
    \Crefname{figure}{Figure}{Figures}

\title{On stable self-similar blowup for corotational wave maps and equivariant Yang-Mills connections}

\author{Roland Donninger}
\address{University of Vienna, Faculty of Mathematics, Oskar-Morgenstern-Platz 1, 1090 Vienna, Austria}
\email{roland.donninger@univie.ac.at}

\author{Matthias Ostermann}
\address{University of Vienna, Faculty of Mathematics, Oskar-Morgenstern-Platz 1, 1090 Vienna, Austria}
\email{matthias.ostermann@univie.ac.at}
\thanks{This research was funded in whole or in part by the Austrian Science Fund (FWF) [10.55776/P34560]. M.~Ostermann acknowledges the support by the Vienna School of Mathematics (VSM)}
\begin{document}
\begin{abstract}
We consider corotational wave maps from Minkowski spacetime into the sphere and the equivariant Yang-Mills equation for all energy-supercritical dimensions. Both models have explicit self-similar finite time blowup solutions, which continue to exist even past the singularity. We prove the nonlinear asymptotic stability of these solutions in spacetime regions that approach the future light cone of the singularity. For this, we develop a general functional analytic framework in adapted similarity coordinates that allows to evolve the stable wave flow near a self-similar blowup solution in such spacetime regions.
\end{abstract}
\maketitle
\tableofcontents
\section{Introduction}
In this paper, we study two prototypical geometric wave equations on $(1+n)$-dimensional Minkowski spacetime $\RR^{1,n}$. As customary in this context, Einstein summation convention is employed, Greek indices $\mu,\nu$ take values in $\{0,1,\ldots,n \}$ and are raised and lowered with respect to the Minkowski metric $\eta = \operatorname{diag}(-1,1,\ldots,1)$.
\medskip\par
The one model we consider are \emph{wave maps} with domain manifold $\RR^{1,n}$ and target manifold the $n$-sphere $\mathbb{S}^{n} \subset \RR^{n+1}$, which is embedded as a Riemannian submanifold in Euclidean space $\RR^{n+1}$. These are maps $U : \RR^{1,n} \rightarrow \mathbb{S}^{n} \subset \RR^{n+1}$ that solve the \emph{wave maps equation}
\begin{equation}
\label{WaveMapsEquation}
\pd^{\mu} \pd_{\mu} U + \langle \pd^{\mu} U, \pd_{\mu} U \rangle U = 0 \,,
\end{equation}
where $\langle \,.\,, \,..\, \rangle$ denotes the Euclidean scalar product on $\RR^{n+1}$.
\medskip\par
The other model we consider are \emph{Yang-Mills connections} represented on the Lie algebra $\mathfrak{so}(n)$ of skew-symmetric $(n\times n)$-matrices. These are $\mathfrak{so}(n)$-valued connection $1$-forms $A_{\mu}: \RR^{1,n} \rightarrow \mathfrak{so}(n)$ with associated curvature $2$-form
\begin{equation*}
F_{\mu\nu} = \pd_{\mu} A_{\nu} - \pd_{\nu} A_{\mu} + [A_{\mu}, A_{\nu}]
\end{equation*}
that solve the \emph{Yang-Mills equation}
\begin{equation}
\label{YangMillsEquation}
\pd^{\mu} F_{\mu\nu} + [A^{\mu}, F_{\mu\nu} ] = 0 \,,
\end{equation}
where the Lie bracket $[ \,.\,, \,..\, ]$ is given by the commutator for $(n\times n)$-matrices.
\medskip\par
The wave maps equation \eqref{WaveMapsEquation} and Yang-Mills equation \eqref{YangMillsEquation} are nonlinear time evolution equations of wave type and as such, their solution theory is concerned with the respective Cauchy problem. A central question of the Cauchy problem is whether solutions exist for all time or develop singularities for given initial data. Properties of these equations, which are important for their analysis, are conservation laws, symmetries, and special solutions. In this respect, we first recall that the wave maps and Yang-Mills equation both appear as Euler-Lagrange equations derived from geometric action functionals with Lagrange densities
\begin{equation*}
\mathcal{L}_{\mathsf{WM}}[U] = \frac{1}{2} \langle \pd^{\mu} U, \pd_{\mu} U \rangle 
\qquad\text{and}\qquad
\mathcal{L}_{\mathsf{YM}}[A] = \frac{1}{4} \langle F_{\mu\nu}, F^{\mu\nu} \rangle \,,
\end{equation*}
where $\langle \,.\,, \,..\, \rangle$ denotes the Euclidean scalar product on $\RR^{n+1}$ and the Hilbert-Schmidt inner product for $(n\times n)$-matrices, respectively. The energy-momentum tensor is defined by its components
\begin{equation*}
\mathrm{T}_{\mathsf{WM}}[U]_{\mu\nu} = \langle \pd_{\mu} U, \pd_{\nu} U \rangle - \mathcal{L}_{\mathsf{WM}}[U] \eta_{\mu\nu}
\quad\text{and}\quad
\mathrm{T}_{\mathsf{YM}}[A]_{\mu\nu} = \langle F_{\mu}{}^{\alpha}, F_{\nu\alpha} \rangle - \mathcal{L}_{\mathsf{YM}}[A] \eta_{\mu\nu}
\end{equation*}
and is divergence-free along the flow. In particular, this implies a conservation law for the positive definite energies
\begin{align}
\label{WMEnergy}
E_{\mathsf{WM}}[U](t) &= \int_{\RR^{n}} \mathrm{T}_{\mathsf{WM}}[U]_{00}(t,\,.\,) = \frac{1}{2} \int_{\RR^{n}} \sum_{\mu=0}^{n} \langle \pd_{\mu} U(t,\,.\,), \pd_{\mu} U(t,\,.\,) \rangle
\intertext{and}
\label{YMEnergy}
E_{\mathsf{YM}}[A](t) &= \int_{\RR^{n}} \mathrm{T}_{\mathsf{YM}}[A]_{00}(t,\,.\,) = \frac{1}{2} \int_{\RR^{n}} \sum_{0 \leq \mu < \nu \leq n} \langle F_{\mu\nu}(t,\,.\,), F_{\mu\nu}(t,\,.\,) \rangle \,.
\end{align}
Moreover, the wave maps equation and Yang-Mills equation have a scaling symmetry. That is, if $U$ solves \Cref{WaveMapsEquation} then so does
\begin{equation}
\label{WMScaling}
U_{\lambda}
\qquad\text{given by}\qquad
U_{\lambda}(t,x) \coloneqq U \big( \tfrac{t}{\lambda}, \tfrac{x}{\lambda} \big)
\end{equation}
and if $A_{\mu}$ solve \Cref{YangMillsEquation}, then so do
\begin{equation}
\label{YMScaling}
{A_{\lambda}}_{\mu}
\qquad\text{given by}\qquad
{A_{\lambda}}_{\mu}(t,x) \coloneqq \lambda^{-1} A_{\mu} \big( \tfrac{t}{\lambda}, \tfrac{x}{\lambda} \big)
\end{equation}
for all $\lambda > 0$. It follows that the energies \eqref{WMEnergy}, \eqref{YMEnergy} exhibit the scaling
\begin{equation*}
E_{\mathsf{WM}}[U_{\lambda}](t) = \lambda^{n-2} E_{\mathsf{WM}}[U] \big( \tfrac{t}{\lambda} \big)
\qquad\text{and}\qquad
E_{\mathsf{YM}}[A_{\lambda}](t) = \lambda^{n-4} E_{\mathsf{YM}}[A] \big( \tfrac{t}{\lambda} \big)
\end{equation*}
for all $\lambda > 0$. Depending on the scaling behaviour of the energies in different space dimensions, the respective model is called
\begin{equation*}
\renewcommand{\arraystretch}{1.2}
\begin{array}{llcl}
\text{ \emph{energy-subcritical}} & \text{if } n = 1 \text{ for wave maps } & \text{and} & n \leq 3 \text{ for Yang-Mills}, \\
\text{ \emph{energy-critical}} & \text{if } n = 2 \text{ for wave maps } & \text{and} & n = 4 \text{ for Yang-Mills}, \\
\text{ \emph{energy-supercritical}} & \text{if } n \geq 3 \text{ for wave maps } & \text{and} & n \geq 5 \text{ for Yang-Mills}.
\end{array}
\end{equation*}
It is precisely the whole range of energy-supercritical dimensions that we are concerned with here. Following the scaling heuristics in this regime, the shrinking of large solutions to smaller scales is energetically favorable, and thereby the formation of a singularity is anticipated. An important role in the blowup dynamics is played by \emph{self-similar} solutions, which are by definition invariant under the respective scaling transformation \eqref{WMScaling}, \eqref{YMScaling}. Before we come to such explicitly known solutions, we mention that self-similar blowup is already observed within the class of \emph{corotational} or \emph{equivariant} solutions, see \cite[p.~109]{MR1674843} and \cite[p.~136]{MR3837560} for general definitions. A map $U: \RR^{1,n} \rightarrow \mathbb{S}^{n} \subset \mathbb{R}^{n+1}$ is called corotational if there is a function $\widetilde{\psi}: \RR \times [0,\infty) \rightarrow \RR$ such that
\begin{equation}
\label{CorotWM}
U(t,x) =
\begin{pmatrix}
|x|^{-1} \sin \big( |x| \widetilde{\psi}(t,|x|) \big) x \\
\cos \big( |x| \widetilde{\psi}(t,|x|) \big)
\end{pmatrix}
\,.
\end{equation}
Similarly, a connection $1$-form $A_{\mu}: \RR^{1,n} \rightarrow \mathfrak{so}(n)$ is called equivariant if there is a function $\widetilde{\psi}: \RR \times [0,\infty) \rightarrow \RR$ such that
\begin{equation}
\label{EquivarYM}
A_{\mu}{}^{ij}(t,x) = \widetilde{\psi}(t,|x|) ( x^{i} \delta_{\mu}{}^{j} - x^{j} \delta_{\mu}{}^{i} ) \,.
\end{equation}
An additional important symmetry for the Yang-Mills equation is the gauge symmetry. However, imposing the equivariant ansatz automatically fixes a gauge and eliminates this symmetry. The gauge symmetry is briefly explained in \Cref{SectionRelatedWork}, where we discuss related works on the well-posedness theory for both equations.
\subsection{Singularity formation and blowup stability}
\label{SubSecSing}
In any energy-supercritical dimension, the wave maps equation \eqref{WaveMapsEquation} and Yang-Mills equation \eqref{YangMillsEquation} have a finite-time blowup solution with self-similar profile given in closed form by
\begin{align}
\label{WaveMapsBlowup}
U_{T}^{\ast}(t,x) &= U^{\ast}(T-t,x)
&\text{where}&&
U^{\ast}(t,x) &=
\frac{1}{(n-2)t^{2} + |x|^{2}}
\begin{pmatrix}
2\sqrt{n-2} t x \\
(n-2)t^{2} - |x|^{2}
\end{pmatrix}
\,, \\
\label{YangMillsBlowup}
{A_{T}^{\ast}}_{\mu}(t,x) &= {A^{\ast}}_{\mu}(T-t,x)
&\text{where}&&
{A^{\ast}}_{\mu}{}^{ij}(t,x) &= \frac{x^{i}}{\alpha_{n} t^{2} + \beta_{n} |x|^{2}} \delta_{\mu}{}^{j} - \frac{x^{j}}{\alpha_{n} t^{2} + \beta_{n} |x|^{2}} \delta_{\mu}{}^{i} \,,
\end{align}
for any blowup time $T>0$, respectively, with coefficients
\begin{align}
\label{YangMillsCoefficients1}
\alpha_{n} &= \frac{1}{12} \frac{3(n-2)(n-4) + (n+2)\sqrt{3(n-2)(n-4)}}{n-1} \,, \\
\beta_{n} &= \frac{1}{4} \frac{3(n-2) - \sqrt{3(n-2)(n-4)}}{n-1} \,,
\end{align}
see \cite{MR3355819}. Note that the solutions \eqref{WaveMapsBlowup}, \eqref{YangMillsBlowup} are smooth in the whole domain away from the blowup event $(T,0) \in \RR^{1,n}$. Moreover, both solutions are corotational in the sense of \eqref{CorotWM}, \eqref{EquivarYM} with profile
\begin{equation}
\label{SelfSimilarSolutionsIntro}
\widetilde{\psi}^{\ast}(t,r) =
\renewcommand{\arraystretch}{1.5}
\left\{
\begin{array}{ll}
\displaystyle{
\frac{4}{r} \arctan\Big( \frac{r}{ \sqrt{n-2} t + \sqrt{(n-2)t^{2} + r^{2}} } \Big)
} & \text{for wave maps,} \\
\displaystyle{
\frac{1}{\alpha_{n} t^{2} + \beta_{n} r^{2}}
} & \text{for Yang-Mills.}
\end{array}
\right.
\end{equation}
The existence of a self-similar corotational wave map from $(1+3)$-dimensional Minkowski spacetime into the $3$-sphere was first proved by J.~Shatah \cite{MR933231}. The explicit solution corresponding to \eqref{WaveMapsBlowup} in case $n=3$ was later found by N.~Turok and D.~Spergel \cite{turok1990global}. The existence of self-similar solutions for the Yang-Mills equation in dimensions $5 \leq n \leq 8$ is due to T.~Cazenave, J.~Shatah and A.~S.~Tahvildar-Zadeh \cite{MR1622539}. For $n=5$, the explicit blowup solution \eqref{YangMillsBlowup} was found by P.~Bizo\'{n} in \cite{MR1923684}. In the lowest energy-supercritical dimensions, he also proved the existence of a countable family of smooth self-similar solutions \cite{MR1799875}, \cite{MR1923684}, with the ``ground state'' corresponding to the explicit solution from above. For wave maps, this existence result was extended to $3 \leq n \leq 6$ in \cite{MR3636308}. The explicit ground state solutions in all higher dimensions were found much more recently in \cite{MR3355819}. P.~Bizo\'{n} et al.~also performed numerical studies in \cite{MR1767966}, \cite{MR1877844}, \cite{Bizo__2005}, \cite{MR3355819} on singularity formation, where they observed that corotational wave maps and equivariant Yang-Mills connections with large data blow up in finite time and exhibit the same blowup as described by the profiles \eqref{SelfSimilarSolutionsIntro}. This led to the conjecture that these self-similar blowup solutions are universal in the sense that the generic large data Cauchy evolution blows up and converges to them.
\medskip\par
A first step to explain this phenomenon rigorously is to prove that the blowup solutions \eqref{WaveMapsBlowup}, \eqref{YangMillsBlowup} are stable under perturbations of their initial data. A stability theory has been initiated by the first author in \cite{MR2839272}, \cite{MR3278903} and focused on wave maps in dimension $n=3$ and Yang-Mills in dimension $n=5$. Nevertheless, a complete proof of stability hinged on the linear stability of the blowup profiles, which is equivalent to the difficult mode stability problem, see \cite{2023arXiv231012016D} for an introduction. This was solved in \cite{MR2412310}, \cite{MR2881965}, \cite{MR3538419} for wave maps and in \cite{MR3475668} for Yang-Mills. The underlying spectral methods were further developed by I.~Glogi\'{c} and applied to prove mode stability of the blowup profiles \eqref{SelfSimilarSolutionsIntro} in all higher dimensions \cite{MR3623242}, \cite{MR4469070}. Subsequently, the stability results were extended to all odd space dimensions as well, see \cite{MR3680948} and \cite{MR4469070}. Blowup stability for wave maps at optimal regularity in dimensions $n=3,4$ was proved by the first author and D.~Wallauch \cite{MR4640202}, \cite{2022arXiv221208374D}. We also mention the original works \cite{2023arXiv231007042C}, \cite{2024arXiv240815345M} on blowup for the related Skyrme model.
\medskip\par
Note, however, that the above results prove stability within a backward light cone and actually none of them uses that the explicit blowup solutions \eqref{WaveMapsBlowup}, \eqref{YangMillsBlowup} feature a natural continuation after the blowup time. This naturally points to the question whether these solutions are also stable in spacetime regions beyond light cones. To address this, ``hyperboloidal similarity coordinates'' have been devised in \cite{MR4338226} to prove the nonlinear stability of the wave maps blowup in dimension $n=3$ under small localized perturbations in spacetime regions that extend even beyond the time of blowup. An implementation of this novel stability theory to higher dimensions requires to evolve the wave flow near self-similar blowup in such curvilinear coordinate systems. This is technically challenging and has so far only been achieved in odd space dimensions and in radial symmetry, where it was implemented for the Yang-Mills blowup \cite{MR4661000}. The underlying functional setting also found application in a proof of co-dimension one radial stability of an explicit blowup solution for the quadratic wave equation \cite{MR4700297}. In a recent complementary approach, I.~Glogi\'{c} \cite{2022arXiv220706952G}, \cite{2023arXiv230510312G} showed in any dimension that the wave maps and Yang-Mills blowup is nonlinearly stable along the hyperplanes $\{t\} \times \RR^{d}$ up to $t = T$ in the whole region $(0,T) \times \RR^{n}$.
\subsection{Statement and discussion of the main results}
The main goal of this paper is to prove stability of the blowup mechanism described by the self-similar solutions $U_{T}^{\ast}$ and ${A_{T}^{\ast}}_{\mu}$ given in \Cref{WaveMapsBlowup,YangMillsBlowup} in spacetime regions beyond light cones that are of the form
\begin{equation*}
\Omega^{1,n}_{T}(\kappa) = \big\{ (t,x) \in \RR^{1,n} \mid 0 < t < T + \kappa |x| \big\}
\end{equation*}
for any fixed $-1 \leq \kappa < 1$ and some suitable parameter $T>0$. The key to access these regions is to utilize flexible coordinate systems.
\begin{definition}
\label{SimilarityCoordinatesFlat}
Let $n \in \NN$, $-1 \leq \kappa < 1$, $r > 0$, and $\varepsilon > 0$. Let $h \in C^{\infty}_{\mathrm{rad}}(\RR^{n})$ be such that
\begin{enumerate}[itemsep=1em,topsep=1em,label=(h\arabic*)]
\item\label{h1} \qquad $h(\xi) = -1$ for $|\xi| < r + \varepsilon$,
\item\label{h2} \qquad $|(\pd h)(\xi)| < 1$ for $\xi \in \RR^{n}$,
\item\label{h3} \qquad $(\pd^{2} h)(\xi)$ is positive semidefinite for $\xi \in \RR^{n}$,
\item\label{h4} \qquad there is an $R \geq r + \varepsilon$ with $h(\xi) = \kappa R$ for $|\xi| = R$.
\end{enumerate}
Let $T > 0$. We define \emph{similarity coordinates}
\begin{equation*}
{\chi\mathstrut}_{T} : \RR \times \RR^{n} \rightarrow \RR^{1,n} \,, \qquad {\chi\mathstrut}_{T}(\tau,\xi) = \big( T + T\ee^{-\tau} h(\xi), T\ee^{-\tau} \xi \big) \,,
\end{equation*}
in the image region
\begin{equation*}
\mathrm{X}^{1,n}_{T}(\kappa) \coloneqq {\chi\mathstrut}_{T} \big( (0,\infty) \times \BB^{n}_{R} \big) \,.
\end{equation*}
\end{definition}
Elementary geometric properties of similarity coordinates are studied in \Cref{AppendixSimilarityCoordinates} and a sketch is given in \Cref{FigHSC}. Comprehensive examples of such coordinate systems are provided in \Cref{GSCExemplary}.
\begin{figure}
\centering
\includegraphics[width=\textwidth]{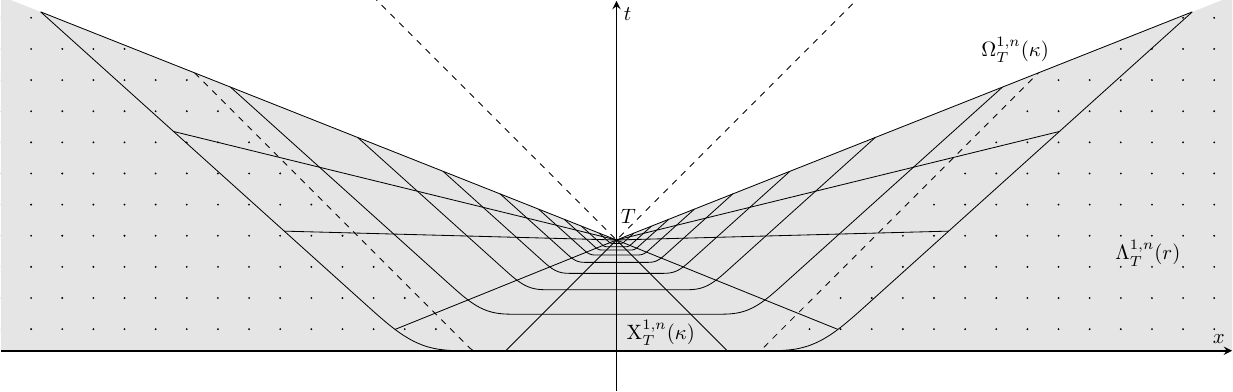}
\caption{Sketch of similarity coordinates centered around $(T,0) \in \RR^{1,n}$ for $n=1$. The gray-shaded part depicts $\Omega^{1,n}_{T}(\kappa)$, which exhausts the exterior of the future light cone of $(T,0) \in \RR^{1,n}$ as $\kappa \to 1$. The region covered by the coordinate grid is the image region $\mathrm{X}^{1,n}_{T}(\kappa)$. Here, the spacelike curves correspond to the hypersurfaces $\Sigma^{1,n}_{T}(t)$ whose flat part is determined by the parameter $r>0$ in condition \ref{h1}. The dotted region shows $\Lambda^{1,n}_{T}(\kappa)$.
}
\label{FigHSC}
\end{figure}
Let us comment on their central features for our stability analysis.
\begin{remark}
\label{RemarkImageRegion}
The function $h$ that appears in similarity coordinates is called a \emph{height function} and it determines the geometry of the coordinate system. The convexity condition \ref{h3} on the height function ensures that ${\chi\mathstrut}_{T} : \RR \times \RR^{n} \rightarrow \RR^{1,n}$ is a diffeomorphism onto its image, see \Cref{GSCLemma}. In particular, we get from \Cref{GSCImage} the explicit image region
\begin{equation*}
\mathrm{X}^{1,n}_{T}(\kappa) = \big\{ (t,x) \in \RR^{1,n} \mid T + T h(\tfrac{x}{T}) < t < T + \kappa |x| \big\} \,.
\end{equation*}
Thus, the parameter $-1 \leq \kappa < 1$ in condition \ref{h4} determines the spacetime region which is covered by a coordinate cylinder in similarity coordinates. For example,
\begin{enumerate}[itemsep=1em,topsep=1em]
\item[] $\kappa = -1$ yields a cover of the backward light cone of $(T,0) \in \RR^{1,n}$,
\item[] $\kappa = 0$ yields an image contained in $(0,T) \times \RR^{n}$,
\item[] $0 < \kappa < 1$ allows to reach the region beyond the blowup time $t=T$ and to approach the complement of the future light cone of $(T,0) \in \RR^{1,n}$.
\end{enumerate}
\end{remark}
\begin{remark}
\label{RemarkSHSF}
Similarity coordinates are adapted to the scaling symmetries of the respective nonlinear wave equations and are therefore well-suited for the study of self-similar solutions. Another important feature is that due to the bound \ref{h2} on the gradient of the height function, the level sets
\begin{equation*}
\Sigma^{1,n}_{T}(t) \coloneqq {\chi\mathstrut}_{T} \big( \{\tau(t)\} \times \BB^{n}_{R} \big) \,, \qquad \tau(t) = \log \frac{T}{T - t} \,,
\end{equation*}
through $(t,0) \in \RR^{1,n}$ for $0 \leq t < T$, are spacelike hypersurfaces that foliate the image region $\mathrm{X}^{1,n}_{T}(\kappa)$ according to \Cref{GeometryFoliation,GSCFoliation}.
\end{remark}
\begin{remark}
\label{RemarkFlatRegion}
The additional condition \ref{h1} that the height function is constant in $\BB^{n}_{r + \varepsilon}$ implies for $T \geq 1 - \frac{\varepsilon}{r + \varepsilon}$ that the region $\Omega^{1,n}_{T}(\kappa) \setminus \mathrm{X}^{1,n}_{T}(\kappa)$, which is not covered by the coordinate grid, is contained in
\begin{equation*}
\Lambda^{1,n}_{T}(r) = \{ (t,x) \in \Omega^{1,n}_{T}(\kappa) \mid 0 < t < |x| - r \} \,,
\end{equation*}
i.e. the exterior of the domain of influence of $\{0\} \times \BB^{n}_{r}$ in $\Omega^{1,n}_{T}(\kappa)$. We also get that the spacelike hypersurfaces from the previous remark contain a flat portion
\begin{equation*}
\{0\} \times \BB^{n}_{\frac{r(T-t)}{T}} \subset \Sigma^{1,n}_{T}(t) \,.
\end{equation*}
This allows us to collect initial data at $t=0$ directly along the initial hypersurface $\Sigma^{1,n}_{T}(0)$, which will come into play later in \Cref{PreparationData}.
\end{remark}
Now, the statements of the stability theorems for wave maps and Yang-Mills connections are as follows. Their content is discussed in the remarks below.
\begin{theorem}[Stable blowup for corotational wave maps]
\label{THMWM}
Let $n,k \in \NN$ with
\begin{equation*}
n \geq 3
\qquad\text{and}\qquad
k \geq \frac{n}{2} + 1 \,.
\end{equation*}
Let $-1 \leq \kappa < 1$ and $r > 0$ and fix a height function as in \Cref{SimilarityCoordinatesFlat}. Fix $0 < \varepsilon_{r} \leq \frac{2(n-2)}{(n-2) + r^{2}}$. Then, there exist constants $\delta^{\ast} > 0$, $M \geq 1$, $\omega > 0$ such that for all $0 < \delta \leq \delta^{\ast}$ and all smooth corotational data $(F,G) : \RR^{n} \rightarrow \mathrm{T}\mathbb{S}^{n} \subset \RR^{n+1} \times \RR^{n+1}$ of the form
\begin{equation*}
F(x) =
\begin{pmatrix}
|x|^{-1} \sin \big( |x| f(x) \big) x \\
\cos \big( |x| f(x) \big)
\end{pmatrix}
\,, \qquad
G(x) =
\begin{pmatrix}
\cos \big( |x| f(x) \big) g(x) x \\
-|x| \sin \big( |x| f(x) \big) g(x)
\end{pmatrix}
\,,
\end{equation*}
with $f,g \in C^{\infty}_{\mathrm{rad}}(\RR^{n})$ which satisfy
\begin{equation*}
\renewcommand{\arraystretch}{1.2}
\begin{array}{rcll}
(F,G) &=& \big( U^{\ast}_{1}(0,\,.\,) , \pd_{0} U^{\ast}_{1}(0,\,.\,) \big) & \text{in } \RR^{n} \setminus \BB^{n}_{r} \,, \\
F^{n+1} &\geq& -1 + \varepsilon_{r} & \text{in } \BB^{n}_{r} \,,
\end{array}
\end{equation*}
and
\begin{equation*}
\big\| \big(F,G\big) - \big( U^{\ast}_{1}(0,\,.\,) , \pd_{0} U^{\ast}_{1}(0,\,.\,) \big) \big\|_{H^{k}(\RR^{n})\times H^{k-1}(\RR^{n})} \leq \frac{\delta}{M} \,,
\end{equation*}
there exist a blowup time $T^{\ast} \in [1-\delta,1+\delta]$ and a unique smooth corotational wave map $U: \Omega^{1,n}_{T^{\ast}}(\kappa) \rightarrow \mathbb{S}^{n}$ to the Cauchy problem
\begin{align*}
\pd^{\mu} \pd_{\mu} U + \big\langle \pd^{\mu} U, \pd_{\mu} U \big\rangle U &=0
\phantom{G \quad }\sbox0{$0$}\hspace{-\the\wd0}
\text{in } \Omega^{1,n}_{T^{\ast}}(\kappa) \,, \\
U(0, \,.\,) = F \,, \quad \pd_{0} U(0, \,.\,) &= G \quad \text{in } \RR^{n} \,,
\end{align*}
which satisfies the stability estimate
\begin{equation*}
\big\| U - U^{\ast}_{T^{\ast}} \big\|_{H^{k}(\Sigma^{1,n}_{T^{\ast}}(t))} \lesssim \delta \Big( \frac{T^{\ast} - t}{T^{\ast}} \Big)^{\omega}
\end{equation*}
for all $0 \leq t < T^{\ast}$, as well as finite speed of propagation
\begin{equation*}
U \equiv U_{1}^{\ast} \quad \text{in } \Lambda^{1,n}_{T^{\ast}}(r) \,.
\end{equation*}
\end{theorem}
\begin{theorem}[Stable blowup for equivariant Yang-Mills connections]
\label{THMYM}
Let $n, k \in \NN$ with
\begin{equation*}
n \geq 5
\qquad\text{and}\qquad
k \geq \frac{n+1}{2} \,.
\end{equation*}
Let $-1 \leq \kappa < 1$ and $r > 0$ and fix a height function as in \Cref{SimilarityCoordinatesFlat}. Then, there exist constants $\delta^{\ast} > 0$, $M \geq 1$, $\omega > 0$ such that for all $0 < \delta \leq \delta^{\ast}$ and all smooth equivariant perturbations $a_{\mu},b_{\mu}: \RR^{n} \rightarrow \mathfrak{so}(n)$ of the form
\begin{equation*}
a_{\mu}{}^{ij}(x) = f(x) ( x^{i} \delta_{\mu}{}^{j} - x^{j} \delta_{\mu}{}^{i}) \,, \qquad b_{\mu}{}^{ij}(x) = g(x) ( x^{i} \delta_{\mu}{}^{j} - x^{j} \delta_{\mu}{}^{i}) \,,
\end{equation*}
with $f,g \in C^{\infty}_{\mathrm{rad}}(\RR^{n})$ which satisfy
\begin{equation*}
\operatorname{supp}(a_{\mu},b_{\mu}) \subset \BB^{n}_{r}
\qquad\text{and}\qquad
\| (a,b) \|_{H^{k}(\RR^{n}) \times H^{k-1}(\RR^{n})} \leq \frac{\delta}{M}
\end{equation*}
there exist a blowup time $T^{\ast} \in [ 1 - \delta, 1 + \delta]$ and a unique smooth equivariant Yang-Mills connection $A_{\mu} : \Omega^{1,n}_{T^{\ast}}(\kappa) \rightarrow \mathfrak{so}(n)$ which solves the Cauchy problem for the hyperbolic Yang-Mills system
\begin{align*}
0 =\pd_{\mu} F^{\mu\nu} + [A_{\mu}, F^{\mu\nu}] \,, \quad
F_{\mu\nu} = \pd_{\mu} A_{\nu} - \pd_{\nu} A_{\mu} + [A_{\mu}, A_{\nu}] \,,
\qquad \text{in } \Omega^{1,n}_{T^{\ast}}(\kappa) \,,
\end{align*}
with data
\begin{equation*}
A_{\mu}(0, \,.\,) = {A_{1}^{\ast}}_{\mu}(0,\,.\,) + a_{\mu} \,,
\quad
\pd_{0} A_{\mu}(0, \,.\,) = \pd_{0} {A_{1}^{\ast}}_{\mu}(0,\,.\,) + b_{\mu} \qquad \text{in } \RR^{n} \,,
\end{equation*}
and satisfies the stability estimate
\begin{equation*}
\frac{\big\| A - A^{\ast}_{T^{\ast}} \big\|_{H^{k}(\Sigma^{1,n}_{T^{\ast}}(t))}}{\big\| A^{\ast}_{T^{\ast}} \big\|_{H^{k}(\Sigma^{1,n}_{T^{\ast}}(t))}} \lesssim \delta \Big( \frac{T^{\ast} - t}{T^{\ast}} \Big)^{\omega}
\end{equation*}
for all $0 \leq t < T^{\ast}$, as well as finite speed of propagation
\begin{equation*}
A_{\mu} \equiv {A_{1}^{\ast}}_{\mu}
\quad \text{in } \Lambda^{1,n}_{T^{\ast}}(r) \,.
\end{equation*}
\end{theorem}
\begin{remark}
\label{RemarkDiscussionContent}
\Cref{THMWM,THMYM} characterize for corotational wave maps and equivariant Yang-Mills connections in all energy-supercritical dimensions the dynamics for large data near the self-similar blowup solutions $U_{1}^{\ast}$ and ${A_{1}^{\ast}}_{\mu}$ in regions $\Omega^{1,n}_{T}(\kappa)$ beyond light cones and partially even past the blowup time $T$. The perturbations of the data at $t=0$ are assumed to be supported in a ball $\BB^{n}_{r}$ of arbitrary but fixed radius $r > 0$. Based on this radius, similarity coordinates are chosen and according to \Cref{RemarkImageRegion,RemarkSHSF,RemarkFlatRegion}, the spacetime domain is decomposed into
\begin{equation*}
\Omega^{1,n}_{T}(\kappa) = \mathrm{X}^{1,n}_{T}(\kappa) \cup \Lambda^{1,n}_{T}(r)
\qquad\text{with foliation}\qquad
\mathrm{X}^{1,n}_{T}(\kappa) = \bigcup_{t\in (0,T)} \Sigma^{1,n}_{T}(t) \,.
\end{equation*}
Then, in the image region $\mathrm{X}^{1,n}_{T}(\kappa)$, the solutions $U$, $A_{\mu}$ exhibit along the spacelike hypersurfaces $\Sigma^{1,n}_{T}(t)$ as $t \to T$ the asymptotic behaviour of the respective blowup solutions $U_{T}^{\ast}$, ${A_{T}^{\ast}}_{\mu}$. In the complementary region $\Lambda^{1,n}_{T}(r)$, the solutions are simply obtained from finite speed of propagation. In particular, the solutions are smooth in $\Omega^{1,n}_{T}(\kappa)$ and hence no other singularities than $(T,0) \in \RR^{1,n}$ form outside the backward light cone of this blowup point. 
\end{remark}
\begin{remark}
The convergence towards the blowup is quantified in Sobolev norms along the hypersurfaces, which are classically defined by
\begin{align*}
\| U \|_{H^{k}(\Sigma^{1,n}_{T}(t))} &\coloneqq \sum_{i=1}^{n} \| U^{i} \circ {\chi\mathstrut}_{T}(\tau(t),\,.\,) \|_{H^{k}(\BB^{n}_{R})}
&& \text{for } U : \RR^{1,n}\rightarrow \mathbb{S}^{n} \subset \RR^{n+1} \,, \\ 
\| A \|_{H^{k}(\Sigma^{1,n}_{T}(t))} &\coloneqq \sum_{i,j=1}^{n} \sum_{\mu=0}^{n} \| A_{\mu}{}^{ij} \circ {\chi\mathstrut}_{T}(\tau(t),\,.\,) \|_{H^{k}(\BB^{n}_{R})}
&& \text{for } A_{\mu} : \RR^{1,n}\rightarrow \mathfrak{so}(n) \,.
\end{align*}
Since no Sobolev norms above the scaling critical regularity $k = n/2$ for wave maps and $k = n/2 - 1$ for Yang-Mills is distinguished for the stability problem, we prove our results for all integer regularities for which we have an elementary local Lipschitz estimate of the effective nonlinearity in \Cref{LocLip}. In fact, due to the flexibility of the coordinate system, we manage to significantly improve the topology used in \cite{MR4338226}, \cite{MR4661000} in hyperboloidal similarity coordinates.
\end{remark}
\begin{remark}
\label{RemarkDiscussionMethod}
The previous works \cite{MR4338226}, \cite{MR4661000} on stability past the blowup were only accessible in odd space dimensions. The dimensional limitation was due to a technically challenging adaptation of the ``descent method'' for the wave equation, which is entirely avoided in this paper. Instead, our current functional analytic approach relies on energies and commuting operators for the wave operator. Not only is this available in all space dimensions, we are also able to control the free wave flow without any symmetry assumptions in similarity coordinates. We outline this in \Cref{SubSubSecNRG,SubSecEnTop}, the general theory is developed in \Cref{AppendixFreeWave} and concluded in \Cref{TheGenerationTHM}.
\end{remark}
\begin{remark}
Moreover, unlike in \cite{MR4338226}, \cite{MR4661000}, where the size of the support of the perturbations is dictated by a small parameter, here we can freely choose the support of the perturbations as large as we please.
\end{remark}
\begin{remark}
The condition $F^{n+1} \geq -1 + \varepsilon_{r}$ for the $(n+1)$-st component of the initial datum ensures that its image has positive distance from the north pole in $\mathbb{S}^{n}$. This is of technical nature and lets us prove radial lemmas for Sobolev norms of corotational maps, see \Cref{CorotDataBound}.
\end{remark}
\begin{remark}
We finally remark that the question whether the stable blowup in \Cref{THMWM,THMYM} can be meaningfully continued beyond the Cauchy horizon of the blowup event appears to be difficult and has to be addressed in another project.
\end{remark}
\subsection{Review of the theory for wave maps and Yang-Mills}
\label{SectionRelatedWork}
In the following, we present some of the main contributions to the rich theory of the wave maps and Yang-Mills equation. As nonlinear geometric wave equations, they share many interesting features and their solution theory is very much influenced by a great interplay of ideas, techniques and methods from the field of nonlinear dispersive equations. The details and cross references are contained in the literature, for an overview we also refer to the survey articles \cite{MR1901147}, \cite{MR2459308}.
\subsubsection{Results for the wave maps equation}
Let $(M,g)$ be a Riemannian manifold. A wave map $U: \RR^{1,n} \rightarrow M$ is defined as a critical point of the action functional
\begin{equation*}
S_{\mathsf{WM}}[U] = \frac{1}{2} \int_{\RR^{1,n}} \operatorname{tr}_{\eta}( U^{\ast} g ) \,.
\end{equation*}
Wave maps were originally introduced as a ``nonlinear sigma model'' in classical field theory \cite{MR140316} and appear in many more applications. In Minkowski coordinates on $\RR^{1,n}$ and local coordinates on $M$, we obtain the Euler-Lagrange equations
\begin{equation*}
\pd^{\mu} \pd_{\mu} \Psi^{i} + \Gamma^{i}{}_{jk}(\Psi) \pd^{\mu} \Psi^{j} \pd_{\mu} \Psi^{k} = 0 \,,
\end{equation*}
where $\Psi = (\Psi^{1}, \ldots,\Psi^{n})$ is the chart representation of the wave map $U$ and $\Gamma^{i}{}_{jk}$ are the Christoffel symbols of the metric $g$. This is the \emph{intrinsic} formulation of the wave maps equation. In case $M \subset \RR^{m}$ is a submanifold of Euclidean space with second fundamental form $A_{p} : \mathrm{T}_{p}M \times \mathrm{T}_{p}M \rightarrow \mathrm{T}_{p}M^{\perp}$ at $p \in M$, the \emph{extrinsic} formulation of the wave maps equation reads
\begin{equation*}
\pd^{\mu} \pd_{\mu} U + A_{U}(\pd^{\mu} U, \pd_{\mu} U) = 0 \,.
\end{equation*}
In the following qualitative overview, we oriented ourselves to a large extent on the modern exposition of the theory of wave maps in D.-A.~Geba's and M.~G.~Grillaki's textbook \cite{MR3585834}. Surveys on wave maps can be found in \cite{MR2488946}, \cite{MR2043751}.
\medskip\par
To begin with, the wave maps equation constitutes a system of semilinear wave equations for which classical energy methods yield local well-posedness for data in Sobolev spaces $H^{s}(\RR^{n}) \times H^{s-1}(\RR^{n})$ of high enough regularity $s > n/2 + 1$. Note that $s = n/2$ is the scaling critical regularity for homogeneous Sobolev spaces $\dot{H}^{s}(\RR^{n}) \times \dot{H}^{s-1}(\RR^{n})$. One of the earliest results on global existence and regularity for wave maps were given in $(1+1)$ dimensions by C.~H.~Gu \cite{MR596432} and J.~Ginibre and G.~Velo \cite{MR678488}. The Cauchy problem for $(1+1)$-dimensional wave maps was also studied by M.~Keel and T.~Tao \cite{MR1663216}, who proved local well-posedness for all data above scaling regularity $s > 1/2$ and global well-posedness for all data if $s \geq 1$, whereas at the scaling critical regularity $s=1/2$ well-posedness fails because of discontinuity of the data-to-solution map \cite{MR1759884}. Wave maps in $(1+2)$ dimensions were extensively studied in corotational symmetry by J.~Shatah and A.~S.~Tahvildar-Zadeh \cite{MR1168115}, \cite{MR1278351} for targets that satisfy a geodesic convexity condition, as well as in radial symmetry for parallelizable targets by D.~Christodoulou and A.~S.~Tahvildar-Zadeh \cite{MR1230285}, \cite{MR1223662} and later by M.~Struwe \cite{MR1985457}, \cite{MR1971037} for general targets. In those settings, global existence of smooth solutions was established for smooth data of any size and the asymptotic behavior of those solutions was described. The work \cite{MR1168115} also gives a local well-posedness result for corotational wave maps at optimal regularity for all $n \geq 2$. The corotational results on blowup from that time are commented on in \Cref{SubSecSing}. We refer to J.~Shatah's and M.~Struwe's textbook \cite{MR1674843} for further notes on early results.
\medskip\par
For the small data global well-posedness theory without symmetry restrictions, several advanced techniques were developed. Building on their works \cite{MR1231427}, \cite{MR1446618}, S.~Klainerman together with S.~Selberg \cite{MR1452172} and M.~Machedon \cite{MR1381973} exploited the ``null structure'' of the wave maps nonlinearity to prove local well-posedness in Sobolev spaces for $s > n/2$ in all dimensions $n \geq 2$, In this context, the difficulties that arise on the way from local to global well-posedness at optimal regularity are known as the ``division problem''. This was solved in the seminal works \cite{MR1641721}, \cite{MR1827277} by D.~Tataru and resulted in global well-posedness for small data in scaling critical homogeneous Besov spaces, also see \cite{MR3883339}. The result in Sobolev spaces has to be addressed in view of the ill-posedness result \cite{MR2020108} due to failure of uniform continuity of the solution operator in scaling critical spaces. Passing from Besov spaces to Sobolev spaces leads to what is called the ``summation problem'', which was solved in a weak global existence result in scaling critical homogeneous Sobolev space for small data wave maps with spherical target by T.~Tao in \cite{MR1820329} for $n \geq 4$ and in \cite{MR1869874} for $n=2,3$. The underlying techniques from harmonic analysis were further improved to cover more general target manifolds. Subsequently, global existence for small critical data was obtained for bounded parallelizable targets in dimensions $n\geq 5$ by S.~Klainerman and I.~Rodnianski \cite{MR1843256} and for compact Lie groups and symmetric spaces in dimensions $n \geq 4$ by A.~Nahmod, A.~Stefanov and K.~Uhlenbeck \cite{MR2016196}. The low-dimensional cases $n=2,3$ are considered to be more difficult and were treated by J.~Krieger \cite{MR1990880}, \cite{MR2094472} for two-dimensional targets. A systematically different approach was pursued by J.~Shatah and M.~Struwe \cite{MR1890048} in dimensions $n \geq 4$ for targets with bounded curvature. The most comprehensive small data global well-posedness and regularity result was given by D.~Tataru \cite{MR2130618} for target manifolds which admit a uniform isometric embedding into Euclidean space.
\medskip\par
This leads to the investigation of large data wave maps and the question whether they are global or blow up in finite time. In energy-supercritical dimensions, it has been known from the early theory of corotational wave maps that the Cauchy problem admits for several target geometries finite time blowup in the form of smooth self-similar solutions. In the energy-critical dimension $n=2$, it was known that a singularity can only form if the energy of the solution concentrates in a backward light cone at the origin. In particular, this scenario rules out self-similar blowup for energy-critical wave maps. M.~Struwe \cite{MR1990477} used this to show that if a corotational wave map develops a singularity, it converges up to a rescaling to a static wave map, i.e. a harmonic map. This indicates that the respective blowup mechanism for energy-critical and supercritical wave maps is substantially different. At that time, however, the existence of blowup solutions in $(1+2)$ dimensions had been observed only numerically for corotational wave maps into the $2$-sphere, see e.g. the work \cite{MR1862811} of P.~Bizo\'{n}, T.~Chmaj and Z.~Tabor. The first rigorous construction of blowup for energy-critical wave maps into the $2$-sphere is due to I.~Rodnianski and J.~Sterbenz \cite{MR2680419} for high equivariance classes. Their construction uses modulation of the harmonic ground state and provides a stable blowup mechanism. Parallel to this work, J.~Krieger, W.~Schlag and D.~Tataru \cite{MR2372807} constructed corotational wave maps into the $2$-sphere which blow up in finite time at a prescribed polynomial blowup rate, also see \cite{MR3359541}. Stability of the latter solutions for sufficiently small rates was established in corotational symmetry by J.~Krieger together with S.~Miao \cite{MR4065147} and recently together with W.~Schlag \cite{2020arXiv200908843K} for general smooth perturbations. A construction of stable blowup solutions for all equivariance classes along with a description of the asymptotics is given in the work \cite{MR2929728} of P.~Rapha\"{e}l and I.~Rodnianski.
\medskip\par
In $(1+2)$ dimensions, the energy provides a conserved quantity that controls the scaling critical Sobolev norm and thus makes a large data global well-posedness theory accessible. Major contributions to global existence, regularity and scattering for large data energy-critical wave maps are due to T.~Tao in a series of preprints \cite{2008arXiv0805.4666T}, \cite{2008arXiv0806.3592T}, \cite{2008arXiv0808.0368T}, \cite{2009arXiv0906.2833T}, \cite{2009arXiv0908.0776T} and J.~Krieger and W.~Schlag \cite{MR2895939} in cases where the target has negative sectional curvature. The latter employed the method of ``concentration compactness'' and obtained from this also asymptotic bounds and scattering. The large data theory for energy-critical wave maps with compact target was established by J.~Sterbenz and D.~Tataru \cite{MR2657817}, \cite{MR2657818}. Their result proves a dichotomy between global regularity and scattering on the one hand, and on the other hand a blowup scenario, where the solution converges up to the symmetries of Minkwoski spacetime to a nontrivial finite energy harmonic map. More precisely, in the case where the compact target manifold does not admit a nontrivial harmonic map with finite energy, the solution is global and scatters for all data in the energy space. In the other case, global regularity and scattering holds for all data with energy below the minimal energy among all nontrivial harmonic maps. This settles the \emph{threshold conjecture} for wave maps. A refined version of the threshold theorem for wave maps with topological degree zero can be found in A.~Lawrie's and S.-J.~Oh's work \cite{MR3465437}.
\medskip\par
A characterization of the global dynamics for $(1+2)$-dimensional wave maps into the $2$-sphere are subject to the \emph{soliton resolution}, which asserts that any wave map in the energy space decomposes asymptotically into a finite sum of solitons and a radiation term. This decomposition along a sequence of times was proved by R.~C\^{o}te \cite{MR3403756} in the corotational case and by H.~Jia and C.~Kenig \cite{MR3730929} for arbitrary equivariance classes. For energies that rule out multi soliton configurations, R.~C\^{o}te, C.~Kenig, A.~Lawrie and W.~Schlag \cite{MR3318089}, \cite{MR3318090} proved the resolution for continuous times. Examples where a single soliton configuration is attained are provided in the fundamental works on the construction of blowup. For the asymptotic configuration of two solitons, constructions and a full classification of the possible dynamics of equivariant wave maps at the energy threshold is provided in the works of J.~Jendrej and A.~Lawrie \cite{MR3842064}, \cite{MR4409881}, \cite{MR4589375}, \cite{MR4599919}, the corotational case is treated by C.~Rodriguez \cite{MR4353567}. They prove that solutions with twice the energy of the harmonic ground state either scatter in both time directions or converge to a two-bubble state in one time direction and scatter in the other. Recently, the soliton resolution conjecture has been proved for corotational wave maps by T.~Duyckaerts, C.~Kenig, Y.~Martel and F.~Merle \cite{MR4397184}, who used the ``channel of energy'' method, and independently by J.~Jendrej and A.~Lawrie \cite{2021arXiv210610738J} for all higher equivariance classes. Versions of the soliton resolution without symmetry restrictions are available in \cite{MR3878592} and \cite{MR3627409}. More detailed explanations, overviews and cross references on the soliton resolution conjecture can be found in the cited works.
\medskip\par
To close our discussion on wave maps, we come back to results in energy-supercritical dimensions $d \geq 3$. B.~Dodson and A.~Lawrie \cite{MR3401013} proved for wave maps into the $3$-sphere a scattering result in the scaling critical norm, which rules out the existence of type II blowup. The existence of finite time blowup in the form of self-similar solutions as well as their stability theory have already been discussed in \Cref{SubSecSing}. Further existence results of self-similar solutions are given in \cite{MR2314330} and \cite{MR2450171}, \cite{MR2494812}. Global solutions into the $3$-sphere with infinite scaling critical Sobolev norm have been constructed in the work \cite{MR3592527} by J.~Krieger and E.~Chiodaroli. For energy-supercritical wave maps, blowup is also possible for negatively curved target geometries \cite{MR1622539}. In this situation, the first author and I.~Glogi\'{c} \cite{MR3861895} have constructed stable self-similar blowup in dimensions $n \geq 8$. A new ``type II'' blowup mechanism has been discovered by T.~Ghoul, S.~Ibrahim and V.~T.~Nguyen \cite{MR3812220} for all $n \geq 7$, which demonstrates that blowup is not necessarily always self-similar. Lastly, we mention that the threshold between blowup and scattering in dimensions $n = 3,4,5,6$ is numerically studied in \cite{MR3636308} and conjectured to be given by the first excited state of the countable family of self-similar solutions.
\subsubsection{Results for the Yang-Mills equation}
\label{SubSubYM}
In their work \cite{MR65437}, C.~N.~Yang and R.~L.~Mills developed the foundations of a gauge theory for the strong nuclear interaction. The mathematical theory is implemented as a non-Abelian gauge theory on principal fiber bundles over spacetime and can be viewed as a nonlinear generalization of Maxwell's theory of the electromagnetic interaction. The geometric setup on principal fiber bundles is treated e.g. in the textbook \cite{MR3837560}.
\medskip\par
As common, we give a more pragmatic setup here. Let $G$ be a matrix group and $\mathfrak{g}$ be its Lie algebra with Lie bracket $[A,B] = AB - BA $. We also assume that there is an inner product $\langle \,.\,, \,.. \rangle$ on $\mathfrak{g}$ which is invariant under the adjoint action of $G$ by conjugation. For $\mu = 0,1,\ldots,n$, let $A_{\mu} : \RR^{1,n} \rightarrow \mathfrak{g}$ be Lie algebra-valued connection $1$-forms and define by
\begin{equation*}
\operatorname{D}_{\mu} S = \pd_{\mu} S + [A_{\mu},S]
\end{equation*}
a metric-compatible connection for sections $S: \RR^{1,n} \rightarrow \mathfrak{g}$ over the vector bundle $\RR^{1,n} \times \mathfrak{g}$. The local curvature $2$-form $F_{\mu\nu} : \RR^{1,n} \rightarrow \mathfrak{g}$ associated to this connection is defined by
\begin{equation*}
\operatorname{D}_{\mu} \operatorname{D}_{\nu} S - \operatorname{D}_{\nu} \operatorname{D}_{\mu} S = [F_{\mu\nu}, S]
\end{equation*}
and given by
\begin{equation*}
F_{\mu\nu} = \pd_{\mu} A_{\nu} - \pd_{\nu} A_{\mu} - [A_{\mu},A_{\nu}] \,.
\end{equation*}
We say that $A_{\mu} : \RR^{1,n} \rightarrow \mathfrak{g}$ is a Yang-Mills connection if it is a critical point of the action functional
\begin{equation*}
S_{\mathsf{YM}}[A] = \frac{1}{2} \int_{\RR^{1,n}} \langle F_{\mu\nu}, F^{\mu\nu} \rangle \,.
\end{equation*}
The associated Euler-Lagrange equations are the Yang-Mills equation
\begin{align*}
0 &= \operatorname{D}^{\mu} F_{\mu\nu} \\&=
\pd^{\mu} \pd_{\mu} A_{\nu} - \pd_{\nu} ( \pd^{\mu} A_{\mu} ) + \pd^{\mu} [A_{\mu}, A_{\nu}] + [A^{\mu}, \pd_{\mu} A_{\nu}] - [A^{\mu}, \pd_{\nu} A_{\mu}] + [ A^{\mu}, [A_{\mu},A_{\nu}] ]
\end{align*}
and they constitute a system of scaling invariant nonlinear wave equations. This system possesses a gauge symmetry. That is, let $O : \RR^{1,n} \rightarrow G$ and consider the gauge transformation
\begin{equation*}
A_{\mu} \mapsto \overline{A}_{\mu} \coloneqq O A_{\mu} O^{-1} - \pd_{\mu} O O^{-1} \,.
\end{equation*}
The curvature $2$-form associated to the connection $1$-forms $\overline{A}_{\mu}$ is then given by $\overline{F}_{\mu\nu} = O F_{\mu\nu} O^{-1}$. Since the inner product on $\mathfrak{g}$ is invariant under conjugation, the Yang-Mills action functional is invariant under gauge transformations. Thus, if $A_{\mu}$ is a Yang-Mills connection then $\overline{A}_{\mu}$ is again a Yang-Mills connection, which also follows from the computation $\overline{\operatorname{D}}^{\mu} \overline{F}_{\mu\nu} = O ( \operatorname{D}^{\mu} F_{\mu\nu} ) O^{-1}$. This implies for the Cauchy problem that certain constraint equations on $A_{\mu}$ have to be imposed to make it well-posed. Examples of gauges that have been used in the study of the Yang-Mills equation are the
\begin{align*}
\renewcommand{\arraystretch}{1.2}
\begin{array}{lr}
\text{temporal gauge} & A_{0} = 0 \,, \\
\text{Coulomb gauge} & \pd^{i} A_{i} = 0 \,, \\
\text{Lorenz gauge} & \pd^{\mu} A_{\mu} = 0 \,,
\end{array}
\end{align*}
see the references for more details. We also note that $s = n/2 - 1$ is the scaling critical Sobolev index.
\medskip\par
An early local well-posedness result in Sobolev spaces for Yang-Mills connections in $(1+3)$ dimensions in the temporal gauge was obtained by I.~Segal \cite{MR546505}. This local existence result was extended to all space dimensions for the more general Yang-Mills-Higgs equation by J.~Ginibre and G.~Velo in the temporal gauge \cite{MR638511} and Lorenz gauge \cite{MR653018}. For $n=1,2$ they also obtained global solutions. The classical global well-posedness and regularity result for Yang-Mills-Higgs in $(1+3)$-dimensions was established by D.~M.~Eardley and V.~Moncrief \cite{MR649158}, \cite{MR649159} in the temporal gauge.
\medskip\par
In $(1+3)$ dimensions, local and global well-posedness for data with finite energy is due to S.~Klainerman and M.~Machedon \cite{MR1338675}. Their approach combined a temporal gauge and a local Coloumb gauge to reveal a null structure in the equations. Low regularity local well-posedness was first proved in \cite{MR1964470} in the temporal gauge for small data and in \cite{MR3519539}, \cite{MR3385623} in the Lorenz gauge. Low regularity global well-posedness in the temporal gauge is considered in the recent work \cite{2022arXiv220801503G}. Finite energy global well-posedness of the Yang-Mills system in $(1+3)$-dimensions was revisited in a new approach via the Yang-Mills heat flow by S.-J.~Oh \cite{MR3190112}, \cite{MR3357182}.
\medskip\par
In $(1+4)$ dimensions, S.~Klainerman and D.~Tataru \cite{MR1626261} proved small data local well-posedness in Sobolev spaces strictly above scaling for a model equation. The global problem was treated by J.~Sterbenz \cite{MR2325100} in Besov spaces. Small data global well-posedness in even dimensions $n \geq 6$ for scaling critical gauge covariant Sobolev norms was proved by J.~Krieger and J.~Sterbenz in \cite{MR3087010}. Based on progress for the related Maxwell-Klein-Gordon equation, small energy global well-posedness for the Yang-Mills equation in $(1+4)$ dimensions in the Coulomb gauge was solved by J.~Krieger and D.~Tataru \cite{MR3664812}. The works \cite{MR3907955}, \cite{MR4113787} of S.-J.~Oh and D.~Tataru on the threshold conjecture and dichotomy theorem also contain local well-posedness at optimal regularity in the temporal gauge in any dimension $n \geq 4$.
\medskip\par
In the energy-critical dimension $n=4$, static energy minimizing solutions to the Yang-Mills equation exist. They are called ``instanton'' and play a key role in the description of singularity formation. Numerics on blowup were carried out in the equivariant case by P.~Bizo\'{n} et al.~in \cite{MR1877844}, \cite{MR1923684}, \cite{MR2069700}. They observed that singularities form via a collapse to the rescaled profile of the instanton and conjectured that large solutions exhibit this universal blowup behavior. The rigorous construction of equivariant blowup solutions is due to J.~Krieger, W.~Schlag and D.~Tataru \cite{MR2522426}. Stable equivariant blowup has been constructed by P.~Rapha\"{e}l and I.~Rodnianski in \cite{MR2929728}.
\medskip\par
The large data theory without symmetry restrictions in the energy-critical dimension $n=4$ has been significantly advanced in the line of works \cite{MR4518477}, \cite{MR4113787}, \cite{MR3907955}, \cite{MR4298746} by S.-J.~Oh and D.~Tataru. Their technique is based on the Yang-Mills heat flow and constructs a global ``caloric'' gauge. The first main result of their work is a proof of the threshold conjecture, which asserts global well-posedness and scattering for data with energy below twice the instanton. This was shown before in the equivariant case in \cite{MR2443303}. The second part of their work establishes a dichotomy theorem, which states that local solutions either extend globally and scatter, or bubble off a soliton. An introduction and overview of the precise results is given in the article \cite{MR3923343}.
\subsection{Notation}
Throughout this paper, we employ Einstein summation convention.
\subsubsection{Sets and relations}
The sets of natural, real and complex numbers are denoted by $\NN$, $\RR$, $\CC$, respectively. The $d$-dimensional Euclidean space is denoted by $\RR^{d}$ and $| \,.\, |$ is the Euclidean norm. The open ball in $\RR^{d}$ of radius $R>0$ is given by $\BB^{d}_{R} = \{ x \in \RR^{d} \mid |x| < R \}$. The $(1+d)$-dimensional Minkowski spacetime is denoted by $\RR^{1,d}$ and $\eta = \operatorname{diag}(-1,1,\ldots,1)$ is the Minkowski metric. For $\omega \in \RR$, the set $\mathbb{H}_{\omega} = \{ z \in \CC \mid \Re(z) > \omega \}$ is an open right-half plane and $\overline{\mathbb{H}}_{\omega} = \{ z \in \CC \mid \Re(z) \geq \omega \}$ is a closed right-half plane in the complex plane.
\medskip\par
If $a_{\lambda}, b_{\lambda} \geq 0$ are nonnegative real numbers which are indexed by a parameter $\lambda \in \Lambda$, we define a relation $a_{\lambda} \lesssim b_{\lambda}$ if there exists a constant $C > 0$ such that the inequality $a_{\lambda} \leq C b_{\lambda}$ holds uniformly for all $\lambda \in \Lambda$. As usual, 
\begin{equation*}
a_{\lambda} \gtrsim b_{\lambda} \quad :\Leftrightarrow \quad b_{\lambda} \lesssim a_{\lambda}
\qquad\text{and}\qquad
a_{\lambda} \simeq b_{\lambda} \quad :\Leftrightarrow \quad
a_{\lambda} \lesssim b_{\lambda} \, \text{ and } \, a_{\lambda} \gtrsim b_{\lambda} \,.
\end{equation*}
\subsubsection{Derivatives and function spaces} The Jacobian matrix of a map $f: \RR^{m} \rightarrow \RR^{n}$ is denoted by $\pd f$ and has entries $(\pd f)^{j}{}_{i} = \pd_{i} f^{j}$ for $i=1,\ldots,m$ and $j=1,\ldots,n$. In particular, for $f:\RR^{d}\rightarrow\RR$, we have
\begin{equation*}
\text{the gradient } \pd f \,, \qquad
\text{the Hessian } \pd^{2} f \,, \qquad
\text{the Laplacian } \Delta f = \pd^{i} \pd_{i} f \,.
\end{equation*}
For a function $u:\RR^{1,d} \rightarrow \RR$ on Minkowski spacetime, the partial derivatives are denoted by $\pd_{\mu} u$ for $\mu = 0,1,\ldots,d$, where $(\pd_{0} u)(t,x) = \pd_{t} u(t,x)$. We frequently consider the space $C^{\infty} ( \overline{\BB^{d}_{R}} )$ which consists of all smooth functions whose partial derivatives of all orders are bounded. The corresponding subspace of radially symmetric functions is given by
\begin{equation*}
C^{\infty}_{\mathrm{rad}}(\overline{\BB^{d}_{R}}) = \{ f \in C^{\infty} ( \overline{\BB^{d}_{R}} ) \mid f \text{ is radially symmetric} \}
\end{equation*}
and the space of smooth even functions on the line is given by
\begin{equation*}
C^{\infty}_{\mathrm{ev}}([0,\infty)) = \{ f\in C^{\infty}([0,\infty)) \mid \forall\,n\in\NN: f^{(2n-1)}(0) = 0 \} \,.
\end{equation*}
We recall from \cite{2022arXiv220902286O} that for any $f \in C^{\infty}_{\mathrm{rad}}(\overline{\BB^{d}_{R}})$ there is a radial profile $\widetilde{f} \in C^{\infty}_{\mathrm{ev}}([0,R])$ such that $f(x) = \widetilde{f}(|x|)$. For integers $k\in \NN_{0}$, we define the Sobolev norm for smooth functions $f \in C^{\infty} ( \overline{\BB^{d}_{R}} )$ by
\begin{equation*}
\| f \|_{H^{k}(\BB^{d}_{R})} = \Big( \sum_{0\leq|\alpha|\leq k} \|\pd^{\alpha} f \|_{L^{2}(\BB^{d}_{R})}^{2} \Big)^{\frac{1}{2}} \,.
\end{equation*}
The space $C^{\infty} ( \overline{\BB^{d}_{R}} )$ endowed with the Sobolev norm $\| \,.\, \|_{H^{k}(\BB^{d}_{R})}$ is a normed space whose completion is the Sobolev space $H^{k}(\BB^{d}_{R})$. The respective radial subspace $H^{k}_{\mathrm{rad}}(\BB^{d}_{R})$ is obtained analogously.
\subsubsection{Operators}
For Hilbert spaces $H_{1}$, $H_{2}$, we denote their product Hilbert space by $\mathfrak{H} = H_{1} \times H_{2}$ and use boldface notation for elements
\begin{equation*}
\mathbf{f} =
\begin{bmatrix}
f_{1} \\
f_{2}
\end{bmatrix}
\in \mathfrak{H} \,,
\qquad\text{where}\qquad
f_{1} \in H_{1} \,, \quad
f_{2} \in H_{2} \,.
\end{equation*}
Linear operators on product spaces are also denoted in boldface notation $\mathbf{L} : \mathfrak{D}(\mathbf{L}) \subset \mathfrak{H} \rightarrow \mathfrak{H}$, where the domain $\mathfrak{D}(\mathbf{L})$ is a linear subspace of $\mathfrak{H}$. The space of bounded linear operators on $\mathfrak{H}$ is denoted by $\mathfrak{L}(\mathfrak{H})$ and endowed with the operator norm. If $\mathbf{L} : \mathfrak{D}(\mathbf{L}) \subset \mathfrak{H} \rightarrow \mathfrak{H}$ is a closed linear operator, its resolvent set and spectrum are given by
\begin{align*}
\varrho(\mathbf{L}) &= \{ z \in \mathbb{C} \mid z\mathbf{I} - \mathbf{L} : \mathfrak{D}(\mathbf{L}) \subset \mathfrak{H} \rightarrow \mathfrak{H} \text{ is a bijection} \} \,, \\
\sigma(\mathbf{L}) &= \CC \setminus \varrho(\mathbf{L}) \,,
\end{align*}
respectively. By the closed graph theorem, $\mathbf{R}_{\mathbf{L}}(z) \coloneqq (z\mathbf{I} - \mathbf{L})^{-1}$ is a bounded linear operator on $\mathfrak{H}$ for all $z \in \varrho(\mathbf{L})$, which defines the resolvent map $\mathbf{R}_{\mathbf{L}} : \varrho(\mathbf{L}) \rightarrow \mathfrak{L}(\mathfrak{H})$.
\section{The stability problem in corotational symmetry}
\label{SecStabProblem}
Here, we describe the symmetry-reduced blowup stability problem for the wave maps equation \eqref{WaveMapsEquation} and the Yang-Mills equation \eqref{YangMillsEquation}. Under the corotational ansatz \eqref{CorotWM}, \eqref{EquivarYM}, the respective models in $\RR^{1,n}$ reduce to a radial semilinear wave equation
\begin{equation}
\label{RadialSemilinearWaveEquation}
\big( -\pd_{t}^{2} + \Delta_{x} \big) \psi(t,x) + F\big( x, \psi(t,x) \big) = 0
\end{equation}
in $\RR^{1,d}$, where $d = n+2$ and $\psi: \RR^{1,d} \rightarrow \RR$ is a radial function in the spatial variable. The main properties of the nonlinearity are captured as follows.
\begin{enumerate}[itemsep=1em,topsep=1em,label=(F\arabic*)]
\item\label{F1} The map $F: \RR^{d} \times \RR \rightarrow \RR$ is given by
\begin{equation}
\label{WaveMapsYangMillsNonlinearity}
F ( x, z ) =
\renewcommand{\arraystretch}{1.5}
\left\{
\begin{array}{ll}
\displaystyle{
- \frac{d-3}{2} \frac{\sin\big( 2 |x| z \big) - 2 |x| z}{|x|^{3}}
} & \text{for wave maps,}
\\
\big( d-4 \big) \big( 3 z^{2} - |x|^{2} z^{3} \big) & \text{for Yang-Mills,}
\end{array}
\right.
\end{equation}
and is smooth and radially symmetric on $\RR^{d}$.
\item\label{F2} The scaling relation
\begin{equation*}
F \big( x, \lambda^{-s} z \big) = \lambda^{-s-2} F \big( \tfrac{x}{\lambda},z \big)
\end{equation*}
holds for all $\lambda>0$ and all $(x,z) \in \RR^{d} \times \RR$, where
\begin{equation*}
s = 1 \quad \text{for wave maps}
\qquad\text{and}\qquad
s = 2 \quad \text{for Yang-Mills.}
\end{equation*}
This implies that the semilinear wave equation \eqref{RadialSemilinearWaveEquation} is invariant under the scaling transformation
\begin{equation*}
\psi \mapsto \psi^{\lambda} \,, \qquad \text{where} \qquad \psi^{\lambda}(t,x) = \lambda^{-s} \psi\big( \tfrac{t}{\lambda}, \tfrac{x}{\lambda} \big)
\end{equation*}
for all $\lambda>0$.
\item\label{F3} For all energy-supercritical dimensions $d \in \NN$, i.e.
\begin{equation*}
\frac{d}{2} - s - 1 > 0 \,,
\end{equation*}
the semilinear wave equation \eqref{RadialSemilinearWaveEquation} admits the solution
\begin{align}
\nonumber
\psi^{\ast}(t,x) &\coloneqq \widetilde{\psi}^{\ast}(t,|x|) \\\label{SelfSimilarSolutions}&\coloneqq
\renewcommand{\arraystretch}{1.5}
\left\{
\begin{array}{ll}
\displaystyle{
\frac{4}{|x|} \arctan\Big( \frac{|x|}{ \sqrt{d-4} t + \sqrt{(d-4)t^{2} + |x|^{2}} } \Big)
} & \text{for wave maps,} \\
\displaystyle{
\frac{1}{\alpha_{d-2} t^{2} + \beta_{d-2} |x|^{2}}
} & \text{for Yang-Mills,}
\end{array}
\right.
\end{align}
in $\RR^{1,d} \setminus \{(0,0)\}$. The respective solution is smooth away from the origin, radially symmetric and self-similar, i.e. invariant under the scaling transformation in \ref{F2}. By time translation invariance of \Cref{RadialSemilinearWaveEquation}, the profile $\psi^{\ast}$ gives rise to a family of finite-time blowup solutions
\begin{equation}
\label{FiniteTimeBlowup}
\psi^{\ast}_{T}(t,x) \coloneqq \widetilde{\psi}^{\ast}_{T}(t,|x|) \coloneqq \psi^{\ast}(T-t,x) \,, \qquad T \in \RR \,.
\end{equation}
\end{enumerate}
The stable Cauchy evolution for \Cref{RadialSemilinearWaveEquation} near the self-similar blowup profile in regions beyond light cones is the content of the following theorem, for a discussion see \Cref{RemarkDiscussionContent,RemarkDiscussionMethod}.
\begin{theorem}
\label{THM}
Let $d,k\in\NN$ with
\begin{equation*}
k \geq \frac{d}{2} > 2 \quad \text{for wave maps}
\qquad\text{and}\qquad
k \geq \frac{d-1}{2} > \frac{5}{2} \quad \text{for Yang-Mills}.
\end{equation*}
Let $-1 \leq \kappa < 1$ and consider the Cauchy problem
\begin{equation*}
\label{StabilityProblem}
\renewcommand{\arraystretch}{1.2}
\left\{
\begin{array}{rcll}
0 &=& \big( -\pd_{t}^{2} + \Delta_{x} \big) \psi(t,x) + F \big( x, \psi(t,x) \big) & \text{in } \Omega^{1,d}_{T}(\kappa) \,, \\
\psi(0,x) &=& \psi^{\ast}_{1}(0,x) + f(x) & \text{in } \RR^{d} \,, \\
(\pd_{0} \psi)(0,x) &=& (\pd_{0}\psi^{\ast}_{1})(0,x) + g(x) & \text{in } \RR^{d} \,,
\end{array}
\right.
\end{equation*}
for the nonlinearity $F: \RR \times \RR^{d} \rightarrow \RR$ from \ref{F1} and the respective blowup profile $\psi^{\ast}$ from \ref{F3} in the spacetime region
\begin{equation*}
\Omega^{1,d}_{T}(\kappa) = \big\{ (t,x) \in \RR^{1,d} \mid 0 < t < T + \kappa |x| \big\} \,.
\end{equation*}
Fix $r > 0$, pick similarity coordinates as in \Cref{SimilarityCoordinatesFlat} and recall from \Cref{RemarkImageRegion,RemarkSHSF,RemarkFlatRegion}
\begin{equation*}
\Omega^{1,d}_{T}(\kappa) = \mathrm{X}^{1,d}_{T}(\kappa) \cup \Lambda^{1,d}_{T}(r)
\qquad\text{with}\qquad
\mathrm{X}^{1,d}_{T}(\kappa) = \bigcup_{t \in (0,T)} \Sigma^{1,d}_{T}(t) \,.
\end{equation*}
Then, there exist constants $\delta^{\ast} > 0$, $C^{\ast} \geq 1$, $\omega > 0$ such that for all $0 < \delta \leq \delta^{\ast}$, $C \geq C^{\ast}$, and all initial data $f,g \in C^{\infty}_{\mathrm{rad}}(\RR^{d})$ with
\begin{equation*}
\operatorname{supp} (f,g) \subset \BB^{d}_{r}
\qquad\text{and}\qquad
\| (f,g) \|_{H^{k}(\RR^{d}) \times H^{k-1}(\RR^{d})} \leq \frac{\delta}{C^{2}} \,,
\end{equation*}
there exist a $T^{\ast} \in [1-\frac{\delta}{C},1+\frac{\delta}{C}]$ and a unique solution $\psi \in C^{\infty} \big( \Omega^{1,d}_{T^{\ast}}(\kappa) \big)$ to the Cauchy problem which satisfies the stability estimates
\begin{align*}
\frac{\big\| \psi - \psi^{\ast}_{T^{\ast}} \big\|_{H^{k}( \Sigma^{1,d}_{T^{\ast}}(t) )}}{\big\| \psi^{\ast}_{T^{\ast}} \big\|_{H^{k}( \Sigma^{1,d}_{T^{\ast}}(t) )}} &\lesssim \delta \Big( \frac{T^{\ast}-t}{T^{\ast}} \Big)^{\omega} \,, \\
\frac{\big\| \pd_{0} \psi - \pd_{0}\psi^{\ast}_{T^{\ast}} \big\|_{H^{k-1}( \Sigma^{1,d}_{T^{\ast}}(t) )}}{\big\| \pd_{0}\psi^{\ast}_{T^{\ast}} \big\|_{H^{k-1}(\Sigma^{1,d}_{T^{\ast}}(t))}} &\lesssim \delta \Big( \frac{T^{\ast}-t}{T^{\ast}} \Big)^{\omega} \,,
\end{align*}
for all $0 \leq t < T^{\ast}$. If $k > \frac{d}{2}$, then additionally
\begin{equation*}
\sup_{(t',x') \in \Sigma^{1,d}_{T^{\ast}}(t)} \frac{|\psi(t',x') - \psi^{\ast}_{T^{\ast}}(t',x')|}{|\psi^{\ast}_{T^{\ast}}(t',x')|} \lesssim \delta \Big( \frac{T^{\ast}-t}{T^{\ast}} \Big)^{\omega}
\end{equation*}
for all $0\leq t < T^{\ast}$. Furthermore, the solution $\psi$ is radially symmetric in the spatial variable and there holds finite speed of propagation
\begin{equation*}
\psi \equiv \psi^{\ast}_{1} \quad \text{in } \Lambda^{1,d}_{T^{\ast}}(r) \,.
\end{equation*}
\end{theorem}
\subsection{Formulation in similarity coordinates and outline of the paper}
\label{Outline}
We give an outline of our approach to prove \Cref{THM}. For any fixed $T > 0$, we reformulate the radial semilinear wave equation \eqref{RadialSemilinearWaveEquation} around the one-parameter family of self-similar blowup solutions $\psi^{\ast}_{T}$ by introducing a radially symmetric perturbation $u$ through
\begin{equation*}
\psi = \psi^{\ast}_{T} + u \,.
\end{equation*}
We denote $F'(x,z) \coloneqq \pd_{z} F(x,z)$ and get with
Taylor's theorem
\begin{equation*}
F \big( x, \psi^{\ast}_{T}(t,x) + u(t,x) \big) = F \big( x, \psi^{\ast}_{T}(t,x) \big) + \mathcal{V}_{T}(t,x) u(t,x) + \mathcal{N}_{T}(u)(t,x) \,,
\end{equation*}
where
\begin{align}
\label{SelfSimilarPotential}
\mathcal{V}_{T}(t,x) &= F'\big( x, \psi^{\ast}(T-t,x) \big) \,, \\
\label{NonlinearRemainder}
\mathcal{N}_{T}(u)(t,x) &= F \big( x, \psi^{\ast}_{T}(t,x) + u(t,x) \big) - F \big( x, \psi^{\ast}_{T}(t,x) \big) - F'\big( x, \psi^{\ast}_{T}(t,x) \big) u(t,x) \,,
\end{align}
is a self-similar potential and a corresponding nonlinear remainder, respectively. Then, \Cref{RadialSemilinearWaveEquation} takes the form
\begin{equation}
\label{DecompositionSLW}
0 = \big( -\pd_{t}^{2} + \Delta_{x} + \mathcal{V}_{T}(t,x) \big) u(t,x) + \mathcal{N}_{T}(u)(t,x) \,.
\end{equation}
In order to analyze this equation in regions beyond the blowup event, we employ the geometric setup of similarity coordinates from \Cref{AppendixSimilarityCoordinates}.
\begin{definition}
\label{SimilarityCoordinates}
Let $h \in C^{\infty}_{\mathrm{rad}}(\RR^{d})$ such that
\begin{equation*}
h(0) < 0 \,, \qquad
|(\pd h)(\xi)| < 1 \,, \qquad
(\pd^{2} h)(\xi) \text{ is positive semidefinite} \,,
\end{equation*}
for all $\xi \in \RR^{d}$. Let $R,T>0$. We define \emph{similarity coordinates}
\begin{equation*}
{\chi\mathstrut}_{T} : \RR \times \RR^{d} \rightarrow \RR^{1,d} \,, \qquad {\chi\mathstrut}_{T}(\tau, \xi) = \big( T + T \ee^{-\tau} h(\xi), T \ee^{-\tau} \xi \big) \,,
\end{equation*}
in the image region
\begin{equation*}
\mathrm{X}^{1,d}_{T,R} \coloneqq {\chi\mathstrut}_{T}\big( (0,\infty)\times \BB^{d}_{R} \big) \,.
\end{equation*}
\end{definition}
Next, we formulate \Cref{DecompositionSLW} as an autonomous first-order system in similarity coordinates. Therefore, we introduce
\begin{equation}
\label{EvoVar}
u_{1}(\tau,\xi) = (T\ee^{-\tau})^{s} ( u \circ{ \chi\mathstrut}_{T} )(\tau,\xi) \,, \qquad
u_{2}(\tau,\xi) = (T\ee^{-\tau})^{s+1} ( (\pd_{0} u) \circ{\chi\mathstrut}_{T} )(\tau,\xi) \,,
\end{equation}
with scaling parameters
\begin{equation*}
s=1 \quad \text{for wave maps} \qquad\text{and}\qquad s=2 \quad \text{for Yang-Mills}
\end{equation*}
and observe the identities
\begin{align}
\label{FreeEvoEq1}
\pd_{\tau} u_{1}(\tau,\xi) &= - s u_{1}(\tau,\xi) - \xi^{i} \pd_{i} u_{1}(\tau,\xi)
+ c(\xi) u_{2}(\tau,\xi) \,, \\
\label{FreeEvoEq2}
\pd_{\tau} u_{2}(\tau,\xi) &= \frac{c(\xi)}{w(\xi)} \pd^{i} \pd_{i} u_{1} (\tau,\xi)
- \Big( s + 1 + \frac{c(\xi)}{w(\xi)} \pd^{i} \pd_{i} h(\xi) \Big) u_{2}(\tau,\xi)
\\\nonumber&\indent- \Big( \xi^{i} + 2 \frac{c(\xi)}{w(\xi)} \pd^{i} h(\xi) \Big) \pd_{i} u_{2}(\tau,\xi)
- (T\ee^{-\tau})^{s+2} \frac{c(\xi)}{w(\xi)} ( \Box u ) \circ {\chi\mathstrut}_{T}(\tau,\xi) \,,
\end{align}
where the coefficients are specified as follows.
\begin{definition}
\label{hcw}
Let $h \in C^{\infty}_{\mathrm{rad}}(\RR^{d})$ be as in \Cref{SimilarityCoordinates}. We define $c,w \in C^{\infty}_{\mathrm{rad}}(\RR^{d})$ by
\begin{equation*}
c(\xi) = \xi^{i} (\pd_{i} h)(\xi) - h(\xi)
\qquad\text{and}\qquad
w(\xi) = 1-(\pd^{i} h)(\xi)(\pd_{i} h)(\xi) \,.
\end{equation*}
\end{definition}
The right-hand side of \Cref{FreeEvoEq1,FreeEvoEq2} governs the free dynamics of \Cref{DecompositionSLW} in similarity coordinates and is implemented in the following map.
\begin{definition}
\label{L0}
For $\mathbf{f} \in C^{\infty}_{\mathrm{rad}}(\overline{\BB^{d}_{R}})^{2}$, we define
\begin{equation*}
\mathbf{L}_{0} \mathbf{f} =
\renewcommand{\arraystretch}{1.5}
\begin{bmatrix}
- s f_{1} - \xi^{i} \pd_{i} f_{1}
+ c f_{2} \\
\displaystyle{
\frac{c}{w} \pd^{i} \pd_{i} f_{1}
- \Big( s + 1 + \frac{c}{w} \pd^{i} \pd_{i} h \Big) f_{2}
- \Big( \xi^{i} + 2 \frac{c}{w} \pd^{i} h \Big) \pd_{i} f_{2}
}
\end{bmatrix}
\in C^{\infty}_{\mathrm{rad}}(\overline{\BB^{d}_{R}})^{2} \,.
\end{equation*}
\end{definition}
The potential $\mathcal{V}_{T}$ from \Cref{SelfSimilarPotential} and nonlinear remainder $\mathcal{N}_{T}$ from \Cref{NonlinearRemainder} are encoded correspondingly.
\begin{definition}
\label{PotentialNonlinearity}
Let $V \in C^{\infty}_{\mathrm{rad}}(\overline{\BB^{d}_{R}})$ and $N : C^{\infty}_{\mathrm{rad}}(\overline{\BB^{d}_{R}}) \rightarrow C^{\infty}_{\mathrm{rad}}(\overline{\BB^{d}_{R}})$ be given by
\begin{align*}
V(\xi) &= F' \big( \xi, \psi^{\ast}(-h(\xi),\xi) \big) \,, \\
N(f)(\xi) &= F \big( \xi, \psi^{\ast}(-h(\xi),\xi) + f(\xi) \big) - F'\big( \xi, \psi^{\ast}(-h(\xi),\xi) \big) f(\xi) - F \big( \xi, \psi^{\ast}(-h(\xi),\xi) \big) \,.
\end{align*}
For $\mathbf{f} \in C^{\infty}_{\mathrm{rad}}(\overline{\BB^{d}_{R}})^{2}$, we define
\begin{align*}
\mathbf{L}_{V}' \mathbf{f}
&=
\begin{bmatrix}
0 \\
\displaystyle{
\frac{c}{w} V f_{1}
}
\end{bmatrix}
\in C^{\infty}_{\mathrm{rad}}(\overline{\BB^{d}_{R}})^{2} \,, \\
\mathbf{N}(\mathbf{f}) &=
\begin{bmatrix}
0 \\
\displaystyle{
\frac{c}{w} N(f_{1})
}
\end{bmatrix}
\in C^{\infty}_{\mathrm{rad}}(\overline{\BB^{d}_{R}})^{2} \,.
\end{align*}
\end{definition}
These maps provide for \Cref{DecompositionSLW} a transition relation to a formulation as first-order nonlinear time evolution equation.
\begin{lemma}
\label{NonlinearTransitionRelation}
Let $u\in C^{\infty}(\overline{\Omega}^{1,d}_{T,R})$ be radial and define $\mathbf{u}
\in C^{\infty}(\overline{ (0,\infty) \times \BB^{d}_{R} })^{2}$ with components
\begin{equation*}
u_{1}(\tau,\xi) = (T\ee^{-\tau})^{s} (u\circ{\chi\mathstrut}_{T})(\tau,\xi) \,, \qquad
u_{2}(\tau,\xi) = (T\ee^{-\tau})^{s+1} ( (\pd_{0} u) \circ{\chi\mathstrut}_{T} ) (\tau,\xi) \,.
\end{equation*}
Then
\begin{equation*}
\pd_{\tau}
\mathbf{u}(\tau,\,.\,)
=
( \mathbf{L}_{0} + \mathbf{L}_{V}' )
\mathbf{u}(\tau,\,.\,)
+
\mathbf{N}\big(
\mathbf{u}(\tau,\,.\,) \big)
-
\begin{bmatrix}
0 \\
\displaystyle{
(T\ee^{-\tau})^{s+2} \frac{c}{w} \big( (\Box u + \mathcal{V}_{T} u + \mathcal{N}_{T}(u)) \circ {\chi\mathstrut}_{T} \big)(\tau,\,.\,)
}
\end{bmatrix}
\,.
\end{equation*}
\end{lemma}
\begin{proof}
The scaling relation \ref{F2} of the nonlinearity implies
\begin{equation*}
F' \big( x, \lambda^{-s} z \big) = \lambda^{-2} F' \big( \tfrac{x}{\lambda},z \big) \,.
\end{equation*}
So we obtain together with self-similarity of the blowup solution in \ref{F3} the scaling properties
\begin{align*}
\big( T\ee^{-\tau} \big)^{s+2} \big( \mathcal{V}_{T} \circ {\chi\mathstrut}_{T} \big)(\tau,\xi) &= \big( T\ee^{-\tau} \big)^{s} V(\xi) \,, \\
\big( T\ee^{-\tau} \big)^{s+2} \big( \mathcal{N}_{T}(u) \circ {\chi\mathstrut}_{T} \big) (\tau,\xi) &= N ( u_{1}(\tau,\,.\,) )(\xi) \,,
\end{align*}
for the potential $\mathcal{V}_{T}$ from \Cref{SelfSimilarPotential} and the nonlinear remainder $\mathcal{N}_{T}$ from \Cref{NonlinearRemainder}, respectively. An application of the identities \eqref{FreeEvoEq1}, \eqref{FreeEvoEq2} yields the transition relation, the underlying computation is carried out in \Cref{TransitionRelation}.
\end{proof}
So in similarity coordinates, our stability problem for the corotational wave maps and equivariant Yang-Mills equation is posed as the abstract nonlinear evolution equation
\begin{equation}
\label{AbstractNonlinearEvolutionEquation}
\pd_{\tau} \mathbf{u}(\tau) = \mathbf{L} \mathbf{u}(\tau) + \mathbf{N} ( \mathbf{u}(\tau) ) \,,
\end{equation}
where the dynamics are comprised of
\begin{equation*}
\renewcommand{\arraystretch}{1.2}
\begin{array}{ll}
\text{a linearized part} & \mathbf{L} \coloneqq \mathbf{L}_{0} + \mathbf{L}_{V}' \,, \\
\text{a nonlinear part} & \mathbf{N} \,.
\end{array}
\end{equation*}
We give a qualitative overview of the further proceeding and contents of this paper.
\subsubsection{Remarks on the analysis of the free wave flow}
\label{SubSubSecNRG}
The linearized part is arranged as a perturbation of $\mathbf{L}_{0}$. Therefore, it is feasible to base the functional analytic framework for the stability problem on the free wave flow in similarity coordinates. For this, we employ the theory of strongly continuous semigroups, particularly the Lumer-Phillips generation theorem. This requires us to realize the map $\mathbf{L}_{0}$ as a densely defined dissipative operator on a Hilbert space. However, semigroup theory alone does not provide a compatible Hilbert space structure. Nevertheless, it is instructive to note that the dissipative estimate in the assumptions of the Lumer-Phillips theorem dictates the exponential growth estimate for the semigroup in its conclusion. This observation suggests that energy estimates for the wave operator in similarity coordinates can inform a proper choice of inner products that are tailored to the dissipative properties of $\mathbf{L}_{0}$. In \Cref{SubSecEnTop}, we manage to construct an energy that exhibits such estimates. With these ideas, a general theory of the wave evolution in similarity coordinates is then developed in \Cref{AppendixFreeWave}. The main result in \Cref{TheGenerationTHM} gives the generation of the free wave flow in similarity coordinates together with sharp growth bounds.
\subsubsection{Remarks on the analysis of the linearized wave flow}
\label{SubSubSecLinearization}
In \Cref{SecLinWaveFlow}, we turn to the well-posedness of the linearized equation associated to the abstract Cauchy problem. Based on \Cref{TheGenerationTHM} and perturbation theory for semigroups, we infer in \Cref{LinearFlow} directly that $\mathbf{L}$ generates in any radial Sobolev space $\mathfrak{H}^{k}_{\mathrm{rad}}(\BB^{d}_{R}) \coloneqq H^{k}_{\mathrm{rad}}(\BB^{d}_{R}) \times H^{k-1}_{\mathrm{rad}}(\BB^{d}_{R})$ a strongly continuous semigroup $\mathbf{S}: [0,\infty) \rightarrow \mathfrak{L}(\mathfrak{H}^{k}_{\mathrm{rad}}(\BB^{d}_{R}))$ which propagates this flow. Ultimately, it is the stable part of this semigroup which accounts for the stability of the family of self-similar blowup solutions. Appropriate notions of stability are provided in terms of the growth bound of a semigroup, which is tied to spectral properties of its generator via spectral mapping theorems. Accordingly, we are led to characterize the unstable spectrum of $\mathbf{L}$. By exploiting coordinate invariance of the linearized operator $\Box + \mathcal{V}_{T}$, we draw an equivalence of this spectral problem to the classical \emph{mode stability problem}. Fortunately, this difficult nonself-adjoint spectral problem has been solved for the respective blowup solutions in the works \cite{MR3623242}, \cite{MR4469070} in the corotational case. With this, we prove that the unstable spectrum of $\mathbf{L}$ consists precisely of the eigenvalue $1 \in \sigma( \mathbf{L} )$, which is induced by time translation invariance of \Cref{RadialSemilinearWaveEquation}. Moreover, this eigenvalue is isolated and simple with associated Riesz projection $\mathbf{P}_{1}$. In \Cref{LinearizedEvolution}, we pin down the unstable part of the linearized wave evolution exactly to the one-dimensional eigenspace of the eigenvalue $1 \in \sigma( \mathbf{L} )$ by showing
\begin{equation*}
\mathbf{S}(\tau) = \mathbf{S}(\tau) ( \mathbf{I} - \mathbf{P}_{1} ) + \ee^{\tau} \mathbf{P}_{1} \,, \qquad \| \mathbf{S}(\tau) ( \mathbf{I} - \mathbf{P}_{1} ) \|_{\mathfrak{L}(\mathfrak{H}^{k}(\BB^{d}_{R}))} \lesssim \ee^{-\omega\tau} \,,
\end{equation*}
for all $\tau \geq 0$. This gives a complete dynamical description of the linearized wave flow around the self-similar blowup solution in spacetime regions even beyond the blowup event.
\subsubsection{Remarks on the analysis of the nonlinear wave flow}
In \Cref{SecNonlin}, we prove mapping properties of the nonlinearity $\mathbf{N}$ in order to give a mild formulation of the abstract nonlinear evolution equation \eqref{AbstractNonlinearEvolutionEquation} with Duhamel's principle, i.e.
\begin{equation*}
\mathbf{u}(\tau) = \mathbf{S}(\tau) \mathbf{u}(0) + \int_{0}^{\tau} \mathbf{S}(\tau-\tau') \mathbf{N} ( \mathbf{u}(\tau') ) \dd\tau' \,.
\end{equation*}
Via a Lyapunov-Perron-type argument, we carry the decaying part of the linear flow over to the nonlinear flow for modified initial data. Note that by definition of the variable $\mathbf{u} = (u_{1}, u_{2})$, the initial datum $\mathbf{u}(0)$ is obtained from an evaluation process along the initial hypersurface $\Sigma^{1,d}_{T}(0)$. Owed to property \ref{h1} of the height function, we have $\{ 0 \} \times \BB^{d}_{r} \subset \Sigma^{1,d}_{T}(0)$. Hence the spacetime region, where we perturb the blowup solution $\psi^{\ast}_{1}$, is contained in the initial hypersurface. Consequently, we obtain together with finite speed of propagation for semilinear wave equations the correct data for the abstract stability problem. In this way, the blowup time $T>0$ enters our analysis through the initial datum. By adjusting this blowup time appropriately, we construct from the modified solution a stable mild solution to the original Cauchy problem in the image region of similarity coordinates. Mapping properties of the semigroup $\mathbf{S}$ allow us to upgrade mild solutions to jointly smooth classical solutions of \Cref{AbstractNonlinearEvolutionEquation}. In \Cref{SecProofMainResults}, an application of this stability theory and the radial lemmas for corotational Sobolev norms proved in \cite{2022arXiv220902286O} ultimately yield \Cref{THMWM,THMYM}.
\subsection{Energy topologies}
\label{SubSecEnTop}
To derive an energy as described in \Cref{SubSubSecNRG}, let us fix $R,T > 0$ and recall from \Cref{GeometryFoliation,GSCFoliation} that the image region $\mathrm{X}^{1,d}_{T,R}$ of similarity coordinates ${\chi\mathstrut}_{T} : \RR \times \RR^{d} \rightarrow \RR^{1,d}$ introduced in \Cref{SimilarityCoordinates} is foliated by the spacelike hypersurfaces
\begin{equation*}
\Sigma^{1,d}_{T,R}(\tau) \coloneqq {\chi\mathstrut}_{T} \big( \{\tau\}\times\BB^{d}_{R} \big) \,, \qquad \tau > 0 \,,
\end{equation*}
and the components of the future-directed timelike unit normal vector field along $\Sigma^{1,d}_{T,R}(\tau)$ are given by
\begin{equation}
\label{FDUNVF}
N^{0} \circ {\chi\mathstrut}_{T}(\tau,\,.\,) = \frac{1}{\sqrt{w}} \,, \qquad
N^{i} \circ {\chi\mathstrut}_{T}(\tau,\,.\,) = \frac{\pd^{i}h}{\sqrt{w}} \,.
\end{equation}
The integral along $\Sigma^{1,d}_{T,R}(\tau)$ of a continuous and bounded function $\varphi: \Sigma^{1,d}_{T,R}(\tau) \rightarrow \RR$ is well-defined by
\begin{equation}
\label{IntegralHypersurface}
\int_{\Sigma^{1,d}_{T,R}(\tau)} \varphi \coloneqq \int_{\BB^{d}_{R}} \big( \varphi\circ{\chi\mathstrut}_{T}(\tau,\,.\,) \big) \sqrt{\det{\gamma\mathstrut}_{T}(\tau)} \,,
\end{equation}
where $\det{\gamma\mathstrut}_{T}(\tau) = \big( T\ee^{-\tau} \big)^{2d} w$ is the determinant of the first fundamental form of $\Sigma^{1,d}_{T,R}(\tau)$. With this, we define the \emph{energy} along the hypersurface $\Sigma^{1,d}_{T,R}(\tau)$ for real-valued maps $u \in C^{\infty}\big( \overline{\Omega}^{1,d}_{T,R} \big)$ as
\begin{equation*}
E_{0}[u](\tau) = \int_{\Sigma^{1,d}_{T,R}(\tau)} \mathrm{T}_{0\mu} N^{\mu} \,,
\end{equation*}
where
\begin{equation}
\label{NRGMOM}
\mathrm{T}_{\mu\nu} = (\pd_{\mu} u) (\pd_{\nu} u) - \frac{1}{2} (\pd^{\alpha} u) (\pd_{\alpha} u) \eta_{\mu\nu}
\end{equation}
are the components of the energy-momentum tensor associated to the wave operator of $u \in C^{\infty}\big( \overline{\Omega}^{1,d}_{T,R} \big)$ in standard Minkowski coordinates, see \cite[p.~147]{MR4182428}. In terms of the rescaled evolution variables
\begin{equation*}
u_{1}(\tau,\xi) = (T\ee^{-\tau})^{s} ( u \circ{ \chi\mathstrut}_{T} )(\tau,\xi) \,, \qquad
u_{2}(\tau,\xi) = (T\ee^{-\tau})^{s+1} ( (\pd_{0} u) \circ{\chi\mathstrut}_{T} )(\tau,\xi) \,,
\end{equation*}
we find
\begin{equation}
\label{NRG}
E_{0}[u](\tau) = (T\ee^{-\tau})^{d-2s-2} \frac{1}{2} \int_{\BB^{d}_{R}} \Big( |\pd u_{1}(\tau,\,.\,)|^{2} + |u_{2}(\tau,\,.\,)|^{2} w \Big) \,.
\end{equation}
We abbreviate $u_{1} \equiv u_{1}(\tau,\,.\,)$, $u_{2} \equiv u_{2}(\tau,\,.\,)$ and ${\chi\mathstrut}_{T} \equiv {\chi\mathstrut}_{T}(\tau,\,.\,)$ as well as $\xi$ for the identity function on $\BB^{d}_{R}$, recall from \Cref{FreeEvoEq1,FreeEvoEq2} the identities
\begin{align*}
\pd_{\tau} u_{1} &= - s u_{1} - \xi^{i} \pd_{i} u_{1}
+ c u_{2} \,, \\
\pd_{\tau} u_{2} &= \frac{c}{w} \pd^{i} \pd_{i} u_{1}
- \Big( s + 1 + \frac{c}{w} \pd^{i} \pd_{i} h \Big) u_{2}
- \Big( \xi^{i} + 2 \frac{c}{w} \pd^{i} h \Big) \pd_{i} u_{2} - (T\ee^{-\tau})^{s+2} \frac{c}{w} (\Box u) \circ {\chi\mathstrut}_{T}	 \,,
\end{align*}
and compute
\begin{align}
\label{DerivativeEnergy}
\frac{1}{2} \pd_{\tau} \Big( |\pd u_{1}|^{2} + |u_{2}|^{2} w \Big) &= \frac{1}{2} (d-2s+2) \Big( |\pd u_{1}|^{2} + |u_{2}|^{2} w \Big) \\\nonumber&\indent- (T\ee^{-\tau})^{s+2} \big( (\Box u) \circ {\chi\mathstrut}_{T} \big) (cu_{2}) \\\nonumber&\indent+
\pd_{i} \Big( - \frac{1}{2} \xi^{i} \big( |\pd u_{1}|^{2} + |u_{2}|^{2} w \big) - (\pd^{i} h) c |u_{2}|^{2} + (\pd^{i} u_{1}) ( c u_{2} ) \Big) \,.
\end{align}
Integration by parts yields
\begin{align*}
\frac{\mathrm{d}}{\mathrm{d}\tau} E_{0}[u](\tau) &= \frac{(T\ee^{-\tau})^{d-2s-2}}{R} \int_{\pd\BB^{d}_{R}} (\xi^{i} \pd_{i} u_{1}) (c u_{2}) - \frac{(T\ee^{-\tau})^{d-2s-2}}{R} \int_{\pd\BB^{d}_{R}} \frac{1}{2} \Big( |\xi|^{2} |\pd u_{1}|^{2} + |cu_{2}|^{2} \Big) \\&\indent+
(T\ee^{-\tau})^{d-2s-2} \frac{\widetilde{h}(R)^{2} - R^{2}}{2R} \int_{\pd\BB^{d}_{R}} |u_{2}|^{2} -
(T\ee^{-\tau})^{d-s} \int_{\BB^{d}_{R}} \big( (\Box u) \circ {\chi\mathstrut}_{T} \big) \big( c u_{2} \big) \,.
\end{align*}
By the Cauchy-Schwarz inequality, the sum of the boundary integrals in the first line is nonpositive and by \Cref{GSCLightCone} we have $\widetilde{h}(\rho)^{2} - \rho^{2} \leq 0$ for all $\rho\geq R_{0}$. Thus, if we fix $R\geq R_{0}$ then
\begin{equation*}
\Box u = 0 \quad \text{in } \mathrm{X}^{1,d}_{T,R}
\qquad\text{implies}\qquad
E_{0}[u](\tau) \leq E_{0}[u](0) \,.
\end{equation*}
However, the energy $E_{0}[u](\tau)$ along a hypersurface does not constitute a norm for $u\in C^{\infty}(\overline{\Omega}^{1,d}_{T,R})$. It turns out that adding to the energy \eqref{NRG} for some fixed $0 < \varepsilon_{1} \leq \frac{d}{2} - 1$ the boundary term
\begin{equation*}
E_{\pd}[u](\tau) = (T\ee^{-\tau})^{d-2s-2} \frac{\varepsilon_{1}}{R} \int_{\pd\BB^{d}_{R}} |u_{1}|^{2}
\end{equation*}
achieves control on a Sobolev norm without spoiling the previous energy estimate. Indeed, put
\begin{equation}
\label{NRGNorm}
E[u](\tau) \coloneqq E_{0}[u](\tau) + E_{\pd}[u](\tau) \,.
\end{equation}
Then, the trace lemma \cite[Lemma 2.1]{MR4778061} yields
\begin{equation*}
E[u](\tau) \simeq \ee^{ -2\big( \frac{d}{2} - s - 1 \big) \tau} \big\| \big( u_{1}(\tau,\,.\,), u_{2}(\tau,\,.\,) \big) \big\|_{H^{1}(\BB^{d}_{R}) \times L^{2}(\BB^{d}_{R})}^{2}
\end{equation*}
for all $u\in C^{\infty}\big( \overline{\Omega}^{1,d}_{T,R} \big)$ and all $\tau \geq 0$. Moreover, we find for the derivative of the boundary term
\begin{equation*}
\frac{\mathrm{d}}{\mathrm{d}\tau} E_{\pd}[u](\tau) = -2\Big( \frac{d}{2} - 1 \Big) E_{\pd}[u](\tau) +
(T\ee^{-\tau})^{d-2s-2} \frac{2 \varepsilon_{1}}{R} \int_{\pd\BB^{d}_{R}} \Big( - (\xi^{i} \pd_{i} u_{1}) u_{1} + c u_{2} u_{1} \Big) \,,
\end{equation*}
so that altogether
\begin{align*}
\frac{\mathrm{d}}{\mathrm{d}\tau} E[u](\tau) &= -2\Big( \frac{d}{2} - 1 \Big) E_{\pd}[u](\tau) \\&\indent+
\frac{(T\ee^{-\tau})^{d-2s-2}}{R} \int_{\pd\BB^{d}_{R}} \Big( (\xi^{i} \pd_{i} u_{1}) (c u_{2}) - (\xi^{i} \pd_{i} u_{1}) (2 \varepsilon_{1} u_{1}) + (c u_{2}) (2 \varepsilon_{1} u_{1}) \Big) \\&\indent- \frac{(T\ee^{-\tau})^{d-2s-2}}{2R} \int_{\pd\BB^{d}_{R}} \Big( |\xi|^{2} |\pd u_{1}|^{2} + |cu_{2}|^{2} \Big) \\&\indent+
(T\ee^{-\tau})^{d-2s-2} \frac{\widetilde{h}(R)^{2} - R^{2}}{2R} \int_{\pd\BB^{d}_{R}} |u_{2}|^{2} -
(T\ee^{-\tau})^{d-s} \int_{\BB^{d}_{R}} \big( (\Box u) \circ {\chi\mathstrut}_{T} \big) \big( c u_{2} \big) \,.
\end{align*}
In the second line above
\begin{equation*}
(\xi^{i} \pd_{i} u_{1}) (c u_{2}) - (\xi^{i} \pd_{i} u_{1}) (2 \varepsilon_{1} u_{1}) + (c u_{2}) (2 \varepsilon_{1} u_{1}) \leq \frac{1}{2} \Big( |\xi^{i} \pd_{i} u_{1}|^{2} + |c u_{2}|^{2} + |2 \varepsilon_{1} u_{1}|^{2} \Big) \,,
\end{equation*}
and now $R\geq R_{0}$ implies
\begin{equation*}
\frac{\mathrm{d}}{\mathrm{d}\tau} E[u](\tau) \leq -2\Big( \frac{d}{2} - 1 - \varepsilon_{1} \Big) E_{\pd}[u](\tau) +
(T\ee^{-\tau})^{d-s} \int_{\BB^{d}_{R}} \big( (\Box u) \circ {\chi\mathstrut}_{T} \big) \big( c u_{2} \big) \,.
\end{equation*}
So, if $\Box u = 0$ in $\mathrm{X}^{1,d}_{T,R}$, we conclude the estimate
\begin{equation}
\label{NormEstimate}
\big\| \big( u_{1}(\tau,\,.\,), u_{2}(\tau,\,.\,) \big) \big\|_{H^{1}(\BB^{d}_{R}) \times L^{2}(\BB^{d}_{R})} \lesssim \ee^{ \big( \frac{d}{2} - s - 1 \big) \tau} \big\| \big( u_{1}(0,\,.\,), u_{2}(0,\,.\,) \big) \big\|_{H^{1}(\BB^{d}_{R}) \times L^{2}(\BB^{d}_{R})} 
\end{equation}
for all $\tau \geq 0$. However, because the stability problem in \Cref{THM} is energy-supercritical, i.e. $\frac{d}{2} - s - 1 > 0$, we need an exponential growth estimate better than in \eqref{NormEstimate} for a viable perturbation theory around self-similar blowup. To achieve this, scaling effects suggest to pass on to higher Sobolev norms. More concretely, the idea is to come up with differential operators that satisfy commutation relations with the wave operator in similarity coordinates and use them to construct higher energies which control higher-order Sobolev norms and exhibit improved growth bounds. These constructions are delicate and therefore carried out in detail in the \Cref{SubSecCommDiffOp,SubSecIP}.
\section{Existence and stability of the linearized flow in similarity coordinates}
\label{SecLinWaveFlow}
In this section, we prove the existence of the linearized corotational wave maps and equivariant Yang-Mills flow in similarity coordinates and determine its stability.
\subsection{Generation of the linearized flow}
As explained in \Cref{Outline}, the linear theory is built on the theory of the free wave flow in similarity coordinates from \Cref{AppendixFreeWave}. We begin with introducing the underlying function spaces.
\begin{definition}
\label{FunctionSpaces}
Let $d,k\in\NN$ and $R>0$. We define the Hilbert space
\begin{equation*}
\mathfrak{H}^{k}_{\mathrm{rad}}(\BB^{d}_{R}) = H^{k}_{\mathrm{rad}}(\BB^{d}_{R}) \times H^{k-1}_{\mathrm{rad}}(\BB^{d}_{R}) \,.
\end{equation*}
\end{definition}
In the following, we implicitly fix $h \in C^{\infty}_{\mathrm{rad}}(\RR^{d})$ as in \Cref{SimilarityCoordinates} and let $R_{0} > 0$ be determined as in \Cref{GSCLightCone}. We also tacitly impose case distinctions through the parameters
\begin{equation*}
s=1 \quad \text{for wave maps} \qquad\text{and}\qquad s=2 \quad \text{for Yang-Mills} \,.
\end{equation*}
With this, we obtain from our generation theorem \ref{TheGenerationTHM} immediately a semigroup and generator for the free dynamics.
\begin{proposition}
\label{SGRadialWaveFlow}
Let $d,k\in\NN$ and $R \geq R_{0}$. Then, the linear map $\mathbf{L}_{0} : C^{\infty}_{\mathrm{rad}}(\overline{\BB^{d}_{R}})^{2} \rightarrow C^{\infty}_{\mathrm{rad}}(\overline{\BB^{d}_{R}})^{2}$ given in \Cref{L0} has an extension to a densely defined closed operator $\mathbf{L}_{0} : \mathfrak{D}(\mathbf{L}_{0}) \subset \mathfrak{H}^{k}_{\mathrm{rad}}(\BB^{d}_{R}) \rightarrow \mathfrak{H}^{k}_{\mathrm{rad}}(\BB^{d}_{R})$
which is the generator of a strongly continuous semigroup $\mathbf{S}_{0}: [0,\infty) \rightarrow \mathfrak{L}( \mathfrak{H}^{k}_{\mathrm{rad}}(\BB^{d}_{R}))$ with the property that for any $0 < \varepsilon < \frac{1}{2}$ there exists a constant $M_{d,k,R,h,\varepsilon} \geq 1$ such that the growth estimate
\begin{equation*}
\| \mathbf{S}_{0}(\tau) \mathbf{f} \|_{ \mathfrak{H}^{k}(\BB^{d}_{R}) } \leq M_{d,k,R,h,\varepsilon} \ee^{\omega_{d,k,s,\varepsilon} \tau} \| \mathbf{f} \|_{ \mathfrak{H}^{k}(\BB^{d}_{R}) } \,, \qquad
\omega_{d,k,s,\varepsilon} = \max\Big\{ \frac{d}{2} - s - k, -s + \varepsilon \Big\} \,,
\end{equation*}
holds for all $\mathbf{f} \in \mathfrak{H}^{k}_{\mathrm{rad}}(\BB^{d}_{R})$ and all $\tau\geq 0$.
\end{proposition}
\begin{proof}
Let $\overline{\mathbf{L}_{\chi} {\mathord{\restriction}}} :
\mathfrak{D}(\overline{\mathbf{L}_{\chi} {\mathord{\restriction}}}) \subset \mathfrak{H}^{k}_{\mathrm{rad}}(\BB^{d}_{R}) \rightarrow \mathfrak{H}^{k}_{\mathrm{rad}}(\BB^{d}_{R})$ be the generator of the semigroup $\mathbf{S} {\mathord{\restriction}} : [0,\infty) \rightarrow \mathfrak{L}(\mathfrak{H}^{k}_{\mathrm{rad}}(\BB^{d}_{R}))$ from \cref{itemrad} of \Cref{TheGenerationTHM}. Then, the operator
\begin{equation*}
\mathbf{L}_{0} \coloneqq \overline{\mathbf{L}_{\chi} {\mathord{\restriction}}} - s \mathbf{I} \,, \qquad
\mathfrak{D}(\mathbf{L}_{0}) = \mathfrak{D}(\overline{\mathbf{L}_{\chi} {\mathord{\restriction}}}) \,,
\end{equation*}
is the generator of the rescaled semigroup given by $\mathbf{S}_{0}(\tau) \coloneqq \ee^{-s \tau} \mathbf{S} {\mathord{\restriction}}(\tau)$ which exhibits the asserted growth estimate. This operator coincides on $C^{\infty}_{\mathrm{rad}}(\overline{\BB^{d}_{R}})^{2}$ with the map from \Cref{L0}.
\end{proof}
The potential which arises from the linearization around the self-similar blowup solution is realized as a compact linear operator in radial Sobolev spaces.
\begin{lemma}
\label{BoundedCompactPotential}
Let $d,k\in \NN$ with $\frac{d}{2} - s > 1$ and $R > 0$. Then, the linear map $\mathbf{L}_{V}': C^{\infty}_{\mathrm{rad}}(\overline{\BB^{d}_{R}})^{2} \rightarrow C^{\infty}_{\mathrm{rad}}(\overline{\BB^{d}_{R}})^{2}$ given in \Cref{PotentialNonlinearity} has a unique extension to a bounded linear operator $\mathbf{L}_{V}' \in \mathfrak{L} ( \mathfrak{H}^{k}_{\mathrm{rad}}(\BB^{d}_{R}) )$ and this operator is compact.
\end{lemma}
\begin{proof}
The map $\mathbf{L}_{V}'$ from \Cref{PotentialNonlinearity} represents a densely defined bounded linear operator on $\mathfrak{H}^{k}_{\mathrm{rad}}(\BB^{d}_{R})$. Hence, there is a unique extension $\mathbf{L}_{V}' \in \mathfrak{L} ( \mathfrak{H}^{k}_{\mathrm{rad}}(\BB^{d}_{R}) )$. Compactness of the extended operator follows from compactness of the embedding $H^{k}_{\mathrm{rad}}(\BB^{d}_{R}) \hookrightarrow H^{k-1}_{\mathrm{rad}}(\BB^{d}_{R})$, see e.g. \cite[proof of Lemma 3.1]{MR4778061} for more details.
\end{proof}
The generator of the linearized dynamics is implemented in this functional analytic setting as follows.
\begin{definition}
\label{GeneratorLinearizedWave}
Let $d,k\in\NN$ with $\frac{d}{2} - s > 1$ and $R \geq R_{0}$. We define the linear operator $\mathbf{L}: \mathfrak{D}(\mathbf{L}) \subset \mathfrak{H}^{k}_{\mathrm{rad}}(\BB^{d}_{R}) \rightarrow \mathfrak{H}^{k}_{\mathrm{rad}}(\BB^{d}_{R})$ by
\begin{equation*}
\renewcommand{\arraystretch}{1.5}
\mathbf{L} = \mathbf{L}_{0} + \mathbf{L}_{V}' \,, \qquad \mathfrak{D}(\mathbf{L}) = \mathfrak{D}(\mathbf{L}_{0}) \,.
\end{equation*}
\end{definition}
This setup allows for a short proof of the existence of the linearized flow in similarity coordinates.
\begin{proposition}
\label{LinearFlow}
Let $d,k\in\NN$ with $\frac{d}{2} - s > 1$ and $R \geq R_{0}$. Then, the operator $\mathbf{L}: \mathfrak{D}(\mathbf{L}) \subset \mathfrak{H}^{k}_{\mathrm{rad}}(\BB^{d}_{R}) \rightarrow \mathfrak{H}^{k}_{\mathrm{rad}}(\BB^{d}_{R})$ is the generator of a strongly continuous semigroup $\mathbf{S}: [0,\infty) \rightarrow \mathfrak{L}( \mathfrak{H}^{k}_{\mathrm{rad}}(\BB^{d}_{R}) )$ with the following properties.
\begin{enumerate}[itemsep=1em,topsep=1em]
\item For any $0 < \varepsilon < \frac{1}{2}$ and for the constants $M_{d,k,R,h,\varepsilon} \geq 1$ and $\omega_{d,k,s,\varepsilon}$ from \Cref{SGRadialWaveFlow}, the estimate
\begin{equation*}
\| \mathbf{S}(\tau) \mathbf{f} \|_{ \mathfrak{H}^{k}(\BB^{d}_{R}) } \leq M_{d,k,R,h,\varepsilon} \ee^{\big( \omega_{d,k,s,\varepsilon} + M_{d,k,R,h,\varepsilon} \| \mathbf{L}_{V}' \|_{\mathfrak{L} ( \mathfrak{H}^{k}(\BB^{d}_{R}) )} \big) \tau} \| \mathbf{f} \|_{ \mathfrak{H}^{k}(\BB^{d}_{R}) }
\end{equation*}
holds for all $\mathbf{f} \in \mathfrak{H}^{k}_{\mathrm{rad}}(\BB^{d}_{R})$ and all $\tau\geq 0$.
\item For any $0 < \varepsilon < \frac{1}{2}$, the set $\sigma(\mathbf{L}) \cap \overline{\mathbb{H}}_{\omega_{d,k,s,\varepsilon}}$ consists of finitely many isolated eigenvalues of $\mathbf{L}$ with finite algebraic multiplicity.
\end{enumerate}
\end{proposition}
\begin{proof}
By \Cref{BoundedCompactPotential}, the operator $\mathbf{L} = \mathbf{L}_{0} + \mathbf{L}_{V}'$ is a bounded perturbation of the generator of the semigroup from \Cref{SGRadialWaveFlow}. Hence, the Bounded Perturbation Theorem \cite[p.~158]{MR1721989} yields the first item. Since $\mathbf{L}_{V}'$ is compact, we can apply the spectral mapping theorem \cite[Theorem B.1]{MR4469070} and infer the second item of the proposition.
\end{proof}
\subsection{Spectral analysis of the generator}
The growth estimate in \Cref{LinearFlow} contains no information about the stable and unstable dynamics in the linearized evolution. Together with the spectral mapping theorem \cite[Theorem B.1]{MR4469070}, however, the second item states that any growth additional to the one described in \Cref{SGRadialWaveFlow} is caused by finitely many eigenvalues in a right half-plane. So the next goal is to characterize this part of the spectrum of the generator. For this, we prepare some technical lemmas for which we recall from \Cref{AssumptionsAppendix} graphical similarity coordinates
\begin{equation*}
\chi : \RR \times \RR^{d} \rightarrow \RR^{1,d} \,, \qquad \chi(\tau,\xi) = \big( \ee^{-\tau} h(\xi), \ee^{-\tau} \xi \big) \,,
\end{equation*}
and from \Cref{GSCLaplaceBeltrami} the Laplace-Beltrami operator
\begin{equation*}
(\Box_{\chi} v) (\tau,\xi) =
\ee^{2\tau} \Big( c^{00}(\xi) \pd_{\tau}^{2} + c^{i0}(\xi)\pd_{\xi^{i}}\pd_{\tau} + c^{ij}(\xi)\pd_{\xi^{i}}\pd_{\xi^{j}} + c^0(\xi)\pd_{\tau} + c^{i}(\xi)\pd_{\xi^{i}} \Big) v(\tau,\xi) \,.
\end{equation*}
The next lemma links the generator of the linearized flow to the wave equation with self-similar potential
\begin{equation*}
\mathcal{V}(t,x) \coloneqq \mathcal{V}_{0}(t,x) = F'\big(x,\psi^{\ast}(-t,x)\big)
\end{equation*}
via a convenient transition relation.
\begin{lemma}
\label{LinearTransitionRelation}
Let $v\in C^{\infty}(\overline{\RR \times \BB^{d}_{R}})$ be radial and define $\mathbf{u}
\in C^{\infty}(\overline{\RR \times \BB^{d}_{R}})^{2}$ with components
\begin{equation*}
u_{1}(\tau,\xi) = \ee^{-s\tau} v(\tau,\xi) \,, \qquad
u_{2}(\tau,\xi) = \ee^{-s\tau} \frac{1}{c(\xi)} ( \pd_{\tau} + \xi^{i} \pd_{\xi^{i}}) v(\tau,\xi) \,.
\end{equation*}
Then
\begin{equation*}
\pd_{\tau}
\mathbf{u}(\tau,\,.\,)
=
\mathbf{L}
\mathbf{u}(\tau,\,.\,)
-
\begin{bmatrix}
0 \\
\displaystyle{
\ee^{-(s+2)\tau} \frac{c}{w} \big( (\Box_{\chi} + \mathcal{V} \circ \chi)v \big)(\tau,\,.\,)
}
\end{bmatrix}
\,.
\end{equation*}
\end{lemma}
\begin{proof}
By the scaling property \ref{F2} of the nonlinearity, we have
\begin{equation*}
\ee^{-(s+2)\tau} ( \mathcal{V} \circ \chi )(\tau,\xi) = \ee^{-s\tau} V(\xi) \,.
\end{equation*}
Now, the transition relation follows from \Cref{TransitionRelation}.
\end{proof}
We give a criterion for the existence of spectral points in a right half-plane in terms of smooth \emph{mode solutions} to the wave equation with self-similar potential in similarity coordinates.
\begin{proposition}
\label{ModeCriterion}
Let $d,k \in \mathbb{N}$ with $1 < \frac{d}{2} - s < k$ and $R \geq R_{0}$. Let $\mathbf{L}: \mathfrak{D}(\mathbf{L}) \subset \mathfrak{H}^{k}_{\mathrm{rad}}(\BB^{d}_{R}) \rightarrow \mathfrak{H}^{k}_{\mathrm{rad}}(\BB^{d}_{R})$ be the operator from \Cref{GeneratorLinearizedWave}. Then, for any fixed $0 > \omega > \max\big\{ \frac{d}{2} - s - k, -s \big\}$ we have $\lambda \in \sigma(\mathbf{L}) \cap \overline{\mathbb{H}}_{\omega}$ if and only if there exists an $f_{\lambda} \in C^{\infty}_{\mathrm{rad}}(\overline{\BB^{d}_{R}})$ such that
\begin{equation*}
\renewcommand{\arraystretch}{1.5}
\left\{
\begin{array}{rcll}
\big( \Box_{\chi} + \mathcal{V} \circ \chi \big) v_{\lambda} &=& 0 & \text{in } \RR\times\BB^{d}_{R} \,, \\
v_{\lambda}(\tau,\xi) &=& \displaystyle{
\ee^{(\lambda+s)\tau} f_{\lambda}(\xi) \,.
} &
\end{array}
\right.
\end{equation*}
Moreover, every $\lambda \in \sigma(\mathbf{L}) \cap \overline{\mathbb{H}}_{\omega}$ is an isolated eigenvalue of $\mathbf{L}$ with finite algebraic multiplicity and with geometric eigenspace
\begin{equation*}
\ker(\lambda\mathbf{I}-\mathbf{L}) = \langle \mathbf{f}_{\lambda} \rangle
\qquad\text{spanned by}\qquad
\mathbf{f}_{\lambda} =
\begin{bmatrix}
f_{\lambda} \\
\displaystyle{
\frac{\lambda+s}{c} f_{\lambda} + \frac{1}{c} \xi^{i} \pd_{i} f_{\lambda}
}
\end{bmatrix}
\in C^{\infty}_{\mathrm{rad}}(\overline{\BB^{d}_{R}})^{2} \,.
\end{equation*}
\end{proposition}
\begin{proof}
If $\lambda \in \sigma(\mathbf{L}) \cap \overline{\mathbb{H}}_{\omega}$ then according to \Cref{LinearFlow}, there exists $\mathbf{f}_{\lambda} \in \mathfrak{D}(\mathbf{L}) \setminus \{ \mathbf{0} \}$ such that
\begin{equation}
\label{SpecEq}
( \lambda\mathbf{I} - \mathbf{L} ) \mathbf{f}_{\lambda} = \mathbf{0} \,.
\end{equation}
In a first step, we reduce this spectral equation to a differential equation. By definition of $\mathbf{L}$, there is a sequence of functions $\mathbf{f}_{\lambda,n} \in C^{\infty}_{\mathrm{rad}}(\overline{\BB^{d}_{R}})^{2}$ for $n\in \mathbb{N}$ such that
\begin{equation}
\label{ClosureImplication}
\lim_{n \to \infty} \| \mathbf{f}_{\lambda,n} - \mathbf{f}_{\lambda} \|_{\mathfrak{H}^{k}(\BB^{d}_{R})} = 0
\qquad\text{and}\qquad
\lim_{n \to \infty} \| ( \lambda\mathbf{I} - \mathbf{L} ) \mathbf{f}_{\lambda,n} \|_{\mathfrak{H}^{k}(\BB^{d}_{R})} = 0 \,.
\end{equation}
Now, we put $\mathbf{F}_{\lambda,n} \coloneqq ( \lambda\mathbf{I} - \mathbf{L} ) \mathbf{f}_{\lambda,n} \in C^{\infty}_{\mathrm{rad}}(\overline{\BB^{d}_{R}})^{2}$ and
\begin{equation}
\label{FlambdaComponents}
f_{\lambda,n,1} \coloneqq [\mathbf{f}_{\lambda,n}]_{1} \,, \qquad F_{\lambda,n,1} \coloneqq [\mathbf{F}_{\lambda,n}]_{1}\,, \qquad F_{\lambda,n,2} \coloneqq [\mathbf{F}_{\lambda,n}]_{2} \,.
\end{equation}
Then
\begin{equation*}
\mathbf{f}_{\lambda,n} =
\begin{bmatrix}
f_{\lambda,n,1} \\
\displaystyle{
\frac{\lambda+s}{c} f_{\lambda,n,1} + \frac{1}{c} \xi^{i} \pd_{i} f_{\lambda,n,1} - \frac{1}{c} F_{\lambda,n,1}
}
\end{bmatrix}
\end{equation*}
and we infer from \Cref{Range} the identity
\begin{equation}
\label{ApproximateMode}
\big( \Box_{\chi} + \mathcal{V} \circ \chi \big) v_{\lambda,n} = G_{\lambda,n} \quad \text{in } \RR\times\BB^{d}_{R} \,,
\end{equation}
where
\begin{equation*}
v_{\lambda,n}(\tau,\xi) = \ee^{(\lambda+s)\tau} f_{\lambda,n,1}(\xi) \,, \qquad G_{\lambda,n}(\tau,\xi) = \ee^{(\lambda+s+2)\tau} F_{\lambda,n}(\xi) \,,
\end{equation*}
and
\begin{align*}
F_{\lambda,n}(\xi) &= - \Big( (\lambda+s+1)\frac{w(\xi)}{c(\xi)^{2}} + \Big( \delta^{ij} - \frac{w(\xi)}{c(\xi)^{2}} \xi^{i}\xi^{j} - 2\frac{(\pd^{i} h)(\xi)\xi^{j}}{c(\xi)} \Big) \frac{(\pd_{i}\pd_{j} h)(\xi)}{c(\xi)} \Big) F_{\lambda,n,1}(\xi) \\&\indent-
\Big( \frac{w(\xi)}{c(\xi)^{2}}\xi^{i} + 2\frac{(\pd^{i} h)(\xi)}{c(\xi)} \Big) (\pd_{i} F_{\lambda,n,1})(\xi) - \frac{w(\xi)}{c(\xi)} F_{\lambda,n,2}(\xi) \Big) \,.
\end{align*}
In terms of the radial profile $\widetilde{f}_{\lambda,n,1} \in C^{\infty}_{\mathrm{ev}}([0,R])$ of $f_{\lambda,n,1} \in C^{\infty}_{\mathrm{rad}}(\overline{\BB^{d}_{R}})$, \Cref{ApproximateMode} reads
\begin{equation}
\label{ApproximateModeRadial}
a_{2}(\rho) \widetilde{f}_{\lambda,n,1}''(\rho) + a_{1}(\rho;\lambda+s,d) \widetilde{f}_{\lambda,n,1}'(\rho) + \big( a_{0}(\rho;\lambda+s,d) + \widetilde{V}(\rho) \big) \widetilde{f}_{\lambda,n,1}(\rho) = \widetilde{F}_{\lambda,n}(\rho)
\end{equation}
in $(0,R)$, where the coefficients $a_{2},a_{1},a_{0}$ are given by
\begin{align}
\label{a2Coeff}
a_{2}(\rho) &= - \frac{\rho^{2} - \widetilde{h}(\rho)^{2}}{\big( \rho \widetilde{h}'(\rho) - \widetilde{h}(\rho) \big)^{2}} \,, \\
\label{a1Coeff}
a_{1}(\rho;\lambda,d) &= 2\Big( \frac{d-3}{2} - \lambda \Big) \frac{\rho - \widetilde{h}(\rho)\widetilde{h}'(\rho)}{\big( \rho \widetilde{h}'(\rho) - \widetilde{h}(\rho) \big)^{2}} + \Big( \frac{d-1}{\rho} - \frac{\rho \widetilde{h}''(\rho)}{\rho \widetilde{h}'(\rho) - \widetilde{h}(\rho)} \Big) a_{2}(\rho) \,, \\
\label{a0Coeff}
a_{0}(\rho;\lambda,d) &= - \frac{\lambda (\lambda+1) (1-\widetilde{h}'(\rho)^{2})}{\big( \rho \widetilde{h}'(\rho) - \widetilde{h}(\rho) \big)^{2}} - \frac{\lambda(d-1)}{\rho \widetilde{h}'(\rho) - \widetilde{h}(\rho)}\frac{\widetilde{h}'(\rho)}{\rho} - \frac{\lambda \widetilde{h}''(\rho)}{\rho \widetilde{h}'(\rho)-\widetilde{h}(\rho)} a_{2}(\rho) \,,
\end{align}
the radial profile $\widetilde{F}_{\lambda,n} \in C^{\infty}_{\mathrm{ev}}([0,R])$ of the inhomogeneity $F_{\lambda,n} \in C^{\infty}_{\mathrm{rad}}(\overline{\BB^{d}_{R}})$ is given by
\begin{align*}
\widetilde{F}_{\lambda,n}(\rho) &= - \Big( (\lambda+s+1) \frac{\widetilde{w}(\rho)}{\widetilde{c}(\rho)^{2}} + \frac{d-1}{\widetilde{c}(\rho)} \frac{\widetilde{h}'(\rho)}{\rho} + \frac{\widetilde{h}(\rho)^{2} - \rho^{2}}{\widetilde{c}(\rho)^{3}} \widetilde{h}''(\rho) \Big) \widetilde{F}_{\lambda,n,1}(\rho) \\&\indent-\Big( 2\frac{\rho - \widetilde{h}(\rho)\widetilde{h}'(\rho)}{\widetilde{c}(\rho)^{2}} - \frac{\widetilde{w}(\rho)}{\widetilde{c}(\rho)^{2}} \rho \Big) \widetilde{F}'_{\lambda,n,1}(\rho) - \frac{\widetilde{w}(\rho)}{\widetilde{c}(\rho)} \widetilde{F}_{\lambda,n,2}(\rho) \,,
\end{align*}
and $\widetilde{V} \in C^{\infty}_{\mathrm{ev}}([0,R])$ is the radial profile of the potential $V \in C^{\infty}_{\mathrm{rad}}(\overline{\BB^{d}_{R}})$ from \Cref{PotentialNonlinearity}.
We use radial Sobolev embedding to get a nonzero function $\widetilde{f}_{\lambda}: (0,R] \rightarrow \CC$ with
\begin{equation}
\label{RadialSobolevImplication}
\widetilde{f}_{\lambda} \in C^{k-1}([\delta,R]) \,, \qquad
\widetilde{f}_{\lambda}^{(k-1)} \text{ is absolutely continuous on $[\delta,R]$,} \qquad
\widetilde{f}_{\lambda}^{(k)} \in L^{2}(\delta,R) \,,
\end{equation}
for any $0 < \delta < R$, such that
\begin{equation}
\label{ModeRadRep}
f_{\lambda} \coloneqq [\mathbf{f}_{\lambda}]_{1} = \widetilde{f}_{\lambda}(|\,.\,|) \,.
\end{equation}
Moreover, \eqref{ClosureImplication} and \eqref{RadialSobolevImplication} imply that $\widetilde{f}_{\lambda,n,1},\widetilde{f}_{\lambda,n,1}'$ and $\widetilde{F}_{\lambda,n}$ converge uniformly on $[\delta,R]$ to $\widetilde{f}_{\lambda}, \widetilde{f}_{\lambda}'$ and $0$, respectively. So we conclude from the weak formulation of \Cref{ApproximateModeRadial} that
\begin{equation}
\label{ModeLimitRadial}
a_{2}(\rho) \widetilde{f}_{\lambda}''(\rho) + a_{1}(\rho;\lambda+s,d) \widetilde{f}_{\lambda}'(\rho) + \big( a_{0}(\rho;\lambda+s,d) + \widetilde{V}(\rho) \big) \widetilde{f}_{\lambda}(\rho) = 0
\end{equation}
a.e. in $(0,R)$. In a second step, we investigate the regularity of this solution and show that $f_{\lambda}$ belongs to $C^{\infty}_{\mathrm{rad}}(\overline{\BB^{d}_{R}})$. The coefficients in the differential equation \eqref{ModeLimitRadial} are smooth in $(0,R)$ but introduce the regular singular points $\{0,R_{0}\}$, where $R_{0}>0$ is from \Cref{GSCLightCone}. Thus, $\widetilde{f}_{\lambda} \in C^{\infty} \big( (0,R_{0}) \cup (R_{0},R] \big)$. To conclude that the radial profile $\widetilde{f}_{\lambda}$ belongs to $C^{\infty}_{\mathrm{ev}}([0,R])$, we apply the Frobenius method \cite[p.~119 f., Theorem 4.5]{MR2961944}. Note, however, that the coefficients in \Cref{ModeLimitRadial} are not necessarily analytic. Therefore, we consider graphical similarity coordinates $\overline{\chi}: \RR \times \RR^{d} \rightarrow \RR^{1,d}$ with a height function $\overline{h} \in C^{\infty}_{\mathrm{rad}}(\RR^{d})$ satisfying $\overline{h} \equiv - 1$ in $\overline{\BB^{d}_{1}}$, see \Cref{GSCExemplary}, and a transition diffeomorphism
\begin{equation*}
\chi^{-1} \circ \overline{\chi} : \RR \times \BB^{d}_{\overline{R}} \rightarrow \RR \times \BB^{d}_{R} \,, \qquad (\overline{\tau}, \overline{\xi}) \mapsto \big( \overline{\tau} - \log h_{+}(\overline{\xi}), h_{+}(\overline{\xi})^{-1} \overline{\xi} \big) \,,
\end{equation*}
as in \Cref{TransitionDiffeo}. Then
\begin{equation}
\label{TransformationSC}
\overline{g}_{\lambda}(\overline{\rho}) = \frac{1}{\widetilde{h}_{+}(\overline{\rho})^{\lambda+s}} \widetilde{f}_{\lambda} \Big( \frac{\overline{\rho}}{\widetilde{h}_{+}(\overline{\rho})} \Big)
\end{equation}
defines a function $\overline{g}_{\lambda} \in C^{\infty} \big( (0,1)\cup(1,\overline{R}) \big)$ and since all derivatives of the transition diffeomorphism are bounded on compact domains, we infer from \eqref{RadialSobolevImplication} that
\begin{equation*}
\overline{g}_{\lambda} \in C^{k-1}([\overline{\delta},\overline{R}]) \,, \qquad
\overline{g}_{\lambda}^{(k-1)} \text{ is absolutely continuous on $[\overline{\delta}, \overline{R}]$,} \qquad
\overline{g}_{\lambda}^{(k)} \in L^{2}(\overline{\delta},\overline{R}) \,,
\end{equation*}
for any $0 < \overline{\delta} < \overline{R}$. It follows from \Cref{ModeLimitRadial} and \Cref{ModeInvariance} that
\begin{align}
\nonumber
(1-\overline{\rho}^{2}) \overline{g}_{\lambda}''(\overline{\rho}) + \Big( 2\Big( \frac{d-3}{2} - \lambda - s \Big) \overline{\rho} &+ (d-1) \frac{1-\overline{\rho}^{2}}{\overline{\rho}} \Big) \overline{g}_{\lambda}'(\overline{\rho}) \\\label{SpectralODE}&+ \Big( \overline{V}(\overline{\rho}) - (\lambda+s)(\lambda+s+1) \Big) \overline{g}_{\lambda}(\overline{\rho}) = 0
\end{align}
in $(0,1)$, where $\overline{V} \in C^{\infty}_{\mathrm{ev}}([0,1])$ is the radial representative of the potential in \Cref{PotentialNonlinearity} defined for the height function $\overline{h} \in C^{\infty}_{\mathrm{rad}}(\RR^{d})$. Furthermore, set
\begin{equation}
\label{TransformationEven}
w_{\lambda}(\overline{\rho}^{2}) = \overline{g}_{\lambda}(\overline{\rho}) \,.
\end{equation}
\Cref{SpectralODE} implies for $w_{\lambda}$ the differential equation
\begin{equation}
\label{HypGeomPot}
z(1-z) w_{\lambda}''(z) + (c - (a+b+1)z) w_{\lambda}'(z) - (ab - W(z)) w_{\lambda}(z) = 0
\end{equation}
in $(0,1)$, where
\begin{equation*}
a = \frac{\lambda+s}{2} \,, \qquad b = \frac{\lambda+s+1}{2} \,, \qquad c = \frac{d}{2} \,,
\end{equation*}
and $W \in C^{\infty}([0,1])$ is determined by $ W(\overline{\rho}^{2}) = \overline{V}(\overline{\rho})$. \Cref{HypGeomPot} is a hypergeometric differential equation with a smooth potential and has Frobenius indices
\begin{equation*} 
\Big\{ 0, -\frac{d-2}{2} \Big\} \quad \text{at } z = 0
\qquad\text{and}\qquad
\Big\{ 0, \frac{d}{2} - s - \frac{1}{2} - \lambda \Big\} \quad \text{at } z = 1 \,.
\end{equation*}
With this, we analyse the regularity of the solution $w_{\lambda}$ at $z=1$ first and distinguish three cases, where each time $\varphi_{1},\varphi_{2}$ denote generic analytic functions near $z=1$ with $\varphi_{1}(1) = \varphi_{2}(1) = 1$.
\begin{itemize}[wide,itemsep=1em,topsep=1em]
\item\textit{Case $m \coloneqq \frac{d}{2} - s - \frac{1}{2} - \lambda \in \NN_{0}$.} A fundamental system of solutions near $z=1$ is given by
\begin{equation*}
\phi_{1}(z) = (1-z)^{m} \varphi_{1}(z) \,, \qquad \phi_{2}(z) = \varphi_{2}(z) + C \log (1-z) \phi_{1}(z) \,.
\end{equation*}
\begin{itemize}[itemsep=1em,topsep=1em]
\item If $C\neq 0$, then $\phi_{2}^{(k)}(z)$ is not square integrable near $z=1$ since
\begin{equation*}
\Re(m - k) = \frac{d}{2} - s - k - \Re(\lambda) - \frac{1}{2} < - \frac{1}{2} \,.
\end{equation*}
But since $\overline{g}_{\lambda}^{(k)}$ is square integrable near $z=1$, also $w_{\lambda}^{(k)}$ is square integrable near $z=1$. Therefore, $w_{\lambda}$ is a multiple of $\phi_{1}$ and analytic near $z=1$.
\item If $C=0$, then any linear combination of $\phi_{1},\phi_{2}$ is analytic near $z=1$ and hence also $w_{\lambda}$.
\end{itemize}
\item\textit{Case $m \coloneqq \frac{d}{2} - s - \frac{1}{2} - \lambda \in -\NN$.} A fundamental system of solutions is given by
\begin{equation*}
\phi_{1}(z) = \varphi_{1}(z) \,, \qquad \phi_{2}(z) = (1-z)^{m} \varphi_{2}(z) + C \log (1-z) \phi_{1}(z) \,.
\end{equation*}
In this case, $\phi_{2}$ fails to be continuous at $z=1$ and hence $w_{\lambda}$ is a multiple of the analytic solution $\phi_{1}$ near $z=1$.
\item\textit{Case $m \coloneqq \frac{d}{2} - s - \frac{1}{2} - \lambda \notin \mathbb{Z}$.} A fundamental system of solutions is given by
\begin{equation*}
\phi_{1}(z) = (1-z)^{m} \varphi_{1}(z) \,, \qquad \phi_{2}(z) = \varphi_{2}(z) \,.
\end{equation*}
Then $\phi_{1}^{(k)}(z)$ is not square integrable near $z=1$ and thus $w_{\lambda}$ is a multiple of the analytic solution $\phi_{2}$ near $z=1$.
\end{itemize}
This proves that $w_{\lambda}$ is analytic near $z=1$. Near $z=0$, a fundamental system of solutions is given by
\begin{equation*}
\phi_{1}(z) = \varphi_{1}(z) \,, \qquad \phi_{2}(z) = z^{-\frac{d-2}{2}} \varphi_{2}(z) + C \log(z) \phi_{1}(z) \,,
\end{equation*}
for some $C\in\CC$ and functions $\varphi_{1},\varphi_{2}$ which are analytic near $z=0$ with $\varphi_{1}(0) = \varphi_{2}(0) = 1$. Now, $\frac{d}{2} - s - 1 > 0$ and so $\phi_{2}$ is not square integrable near $z=0$. But $f_{\lambda} \in L^{2}(\BB^{d}_{R})$ and
\begin{equation*}
\| f_{\lambda} \|_{L^{2}(\BB^{d}_{R_{0}})} \simeq \big\| (\,.\,)^{\frac{d-1}{2}} \widetilde{f}_{\lambda} \big\|_{L^{2}(0,R_{0})} \simeq \big\| (\,.\,)^{\frac{d-1}{2}} \overline{g}_{\lambda} \big\|_{L^{2}(0,1)} \simeq \big\| (\,.\,)^{\frac{d-2}{4}} w_{\lambda} \big\|_{L^{2}(0,1)} \,,
\end{equation*}
so $w_{\lambda}$ is a multiple of the analytic fundamental solution near $z=0$. In particular, the solution $w_{\lambda}$ is analytic in a neighbourhood of $[0,1]$ and any further analytic solution of \Cref{HypGeomPot} in $[0,1]$ is a multiple of this solution. We conclude from this and the transformations \eqref{ModeRadRep}, \eqref{TransformationSC}, \eqref{TransformationEven} that $f_{\lambda} \in C^{\infty}_{\mathrm{rad}}(\overline{\BB^{d}_{R}})$ and get from \Cref{SpecEq} and \Cref{Range} as above
\begin{equation*}
\big( \Box_{\chi} + \mathcal{V} \circ \chi \big) v_{\lambda} = 0 \qquad \text{in } \RR\times\BB^{d}_{R} \,,
\end{equation*}
where
\begin{equation*}
v_{\lambda}(\tau,\xi) = \ee^{(\lambda+s)\tau} f_{\lambda}(\xi) \,.
\end{equation*}
Conversely, if $f_{\lambda} \in C^{\infty}_{\mathrm{rad}}(\overline{\BB^{d}_{R}})$ defines a mode solution $v_{\lambda} \in C^{\infty}(\overline{\RR\times\BB^{d}_{R}})$ as in the proposition, put
\begin{align*}
u_{\lambda,1}(\tau,\xi) &= \ee^{-s\tau} v_{\lambda}(\tau,\xi) = \ee^{\lambda\tau} f_{\lambda}(\xi) \,, \\
u_{\lambda,2}(\tau,\xi) &= \ee^{-s\tau} \frac{1}{c(\xi)} \big( \pd_{\tau} + \xi^{i} \pd_{\xi^{i}} \big) v_{\lambda}(\tau,\xi) = \ee^{\lambda\tau} \Big( \frac{\lambda+s}{c(\xi)} f_{\lambda}(\xi) + \frac{1}{c(\xi)} \xi^{i} \pd_{i} f_{\lambda}(\xi) \Big) \,.
\end{align*}
Then, \Cref{LinearTransitionRelation} implies for $\mathbf{u}_{\lambda} = ( u_{\lambda,1}, u_{\lambda,2} ) \in C^{\infty}\big(\overline{\RR\times\BB^{d}_{R}}\big)^{2}$ the equation $\pd_{\tau} \mathbf{u}_{\lambda} (\tau,\,.\,) = \mathbf{L} \mathbf{u}_{\lambda} (\tau,\,.\,)$, which is equivalent to $(\lambda\mathbf{I} - \mathbf{L} ) \mathbf{f}_{\lambda} = \mathbf{0}$.
\end{proof}
Time translation invariance of \Cref{RadialSemilinearWaveEquation} induces an unstable eigenvalue of the generator.
\begin{lemma}
\label{SymmetryMode}
Let $\psi^{\ast}$ be the blowup profile given in \Cref{SelfSimilarSolutions}. Let $f_{1}^{\ast} \in C^{\infty}_{\mathrm{rad}}(\overline{\BB^{d}_{R}})$ and $\mathbf{f}_{1}^{\ast} \in C^{\infty}_{\mathrm{rad}}(\overline{\BB^{d}_{R}})^{2}$ be given by
\begin{equation*}
f_{1}^{\ast}(\xi) = (\pd_{0}\psi^{\ast})(-h(\xi),\xi)
\qquad\text{and}\qquad
\mathbf{f}_{1}^{\ast}(\xi) =
\begin{bmatrix}
f_{1}^{\ast}(\xi) \\
\displaystyle{
\frac{s+1}{c(\xi)} f_{1}^{\ast}(\xi) + \frac{1}{c(\xi)} \xi^{i} \pd_{\xi^{i}} f_{1}^{\ast}(\xi)
}
\end{bmatrix}
\,,
\end{equation*}
respectively. Then
\begin{equation*}
(\mathbf{I} - \mathbf{L}) \mathbf{f}_{1}^{\ast} = \mathbf{0} \,.
\end{equation*}
\end{lemma}
\begin{proof}
By \ref{F3}, for any $T' \in \RR$ and $(t,x) \in \RR^{1,d} \setminus \{ (T',0) \}$, the blowup solution $\psi^{\ast}_{T'}(t,x) = \psi^{\ast}(T'-t,x)$ solves
\begin{equation*}
\big( - \pd_{t}^{2} + \Delta_{x} \big) \psi^{\ast}_{T'}(t,x) + F\big( x, \psi^{\ast}_{T'}(t,x) \big) = 0 \,.
\end{equation*}
Differentiating this equation with respect to the parameter yields
\begin{equation*}
\big( - \pd_{t}^{2} + \Delta_{x} + \mathcal{V}_{T'}(t,x) \big) \pd_{T'} \psi^{\ast}_{T'}(t,x) = 0 \,.
\end{equation*}
After evaluating this equation at $T'=0$ and transforming it to graphical similarity coordinates, we find that
\begin{equation*}
v_{1}(\tau,\xi) \coloneqq \big( (\left.\pd_{T'}\psi^{\ast}_{T'})\right|_{T'=0} \circ \chi \big) (\tau,\xi) = \ee^{(s+1)\tau} f_{1}^{\ast}(\xi)
\end{equation*}
solves
\begin{equation*}
\big( \Box_{\chi} + \mathcal{V} \circ \chi \big) v_{1} = 0 \,.
\end{equation*}
Now the assertion of the lemma follows from \Cref{ModeCriterion}.
\end{proof}
\subsection{Characterization of the unstable spectrum}
\label{SecModeStability}
So far, the analysis of the linearized wave equation has only relied on general theory within our functional analytic framework, without taking the explicit form of the self-similar potential
\begin{equation}
\label{SelfSimilarPotentialExplicit}
\mathcal{V}(t,x) =
{
\renewcommand{\arraystretch}{2}
\left\{
\begin{array}{ll}
\displaystyle{
\frac{8(d-3)(d-4)t^{2}}{( (d-4)t^{2} + |x|^{2} )^{2}}
} & \text{for wave maps,} \\
\displaystyle{
3 (d-4) \frac{2 \alpha_{d-2}t^{2} + (2 \beta_{d-2}-1)|x|^{2}}{(\alpha_{d-2}t^{2} + \beta_{d-2} |x|^{2})^{2}}
}
& \text{for Yang-Mills,}
\end{array}
\right.
}
\end{equation}
into account, despite of it carrying sensible information about the precise location of unstable eigenvalues. This comes into play now, when we show that the symmetry-induced eigenvalue $\lambda=1$ from \Cref{SymmetryMode} is the only unstable eigenvalue of the generator. The core of this difficult nonself-adjoint spectral problem is reduced to the classical \emph{mode stability problem} for the corotational wave maps and Yang-Mills equation. The mode stability problem is essentially detached from the functional setting and the choice of similarity coordinates and it has been solved in \cite{MR3623242} for wave maps and in \cite{MR3475668}, \cite{MR4469070} for Yang-Mills, also see the survey article \cite{2023arXiv231012016D} for a self-contained introduction.
\begin{proposition}
\label{ModeStability}
Consider the assumptions of \Cref{ModeCriterion}. Then, there exists an $\omega < 0$ such that
\begin{equation*}
\sigma(\mathbf{L})\cap \mathbb{H}_{\omega} = \{ 1 \} \,.
\end{equation*}
\end{proposition}
\begin{proof}
Assume that $\lambda\in\sigma(\mathbf{L})$ with $\Re(\lambda) \geq 0$. By \Cref{ModeCriterion}, there exists a nonzero $\widetilde{f}_{\lambda} \in C^{\infty}_{\mathrm{ev}}([0,R])$ such that
\begin{equation*}
a_{2}(\rho) \widetilde{f}_{\lambda}''(\rho) + a_{1}(\rho;\lambda+s,d) \widetilde{f}_{\lambda}'(\rho) + \big( a_{0}(\rho;\lambda+s,d) + \widetilde{V}(\rho) \big) \widetilde{f}_{\lambda}(\rho) = 0
\end{equation*}
in $(0,R)$ with coefficients given in \Cref{a2Coeff,a1Coeff,a0Coeff}. To transform this differential equation, we employ the transition diffeomorphism
\begin{equation*}
\chi^{-1} \circ \overline{\chi} : \RR \times \BB^{d}_{\overline{R}} \rightarrow \RR \times \BB^{d}_{R} \,, \qquad (\overline{\tau}, \overline{\xi}) \mapsto \big( \overline{\tau} - \log h_{+}(\overline{\xi}), h_{+}(\overline{\xi})^{-1} \overline{\xi} \big) \,,
\end{equation*}
to similarity coordinates with a height function $\overline{h} \in C^{\infty}_{\mathrm{rad}}(\RR^{d})$ satisfying $\overline{h} \equiv -1$ in $\overline{\BB^{d}_{1}}$, see \Cref{TransitionDiffeo,GSCExemplary}. We switch to the variable
\begin{equation*}
\overline{g}_{\lambda}(\overline{\rho}) = \frac{\overline{\rho}^{s}}{\widetilde{h}_{+}(\overline{\rho})^{\lambda+s}} \widetilde{f}_{\lambda} \Big( \frac{\overline{\rho}}{\widetilde{h}_{+}(\overline{\rho})} \Big) \,,
\end{equation*}
and infer from \Cref{ModeInvariance} the generalized eigenvalue equation
\begin{align}
\nonumber
(1-\overline{\rho}^{2}) \overline{g}_{\lambda}''(\overline{\rho}) + \Big( \frac{d-2s-1}{\overline{\rho}} - 2( \lambda + 1 ) \overline{\rho} \Big) \overline{g}_{\lambda}'(\overline{\rho}) &- \lambda ( \lambda + 1 ) \overline{g}_{\lambda}(\overline{\rho}) \\\label{GenEvEq}&+
\Big( \overline{V}(\overline{\rho}) - \frac{s(d-s-2)}{\overline{\rho}^{2}} \Big) \overline{g}_{\lambda}(\overline{\rho}) = 0
\end{align}
in $(0,1)$, where $\overline{V} \in C^{\infty}_{\mathrm{ev}}([0,1])$ given by
\begin{equation*}
\overline{V}(\overline{\rho}) =
{
\renewcommand{\arraystretch}{2}
\left\{
\begin{array}{ll}
\displaystyle{
\frac{8(d-3)(d-4)}{( (d-4) + \overline{\rho}^{2} )^{2}}
} & \text{for wave maps,} \\
\displaystyle{
3 (d-4) \frac{2 \alpha_{d-2} + (2\beta_{d-2}-1)\overline{\rho}^{2}}{(\alpha_{d-2} + \beta_{d-2} \overline{\rho}^{2})^{2}}
}
& \text{for Yang-Mills,}
\end{array}
\right.
}
\end{equation*}
is the radial representative of the potential $F' \big( \,.\,, \psi^{\ast}(1,\,.\,) \big) \in C^{\infty}_{\mathrm{rad}}(\overline{\BB^{d}_{1}})$ in \Cref{SelfSimilarPotentialExplicit}. That is, $\overline{g}_{\lambda} \in C^{\infty}([0,1])$ poses for $\Re(\lambda) \geq 0$ a smooth solution to \cite[Eq. (2.7)]{MR3623242} for wave maps and \cite[Eq. (3.14)]{MR4469070} for Yang-Mills, respectively. However, according to \cite[Theorem 2.2]{MR3623242} and \cite[claim in the proof of Proposition 3.2]{MR4469070}, there exists no such smooth solution to \Cref{GenEvEq} unless $\lambda = 1$. Since $1 \in \sigma(\mathbf{L})$ is an eigenvalue by \Cref{SymmetryMode} and $\sigma(\mathbf{L}) \subseteq \CC$ is a closed subset, the assertion of the proposition follows.
\end{proof}
Next, we determine the roles of the unstable eigenvalue and symmetry mode in the linearized wave evolution. For this, we employ the Riesz projection from spectral theory.
\begin{definition}
\label{RieszProjection}
Consider the assumptions of \Cref{ModeCriterion}. We define the \emph{Riesz projection} $\mathbf{P}_{1} \in \mathfrak{L}( \mathfrak{H}^{k}_{\mathrm{rad}}(\BB^{d}_{R}) )$ associated to the isolated eigenvalue $1 \in \sigma(\mathbf{L})$ by
\begin{equation*}
\mathbf{P}_{1} = \frac{1}{2\pi\ii} \int_{\pd\mathbb{D}(1)} \mathbf{R}_{\mathbf{L}}(z) \dd z \,.
\end{equation*}
\end{definition}
This operator is well-defined and a bounded linear projection that commutes with the generator. Hence, it yields a decomposition of the Hilbert space into two closed complementary invariant subspaces for the generator. We also mention that this decomposition separates the unstable eigenvalue from the spectrum, see \cite[p.~178, Theorem 6.17]{MR1335452} for the general theory. In our setting, this leads to the following explicit subspace.
\begin{proposition}
\label{AlgMult}
Consider the assumptions of \Cref{ModeCriterion}. Then, $1 \in \sigma(\mathbf{L})$ is a simple eigenvalue whose geometric and algebraic eigenspaces are both spanned by the symmetry mode $\mathbf{f}_{1}^{\ast}$ from \Cref{SymmetryMode}, that is
\begin{equation*}
\operatorname{ran}(\mathbf{P}_{1}) = \ker(\mathbf{I} - \mathbf{L}) = \langle \mathbf{f}_{1}^{\ast} \rangle \,.
\end{equation*}
\end{proposition}
\begin{proof}
We have shown in \Cref{SymmetryMode} and \Cref{ModeCriterion} that $\ker(\mathbf{I} - \mathbf{L}) = \langle \mathbf{f}_{1}^{\ast} \rangle$. Spectral theory of linear operators yields the inclusion $\ker(\mathbf{I} - \mathbf{L}) \subseteq \operatorname{ran}(\mathbf{P}_{1})$ as well that the operator $\mathbf{I} - \mathbf{L}\mathord{\restriction}_{\operatorname{ran}(\mathbf{P}_{1})}$ is nilpotent, see e.g. \cite[proof of Lemma 3.4]{MR4778061} for details. If $\mathbf{I} - \mathbf{L}\mathord{\restriction}_{\operatorname{ran}(\mathbf{P}_{1})} = \mathbf{0}$ then $\big( \mathbf{I} - \mathbf{L} \big) \mathbf{P} \mathbf{f} = \mathbf{0}$ for all $\mathbf{f} \in \mathfrak{H}^{k}_{\mathrm{rad}}(\BB^{d}_{R})$ and therefore the other inclusion $\operatorname{ran}(\mathbf{P}_{1}) \subseteq \ker(\mathbf{I} - \mathbf{L})$ follows.
\medskip\par
Now, let us assume that $\mathbf{I} - \mathbf{L}\mathord{\restriction}_{\operatorname{ran}(\mathbf{P}_{1})} \neq \mathbf{0}$. Then there exists an $\mathbf{f} \in \operatorname{ran}( \mathbf{P}_{1} ) \subset \mathfrak{H}^{k}_{\mathrm{rad}}(\BB^{d}_{R})$ such that
\begin{equation}
\label{NilpotentGtr1}
\big( \mathbf{I} - \mathbf{L} \big) \mathbf{f} = \mathbf{f}_{1}^{\ast} \,.
\end{equation}
Radial Sobolev embedding yields that there is a nonzero function $\widetilde{f}_{1}: (0,R] \rightarrow \CC$ with
\begin{equation*}
\widetilde{f}_{1} \in C^{k-1}([\delta,R]) \,, \qquad
\widetilde{f}_{1}^{(k-1)} \text{ is absolutely continuous on } [\delta,R] \,, \qquad
\widetilde{f}_{1}^{(k)} \in L^{2}(\delta,R) \,,
\end{equation*}
for any $0 < \delta < R$, such that $[\mathbf{f}]_{1} = \widetilde{f}_{1}(|\,.\,|)$. As in the proof of \Cref{ModeCriterion}, we conclude from this and \Cref{NilpotentGtr1} the differential equation
\begin{equation}
\label{ModeInhom}
a_{2}(\rho) \widetilde{f}_{1}''(\rho) + a_{1}(\rho;s+1,d) \widetilde{f}_{1}'(\rho) + \big( a_{0}(\rho;s+1,d) + \widetilde{V}(\rho) \big) \widetilde{f}_{1}(\rho) = \widetilde{F}(\rho)
\end{equation}
a.e. in $(0,R)$, where the coefficients are given in \Cref{a2Coeff,a1Coeff,a0Coeff} and
\begin{equation*}
\widetilde{F}(\rho) = - \Big( (2s+3) \frac{\widetilde{w}(\rho)}{\widetilde{c}(\rho)^{2}} + \frac{d-1}{\widetilde{c}(\rho)} \frac{\widetilde{h}'(\rho)}{\rho} + \frac{\widetilde{h}(\rho)^{2} - \rho^{2}}{\widetilde{c}(\rho)^{3}} \widetilde{h}''(\rho) \Big) \widetilde{f}_{1}^{\ast}(\rho) - 2 \frac{\rho - \widetilde{h}(\rho)\widetilde{h}'(\rho)}{\widetilde{c}(\rho)^{2}} \widetilde{f}_{1}^{*'}(\rho)
\end{equation*}
and $\widetilde{f}_{1}^{\ast} \in C^{\infty}_{\mathrm{ev}}([0,R])$ is the radial representative of the symmetry mode $f_{1}^{\ast} \in C^{\infty}_{\mathrm{rad}}(\overline{\BB^{d}_{R}})$ from \Cref{SymmetryMode}. As previously, we consider similarity coordinates ${\overline{\chi}\mathstrut}_{T}: \RR \times \RR^{d} \rightarrow \RR^{1,d}$ with height function $\overline{h} \in C^{\infty}_{\mathrm{rad}}(\RR^{d})$ satisfying $\overline{h} \equiv - 1$ in $\overline{\BB^{d}_{1}}$ and the transition diffeomorphism
\begin{equation*}
\chi^{-1} \circ \overline{\chi} : \RR \times \BB^{d}_{\overline{R}} \rightarrow \RR \times \BB^{d}_{R} \,, \qquad ( \overline{\tau}, \overline{\xi} ) \mapsto \big( \overline{\tau} - \log h_{+}(\overline{\xi}), h_{+}(\overline{\xi})^{-1} \overline{\xi} \big) \,,
\end{equation*}
see \Cref{TransitionDiffeo,GSCExemplary}. For $\overline{\rho} \in (0,1)$, we set
\begin{equation}
\label{TransformationInhomLinWave}
\overline{g}(\overline{\rho}) = \frac{1}{\widetilde{h}_{+}(\overline{\rho})^{s+1}} \widetilde{f}_{1} \Big( \frac{\overline{\rho}}{\widetilde{h}_{+}(\overline{\rho})} \Big) \,, \qquad \overline{F}(\overline{\rho}) = \frac{1}{\widetilde{h}_{+}(\overline{\rho})^{s+3}} \widetilde{F} \Big( \frac{\overline{\rho}}{\widetilde{h}_{+}(\overline{\rho})} \Big) \,.
\end{equation}
It follows that $\overline{g} \in C^{\infty}((0,1)) \cap H^{k}((\frac{1}{2},1))$. To compute the inhomogeneity $\overline{F}$, we note the transformation
\begin{equation*}
\widetilde{h} \big( \widetilde{h}_{+}(\overline{\rho})^{-1} \overline{\rho} \big) = \frac{-1}{\widetilde{h}_{+}(\overline{\rho})}
\end{equation*}
and derive the identities
\begin{align*}
\widetilde{h}' \Big( \frac{\overline{\rho}}{\widetilde{h}_{+}(\overline{\rho})} \Big) &= -
\frac{H(\overline{\rho})}{\overline{\rho} H(\overline{\rho}) - 1 } \,, &
\widetilde{h}'' \Big( \frac{\overline{\rho}}{\widetilde{h}_{+}(\overline{\rho})} \Big) &= - \widetilde{h}_{+}(\overline{\rho}) \frac{H'( \overline{\rho} ) + H( \overline{\rho} )^{2}}{( \overline{\rho} H( \overline{\rho} ) - 1 )^{3}} \,, \\
\widetilde{c} \Big( \frac{\overline{\rho}}{\widetilde{h}_{+}(\overline{\rho})} \Big) &= -\frac{1}{\widetilde{h}_{+}(\overline{\rho})} \frac{1}{\overline{\rho} H(\overline{\rho}) - 1} \,, &
\widetilde{w} \Big( \frac{\overline{\rho}}{\widetilde{h}_{+}(\overline{\rho})} \Big) &= \frac{-(1-\overline{\rho}^{2})H(\overline{\rho})^{2} - 2 \overline{\rho} H(\overline{\rho}) + 1 }{( \overline{\rho} H(\overline{\rho}) - 1 )^{2}} \,,
\end{align*}
where
\begin{equation*}
H(\overline{\rho}) \coloneqq \frac{\widetilde{h}_{+}'(\overline{\rho})}{\widetilde{h}_{+}(\overline{\rho})} \,.
\end{equation*}
Furthermore, we recall that the symmetry mode from \Cref{SymmetryMode} is given in the respective similarity coordinates by
\begin{equation*}
f_{1}^{\ast}(\xi) = \ee^{-(s+1)\tau} \big( (\left.\pd_{T'}\psi^{\ast}_{T'})\right|_{T'=0} \circ \chi \big) (\tau,\xi) \,, \qquad
\overline{f}_{1}^{\ast}(\overline{\xi}) = \ee^{-(s+1)\overline{\tau}} \big( (\left.\pd_{T'}\psi^{\ast}_{T'})\right|_{T'=0} \circ \overline{\chi} \big) (\overline{\tau},\overline{\xi}) \,,
\end{equation*}
so the respective radial representatives $\widetilde{f}_{1}^{\ast}$, $\overline{g}_{1}^{\ast}$ are related via
\begin{align*}
\widetilde{f}_{1}^{\ast} \Big( \frac{\overline{\rho}}{\widetilde{h}_{+}(\overline{\rho})} \Big) &= \widetilde{h}_{+}(\overline{\rho})^{s+1} \overline{g}_{1}^{\ast}(\overline{\rho}) \,, \\
\widetilde{f}_{1}^{*'} \Big( \frac{\overline{\rho}}{\widetilde{h}_{+}(\overline{\rho})} \Big) &= - \frac{\widetilde{h}_{+}(\overline{\rho})^{s+2}}{\overline{\rho} H(\overline{\rho}) - 1} \Big( (s+1) H(\overline{\rho}) \overline{g}_{1}^{\ast}(\overline{\rho}) + \overline{g}_{1}^{\ast'}(\overline{\rho}) \Big) \,.
\end{align*}
With this at hand, a computation yields $\overline{F} = \overline{F}^{\ast} + \overline{F}^{h}$ for $\overline{F}^{\ast}, \overline{F}^{h} \in C^{\infty}_{\mathrm{ev}}([0,1])$ given by
\begin{align*}
\overline{F}^{\ast}(\overline{\rho}) &= - (2s+3) \overline{g}_{1}^{\ast}(\overline{\rho}) - 2 \overline{\rho} \, \overline{g}_{1}^{\ast'}(\overline{\rho}) \,, \\
\overline{F}^{h}(\overline{\rho}) &= - \frac{1-\overline{\rho}^{2}}{\overline{g}_{1}^{\ast}(\overline{\rho})} \frac{\big( 1 - \overline{\rho}^{2} \big)^{\frac{d-5}{2}-s}}{\overline{\rho}^{d-1}} \pd_{\overline{\rho}} \Big( \frac{ \overline{\rho}^{d-1} }{ \big( 1 - \overline{\rho}^{2} \big)^{\frac{d-5}{2}-s} } H(\overline{\rho}) \overline{g}_{1}^{\ast}(\overline{\rho})^{2} \Big) \,.
\end{align*}
This implies together with \Cref{ModeInvariance} that $\overline{g}$ solves the inhomogeneous differential equation
\begin{align}
\nonumber
(1-\overline{\rho}^{2}) \overline{g}''(\overline{\rho}) + \Big( 2\Big( \frac{d-5}{2} - s \Big) \overline{\rho} &+ (d-1) \frac{1-\overline{\rho}^{2}}{\overline{\rho}} \Big) \overline{g}'(\overline{\rho}) \\\label{InhomODE}&+ \Big( \overline{V}(\overline{\rho}) - (s+1)(s+2) \Big) \overline{g}(\overline{\rho}) = \overline{F}(\overline{\rho})
\end{align}
in $(0,1)$. By \Cref{SymmetryMode} and \Cref{ModeCriterion}, the function $\overline{g}_{1}^{\ast} \in C^{\infty}_{\mathrm{ev}}([0,1])$ is a solution to the homogeneous version of \Cref{InhomODE} and therefore, a fundamental system of solutions is given by
\begin{align*}
\phi_{1} &= \overline{g}_{1}^{\ast} \,, \\
\phi_{2} &= \overline{g}_{1}^{\ast} I \,,
\quad\text{where}\quad
I'(\overline{\rho}) = \frac{\big( 1 - \overline{\rho}^{2} \big)^{\frac{d-5}{2}-s}}{\overline{\rho}^{d-1}} \frac{1}{\overline{g}_{1}^{\ast}(\overline{\rho})^{2}} \,,
\end{align*}
with Wronski determinant
\begin{equation*}
W(\overline{\rho}) = W(\phi_{1},\phi_{2})(\overline{\rho}) = \overline{g}_{1}^{\ast}(\overline{\rho})^{2} I'(\overline{\rho}) = \frac{\big( 1 - \overline{\rho}^{2} \big)^{\frac{d-5}{2}-s}}{\overline{\rho}^{d-1}} \,.
\end{equation*}
The variations of constants formula yields
\begin{align*}
\overline{g}(\overline{\rho}) &=
\alpha_{1} \phi_{1}(\overline{\rho}) + \alpha_{2} \phi_{2}(\overline{\rho}) - \phi_{1}(\overline{\rho}) \int_{0}^{\overline{\rho}} \frac{\phi_{2}(z)}{W(z)} \frac{\overline{F}(z)}{1-z^{2}} \dd z + \phi_{2}(\overline{\rho}) \int_{0}^{\overline{\rho}} \frac{\phi_{1}(z)}{W(z)} \frac{\overline{F}(z)}{1-z^{2}} \dd z \\&=
g(\overline{\rho}) + \overline{g}^{h}(\overline{\rho})
\end{align*}
for $\overline{\rho} \in (0,1)$ and some constants $\alpha_{1},\alpha_{2} \in \CC$, where we have split the solution into a term given by
\begin{equation*}
g(\overline{\rho}) = \alpha_{1} \phi_{1}(\overline{\rho}) + \alpha_{2} \phi_{2}(\overline{\rho}) - \phi_{1}(\overline{\rho}) \int_{0}^{\overline{\rho}} \frac{\phi_{2}(z)}{W(z)} \frac{\overline{F}^{\ast}(z)}{1-z^{2}} \dd z + \phi_{2}(\overline{\rho}) \int_{0}^{\overline{\rho}} \frac{\phi_{1}(z)}{W(z)} \frac{\overline{F}^{\ast}(z)}{1-z^{2}} \dd z
\end{equation*}
and another term given by
\begin{equation*}
\overline{g}^{h}(\overline{\rho}) =
- \phi_{1}(\overline{\rho}) \int_{0}^{\overline{\rho}} \frac{\phi_{2}(z)}{W(z)} \frac{\overline{F}^{h}(z)}{1-z^{2}} \dd z + \phi_{2}(\overline{\rho}) \int_{0}^{\overline{\rho}} \frac{\phi_{1}(z)}{W(z)} \frac{\overline{F}^{h}(z)}{1-z^{2}} \dd z =
- \overline{g}_{1}^{\ast}(\overline{\rho}) \log \frac{\widetilde{h}_{+}(\overline{\rho})}{\widetilde{h}_{+}(0)} \,,
\end{equation*}
as follows from integration by parts. Note that $\overline{g}^{h} \in C^{\infty}_{\mathrm{ev}}([0,1])$ and that $g \in C^{\infty}((0,1))$ is the general solution to the version of \Cref{ModeInhom} in $(0,1)$ in standard similarity coordinates with $\overline{h} \equiv -1$. In particular, we conclude from \cite[Proposition A.2]{MR3680948}, \cite[proof of Proposition 6.1]{2022arXiv220706952G} for wave maps and \cite[proof of Proposition 3.4]{MR4469070} for Yang-Mills that $g^{(k)} \notin L^{2}((\frac{1}{2},1))$. However, this contradicts $f_{1} \in H^{k}(\BB^{d}_{R})$ via the relation \eqref{TransformationInhomLinWave} and so the assumption $\mathbf{I} - \mathbf{L}\mathord{\restriction}_{\operatorname{ran}(\mathbf{P}_{1})} \neq \mathbf{0}$ does not hold.
\end{proof}
\subsection{Co-dimension one stable subspace for the linearized wave flow}
As a consequence, we get that the linearized wave evolution in similarity coordinates grows exponentially on the one-dimensional subspace spanned by the symmetry mode, whereas on the complementary subspace it decays exponentially. This completely characterizes the global dynamics of the linearized flow.
\begin{proposition}
\label{LinearizedEvolution}
Let $d,k \in \mathbb{N}$ with $1 < \frac{d}{2} - s < k$ and $R \geq R_{0}$. Let $\mathbf{S}: [0,\infty) \rightarrow \mathfrak{L}( \mathfrak{H}^{k}_{\mathrm{rad}}(\BB^{d}_{R}) )$ be the semigroup from \Cref{LinearFlow} and $\mathbf{P}_{1} \in \mathfrak{L}( \mathfrak{H}^{k}_{\mathrm{rad}}(\BB^{d}_{R}) )$ be the spectral projection from \Cref{RieszProjection}. Then
\begin{equation*}
\mathbf{S}(\tau) \mathbf{P}_{1} - \mathbf{P}_{1} \mathbf{S}(\tau) = \mathbf{0} \,, \qquad \mathbf{S}(\tau) \mathbf{P}_{1} = \ee^{\tau} \mathbf{P}_{1} \,,
\end{equation*}
and there exist constants $\omega > 0$ and $M \geq 1$ such that
\begin{equation*}
\| \mathbf{S}(\tau) ( \mathbf{I} - \mathbf{P}_{1} ) \mathbf{f} \|_{\mathfrak{H}^{k}(\BB^{d}_{R})} \leq M \ee^{-\omega \tau} \| ( \mathbf{I} - \mathbf{P}_{1} ) \mathbf{f} \|_{\mathfrak{H}^{k}(\BB^{d}_{R})}
\end{equation*}
for all $\mathbf{f} \in \mathfrak{H}^{k}_{\mathrm{rad}}(\BB^{d}_{R})$ and all $\tau \geq 0$.
\end{proposition}
\begin{proof}
Since $\mathbf{S}(\tau)$ is a bounded linear operator which commutes with the generator $\mathbf{L}$, it also commutes with the spectral projection $\mathbf{P}_{1}$. Using this and \Cref{AlgMult}, we find the differential equation $\pd_{\tau} \mathbf{S}(\tau) \mathbf{P}_{1} \mathbf{f} = \mathbf{S}(\tau) \mathbf{P}_{1} \mathbf{f}$ and obtain the unique solution $\mathbf{S}(\tau) \mathbf{P}_{1} \mathbf{f} = \ee^{\tau} \mathbf{P}_{1} \mathbf{f}$. By \Cref{BoundedCompactPotential}, $\mathbf{L} = \mathbf{L}_{0} + \mathbf{L}_{V}'$ is a compact perturbation of the generator of the semigroup from \Cref{SGRadialWaveFlow} with $\sup \big\{ \Re (\lambda) \in \CC \mid \lambda \in \sigma(\mathbf{L}) \setminus \{ 1 \} \big\} < 0$ according to \Cref{ModeStability}. Now, the stability estimate on the co-dimension one subspace follows directly from \cite[Theorem B.1]{MR4469070}.
\end{proof}
\section{Nonlinear stability analysis}
\label{SecNonlin}
In this section, we analyze the full nonlinear stability problem in similarity coordinates.
\subsection{Estimates for the nonlinearity}
We incorporate the nonlinear remainder from \Cref{PotentialNonlinearity} as a locally Lipschitz continuous map in the functional setting of the linear theory.
\begin{lemma}
\label{LocLip}
Let $d,k\in\NN$ and $R>0$ with
\begin{equation*}
k \geq \frac{d}{2} > 2 \text{ for wave maps}
\qquad\text{and}\qquad
k \geq \frac{d-1}{2} > \frac{5}{2} \text{ for Yang-Mills.}
\end{equation*}
Then, the map $\mathbf{N}: C^{\infty}_{\mathrm{rad}}(\overline{\BB^{d}_{R}})^{2} \rightarrow C^{\infty}_{\mathrm{rad}}(\overline{\BB^{d}_{R}})^{2}$ given in \Cref{PotentialNonlinearity} has a unique extension to a map $\mathbf{N}: \mathfrak{H}^{k}_{\mathrm{rad}}(\BB^{d}_{R}) \rightarrow \mathfrak{H}^{k}_{\mathrm{rad}}(\BB^{d}_{R})$ with the property that for any $M > 0$ the estimate
\begin{equation*}
\| \mathbf{N}(\mathbf{f}) - \mathbf{N}(\mathbf{g}) \|_{\mathfrak{H}^{k}(\BB^{d}_{R})} \lesssim \big( \| \mathbf{f} \|_{\mathfrak{H}^{k}(\BB^{d}_{R})} + \| \mathbf{g} \|_{\mathfrak{H}^{k}(\BB^{d}_{R})} \big) \| \mathbf{f} - \mathbf{g} \|_{\mathfrak{H}^{k}(\BB^{d}_{R})}
\end{equation*}
holds for all $\mathbf{f},\mathbf{g} \in \mathfrak{H}^{k}_{\mathrm{rad}}(\BB^{d}_{R})$ with $\| \mathbf{f} \|_{\mathfrak{H}^{k}(\BB^{d}_{R})}, \| \mathbf{g} \|_{\mathfrak{H}^{k}(\BB^{d}_{R})} \leq M$.
\end{lemma}
\begin{proof}
We have
\begin{align*}
\| \mathbf{N}(\mathbf{f}) - \mathbf{N}(\mathbf{g}) \|_{\mathfrak{H}^{k}(\BB^{d}_{R})} &= \Big\| \frac{c}{w} \big( N(f_{1}) - N(g_{1}) \big) \Big\|_{H^{k-1}(\BB^{d}_{R})} \\&\lesssim \| N(f_{1}) - N(g_{1}) \|_{H^{k-1}(\BB^{d}_{R})}
\end{align*}
for all $\mathbf{f}, \mathbf{g} \in C^{\infty} ( \overline{\BB^{d}_{R}} )^{2}$. We proceed by splitting the proof into three cases by treating the Yang-Mills nonlinearity for $\frac{d-1}{2} \leq k_{d} \leq \frac{d}{2}$, the wave maps nonlinearity for $k_{d} = \frac{d}{2}$ and both nonlinearities for $k > \frac{d}{2}$.
\begin{enumerate}[wide,itemsep=1em,topsep=1em]
\item[$\bullet$ \textit{Yang-Mills in case $\frac{d-1}{2} \leq k_{d} \leq \frac{d}{2}$.}] The Yang-Mills nonlinearity reads explicitly
\begin{equation}
\label{NYM}
N_{\mathsf{YM}}(f)(\xi) = \big( d-4 \big) \big( 3 \big( 1 - |\xi|^{2} \psi^{\ast}(-h(\xi),\xi) \big) f(\xi)^{2} - |\xi|^{2} f(\xi)^{3} \big) \,.
\end{equation}
So, it suffices to control the Sobolev norm of the product of three functions. The product rule and H\"{o}lder's inequality for exponents $\frac{1}{p_{1}} + \frac{1}{p_{2}} + \frac{1}{p_{3}} \leq \frac{1}{2}$ show that
\begin{align*}
\| f_{1} f_{2} f_{3} \|_{H^{k_{d}-1}(\BB^{d}_{R})} &\lesssim \sum_{0 \leq |\alpha_{1}| + |\alpha_{2}| + |\alpha_{3}| \leq k_{d} - 1} \big\| \big( \pd^{\alpha_{1}} f_{1} \big) \big( \pd^{\alpha_{2}} f_{2} \big) \big( \pd^{\alpha_{3}} f_{3} \big) \big\|_{L^{2}(\BB^{d}_{R})} \\&\lesssim
\sum_{0 \leq |\alpha_{1}| + |\alpha_{2}| + |\alpha_{3}| \leq k_{d} - 1}
\big\| \pd^{\alpha_{1}} f_{1} \big\|_{L^{p_{1}}(\BB^{d}_{R})}
\big\| \pd^{\alpha_{2}} f_{2} \big\|_{L^{p_{2}}(\BB^{d}_{R})}
\big\| \pd^{\alpha_{3}} f_{3} \big\|_{L^{p_{3}}(\BB^{d}_{R})}
\end{align*}
for all $f_{1},f_{2},f_{3} \in C^{\infty} ( \overline{\BB^{d}_{R}} )$. If $d\in\NN$ is odd, then $k_{d} = \frac{d-1}{2}$ and we choose exponents
\begin{equation*}
p_{1} = \frac{2d}{d - 2( k_{d} - |\alpha_{1}| )} \,, \qquad
p_{2} = \frac{2d}{d - 2( k_{d} - |\alpha_{2}| )} \,, \qquad
p_{3} = \frac{2d}{d - 2( k_{d} - |\alpha_{3}| )} \,,
\end{equation*}
which ensures
\begin{equation*}
\frac{1}{p_{1}} + \frac{1}{p_{2}} + \frac{1}{p_{3}} = \frac{1 + 2|\alpha_{1}|}{2d} + \frac{1 + 2|\alpha_{2}|}{2d} + \frac{1 + 2|\alpha_{3}|}{2d} \leq
\frac{3 + 2( k_{d} - 1 )}{2d} = \frac{1}{2} \,.
\end{equation*}
If $d\in\NN$ is even, then $k_{d} = \frac{d}{2}$ and we choose the H\"{o}lder exponents as above if $\alpha_{i} \neq 0$. If $\alpha_{i} = 0$, we set $p_{i} = p$ for a suitable $p > 2$ such that $\frac{1}{p_{1}} + \frac{1}{p_{2}} + \frac{1}{p_{3}} = \frac{1}{2}$. In both cases, the Sobolev embeddings
\begin{equation}
\label{SobolevEmbedding}
H^{k}(\BB^{d}_{R}) \hookrightarrow L^{q}(\BB^{d}_{R})
\quad\text{for } 2 \leq q \text{ if } k = \frac{d}{2}
\quad\text{and}\quad
2 \leq q \leq \frac{2d}{d-2k} \text{ if } 0 \leq k < \frac{d}{2}
\end{equation}
give
\begin{equation*}
\| f_{1} f_{2} f_{3} \|_{H^{k_{d}-1}(\BB^{d}_{R})} \lesssim \| f_{1} \|_{H^{k_{d}}(\BB^{d}_{R})} \| f_{2} \|_{H^{k_{d}}(\BB^{d}_{R})} \| f_{3} \|_{H^{k_{d}}(\BB^{d}_{R})}
\end{equation*}
for all $f_{1},f_{2},f_{3} \in C^{\infty} ( \overline{\BB^{d}_{R}} )$, which yields the desired local Lipschitz bound for the Yang-Mills nonlinearity.
\item[$\bullet$ \textit{Wave maps in case $k_{d} = \frac{d}{2}$.}] By \cite[Lemma 4.1]{MR4778061}, we have
\begin{equation}
\label{WMNonlinearityRepresentation}
N_{\mathsf{WM}}(f)(\xi) - N_{\mathsf{WM}}(g)(\xi) =
\int_{0}^{1} \int_{0}^{1} A(\xi) B(\xi;s) F_{\mathsf{WM}}'' \big( \xi, C(\xi;s,t) \big) \dd t \dd s \,,
\end{equation}
where we denote
\begin{align*}
A(\xi) &\coloneqq f(\xi) - g(\xi) \,, &
C(\xi;s,t) &\coloneqq \psi^{\ast}(-h(\xi),\xi) + t \big( s f(\xi) + (1 - s ) g(\xi) \big) \,, \\
B(\xi;s) &\coloneqq s f(\xi) + (1 - s ) g(\xi) \,, &
F_{\mathsf{WM}}''(x,z) &\coloneqq \pd_{z}^{2} F_{\mathsf{WM}}(x,z) = 2(d-3) \frac{\sin\big( 2 |x| z \big)}{|x|} \,,
\end{align*}
for $f,g \in C^{\infty} ( \overline{\BB^{d}_{R}} )$. So,
\begin{align}
\begin{split}
\label{2ndComponentEstimateWM}
&
\| N_{\mathsf{WM}}(f) - N_{\mathsf{WM}}(g) \|_{H^{k_{d}-1}(\BB^{d}_{R})} \lesssim \int_{0}^{1} \int_{0}^{1} \Big\| A B(\,.\,;s) F_{\mathsf{WM}}'' \big( \,.\,, C(\,.\,;s,t) \big) \Big\|_{H^{k_{d}-1}(\BB^{d}_{R})} \dd t \dd s \\
&\indent\lesssim
\int_{0}^{1} \int_{0}^{1}
\sum_{0 \leq |\alpha| + |\beta| + |\gamma| \leq k_{d} - 1}
\Big\| \big( \pd^{\alpha} A \big) \big( \pd^{\beta} B(\,.\,;s) \big) \big( \pd^{\gamma} F_{\mathsf{WM}}'' \big( \,.\,, C(\,.\,;s,t) \big) \big) \Big\|_{L^{2}(\BB^{d}_{R})} \dd t \dd s
\end{split}
\end{align}
for all $f,g \in C^{\infty} ( \overline{\BB^{d}_{R}} )$. The chain rule yields
\begin{align}
\begin{split}
\label{ChainRuleEstimateWM}
&
\Big| \pd_{\xi}^{\gamma} F_{\mathsf{WM}}'' \big( \xi, C(\xi;s,t) \big) \Big| \lesssim \Big| (\pd^{(\gamma,0)}F_{\mathsf{WM}}'') \big( \xi, C(\xi;s,t) \big) \Big| \\&\indent\indent+
\sum_{\ell=1}^{|\gamma|} \sum_{0 \leq |\alpha| \leq |\gamma| - \ell} \Big| (\pd^{(\alpha,\ell)}F_{\mathsf{WM}}'') \big( \xi, C(\xi;s,t) \big) \Big| \sum_{
\substack{
|\gamma_{1}| + \ldots + |\gamma_{\ell}| = |\gamma| - |\alpha| \\
\gamma_{1}, \ldots, \gamma_{\ell} \neq 0
}
} \prod_{i=1}^{\ell} \Big| \pd^{\gamma_{i}}_{\xi} C(\xi;s,t) \Big|
\end{split}
\end{align}
and so we continue with pointwise bounds on derivatives of the wave maps nonlinearity. We have for any multi-index $\alpha \in \NN_{0}^{d}$ the bounds
\begin{align*}
| F_{\mathsf{WM}}(x,z) | &\lesssim \frac{|z|^{3}}{1 + |xz|^{2}} \,,
&
| \pd_{x}^{\alpha} F_{\mathsf{WM}}(x,z) | &\lesssim \frac{|z|^{3+|\alpha|}}{1 + |xz|^{3}} \,, \\
| \pd_{z} F_{\mathsf{WM}}(x,z) | &\lesssim \frac{|z|^{2}}{1 + |xz|^{2}} \,, &
| \pd_{x}^{\alpha} \pd_{z} F_{\mathsf{WM}}(x,z) | &\lesssim \frac{|z|^{2+|\alpha|}}{1 + |xz|^{2}} \,, \\
| \pd_{z}^{2} F_{\mathsf{WM}}(x,z) | &\lesssim \frac{|z|}{1 + |xz|} \,,
&
| \pd_{x}^{\alpha} \pd_{z}^{2} F_{\mathsf{WM}}(x,z) | &\lesssim \frac{|z|^{1+|\alpha|}}{1 + |xz|} \,,
\end{align*}
for all $(x,z) \in \RR^{d} \times \RR$. Furthermore,
\begin{equation*}
\pd_{z} F_{\mathsf{WM}}''(x,z) = 4(d-3) \cos \big( |2 z x| \big) \eqqcolon \varphi(zx) \,,
\end{equation*}
where $\varphi \in C^{\infty}_{\mathrm{rad}}(\RR^{d}) \cap C^{\infty}_{\mathrm{b}}(\RR^{d})$ by \cite[Lemma 2.3]{2022arXiv220902286O}. Thus, we have for any multi-index $\alpha \in \NN_{0}^{d}$ and for any $\ell\in \NN$
\begin{align*}
\pd_{x}^{\alpha} \pd_{z} F_{\mathsf{WM}}''(x,z) &= z^{|\alpha|} (\pd^{\alpha} \varphi)(zx) \,, \\
\pd_{z}^{\ell} (\pd^{\alpha} \varphi)(zx) &= \sum_{i_{1},\ldots,i_{\ell} = 1}^{d} x^{i_{1}} \ldots x^{i_{\ell}} \big( \pd_{i_{1}} \ldots \pd_{i_{\ell}} \pd^{\alpha} \varphi \big)(zx) \,,
\end{align*}
which implies the pointwise bounds
\begin{equation}
\label{pdxalphapdzellF''}
\big| \pd_{z}^{\ell} F_{\mathsf{WM}}''(x.z) \big| \lesssim |x|^{\ell-1} \,, \qquad
\big| \pd_{x}^{\alpha} \pd_{z}^{\ell} F_{\mathsf{WM}}''(x,z) \big| \lesssim \sum_{j = \max\{ 0, |\alpha| - \ell \}}^{|\alpha|} |x|^{j - |\alpha| + \ell - 1} |z|^{j} \,,
\end{equation}
for all $(x,z) \in \RR^{d} \times \RR$. Hence, according to the estimates \eqref{2ndComponentEstimateWM}, \eqref{ChainRuleEstimateWM}, \eqref{pdxalphapdzellF''}, it suffices to prove bounds for the terms
\begin{align}
\label{Term1}
\Big\| \big| \pd^{\alpha} A \big| \big| \pd^{\beta} B(\,.\,;s) \big| \big| C(\,.\,;s,t) \big|^{m} \prod_{i=1}^{\ell} \big| \pd^{\gamma_{i}} C(\,.\,;s,t) \big| \Big\|_{L^{2}(\BB^{d}_{R})} \,, \\
\label{Term2}
\Big\| \big| \pd^{\alpha} A \big| \big| \pd^{\beta} B(\,.\,;s) \big| \big| C(\,.\,;s,t) \big|^{1 + |\gamma|} \Big\|_{L^{2}(\BB^{d}_{R})} \,,
\end{align}
for any $\alpha, \beta, \gamma, \gamma_{i} \in \NN_{0}^{d}$ with $\gamma_{i} \neq 0$, $1 \leq \ell \leq k_{d} - 1$, $0 \leq m \leq k_{d} - 1 - \ell$, and
\begin{equation}
\label{MultiIndexNumerology}
\ell \leq |\alpha| + |\beta| + \sum_{i=1}^{\ell} |\gamma_{i}| \leq k_{d} - 1
\qquad \text{resp.} \qquad
0 \leq |\alpha| + |\beta| + |\gamma| \leq k_{d} - 1 \,.
\end{equation}
For the term in \eqref{Term1}, we use H\"{o}lder's inequality with exponents
\begin{equation}
\label{HoelderExponents}
\frac{1}{2} = \frac{1}{p_{\alpha}} + \frac{1}{p_{\beta}} + \frac{m}{p} + \sum_{i=1}^{\ell} \frac{1}{p_{\gamma_{i}}}
\end{equation}
and get
\begin{align}
\begin{split}
\label{HoelderBound}
&
\Big\| \big| \pd^{\alpha} A \big| \big| \pd^{\beta} B(\,.\,;s) \big| \big| C(\,.\,;s,t) \big|^{m} \prod_{i=1}^{\ell} \big| \pd^{\gamma_{i}} C(\,.\,;s,t) \big| \Big\|_{L^{2}(\BB^{d}_{R})} \\&\indent\leq
\big\| \pd^{\alpha} A \big\|_{L^{p_{\alpha}}(\BB^{d}_{R})}
\big\| \pd^{\beta} B(\,.\,;s) \big\|_{L^{p_{\beta}}(\BB^{d}_{R})}
\big\| C(\,.\,;s,t) \big\|_{L^{p}(\BB^{d}_{R})}^{m}
\prod_{i=1}^{\ell} \big\| \pd^{\gamma_{i}} C(\,.\,;s,t) \big\|_{L^{p_{\gamma_{i}}}(\BB^{d}_{R})} \,.
\end{split}
\end{align}
As $\gamma_{i} \neq 0$, we choose
\begin{equation*}
p_{\gamma_{i}} = \frac{2d}{d - 2(k_{d}-|\gamma_{i}|)} = \frac{d}{|\gamma_{i}|}
\end{equation*}
and if $\alpha,\beta\neq 0$
\begin{equation*}
p_{\alpha} = \frac{2d}{d - 2(k_{d}-|\alpha|)} = \frac{d}{|\alpha|} \,, \qquad
p_{\beta} = \frac{2d}{d - 2(k_{d}-|\beta|)} = \frac{d}{|\beta|} \,,
\end{equation*}
or else $p_{\alpha} = p$ resp. $p_{\beta} = p$. Thus, the conditions in \eqref{MultiIndexNumerology} yield
\begin{equation*}
\frac{1}{p_{\alpha}} + \frac{1}{p_{\beta}} + \sum_{i=1}^{\ell} \frac{1}{p_{\gamma_{i}}} \leq \frac{1}{d} \Big( |\alpha| + |\beta| + \sum_{i=1}^{d} |\gamma_{i}| \Big) \leq \frac{1}{2} - \frac{1}{d} \,,
\end{equation*}
and so we can choose $p>2$ such that \Cref{HoelderExponents} is satisfied. The exponents for the term in \eqref{Term2} are chosen similarly. Then, the Sobolev embeddings \eqref{SobolevEmbedding} applied in \eqref{HoelderBound} and analogously in \eqref{Term2} imply the desired local Lipschitz bound.
\item[$\bullet$ \textit{Wave maps and Yang-Mills in case $k > \frac{d}{2}$.}] In this case, the proof is based on
\begin{align}
\begin{split}
\label{LocLipx}
\| \mathbf{N}(\mathbf{f}) - \mathbf{N}(\mathbf{g}) \|_{\mathfrak{H}^{k}(\BB^{d}_{R})} &\leq \| \mathbf{N}(\mathbf{f}) - \mathbf{N}(\mathbf{g}) \|_{\mathfrak{H}^{k+1}(\BB^{d}_{R})} \\&\lesssim
\| N(f_{1}) - N(g_{1}) \|_{H^{k}(\BB^{d}_{R})}
\end{split}
\end{align}
and the algebra property of the Sobolev embedding $H^{k}(\BB^{d}_{R}) \hookrightarrow C(\overline{\BB^{d}_{R}})$. For wave maps, we use the representation \eqref{WMNonlinearityRepresentation} and H\"{o}lder and Sobolev estimates as in \cite[p.~221 f.]{MR2918544} to get
\begin{align}
\begin{split}
\label{LocLipWM}
&
\| N_{\mathsf{WM}}(f_{1}) - N_{\mathsf{WM}}(g_{1}) \|_{H^{k}(\BB^{d}_{R})} \\&\indent\lesssim
\| f_{1} - g_{1} \|_{H^{k}(\BB^{d}_{R})} \big( \| f_{1} \|_{H^{k}(\BB^{d}_{R})} + \| g_{1} \|_{H^{k}(\BB^{d}_{R})} \big) \int_{0}^{1} \int_{0}^{1} \big\| F_{\mathsf{WM}}'' \big( \,.\,, C(\,.\,;s,t) \big) \big\|_{H^{k}(\BB^{d}_{R})} \dd t \dd s \\&\indent\lesssim
\| f_{1} - g_{1} \|_{H^{k}(\BB^{d}_{R})} \big( \| f_{1} \|_{H^{k}(\BB^{d}_{R})} + \| g_{1} \|_{H^{k}(\BB^{d}_{R})} \big)
\end{split}
\end{align}
for all $\mathbf{f},\mathbf{g} \in C^{\infty} ( \overline{\BB^{d}_{R}} )^{2}$ with $\| \mathbf{f} \|_{\mathfrak{H}^{k}(\BB^{d}_{R})}, \| \mathbf{g} \|_{\mathfrak{H}^{k}(\BB^{d}_{R})} \leq M$. The local Lipschitz estimate for the Yang-Mills nonlinearity follows from
\begin{align}
\begin{split}
\label{LocLipYM}
\| N_{\mathsf{YM}}(f_{1}) &- N_{\mathsf{YM}}(g_{1}) \|_{H^{k}(\BB^{d}_{R})} \lesssim \| f_{1}^{2} - g_{1}^{2} \|_{H^{k}(\BB^{d}_{R})} + \| f_{1}^{3} - g_{1}^{3} \|_{H^{k}(\BB^{d}_{R})} \\&\indent\lesssim
\| f_{1} - g_{1} \|_{H^{k}(\BB^{d}_{R})} \big( \| f_{1} \|_{H^{k}(\BB^{d}_{R})} + \| g_{1} \|_{H^{k}(\BB^{d}_{R})} \big) \\&\indent\indent+
\| f_{1} - g_{1} \|_{H^{k}(\BB^{d}_{R})} \big( \| f_{1} \|_{H^{k}(\BB^{d}_{R})}^{2} + \| f_{1} \|_{H^{k}(\BB^{d}_{R})} \| g_{1} \|_{H^{k}(\BB^{d}_{R})} + \| g_{1} \|_{H^{k}(\BB^{d}_{R})}^{2} \big) \,.
\end{split}
\end{align}
\end{enumerate}
The rest of the lemma is a continuous extension argument.
\end{proof}
We also prove another estimate for the nonlinearity that allows to upgrade the regularity of solutions later.
\begin{lemma}
\label{NonlinEstimates}
Let $R>0$ and $d,k_{d}\in\NN$ with
\begin{equation*}
k_{d} = \frac{d}{2} > 2 \text{ for wave maps}
\qquad\text{and}\qquad
\frac{d}{2} \geq k_{d} \geq \frac{d-1}{2} > \frac{5}{2} \text{ for Yang-Mills.}
\end{equation*}
Then,
\begin{equation*}
\| \mathbf{N} (\mathbf{f}) \|_{\mathfrak{H}^{k_{d}+1}(\BB^{d}_{R})} \lesssim
\renewcommand{\arraystretch}{1.2}
\left\{
\begin{array}{ll}
\big( 1 + \| \mathbf{f} \|_{\mathfrak{H}^{k_{d}}(\BB^{d}_{R})}^{k_{d}+1} \big) \| \mathbf{f} \|_{\mathfrak{H}^{k_{d}}(\BB^{d}_{R})} \| \mathbf{f} \|_{\mathfrak{H}^{k_{d}+1}(\BB^{d}_{R})} & \text{for wave maps,} \\
\big( 1 + \| \mathbf{f} \|_{\mathfrak{H}^{k_{d}}(\BB^{d}_{R})} \big) \| \mathbf{f} \|_{\mathfrak{H}^{k_{d}}(\BB^{d}_{R})} \| \mathbf{f} \|_{\mathfrak{H}^{k_{d}+1}(\BB^{d}_{R})} & \text{for Yang-Mills,}
\end{array}
\right.
\end{equation*}
for all $\mathbf{f} \in \mathfrak{H}^{k_{d}+1}_{\mathrm{rad}}(\BB^{d}_{R})$. Moreover, if $k > \frac{d}{2}$, then $\mathbf{N} (\mathbf{f}) \in \mathfrak{H}^{k+1}_{\mathrm{rad}}(\BB^{d}_{R}) \subset \mathfrak{H}^{k}_{\mathrm{rad}}(\BB^{d}_{R})$ for all $\mathbf{f} \in \mathfrak{H}^{k}_{\mathrm{rad}}(\BB^{d}_{R})$.
\end{lemma}
\begin{proof}
We have
\begin{equation*}
\| \mathbf{N}(\mathbf{f}) \|_{\mathfrak{H}^{k+1}(\BB^{d}_{R})} = \Big\| \frac{c}{w} N(f_{1}) \Big\|_{H^{k}(\BB^{d}_{R})} \lesssim \| N(f_{1}) \|_{H^{k}(\BB^{d}_{R})}
\end{equation*}
for all $\mathbf{f} \in C^{\infty} ( \overline{\BB^{d}_{R}} )^{2}$. We distinguish between the cases $\frac{d-1}{2} \leq k_{d} \leq \frac{d}{2}$ for the Yang-Mills nonlinearity and $k_{d} = \frac{d}{2}$ for the wave maps nonlinearity.
\begin{enumerate}[wide,itemsep=1em,topsep=1em]
\item[$\bullet$ \textit{Yang-Mills in case $\frac{d-1}{2} \leq k_{d} \leq \frac{d}{2}$.}] First, suppose that $d\in\NN$ is odd, so $k_{d} = \frac{d-1}{2}$. For the Yang-Mills nonlinearity \eqref{NYM}, we have
\begin{align*}
\| N_{\mathsf{YM}}(f) \|_{H^{k_{d}}(\BB^{d}_{R})} &\lesssim \| f^{2} \|_{H^{k_{d}}(\BB^{d}_{R})} + \| f^{3} \|_{H^{k_{d}}(\BB^{d}_{R})} \\&\lesssim \sum_{0 \leq |\alpha_{1}| + |\alpha_{2}| \leq k_{d}} \big\| \big( \pd^{\alpha_{1}} f \big) \big( \pd^{\alpha_{2}} f \big) \big\|_{L^{2}(\BB^{d}_{R})} \\&\indent+
\sum_{0 \leq |\alpha_{1}| + |\alpha_{2}| + |\alpha_{3}| \leq k_{d}} \big\| \big( \pd^{\alpha_{1}} f \big) \big( \pd^{\alpha_{2}} f \big) \big( \pd^{\alpha_{3}} f \big) \big\|_{L^{2}(\BB^{d}_{R})}
\end{align*}
for all $f \in C^{\infty} ( \overline{\BB^{d}_{R}} )$. To control the quadratic term, we fix multi-indices $\alpha_{1},\alpha_{2} \in \NN_{0}^{d}$ with $0 \leq |\alpha_{1}| + |\alpha_{2}| \leq k_{d}$ and $|\alpha_{1}| \geq |\alpha_{2}|$. We choose corresponding H\"{o}lder exponents
\begin{equation*}
p_{1} = \frac{2d}{d - 2( k_{d} + 1 - |\alpha_{1}|)} = \frac{2d}{2|\alpha_{1}| - 1} \,, \qquad
p_{2} = \frac{2d}{d - 2( k_{d} - |\alpha_{2}|)} = \frac{2d}{2|\alpha_{2}| + 1} \,.
\end{equation*}
This ensures $\frac{1}{p_{1}} + \frac{1}{p_{2}} \leq \frac{1}{2}$ and H\"{o}lder's inequality and Sobolev embedding give
\begin{align*}
\big\| \big( \pd^{\alpha_{1}} f \big) \big( \pd^{\alpha_{2}} f \big) \big\|_{L^{2}(\BB^{d}_{R})} &\leq \big\| \pd^{\alpha_{1}} f \big\|_{L^{p_{1}}(\BB^{d}_{R})} \big\| \pd^{\alpha_{2}} f \big\|_{L^{p_{2}}(\BB^{d}_{R})} \\&\lesssim
\| f \|_{H^{k_{d}}(\BB^{d}_{R})} \| f \|_{H^{k_{d}+1}(\BB^{d}_{R})}
\end{align*}
for all $f \in C^{\infty} ( \overline{\BB^{d}_{R}} )$. For the cubic term, we fix multi-indices $\alpha_{1},\alpha_{2},\alpha_{3} \in \NN_{0}^{d}$ with $0 \leq |\alpha_{1}| + |\alpha_{2}| + |\alpha_{3}| \leq k_{d}$ and $|\alpha_{1}| \geq |\alpha_{2}| \geq |\alpha_{3}|$. We choose corresponding H\"{o}lder exponents $p_{1},p_{2}$ as above and
\begin{equation*}
p_{3} = \frac{2d}{d - 2( k_{d} - |\alpha_{3}|)} = \frac{2d}{2|\alpha_{3}| + 1} \,.
\end{equation*}
As before, $\frac{1}{p_{1}} + \frac{1}{p_{2}} + \frac{1}{p_{3}} \leq \frac{1}{2}$ and H\"{o}lder's inequality and Sobolev embedding yield
\begin{align*}
\big\| \big( \pd^{\alpha_{1}} f \big) \big( \pd^{\alpha_{2}} f \big) \big( \pd^{\alpha_{3}} f \big) \big\|_{L^{2}(\BB^{d}_{R})} &\leq \big\| \pd^{\alpha_{1}} f \big\|_{L^{p_{1}}(\BB^{d}_{R})} \big\| \pd^{\alpha_{2}} f \big\|_{L^{p_{2}}(\BB^{d}_{R})} \big\| \pd^{\alpha_{3}} f \big\|_{L^{p_{3}}(\BB^{d}_{R})} \\&\lesssim
\| f \|_{H^{k_{d}}(\BB^{d}_{R})}^{2} \| f \|_{H^{k_{d}+1}(\BB^{d}_{R})}
\end{align*}
for all $f \in C^{\infty} ( \overline{\BB^{d}_{R}} )$. These estimates imply the result for the Yang-Mills nonlinearity. In case $k_{d} = \frac{d}{2}$, the result follows directly from the inequality \eqref{NonlinearEstimate} below.
\item[$\bullet$ \textit{Wave maps in case $k_d = \frac{d}{2}$.}] We use the representation
\begin{equation*}
N_{\mathsf{WM}}(f)(\xi) =
\int_{0}^{1} f(\xi)^{2} F_{\mathsf{WM}}'' \big( \xi, \psi^{\ast}(-h(\xi),\xi) + t f(\xi) \big) ( 1 - t ) \dd t
\end{equation*}
from \cite[Lemma 4.1]{MR4778061} and get
\begin{equation*}
\| N_{\mathsf{WM}}(f) \|_{H^{k_{d}}(\BB^{d}_{R})}
\lesssim
\int_{0}^{1}
\sum_{0 \leq |\alpha| + |\beta| + |\gamma| \leq k_{d}}
\Big\| \big( \pd^{\alpha} f \big) \big( \pd^{\beta} f \big) \big( \pd^{\gamma} F_{\mathsf{WM}}'' \big( \,.\,, C(\,.\,;t) \big) \big) \Big\|_{L^{2}(\BB^{d}_{R})} \dd t
\end{equation*}
for all $f \in C^{\infty} ( \overline{\BB^{d}_{R}} )$, where $C(\xi;t) = \psi^{\ast}(-h(\xi),\xi) + t f(\xi)$. We recall the pointwise estimates \eqref{pdxalphapdzellF''} for the wave maps nonlinearity and note that the asserted bound in the lemma follows from estimating terms of the form
\begin{equation*}
\Big\| \big| f \big|^{m} \prod_{i=1}^{\ell} \big| \pd^{\alpha_{i}} f \big| \Big\|_{L^{2}(\BB^{d}_{R})} \,,
\end{equation*}
for $\ell,m\in\NN$ and $\alpha_{i} \in \NN_{0}^{d}$ with $1 \leq \ell \leq k_{d}$, $0 \leq m \leq k_{d} - \ell$ and $0 \leq |\alpha_{1}| + \ldots + |\alpha_{\ell}| \leq k_{d}$. By H\"{o}lder's inequality for exponents with
\begin{equation}
\label{HoelderNumerology}
\frac{m}{p} + \sum_{i=1}^{\ell} \frac{1}{p_{i}} \leq \frac{1}{2} \,,
\end{equation}
we have
\begin{equation*}
\Big\| \big| f \big|^{m} \prod_{i=1}^{\ell} \big| \pd^{\alpha_{i}} f \big| \Big\|_{L^{2}(\BB^{d}_{R})} \leq
\big\| f \big\|_{L^{p}(\BB^{d}_{R})}^{m} \prod_{i=1}^{\ell} \big\| \pd^{\alpha_{i}} f \big\|_{L^{p_{i}}(\BB^{d}_{R})} \,.
\end{equation*}
We order the multi-indices so that $|\alpha_{\ell}| \geq \ldots \geq |\alpha_{1}|$. If $|\alpha_{\ell}| > 1$, set $p_{\ell} = d/(|\alpha_{\ell}|-1)$ and $p_{i} = d/|\alpha_{i}|$ for the remaining indices if $\alpha_{i} \neq 0$ and $p_{i} = p$ if $\alpha_{i} = 0$. Then, $p > 2$ can be chosen such that equality is attained in \eqref{HoelderNumerology}. If $|\alpha_{\ell}| = 1$, choose $p_{\ell} = p$ and the remaining indices as above. Then, the Sobolev embedding \eqref{SobolevEmbedding} yields
\begin{equation}
\label{NonlinearEstimate}
\Big\| \big| f \big|^{m} \prod_{i=1}^{\ell} \big| \pd^{\alpha_{i}} f \big| \Big\|_{L^{2}(\BB^{d}_{R})} \lesssim
\| f \|_{H^{k_{d}}(\BB^{d}_{R})}^{m+\ell-1} \| f \|_{H^{k_{d}+1}(\BB^{d}_{R})}
\end{equation}
for all $f \in C^{\infty} ( \overline{\BB^{d}_{R}} )$.
\end{enumerate}
The final statement of the lemma in case $k > \frac{d}{2}$ follows from the estimates \eqref{LocLipx}, \eqref{LocLipWM}, \eqref{LocLipYM} in the proof of \Cref{LocLip}.
\end{proof}
\subsection{Notion of solution}
\Cref{LocLip} gives control on the mapping properties of the nonlinearity in Sobolev spaces. In particular, Duhamel's principle can be meaningfully applied to the semigroup from \Cref{LinearFlow} which allows to adopt the notion of solution from \cite[p.~184, Definition 1.1]{MR710486}.
\begin{definition}
\label{NotionSolution}
Let $d,k\in\NN$ and $R \geq R_{0}$ with
\begin{equation*}
k \geq \frac{d}{2} > 2 \text{ for wave maps}
\qquad\text{and}\qquad
k \geq \frac{d-1}{2} > \frac{5}{2} \text{ for Yang-Mills.}
\end{equation*}
Let
\begin{enumerate}[itemsep=1em,topsep=1em]
\item[] $\mathbf{S}: [0,\infty) \rightarrow \mathfrak{L}( \mathfrak{H}^{k}_{\mathrm{rad}}(\BB^{d}_{R}) )$ be the semigroup from \Cref{LinearFlow},
\item[] $\mathbf{L}: \mathfrak{D}(\mathbf{L}) \subset \mathfrak{H}^{k}_{\mathrm{rad}}(\BB^{d}_{R}) \rightarrow \mathfrak{H}^{k}_{\mathrm{rad}}(\BB^{d}_{R})$ be its generator from \Cref{GeneratorLinearizedWave},
\item[] $\mathbf{N}: \mathfrak{H}^{k}_{\mathrm{rad}}(\BB^{d}_{R}) \rightarrow \mathfrak{H}^{k}_{\mathrm{rad}}(\BB^{d}_{R})$ be the nonlinearity from \Cref{LocLip}.
\end{enumerate}
Let $\mathbf{f} \in \mathfrak{H}^{k}_{\mathrm{rad}}(\BB^{d}_{R})$.
\medskip\\
A map $\mathbf{u} \in C \big( [0,\infty), \mathfrak{H}^{k}_{\mathrm{rad}}(\BB^{d}_{R}) \big) \cap C^{1} \big( (0,\infty), \mathfrak{H}^{k}_{\mathrm{rad}}(\BB^{d}_{R}) \big)$ which satisfies $\mathbf{u}(\tau) \in \mathfrak{D}(\mathbf{L})$ and
\begin{equation}
\label{AbstractCauchyProblem}
\renewcommand{\arraystretch}{1.2}
\left\{
\begin{array}{rcl}
\pd_{\tau} \mathbf{u}(\tau) &=& \mathbf{L} \mathbf{u}(\tau) + \mathbf{N} ( \mathbf{u}(\tau) ) \,, \\
\mathbf{u}(0) &=& \mathbf{f} \,,
\end{array}
\right.
\end{equation}
for all $\tau > 0$, is called a \emph{classical solution} to the abstract Cauchy problem.
\medskip\\
A map $\mathbf{u} \in C \big( [0,\infty), \mathfrak{H}^{k}_{\mathrm{rad}}(\BB^{d}_{R}) \big)$ which satisfies
\begin{equation*}
\mathbf{u}(\tau) = \mathbf{S}(\tau) \mathbf{f} + \int_{0}^{\tau} \mathbf{S}(\tau-\tau') \mathbf{N}( \mathbf{u}(\tau') ) \dd\tau'
\end{equation*}
for all $\tau > 0$, is called a \emph{mild solution} to the abstract Cauchy problem.
\end{definition}
We recall that the abstract Cauchy problem \eqref{AbstractCauchyProblem} is locally well-posed and obeys a blowup alternative.
\begin{proposition}
\label{LocalMildSolution}
Let $M > 0$. Then, there is a $t_{+} \in (0,\infty]$ such that for all $\mathbf{f} \in \mathfrak{H}^{k}_{\mathrm{rad}}(\BB^{d}_{R})$ with $\| \mathbf{f} \|_{\mathfrak{H}^{k}(\BB^{d}_{R})} \leq M$ there exists a unique local-in-time mild solution $\mathbf{u} \in C \big( [0,t_{+}), \mathfrak{H}^{k}_{\mathrm{rad}}(\BB^{d}_{R}) \big)$ in the sense of \Cref{NotionSolution} to the abstract Cauchy problem \eqref{AbstractCauchyProblem}. If $t_{+} < \infty$, then
\begin{equation*}
\limsup_{t \nearrow t_{+}} \| \mathbf{u}(\tau) \|_{\mathfrak{H}^{k}(\BB^{d}_{R})} = \infty \,.
\end{equation*}
\end{proposition}
\begin{proof}
See \cite[p.~185 f., Theorem 1.4]{MR710486}.
\end{proof}
We prove that for smooth initial data the mild solution is in fact a jointly smooth classical solution. This crucially uses mapping properties of the linearized flow in Sobolev spaces, which are inherited from our construction of the free wave flow in \Cref{TheGenerationTHM}.
\begin{proposition}
\label{SmoothClassicalSolution}
Suppose $\mathbf{u} \in C \big( [0,\infty), \mathfrak{H}^{k}_{\mathrm{rad}}(\BB^{d}_{R}) \big)$ is a mild solution to the abstract Cauchy problem in the sense of \Cref{NotionSolution} for some smooth initial datum $\mathbf{f} \in C^{\infty}_{\mathrm{rad}} ( \overline{ \BB^{d}_{R}} )^{2}$. Then $\mathbf{u} \in C^{\infty} \big( \overline{ (0,\infty) \times \BB^{d}_{R}} \big)^{2}$ and it is the unique classical solution to the abstract Cauchy problem.
\end{proposition}
\begin{proof}
First, consider $k_{d} = \frac{d}{2}$ for wave maps and $\frac{d-1}{2} \leq k_{d} \leq \frac{d}{2}$ for Yang-Mills. Let $\mathbf{u} \in C \big( [0,\infty), \mathfrak{H}^{k_{d}}_{\mathrm{rad}}(\BB^{d}_{R}) \big)$ such that
\begin{equation*}
\mathbf{u}(\tau) = \mathbf{S}(\tau) \mathbf{f} + \int_{0}^{\tau} \mathbf{S}(\tau-\tau') \mathbf{N}( \mathbf{u}(\tau') ) \dd\tau'
\end{equation*}
for all $\tau \geq 0$. By uniqueness, $\mathbf{u}$ coincides with the local mild solution in $C \big( [0,t_{+}), \mathfrak{H}^{k_{d}+1}_{\mathrm{rad}}(\BB^{d}_{R}) \big)$ from \Cref{LocalMildSolution} on its maximal interval of existence. To show that this solution is actually a global one, we assume to the contrary that $t_{+} < \infty$. Then, since $\| \mathbf{u}(\tau) \|_{\mathfrak{H}^{k_{d}}(\BB^{d}_{R})}$ is uniformly bounded on $[0,t_{+}]$, we can exploit the nonlinear estimates from \Cref{NonlinEstimates} to obtain from the above fixed-point equation
\begin{align*}
\| \mathbf{u}(\tau) \|_{\mathfrak{H}^{k_{d}+1}(\BB^{d}_{R})} &\lesssim \| \mathbf{S}(\tau) \mathbf{f} \|_{\mathfrak{H}^{k_{d}+1}(\BB^{d}_{R})} + \int_{0}^{\tau} \| \mathbf{S}(\tau-\tau') \mathbf{N}( \mathbf{u}(\tau') ) \|_{\mathfrak{H}^{k_{d}+1}(\BB^{d}_{R})} \dd\tau' \\&\lesssim
1 + \int_{0}^{\tau} \| \mathbf{u}(\tau') \|_{\mathfrak{H}^{k_{d}+1}(\BB^{d}_{R})} \dd\tau'
\end{align*}
for all $\tau \in [0,t_{+})$. Thus, Gr\"{o}nwall's inequality shows that $\| \mathbf{u}(\tau) \|_{\mathfrak{H}^{k_{d}+1}(\BB^{d}_{R})}$ is uniformly bounded on the maximal interval of existence, which contradicts the blowup alternative in \Cref{LocalMildSolution}. Therefore, $\mathbf{u} \in C \big( [0,\infty), \mathfrak{H}^{k_{d}+1}_{\mathrm{rad}}(\BB^{d}_{R}) \big)$. Now, for any mild solution in $C \big( [0,\infty), \mathfrak{H}^{k}_{\mathrm{rad}}(\BB^{d}_{R}) \big)$ with $k > \frac{d}{2}$, we conclude from the fixed-point equation together with the restriction property \cite[Lemma C.1]{2022arXiv220706952G} satisfied by the semigroups $\mathbf{S}: [0,\infty) \rightarrow \mathfrak{L}( \mathfrak{H}^{k}_{\mathrm{rad}}(\BB^{d}_{R}) )$ from \Cref{LinearFlow} and the second part of \Cref{NonlinEstimates} that $\mathbf{u}(\tau) \in \mathfrak{H}^{k}_{\mathrm{rad}}(\BB^{d}_{R})$ for all $k > \frac{d}{2}$ and all $\tau \geq 0$. Now, Sobolev embedding yields $\mathbf{u}(\tau) \in C^{\infty}_{\mathrm{rad}} ( \overline{ \BB^{d}_{R}} )^{2}$ for all $\tau \geq 0$. To prove that this solution is classical and jointly smooth on $(0,\infty) \times \BB^{d}_{R}$, we first use \cite[p.~189, Theorem 1.6]{MR710486} to infer that the mild solution is classical in the sense of \Cref{NotionSolution}. So, the fixed-point equation can be differentiated with respect to $\tau$, which gives
\begin{equation*}
\pd_{\tau} \mathbf{u}(\tau) = \mathbf{S}(\tau) \mathbf{L} \mathbf{f} + \int_{0}^{\tau} \mathbf{S}(\tau-\tau') \mathbf{L} \mathbf{N}( \mathbf{u}(\tau') ) \dd\tau' + \mathbf{N}( \mathbf{u}(\tau) )
\end{equation*}
and so the restriction property \cite[Lemma C.1]{2022arXiv220706952G} and smoothness of $\mathbf{u}(\tau)$ give $\pd_{\tau} \mathbf{u}(\tau) \in C^{\infty}_{\mathrm{rad}}(\overline{\BB^{d}_{R}})$ for all $\tau \geq 0$. By induction, we conclude for any $\ell \in \NN$ that $\pd_{\tau}^{\ell} \mathbf{u}(\tau) \in C^{\infty} ( \overline{\BB^{d}_{R}} )$. An application of Schwarz's theorem \cite[p.~235, Theorem 9.41]{MR385023} yields $\mathbf{u} \in C^{\infty} \big( \overline{ (0,\infty) \times \BB^{d}_{R}} \big)^{2}$.
\end{proof}
\subsection{Stabilized evolution}
Our next goal is to prove the existence of global-in-time stable solutions to the abstract Cauchy problem \eqref{AbstractCauchyProblem}. For this, we incorporate the decay of the linearized wave flow on the stable subspace into a function space for our main fixed-point argument.
\begin{definition}
Let $\omega > 0$ be the growth constant from \Cref{LinearizedEvolution}. We define a Banach space $\big( \mathfrak{X}^{k}_{\mathrm{rad}}( \BB^{d}_{R}), \| \,.\, \|_{\mathfrak{X}^{k}( \BB^{d}_{R})} \big)$ by
\begin{align*}
\mathfrak{X}^{k}_{\mathrm{rad}}( \BB^{d}_{R}) &\coloneqq \big\{ \mathbf{u} \in C \big( [0,\infty), \mathfrak{H}^{k}_{\mathrm{rad}}(\BB^{d}_{R}) \big) \mid \| \mathbf{u}(\tau) \|_{\mathfrak{H}^{k}(\BB^{d}_{R})} \lesssim \ee^{-\omega\tau} \text{ for all } \tau \geq 0 \big\} \,, \\
\| \mathbf{u} \|_{\mathfrak{X}^{k}( \BB^{d}_{R})} &\coloneqq \sup_{\tau\in[0,\infty)} \Big( \ee^{\omega\tau} \| \mathbf{u}(\tau) \|_{\mathfrak{H}^{k}(\BB^{d}_{R})} \Big) \,.
\end{align*}
\end{definition}
In this setting, the Lyapunov-Perron method provides an approach for suppressing the linear instabilities in the nonlinear evolution by means of a correction term.
\begin{definition}
Let $\mathbf{P}_{1} \in \mathfrak{L}( \mathfrak{H}^{k}_{\mathrm{rad}}(\BB^{d}_{R}) )$ be the spectral projection from \Cref{RieszProjection}. We define
\begin{equation*}
\mathbf{C}: \mathfrak{H}^{k}_{\mathrm{rad}}(\BB^{d}_{R}) \times \mathfrak{X}^{k}_{\mathrm{rad}}(\BB^{d}_{R}) \rightarrow \mathfrak{H}^{k}_{\mathrm{rad}}(\BB^{d}_{R}) \,, \quad \mathbf{C} ( \mathbf{f} , \mathbf{u} ) = \mathbf{P}_{1} \mathbf{f} + \mathbf{P}_{1} \int_{0}^{\infty} \ee^{-\tau'} \mathbf{N}(\mathbf{u}(\tau')) \dd\tau' \,.
\end{equation*}
\end{definition}
Modifying the initial data by subtracting the correction term leads for small data to a well-posed global-in-time nonlinear evolution of exponentially stable solutions.
\begin{proposition}
\label{StabilizedEvolution}
Let $d,k\in\NN$ and $R \geq R_{0}$ be as in \Cref{NotionSolution}. Then, there are $0 < \delta_{0} < 1$ and $C_{0} > 1$ such that for all $0 < \delta \leq \delta_{0}$, $C \geq C_{0}$ and all $\mathbf{f} \in \mathfrak{H}^{k}_{\mathrm{rad}}(\BB^{d}_{R})$ with $\| \mathbf{f} \|_{\mathfrak{H}^{k}(\BB^{d}_{R})} \leq \frac{\delta}{C}$ there is a unique $\mathbf{u}_{\mathbf{f}} \in \mathfrak{X}^{k}_{\mathrm{rad}}(\BB^{d}_{R})$ with $\| \mathbf{u}_{\mathbf{f}} \|_{\mathfrak{X}^{k}(\BB^{d}_{R})} \leq \delta$ and
\begin{equation*}
\mathbf{u}_{\mathbf{f}}(\tau) = \mathbf{S}(\tau) ( \mathbf{f} - \mathbf{C}(\mathbf{f}, \mathbf{u}_{\mathbf{f}} ) ) + \int_{0}^{\tau} \mathbf{S}(\tau-\tau') \mathbf{N}( \mathbf{u}_{\mathbf{f}}(\tau') ) \dd\tau'
\end{equation*}
for all $\tau \geq 0$. Moreover, we have
\begin{equation*}
\| \mathbf{u}_{\mathbf{f}} - \mathbf{u}_{\mathbf{g}} \|_{\mathfrak{X}^{k}(\BB^{d}_{R})} \lesssim \| \mathbf{f} - \mathbf{g} \|_{\mathfrak{H}^{k}(\BB^{d}_{R})}
\end{equation*}
for all $\mathbf{f},\mathbf{g} \in \mathfrak{H}^{k}_{\mathrm{rad}}(\BB^{d}_{R})$ with $\| \mathbf{f} \|_{\mathfrak{H}^{k}(\BB^{d}_{R})}, \| \mathbf{g} \|_{\mathfrak{H}^{k}(\BB^{d}_{R})} \leq \frac{\delta}{C}$.
\end{proposition}
\begin{proof}
For now, let $\delta_{0} > 0$ and $C_{0} > 1$ be arbitrary but fixed and $0 < \delta \leq \delta_{0}$ and $C \geq C_{0}$. We set
\begin{equation*}
\mathfrak{B}_{\delta} = \big\{ \mathbf{u} \in \mathfrak{X}^{k}_{\mathrm{rad}}(\BB^{d}_{R}) \mid \| \mathbf{u} \|_{\mathfrak{X}^{k}(\BB^{d}_{R})} \leq \delta \big\}
\end{equation*}
for a closed ball in our Banach space. For any $\mathbf{f} \in \mathfrak{H}^{k}_{\mathrm{rad}}(\BB^{d}_{R})$ with $\| \mathbf{f} \|_{\mathfrak{H}^{k}(\BB^{d}_{R})} \leq \frac{\delta}{C}$ and $\mathbf{u} \in \mathfrak{X}^{k}_{\mathrm{rad}}(\BB^{d}_{R})$ we define
\begin{equation*}
\mathbf{K}_{\mathbf{f}}(\mathbf{u})(\tau) = \mathbf{S}(\tau) ( \mathbf{f} - \mathbf{C} ( \mathbf{f}, \mathbf{u} ) ) + \int_{0}^{\tau} \mathbf{S}(\tau-\tau') \mathbf{N}( \mathbf{u}(\tau') ) \dd\tau' \,.
\end{equation*}
Thanks to the correction term, we can conclude from \Cref{LinearizedEvolution} and \Cref{LocLip} the estimates
\begin{align*}
\| \mathbf{K}_{\mathbf{f}}(\mathbf{u})(\tau) \|_{\mathfrak{H}^{k}(\BB^{d}_{R})} &\lesssim \frac{\delta}{C} \ee^{-\omega\tau} + \delta^{2} \ee^{-\omega\tau} \,, \\
\| \mathbf{K}_{\mathbf{f}}(\mathbf{u})(\tau) - \mathbf{K}_{\mathbf{f}}(\mathbf{v})(\tau) \|_{\mathfrak{H}^{k}(\BB^{d}_{R})} &\lesssim \delta \| \mathbf{u}(\tau) - \mathbf{v}(\tau) \|_{\mathfrak{H}^{k}(\BB^{d}_{R})} \ee^{-\omega\tau} \,,
\end{align*}
for all $\mathbf{u},\mathbf{v} \in \mathfrak{X}^{k}_{\mathrm{rad}}(\BB^{d}_{R})$ with $\| \mathbf{u} \|_{\mathfrak{X}^{k}(\BB^{d}_{R})}, \| \mathbf{v} \|_{\mathfrak{X}^{k}(\BB^{d}_{R})} \leq \delta$. Hence, there are $\delta_{0} > 0$ and $C_{0} > 1$ such that the map
\begin{equation*}
\mathbf{K}_{\mathbf{f}}: \mathfrak{B}_{\delta} \rightarrow \mathfrak{B}_{\delta} \,, \qquad \mathbf{u} \mapsto \mathbf{K}_{\mathbf{f}} ( \mathbf{u} ) \,,
\end{equation*}
is well-defined and a contraction for any $\mathbf{f} \in \mathfrak{H}^{k}_{\mathrm{rad}}(\BB^{d}_{R})$ with $\| \mathbf{f} \|_{\mathfrak{H}^{k}(\BB^{d}_{R})} \leq \frac{\delta}{C}$. According to Banach's fixed-point theorem, there is a unique $\mathbf{u}_{\mathbf{f}} \in \mathfrak{B}_{\delta}$ such that $\mathbf{K}_{\mathbf{f}}(\mathbf{u}_{\mathbf{f}}) = \mathbf{u}_{\mathbf{f}}$. By Gr\"{o}nwall's inequality, this solution is unique in $\mathfrak{X}^{k}_{\mathrm{rad}}(\BB^{d}_{R})$. Finally, the above estimates also imply
\begin{equation*}
\| \mathbf{u}_{\mathbf{f}}(\tau) - \mathbf{u}_{\mathbf{g}}(\tau) \|_{\mathfrak{H}^{k}(\BB^{d}_{R})} \lesssim \| \mathbf{f} - \mathbf{g} \|_{\mathfrak{H}^{k}(\BB^{d}_{R})} \ee^{-\omega\tau}
\end{equation*}
for all $\mathbf{f}, \mathbf{g} \in \mathfrak{H}^{k}_{\mathrm{rad}}(\BB^{d}_{R})$ with $\| \mathbf{f} \|_{\mathfrak{H}^{k}(\BB^{d}_{R})}, \| \mathbf{g} \|_{\mathfrak{H}^{k}(\BB^{d}_{R})} \leq \frac{\delta}{C}$ and all $\tau \geq 0$.
\end{proof}
\subsection{Data along the initial hypersurface}
\label{PreparationData}
The initial data in \Cref{THM} are prescribed at time $t=0$ and have to be transported to data for the abstract Cauchy problem \eqref{AbstractCauchyProblem} along an initial hypersurface at $\tau = 0$ in similarity coordinates. We achieve this by adapting our coordinate system in such a way that the initial hypersurface directly picks up the perturbations for the blowup data.
\begin{lemma}
\label{InitialHypersurface}
Let $r>0$, $\varepsilon>0$ and $R \geq r+\varepsilon$, $T \geq 1 - \frac{\varepsilon}{r + \varepsilon}$. Let $h \in C^{\infty}_{\mathrm{rad}}(\RR^{d})$ and ${\chi\mathstrut}_{T} : \RR \times \RR^{d} \rightarrow \RR^{1,d}$ be as in \Cref{SimilarityCoordinates} with the additional property $h \equiv -1 \text{ in } \BB^{d}_{r + \varepsilon}$. Consider the domain
\begin{equation*}
\Lambda^{1,d}(r) = \big\{ (t,x) \in \RR^{1,d} \mid 0 \leq t \leq |x| - r \big\}
\end{equation*}
and the spacelike hypersurface
\begin{equation*}
\Sigma^{1,d}_{T,R}(0) = {\chi\mathstrut}_{T} \big( \{0\} \times \BB^{d}_{R} \big) \,.
\end{equation*}
Then
\begin{equation*}
\{ 0 \} \times \BB^{d}_{r} \subset \Sigma^{1,d}_{T,R}(0)
\qquad\text{and}\qquad
\Sigma^{1,d}_{T,R}(0) \setminus \big( \{ 0 \} \times \BB^{d}_{r} \big) \subset \Lambda^{1,d}(r) \,.
\end{equation*}
\end{lemma}
\begin{proof}
Note that $\Sigma^{1,d}_{T,R}(0) = \{ (t_{0}(x),x) \in \RR^{1,d} \mid x \in \BB^{d}_{RT} \}$ is the graph of the function given by $t_{0}(x) = T + T h\big( \tfrac{x}{T}\big)$. If $|x| < r$, then $|x|/T < r/T \leq r + \varepsilon$ and therefore $h\big( \tfrac{x}{T}\big) = -1$. So $t_{0}(x) = 0$ for all $|x| < r$ which yields the inclusion $\{ 0 \} \times \BB^{d}_{r} \subset \Sigma^{1,d}_{T,R}(0)$. Moreover, the assumptions on $h$ imply the inequality
\begin{equation*}
0 \leq t_{0}(x) = T + T h\big( \tfrac{x}{T}\big) < T + T \big( \tfrac{|x|-r}{T} - 1 \big) = |x| - r
\end{equation*}
for all $|x| \geq r$, which yields the other inclusion $\Sigma^{1,d}_{T,R}(0) \setminus \big( \{ 0 \} \times \BB^{d}_{r} \big) \subset \Lambda^{1,d}_{r}$.
\end{proof}
Now, if
\begin{equation*}
\psi(0,\,.\,) = \psi^{\ast}_{1}(0,\,.\,) + f
\qquad\text{and}\qquad
(\pd_{0}\psi)(0,\,.\,) = (\pd_{0}\psi^{\ast}_{1})(0,\,.\,) + g
\end{equation*}
are data at $t = 0$ for the semilinear wave equation \eqref{RadialSemilinearWaveEquation} and the perturbations $f,g$ are supported in $\BB^{d}_{r}$, then the corresponding solution $\psi$ has to coincide with $\psi^{\ast}_{1}$ in $\Lambda^{1,d}_{r}$ by finite speed of propagation. Therefore, the solution $\psi$ is known along the initial hypersurface $\Sigma^{1,d}_{T,R}(0)$ from \Cref{InitialHypersurface} beforehand. Evaluating the corresponding rescaled evolution variables $u_{1},u_{2}$ from \Cref{EvoVar} at $\tau=0$ then yields precisely the correct data for the abstract Cauchy problem, which are provided by the following initial data operator.
\begin{definition}
\label{InitialDataOperator}
Let $r>0$, $\varepsilon>0$ and $R \geq r+\varepsilon$. Let $h \in C^{\infty}_{\mathrm{rad}}(\RR^{d})$ be as in \Cref{SimilarityCoordinates} with $h \equiv -1 \text{ in } \BB^{d}_{r+\varepsilon}$. We define the map
\begin{equation*}
\mathbf{U}: \big( C^{\infty}_{\mathrm{rad}}(\RR^{d}) \cap C^{\infty}_{\mathrm{c}}(\BB^{d}_{r}) \big)^{2} \times \big[ 1 - \tfrac{\varepsilon}{r + \varepsilon}, 1 + \tfrac{\varepsilon}{r + \varepsilon} \big] \rightarrow C^{\infty}_{\mathrm{rad}}(\overline{\BB^{d}_{R}})^{2} \,, \qquad
\mathbf{U}(\bm{f},T) = \bm{f}_{T} + \bm{\Psi}^{1}_{T} - \bm{\Psi}^{T}_{T} \,,
\end{equation*}
where
\begin{align*}
\bm{f}_{T}(\xi) =
\begin{bmatrix}
\hfill T^{s} f(T\xi) \\
T^{s+1} g(T\xi)
\end{bmatrix}
\,, \qquad
\mathbf{\Psi}_{T}^{1}(\xi) &=
\begin{bmatrix}
\hfill T^{s} \psi^{\ast}_{1} \big( T + Th(\xi) , T \xi \big) \\
T^{s+1} (\pd_{0}\psi^{\ast}_{1}) \big( T + Th(\xi) , T \xi \big)
\end{bmatrix}
\,, \\
\mathbf{\Psi}_{T}^{T}(\xi) &=
\begin{bmatrix}
\hfill T^{s} \psi^{\ast}_{T} \big( T + Th(\xi) , T \xi \big) \\
T^{s+1} (\pd_{0}\psi^{\ast}_{T}) \big( T + Th(\xi) , T \xi \big)
\end{bmatrix}
\,.
\end{align*}
\end{definition}
This initial data operator takes the following form.
\begin{lemma}
\label{InitialDataExpansion}
Recall the symmetry mode $\mathbf{f}_{1}^{\ast} \in C^{\infty}_{\mathrm{rad}}(\overline{\BB^{d}_{R}})^{2}$ from \Cref{SymmetryMode}. Then, there is a map $\mathbf{r} \in C^{\infty} \big( \overline{ \big( 1 - \tfrac{\varepsilon}{r + \varepsilon}, 1 + \tfrac{\varepsilon}{r + \varepsilon} \big) \times \BB^{d}_{R}} \big)^{2}$ such that
\begin{equation*}
\mathbf{U}(\bm{f},T) = \bm{f}_{T} + (1-T) \mathbf{f}_{1}^{\ast} + (1-T)^{2} \mathbf{r}({T}, \,.\,)
\end{equation*}
for all $\bm{f} \in \big( C^{\infty}_{\mathrm{rad}}(\RR^{d}) \cap C^{\infty}_{\mathrm{c}}(\BB^{d}_{r}) \big)^{2}$ and all $T \in \big[ 1 - \tfrac{\varepsilon}{r + \varepsilon}, 1 + \tfrac{\varepsilon}{r + \varepsilon} \big]$.
\end{lemma}
\begin{proof}
Note with the scaling \ref{F2} and self-similarity of the blowup profile from \ref{F3} that
\begin{equation*}
\mathbf{\Psi}_{T}^{T'}(\xi) \coloneqq
\begin{bmatrix}
\hfill T^{s} \psi^{\ast}_{T'} \big( T + Th(\xi) , T \xi \big) \\
T^{s+1} (\pd_{0}\psi^{\ast}_{T'}) \big( T + Th(\xi) , T \xi \big)
\end{bmatrix}
=
\begin{bmatrix}
\hfill\psi^{\ast} \big( \frac{T'}{T} - 1 - h(\xi) , \xi \big) \\
-(\pd_{0}\psi^{\ast}) \big( \frac{T'}{T} - 1 - h(\xi) , \xi \big)
\end{bmatrix}
\,.
\end{equation*}
Taylor's theorem gives
\begin{equation*}
\mathbf{\Psi}_{T}^{1}(\xi) - \mathbf{\Psi}_{T}^{T}(\xi) = (1-T) \left.\pd_{T'} \mathbf{\Psi}_{T}^{T'}(\xi)\right|_{T'=T} + (1-T)^{2} \mathbf{r}(T, \xi) \,,
\end{equation*}
where
\begin{equation*}
\left.\pd_{T'} \mathbf{\Psi}_{T}^{T'}(\xi)\right|_{T'=T} =
\begin{bmatrix}
\hfill
(\pd_{0}\psi^{\ast}) \big( -h(\xi) , \xi \big) \\
-(\pd_{0}^{2} \psi^{\ast}) \big( -h(\xi) , \xi \big)
\end{bmatrix}
=
\mathbf{f}_{1}^{\ast}(\xi)
\end{equation*}
and the integral remainder
\begin{equation*}
\mathbf{r}(T,\xi) = \int_{0}^{1} \left.\pd_{T'}^{2}\mathbf{\Psi}_{T}^{T'}(\xi)\right|_{T'=T+z(1-T)} (1-z) \dd z
\end{equation*}
defines a jointly smooth function.
\end{proof}
\subsection{Stable global wave evolution}
The parameter for the blowup time has entered through the initial data operator from \Cref{InitialDataOperator} in the analysis of the stability problem in similarity coordinates. By adjusting this parameter, we can remove the correction term which has stabilized the evolution in \Cref{StabilizedEvolution}. This yields the existence of stable solutions to the abstract Cauchy problem \eqref{AbstractCauchyProblem} which evolve from perturbations of the self-similar blowup solution posed at time $t = 0$.
\begin{proposition}
\label{GlobalMildSolution}
Let $d,k\in\NN$ and $R>r>0$ be as in \Cref{NotionSolution,InitialDataOperator}. Then, there are constants $\delta^{\ast} > 0$ and $C^{\ast} \geq 1$ such that for all $0 < \delta \leq \delta^{\ast}$ and $C \geq C^{\ast}$ and all $\bm{f} \in \big( C^{\infty}_{\mathrm{rad}}(\RR^{d}) \cap C^{\infty}_{\mathrm{c}}(\BB^{d}_{r}) \big)^{2}$ with $\| \bm{f} \|_{H^{k}(\RR^{d}) \times H^{k-1}(\RR^{d})} \leq \frac{\delta}{C^{2}}$ there exists a $T^{\ast} \in [ 1 - \frac{\delta}{C}, 1 + \frac{\delta}{C} ]$ and a unique $\mathbf{u}_{\bm{f}, T^{\ast}} \in C \big( [0,\infty), \mathfrak{H}^{k}_{\mathrm{rad}}(\BB^{d}_{R}) \big)$ with $\| \mathbf{u}_{\bm{f}, T^{\ast}}(\tau) \|_{\mathfrak{H}^{k}(\BB^{d}_{R})} \leq \delta \ee^{-\omega\tau}$ and
\begin{equation*}
\mathbf{u}_{\bm{f}, T^{\ast}}(\tau) = \mathbf{S}(\tau) \mathbf{U} ( \bm{f}, T^{\ast} ) + \int_{0}^{\tau} \mathbf{S}(\tau-\tau') \mathbf{N}( \mathbf{u}_{\bm{f}, T^{\ast}}(\tau') ) \dd\tau'
\end{equation*}
for all $\tau \geq 0$. In particular, $\mathbf{u}_{\bm{f},T^{\ast}} \in C^{\infty} \big( \overline{ (0,\infty) \times \BB^{d}_{R} } \big)^{2}$ with
\begin{equation*}
\renewcommand{\arraystretch}{1.2}
\left\{
\begin{array}{rcl}
\pd_{\tau} \mathbf{u}_{\bm{f}, T^{\ast}}(\tau, \,.\,) &=& \mathbf{L} \mathbf{u}_{\bm{f}, T^{\ast}}(\tau, \,.\,) + \mathbf{N} ( \mathbf{u}_{\bm{f}, T^{\ast}}(\tau, \,.\,) ) \,, \\
\mathbf{u}_{\bm{f}, T^{\ast}}(0, \,.\,) &=& \mathbf{U} ( \bm{f}, T^{\ast} ) \,.
\end{array}
\right.
\end{equation*}
\end{proposition}
\begin{proof}
Pick $\delta_{0} > 0$ and $C_{0} \geq 1$ from \Cref{StabilizedEvolution}. For now, let $0 < \delta^{\ast} \leq \delta_{0}$ and $C^{\ast} \geq C_{0}$ be arbitrary but fixed and let
\begin{equation*}
0 < \delta \leq \delta^{\ast}
\qquad\text{and}\qquad
C \geq C^{\ast} \,.
\end{equation*}
Let $\bm{f} \in C^{\infty}_{\mathrm{rad}}(\overline{\BB^{d}_{R}})^{2} \cap C^{\infty}_{\mathrm{c}}(\BB^{d}_{r})^{2}$ with $\| \bm{f} \|_{\mathfrak{H}^{k}(\BB^{d}_{R})} \leq \frac{\delta}{C^{2}}$. If $C^{\ast} \geq C_{0}$ is large enough, \Cref{InitialDataExpansion} yields
\begin{align*}
\| \mathbf{U}(\bm{f},T) \|_{\mathfrak{H}^{k}(\BB^{d}_{R})} &\leq \| \bm{f}_{T} \|_{\mathfrak{H}^{k}(\BB^{d}_{R})} + |T-1| \| \mathbf{f}_{1}^{\ast} \|_{\mathfrak{H}^{k}(\BB^{d}_{R})} + |T-1|^{2} \| \mathbf{r}(T,\,.\,) \|_{\mathfrak{H}^{k}(\BB^{d}_{R})} \\&\leq \frac{\delta}{C_{0}}
\end{align*}
for all $T \in [ 1 - \frac{\delta}{C}, 1 + \frac{\delta}{C}]$ and so $\mathbf{U}(\bm{f},T) \in \mathfrak{H}^{k}_{\mathrm{rad}}(\BB^{d}_{R})$ satisfies the assumptions for the initial data in \Cref{StabilizedEvolution}. Thus, for any $T \in [ 1 - \frac{\delta}{C}, 1 + \frac{\delta}{C}]$ there is a $\mathbf{u}_{\bm{f},T} \in C \big( [0,\infty), \mathfrak{H}^{k}_{\mathrm{rad}}(\BB^{d}_{R}) \big)$ satisfying $\| \mathbf{u}_{\bm{f},T}(\tau) \|_{\mathfrak{H}^{k}(\BB^{d}_{R})} \leq \delta \ee^{-\omega\tau}$ and
\begin{equation*}
\mathbf{u}_{\bm{f},T}(\tau) = \mathbf{S}(\tau) \Big( \mathbf{U} ( \bm{f}, T ) - \mathbf{C} ( \mathbf{U}( \bm{f}, T ), \mathbf{u}_{\bm{f},T} ) \Big) + \int_{0}^{\tau} \mathbf{S}(\tau-\tau') \mathbf{N}( \mathbf{u}_{\bm{f},T}(\tau') ) \dd\tau'
\end{equation*}
for all $\tau \geq 0$. Note that this solution depends continuously on the parameter $T \in [ 1 - \frac{\delta}{C}, 1 + \frac{\delta}{C}]$ for the blowup time due to the continuous dependence on the initial data in \Cref{StabilizedEvolution} and \Cref{InitialDataExpansion}. In the rest of the proof, we extract such a time parameter along with a solution for which the corresponding correction term vanishes. Since $\mathbf{C} ( \mathbf{U}( \bm{f}, T ), \mathbf{u}_{\bm{f},T} ) \in \operatorname{ran}( \mathbf{P}_{1} )$ and $\operatorname{ran}( \mathbf{P}_{1} ) = \langle \mathbf{f}_{1}^{\ast} \rangle$ according to \Cref{AlgMult}, it is sufficient to prove that the continuous function
\begin{equation}
\label{SelfMap}
\Big[ 1 - \frac{\delta}{C}, 1 + \frac{\delta}{C} \Big] \rightarrow \RR \,, \qquad T \mapsto T + \Big( \mathbf{C} ( \mathbf{U} ( \bm{f}, T ), \mathbf{u}_{\bm{f},T} ) \,\Big|\, \mathbf{g}_{1}^{\ast} \Big)_{\mathfrak{H}^{k}(\BB^{d}_{R})} \,,
\end{equation}
where we set $\mathbf{g}_{1}^{\ast} \coloneqq \mathbf{f}_{1}^{\ast} / \| \mathbf{f}_{1}^{\ast} \|_{\mathfrak{H}^{k}(\BB^{d}_{R})}^{2}$, has a fixed point. Using the expansion for the initial data operator from \Cref{InitialDataExpansion}, we get
\begin{align*}
&
\Big( \mathbf{C} ( \mathbf{U} ( \bm{f}, T ), \mathbf{u}_{\bm{f},T} ) \,\Big|\, \mathbf{g}_{1}^{\ast} \Big)_{\mathfrak{H}^{k}(\BB^{d}_{R})} \\&\indent= \Big( \mathbf{P}_{1} \mathbf{U}(\bm{f},T) \,\Big|\, \mathbf{g}_{1}^{\ast} \Big)_{\mathfrak{H}^{k}(\BB^{d}_{R})} + \Big( \mathbf{P}_{1} \int_{0}^{\infty} \ee^{-\tau'} \mathbf{N}( \mathbf{u}_{\bm{f},T}(\tau') ) \dd\tau' \,\Big|\, \mathbf{g}_{1}^{\ast} \Big)_{\mathfrak{H}^{k}(\BB^{d}_{R})} \\&\indent=
T^{s} \Big( \mathbf{P}_{1} \bm{f}(T \,.\,) \,\Big|\, \mathbf{g}_{1}^{\ast} \Big)_{\mathfrak{H}^{k}(\BB^{d}_{R})} + (1-T) + (1-T)^{2} \Big( \mathbf{P}_{1} \mathbf{r}(T,\,.\,) \,\Big|\, \mathbf{g}_{1}^{\ast} \Big)_{\mathfrak{H}^{k}(\BB^{d}_{R})} \\&\indent\indent+ \Big( \mathbf{P}_{1} \int_{0}^{\infty} \ee^{-\tau'} \mathbf{N}( \mathbf{u}_{\bm{f},T}(\tau') ) \dd\tau' \,\Big|\, \mathbf{g}_{1}^{\ast} \Big)_{\mathfrak{H}^{k}(\BB^{d}_{R})} \,.
\end{align*}
Together with the Cauchy-Schwarz inequality, \Cref{LocLip} and \Cref{StabilizedEvolution}, we estimate
\begin{align*}
&
\Big| T + \Big( \mathbf{C} ( \mathbf{U} ( \bm{f}, T ), \mathbf{u}_{\bm{f},T} ) \,\Big|\, \mathbf{g}_{1}^{\ast} \Big)_{\mathfrak{H}^{k}(\BB^{d}_{R})} - 1 \Big| \\&\indent\indent\lesssim \| \bm{f} \|_{H^{k}(\RR^{d}) \times H^{k-1}(\RR^{d})} + |T-1|^{2} + \int_{0}^{\infty} \ee^{-\tau'} \| \mathbf{N}( \mathbf{u}_{\bm{f},T}(\tau') ) \big\|_{\mathfrak{H}^{k}(\BB^{d}_{R})} \dd\tau' \\&\indent\indent\lesssim
\frac{\delta}{C^{2}} + \frac{\delta^{2}}{C^{2}} + \delta^{2} 
\end{align*}
for all $T \in [1 - \frac{\delta}{C}, 1+\frac{\delta}{C}]$.
Thus, we can fix $C^{\ast} \geq C_{0}$ large enough and $0 < \delta^{\ast} \leq \delta_{0}$ small enough such that the map in \eqref{SelfMap} becomes a continuous self-map for any $0 < \delta \leq \delta^{\ast}$ and $C \geq C^{\ast}$. Thus, this map has a fixed point $T^{\ast} \in [1 - \frac{\delta}{C}, 1+\frac{\delta}{C}]$ by a consequence of the intermediate value theorem. Uniqueness of the associated mild solution $\mathbf{u}_{\bm{f},T^{\ast}} \in C \big( [0,\infty), \mathfrak{H}^{k}_{\mathrm{rad}}(\BB^{d}_{R}) \big)$ follows from Gr\"{o}nwall's inequality. Joint smoothness $\mathbf{u}_{\bm{f},T^{\ast}} \in C^{\infty} \big( \overline{ (0,\infty) \times \BB^{d}_{R} } \big)^{2}$ of the solution follows from \Cref{SmoothClassicalSolution}.
\end{proof}
\section{Proof of the main results}
\label{SecProofMainResults}
Here, we employ similarity coordinates to transform the solution of the abstract Cauchy problem into a solution of the nonlinear stability problem for the corotational wave maps equation and equivariant Yang-Mills equation in spacetime regions beyond light cones.
\subsection{Proof of Theorem \ref{THM}}
In the final step of our construction, we extend the solution in the image region of similarity coordinates by the solution outside of this region via finite speed of propagation. To ensure that the resulting map is smooth and solves the Cauchy problem, we recall the following version of finite speed of propagation for semilinear wave equations, also see \Cref{Fig_SolutionConstruction} for an illustration.
\begin{lemma}
\label{FiniteSpeedOfPropagation}
Let ${\chi\mathstrut}_{T} : \RR \times \RR^{d} \rightarrow \RR^{1,d}$ be similarity coordinates as in \Cref{SimilarityCoordinates}. Let $R \geq R_{0}$, where $R_{0} > 0$ is determined in \Cref{GSCLightCone} and consider the spacetime regions
\begin{equation*}
\mathrm{X}^{1,d}_{T,R} = {\chi\mathstrut}_{T} \big( (0,\infty) \times \BB^{d}_{R} \big) \,, \qquad \Sigma^{1,d}_{T,R}(\tau) = {\chi\mathstrut}_{T} \big( \{\tau\} \times \BB^{d}_{R} \big) \,.
\end{equation*}
For fixed $0 < \rho_{0} < R$ set $( t_{0},r_{0} ) = \big( T + T \widetilde{h}(\rho_{0}), T \rho_{0} \big) \in \Sigma^{1,d}_{T,R}(0)$ and define the domain
\begin{equation*}
\Lambda^{1,d}(t_{0},r_{0}) = \big\{ (t,x) \in \RR^{1,d} \mid |t-t_{0}| + r_{0} < |x| \big\} \,.
\end{equation*}
Assume that $\psi^{\mathrm{I}}, \psi^{\mathrm{I\!I}} \in C^{\infty} \big( \overline{\mathrm{X}}^{1,d}_{T,R} \big)$ are solutions to the semilinear wave equation
\begin{equation*}
(\Box \psi)(t,x) + F \big( x, \psi(t,x) \big) = 0 \quad \text{in } \mathrm{X}^{1,d}_{T,R}
\end{equation*}
with
\begin{equation*}
\big( \psi^{\mathrm{I}}, \pd_{0} \psi^{\mathrm{I}} \big) \big|_{\Sigma^{1,d}_{T,R}(0) \cap \Lambda^{1,d}(t_{0},r_{0})} =
\big( \psi^{\mathrm{I\!I}}, \pd_{0} \psi^{\mathrm{I\!I}} \big) \big|_{\Sigma^{1,d}_{T,R}(0) \cap \Lambda^{1,d}(t_{0},r_{0})} \,.
\end{equation*}
Then $\psi^{\mathrm{I}} = \psi^{\mathrm{I\!I}}$ in $\mathrm{X}^{1,d}_{T,R} \cap \Lambda^{1,d}(t_{0},r_{0})$.
\end{lemma}
\begin{proof}
The assumptions on the similarity coordinates and \Cref{GeometryFoliation,GSCFoliation} imply that $\mathrm{X}^{1,d}_{T,R} \cap \Lambda^{1,d}(t_{0},r_{0})$ is nonempty and foliated by the spacelike hypersurfaces
\begin{equation*}
\Sigma(\tau) \coloneqq \Sigma^{1,d}_{T,R}(\tau) \cap \Lambda^{1,d}(t_{0},r_{0}) =
{\chi\mathstrut}_{T} \big( \{\tau\} \times (\BB^{d}_{R} \setminus \overline{\BB^{d}_{\rho(\tau)}}) \big) \,,
\end{equation*}
where $\rho_{0} \leq \rho(\tau) \leq R$ is defined implicitly through the equation
\begin{align*}
r(\tau) = r_{0} + t(\tau) - t_{0} \qquad\text{where}\qquad
t(\tau) &= T + T \ee^{-\tau} \widetilde{h}(\rho(\tau)) \,, \\
r(\tau) &= T \ee^{-\tau} \rho(\tau) \,,
\end{align*}
that is,
\begin{equation*}
\rho(\tau) = \ee^{\tau} \big( \rho_{0} - \widetilde{h}(\rho_{0}) \big) + \widetilde{h} \big( \rho(\tau) \big) \,.
\end{equation*}
We define the integral along $\Sigma(\tau)$ as in \Cref{IntegralHypersurface}. We recall the unit normal vector field of $\Sigma(\tau)$ from \eqref{FDUNVF} and the energy-momentum tensor associated to $\psi$ from \eqref{NRGMOM} and define an energy
\begin{equation*}
E_{\Sigma}[\psi](\tau) \coloneqq \int_{\Sigma(\tau)} \mathrm{T}_{0\mu} N^{\mu} =
(T\ee^{-\tau})^{d-2s-2} \frac{1}{2} \int_{\BB^{d}_{R} \setminus \overline{ \BB^{d}_{\rho(\tau)} } } \Big( |\pd_{\xi} u_{1}|^{2} + |u_{2}|^{2} w \Big) \,,
\end{equation*}
and
\begin{equation*}
u_{1}(\tau,\xi) = (T\ee^{-\tau})^{s} (\psi\circ{\chi\mathstrut}_{T})(\tau,\xi) \,, \qquad
u_{2}(\tau,\xi) = (T\ee^{-\tau})^{s+1} ( (\pd_{0} \psi) \circ{\chi\mathstrut}_{T} ) (\tau,\xi) \,.
\end{equation*}
A computation as in \Cref{DerivativeEnergy} and the Leibniz integral rule lead to the identity
\begin{align*}
&
\frac{\mathrm{d}}{\mathrm{d}\tau} E_{\Sigma}[\psi](\tau) = -(T\ee^{-\tau})^{d-s} \int_{\BB^{d}_{R} \setminus \overline{ \BB^{d}_{\rho(\tau)} } } \big( \Box \psi \circ {\chi\mathstrut}_{T} \big)(\tau,\,.\,) \big( c u_{2} \big) \\&\indent+
\frac{(T\ee^{-\tau})^{d-2s-2}}{R} \int_{\pd\BB^{d}_{R}} \Big( - \frac{|\xi|^{2}}{2} \big( |\pd_{\xi} u_{1}|^{2} + |u_{2}|^{2} w \big) - \xi^{i} (\pd_{i} h) c |u_{2}|^{2} + \xi^{i} (\pd_{i} u_{1}) ( c u_{2} ) \Big) \\&\indent-
\frac{(T\ee^{-\tau})^{d-2s-2}}{\rho(\tau)} \int_{\pd\BB^{d}_{\rho(\tau)}} \Big( - \frac{|\xi|^{2}}{2} \big( |\pd_{\xi} u_{1}|^{2} + |u_{2}|^{2} w \big) - \xi^{i} (\pd_{i} h) c |u_{2}|^{2} + \xi^{i} (\pd_{i} u_{1}) ( c u_{2} ) \Big) \\&\indent-
(T\ee^{-\tau})^{d-2s-2} \int_{\pd\BB^{d}_{\rho(\tau)}} \frac{\dot{\rho}(\tau)}{2} \big( |\pd_{\xi} u_{1}|^{2} + |u_{2}|^{2} w \big) \,.
\end{align*}
We find for the first boundary integral
\begin{align*}
&
\int_{\pd\BB^{d}_{R}} \Big( - \frac{|\xi|^{2}}{2} \big( |\pd_{\xi} u_{1}|^{2} + |u_{2}|^{2} w \big) - \xi^{i} (\pd_{i} h) c |u_{2}|^{2} + \xi^{i} (\pd_{i} u_{1}) ( c u_{2} ) \Big) \\&\indent\leq
\int_{\pd\BB^{d}_{R}} \Big( - \frac{|\xi|^{2}}{2} w - \xi^{i} (\pd_{i} h) c + \frac{1}{2} c^{2} \Big) |u_{2}|^{2} \\&\indent=
\frac{\widetilde{h}(R)^{2} - R^{2}}{2} \int_{\pd\BB^{d}_{R}} |u_{2}|^{2} \leq 0
\end{align*}
by \Cref{GSCLightCone} and since $R \geq R_{0}$. To estimate the remaining boundary integrals, we first compute via implicit differentiation
\begin{equation*}
\dot{\rho}(\tau) =
\frac{\rho(\tau) - \widetilde{h} \big( \rho(\tau) \big)}{1 - \widetilde{h}' \big( \rho(\tau) \big) } \,,
\qquad\text{thus}\qquad
\rho(\tau) - \dot{\rho}(\tau) = -\frac{\widetilde{c}\big( \rho(\tau) \big)}{1 - \widetilde{h}' \big( \rho(\tau) \big)} \leq \widetilde{h}(0) < 0 \,.
\end{equation*}
So we can estimate
\begin{equation*}
\big| \xi^{i} (\pd_{i} u_{1}) ( c u_{2} ) \big| \leq \frac{\rho(\tau) \dot{\rho}(\tau) - \rho(\tau)^{2}}{2} |\pd_{\xi}u_{1}|^{2} + \frac{1}{2} \frac{\rho(\tau) c^{2}}{\dot{\rho}(\tau) - \rho(\tau)} |u_{2}|^{2}
\end{equation*}
for $|\xi| = \rho(\tau)$ and get
\begin{align*}
&
\frac{1}{\rho(\tau)} \int_{\pd\BB^{d}_{\rho(\tau)}} \Big( \frac{\rho(\tau)^{2} - \rho(\tau)\dot{\rho}(\tau)}{2} \big( |\pd_{\xi} u_{1}|^{2} + |u_{2}|^{2} w \big) + \xi^{i} (\pd_{i} h) c |u_{2}|^{2} - \xi^{i} (\pd_{i} u_{1}) ( c u_{2} ) \Big) \\&\indent\leq
\int_{\pd\BB^{d}_{\rho(\tau)}} \Big( \frac{\rho(\tau) - \dot{\rho}(\tau)}{2} w + \widetilde{h}' c + \frac{1}{2} \frac{c^{2}}{\dot{\rho}(\tau) - \rho(\tau)} \Big) |u_{2}|^{2} = 0 \,.
\end{align*}
This implies
\begin{align*}
\frac{\mathrm{d}}{\mathrm{d}\tau} E_{\Sigma}[\psi](\tau) &\leq -(T\ee^{-\tau})^{d-s} \int_{\BB^{d}_{R} \setminus \overline{ \BB^{d}_{\rho(\tau)} } } \big( \Box \psi \circ {\chi\mathstrut}_{T} \big)(\tau,\,.\,) \big( c u_{2} \big) \\&\leq
E_{\Sigma}[\psi](\tau) +
\frac{1}{2} (T\ee^{-\tau})^{d - 2s - 2} \Big\| (T\ee^{-\tau})^{s+2} \sqrt{\tfrac{c^{2}}{w}} \big( \Box \psi \circ {\chi\mathstrut}_{T} \big)(\tau,\,.\,) \Big\|_{L^{2}( \BB^{d}_{R} \setminus \overline{ \BB^{d}_{\rho(\tau)} } )}^{2}
\end{align*}
for all $\psi \in C^{\infty} \big( \overline{\mathrm{X}}^{1,d}_{T,R} \big)$. Furthermore, we compute
\begin{align*}
&
\frac{\mathrm{d}}{\mathrm{d}\tau} \int_{\BB^{d}_{R} \setminus \overline{ \BB^{d}_{\rho(\tau)} } } |u_{1}|^{2} = \int_{\BB^{d}_{R} \setminus \overline{ \BB^{d}_{\rho(\tau)} } } \frac{1}{2} \pd_{\tau} |u_{1}|^{2} - \frac{\dot{\rho}(\tau)}{2} \int_{\pd \BB^{d}_{\rho(\tau)}} |u_{1}|^{2} \\&\indent=
\frac{1}{2} (d-2s) \int_{\BB^{d}_{R} \setminus \overline{ \BB^{d}_{\rho(\tau)} } } |u_{1}|^{2} + \int_{\BB^{d}_{R} \setminus \overline{ \BB^{d}_{\rho(\tau)} } } \big( c u_{1} u_{2} \big) -
\frac{R}{2} \int_{\pd\BB^{d}_{R}} |u_{1}|^{2} + \frac{\rho(\tau) - \dot{\rho}(\tau)}{2} \int_{\pd\BB^{d}_{\rho(\tau)}} |u_{1}|^{2} \,.
\end{align*}
This yields for the total energy
\begin{equation*}
E[\psi](\tau) \coloneqq \frac{1}{2} \int_{\BB^{d}_{R} \setminus \overline{ \BB^{d}_{\rho(\tau)} } } \Big( |u_{1}|^{2} + |\pd_{\xi} u_{1}|^{2} + |u_{2}|^{2} w \Big)
\end{equation*}
the estimate
\begin{align}
\label{TotalEnergyEstimate}
\frac{\mathrm{d}}{\mathrm{d}\tau} E[\psi](\tau) &\lesssim E[\psi](\tau) + \big\| (T\ee^{-\tau})^{s+2} \big( \Box \psi \circ {\chi\mathstrut}_{T} \big)(\tau,\,.\,) \big\|_{L^{2}( \BB^{d}_{R} \setminus \overline{ \BB^{d}_{\rho(\tau)} } )}^{2}
\end{align}
for all $\psi \in C^{\infty} \big( \overline{\mathrm{X}}^{1,d}_{T,R} \big)$ and all $0 \leq \tau \leq \rho^{-1}(R)$. Now, we apply the energy estimate \eqref{TotalEnergyEstimate} to the difference of the solutions $\psi^{\mathrm{I}},\psi^{\mathrm{I\!I}}$ to the semilinear wave equation. Since $F \in C^{\infty}(\RR^{d} \times \RR)$ is locally Lipschitz continuous in the second argument, we have
\begin{align*}
\big\| (T\ee^{-\tau})^{s+2} \big( \Box ( \psi^{\mathrm{I}} - \psi^{\mathrm{I\!I}} ) \circ {\chi\mathstrut}_{T} \big)(\tau,\,.\,) \big\|_{L^{2}( \BB^{d}_{R} \setminus \overline{ \BB^{d}_{\rho(\tau)} } )}^{2} &= \big\| F \big( \,.\,, u_{1}^{\mathrm{I}} (\tau,\,.\,) \big) - F \big( \,.\,, u_{1}^{\mathrm{I\!I}} (\tau,\,.\,) \big) \big\|_{L^{2}( \BB^{d}_{R} \setminus \overline{ \BB^{d}_{\rho(\tau)} } )}^{2} \\&\lesssim
E[ \psi^{\mathrm{I}} - \psi^{\mathrm{I\!I}} ](\tau)
\end{align*}
for all $0 \leq \tau \leq \rho^{-1}(R)$. Gr\"{o}nwall's inequality implies $E[ \psi^{\mathrm{I}} - \psi^{\mathrm{I\!I}} ](\tau) = 0$ from which we conclude the assertion of the lemma.
\end{proof}
We employ a fixed height function $h \in C^{\infty}_{\mathrm{rad}}(\RR^{d})$ as in \Cref{SimilarityCoordinatesFlat} and are set to translate our nonlinear stability theory back to the classical setting.
\begin{proof}[Proof of \Cref{THM}]
Pick $\delta^{*} > 0$ and $C^{*} \geq 1$ from \Cref{GlobalMildSolution} and let $0 < \delta \leq \delta^{\ast}$ and $C \geq C^{\ast}$ be arbitrary but fixed. Then, for any data $(f,g) \coloneqq \bm{f} \in C^{\infty}_{\mathrm{rad}}(\RR^{d})^{2}$ with $\operatorname{supp}(f,g) \subset \BB^{d}_{r}$ and $\| (f,g) \|_{H^{k}(\RR^{d}) \times H^{k-1}(\RR^{d})} \leq \delta / C^{2}$ there exist a blowup time $T^{\ast} \in [1-\frac{\delta}{C},1+\frac{\delta}{C}]$ and a unique solution
\begin{equation*}
\mathbf{u}_{\bm{f}, T^{\ast}} \in C^{\infty} \big( \overline{ (0,\infty) \times \BB^{d}_{R}} \big)^{2}
\qquad\text{with}\qquad
\| \mathbf{u}_{\bm{f}, T^{\ast}} (\tau) \|_{\mathfrak{H}^{k}(\BB^{d}_{R})} \leq \delta \ee^{-\omega \tau}
\end{equation*}
to the equation
\begin{equation*}
\renewcommand{\arraystretch}{1.2}
\left\{
\begin{array}{rcl}
\pd_{\tau} \mathbf{u}_{\bm{f}, T^{\ast}}(\tau, \,.\,) &=& \mathbf{L} \mathbf{u}_{\bm{f}, T^{\ast}}(\tau, \,.\,) + \mathbf{N} ( \mathbf{u}_{\bm{f}, T^{\ast}}(\tau, \,.\,) ) \,, \\
\mathbf{u}_{\bm{f}, T^{\ast}}(0, \,.\,) &=& \mathbf{U}(\bm{f},T^{\ast}) \,,
\end{array}
\right.
\end{equation*}
for all $\tau \geq 0$, where the initial data $\mathbf{U}(\bm{f},T^{\ast})$ are given in \Cref{InitialDataOperator}. Fix similarity coordinates ${\chi\mathstrut}_{T^{\ast}}: \RR \times \RR^{d} \rightarrow \RR^{1,d}$ for the given height function and recall from \Cref{RemarkImageRegion,RemarkSHSF,RemarkFlatRegion} the spacetime regions
\begin{align*}
\Omega^{1,d}_{T^{\ast}}(\kappa) &= \big\{ (t,x) \in \RR^{1,d} \mid 0 < t < T^{\ast} + \kappa |x| \big\} \,, \\
\mathrm{X}^{1,d}_{T^{\ast}}(\kappa) &= \big\{ (t,x) \in \RR^{1,d} \mid T^{\ast} + T^{\ast} h \big( \tfrac{x}{T^{\ast}}\big) < t < T^{\ast} + \kappa |x| \big\} \,,
\end{align*}
also see \Cref{FigHSC}. With this, we define
\begin{equation*}
u \in C^{\infty} \big( \overline{\mathrm{X}}^{1,d}_{T^{\ast}}(\kappa) \big)
\qquad\text{through}\qquad
(u \circ {\chi\mathstrut}_{T^{\ast}})(\tau,\xi) = \Big( \frac{\ee^{\tau}}{T^{\ast}} \Big)^{s} [\mathbf{u}_{\bm{f}, T^{\ast}}]_{1}(\tau,\xi)
\end{equation*}
and
\begin{equation}
\label{SolutionCoordinateRegion}
\psi \in C^{\infty} \big( \overline{\mathrm{X}}^{1,d}_{T^{\ast}}(\kappa) \big)
\qquad\text{by}\qquad
\psi = \psi^{\ast}_{T^{\ast}} + u \,.
\end{equation}
Note that $u(t,\,.\,)$ and $\psi(t,\,.\,)$ are smooth and radially symmetric. \Cref{NonlinearTransitionRelation} immediately gives
\begin{equation*}
\Box \psi + F(\,.\,,\psi) = \Box u + \mathcal{V}_{T^{\ast}} u + \mathcal{N}_{T^{\ast}}(u) = 0 \quad \text{in } \mathrm{X}^{1,d}_{T^{\ast}}(\kappa) \,.
\end{equation*}
By construction of the initial data for the abstract Cauchy problem, the evaluation of $\psi, \pd_{0} \psi$ along the initial hypersurface
\begin{align*}
\Sigma^{1,d}_{T^{\ast}}(0) &= {\chi\mathstrut}_{T^{\ast}} \big( \{0\} \times \overline{ \BB^{d}_{R} } \big) \\&=
\big\{ (t_{0}(x),x) \in \RR^{1,d} \mid x\in \overline{\BB^{d}_{T^{\ast}R}} \big\} \,, \quad \text{where}\quad t_{0}(x) = T^{\ast} + T^{\ast} h \big( \tfrac{x}{T^{\ast}} \big) \,,
\end{align*}
yields
\begin{align}
\nonumber
\psi(t_{0}(x),x) &= \big( \psi \circ {\chi\mathstrut}_{T^{\ast}} \big) \big(0,\tfrac{x}{T^{\ast}} \big) \\\nonumber&=
\big( \psi^{\ast}_{T^{\ast}} \circ {\chi\mathstrut}_{T^{\ast}} \big) \big(0,\tfrac{x}{T^{\ast}} \big) + (T^{\ast})^{-s} [\mathbf{u}_{\bm{f}, T^{\ast}}]_{1}\big(0,\tfrac{x}{T^{\ast}} \big) \\\label{psionSigma0}&=
f(x) + \psi^{\ast}_{1} \big( t_{0}(x),x \big) \,,
\end{align}
and
\begin{align}
\nonumber
(\pd_{0} \psi)(t_{0}(x),x) &= \big( (\pd_{0} \psi) \circ {\chi\mathstrut}_{T^{\ast}} \big) \big(0,\tfrac{x}{T^{\ast}} \big) \\\nonumber&=
\big( (\pd_{0}\psi^{\ast}_{T^{\ast}}) \circ {\chi\mathstrut}_{T^{\ast}} \big) \big(0,\tfrac{x}{T^{\ast}} \big) + (T^{\ast})^{-s-1} [\mathbf{u}_{\bm{f}, T^{\ast}}]_{2}\big(0,\tfrac{x}{T^{\ast}} \big) \\\label{pd0psionSigma0}&=
g(x) + (\pd_{0}\psi^{\ast}_{1}) \big( t_{0}(x),x \big) \,.
\end{align}
Since $\operatorname{supp}(f,g) \subset \BB^{d}_{r}$ and $h \equiv -1$ in $\BB^{d}_{r+\varepsilon}$, we have that $\psi$, $\pd_{0} \psi$ coincides with $\psi^{\ast}_{1}$, $\pd_{0}\psi^{\ast}_{1}$ on $\Sigma^{1,d}_{T^{\ast}}(0) \cap \Lambda^{1,d}_{T^{\ast}}(r)$, respectively, where
\begin{equation*}
\Lambda^{1,d}_{T^{\ast}}(r) = \{ (t,x) \in \Omega^{1,d}_{T^{\ast}}(\kappa) \mid 0 < t < |x| - r \} \,.
\end{equation*}
Finite speed of propagation in the version of \Cref{FiniteSpeedOfPropagation} implies that $\psi$ coincides with $\psi^{\ast}_{1}$ also in the region $\mathrm{X}^{1,d}_{T^{\ast}}(\kappa) \cap \Lambda^{1,d}_{T^{\ast}}(r)$. With this, let us extend the solution \eqref{SolutionCoordinateRegion} to
\begin{equation*}
\psi: \Omega^{1,d}_{T^{\ast}}(\kappa) \rightarrow \RR
\qquad\text{by}\qquad
\psi \coloneqq
\renewcommand{\arraystretch}{1.2}
\left\{
\begin{array}{ll}
\psi^{\ast}_{T^{\ast}} + u & \text{in } \mathrm{X}^{1,d}_{T^{\ast}}(\kappa) \,, \\
\psi^{\ast}_{1} & \text{in } \Omega^{1,d}_{T^{\ast}}(\kappa) \setminus \mathrm{X}^{1,d}_{T^{\ast}}(\kappa) \,.
\end{array}
\right.
\end{equation*}
We check that $\psi$ is smooth in the whole domain $\Omega^{1,d}_{T^{\ast}}(\kappa)$ by proving that $\psi$ defines a smooth map in an open neighbourhood of $\Sigma^{1,d}_{T^{\ast}}(0) \setminus \big( \{ 0 \} \times \RR^{d} \big)$, also see \Cref{Fig_SolutionConstruction}. By design of the similarity coordinates,
\begin{equation*}
\Lambda^{1,d}_{T^{\ast}}(r) = \big( \mathrm{X}^{1,d}_{T^{\ast}}(\kappa) \cap \Lambda^{1,d}_{T^{\ast}}(r) \big) \cup \big( \Omega^{1,d}_{T^{\ast}}(\kappa) \setminus \mathrm{X}^{1,d}_{T^{\ast}}(\kappa) \big)
\end{equation*}
is a relatively open region in $\Omega^{1,d}_{T^{\ast}}(\kappa)$ which contains $\Sigma^{1,d}_{T^{\ast}}(0) \setminus \big( \{ 0 \} \times \RR^{d} \big)$ and from above we infer that $\psi$ and $\psi^{\ast}_{1}$ coincide in $\Lambda^{1,d}_{T^{\ast}}(r)$. It follows that $\psi \in C^{\infty} \big( \Omega^{1,d}_{T^{\ast}}(\kappa) \big)$. As $t_{0}(x) = 0$ for $|x| < r$, \Cref{psionSigma0,pd0psionSigma0} show that $\psi$ satisfies the initial data for the Cauchy problem at $t=0$ as well. Hence $\psi$ is the unique solution to the Cauchy problem in the theorem. Finally, we verify the stability estimate. Since
\begin{align*}
(\psi^{\ast}_{T^{\ast}} \circ {\chi\mathstrut}_{T^{\ast}})(\tau,\xi) &= \Big(\frac{\ee^{\tau}}{T^{\ast}} \Big)^{s} \psi^{\ast}(-h(\xi),\xi) \,, \\
(\pd_{0} \psi^{\ast}_{T^{\ast}} \circ {\chi\mathstrut}_{T^{\ast}})(\tau,\xi) &= - \Big(\frac{\ee^{\tau}}{T^{\ast}} \Big)^{s+1} \pd_{0} \psi^{\ast}(-h(\xi),\xi) \,,
\end{align*}
the stability estimate follows directly from
\begin{align*}
\frac{\big\| \big( \psi - \psi^{\ast}_{T^{\ast}} \big) \circ {\chi\mathstrut}_{T^{\ast}}(\tau,\,.\,) \big\|_{H^{k}(\BB^{d}_{R})}}{\big\| \psi^{\ast}_{T^{\ast}} \circ {\chi\mathstrut}_{T^{\ast}} (\tau,\,.\,) \big\|_{H^{k}(\BB^{d}_{R})}} &=
\frac{\| [\mathbf{u}_{\bm{f}, T^{\ast}}]_{1}(\tau,\,.\,) \|_{H^{k}(\BB^{d}_{R})}}{\| \psi^{\ast}(-h(\,.\,), \,.\,) \|_{H^{k}(\BB^{d}_{R})}} \lesssim \delta \ee^{-\omega \tau} \,, \\
\frac{\big\| \big( \pd_{0} \psi - \pd_{0}\psi^{\ast}_{T^{\ast}} \big) \circ {\chi\mathstrut}_{T^{\ast}}(\tau,\,.\,) \big\|_{H^{k-1}(\BB^{d}_{R})}}{\big\| (\pd_{0}\psi^{\ast}_{T^{\ast}}) \circ {\chi\mathstrut}_{T^{\ast}} (\tau,\,.\,) \big\|_{H^{k-1}(\BB^{d}_{R})}} &=
\frac{\| [\mathbf{u}_{\bm{f}, T^{\ast}}]_{2}(\tau,\,.\,) \|_{H^{k-1}(\BB^{d}_{R})}}{\| (\pd_{0}\psi^{\ast})(-h(\,.\,), \,.\,) \|_{H^{k-1}(\BB^{d}_{R})}} \lesssim \delta \ee^{-\omega \tau} \,,
\end{align*}
for all $\tau \geq 0$. If $k > \frac{d}{2}$, we get additionally from Sobolev embedding
\begin{equation*}
\sup_{(t',x') \in \Sigma^{1,d}_{T^{\ast},R}(\tau)} \frac{|\psi(t',x') - \psi^{\ast}_{T^{\ast}}(t',x')|}{|\psi^{\ast}_{T^{\ast}}(t',x')|} \lesssim \big( T^{\ast} \ee^{-\tau} \big)^{s} \big\| \big( \psi - \psi^{\ast}_{T^{\ast}} \big) \circ {\chi\mathstrut}_{T^{\ast}}(\tau,\,.\,) \big\|_{L^{\infty}(\BB^{d}_{R})} \lesssim \delta \ee^{-\omega \tau}
\end{equation*}
for all $\tau \geq 0$.
\end{proof}
\begin{figure}
\centering
\includegraphics[width=\textwidth]{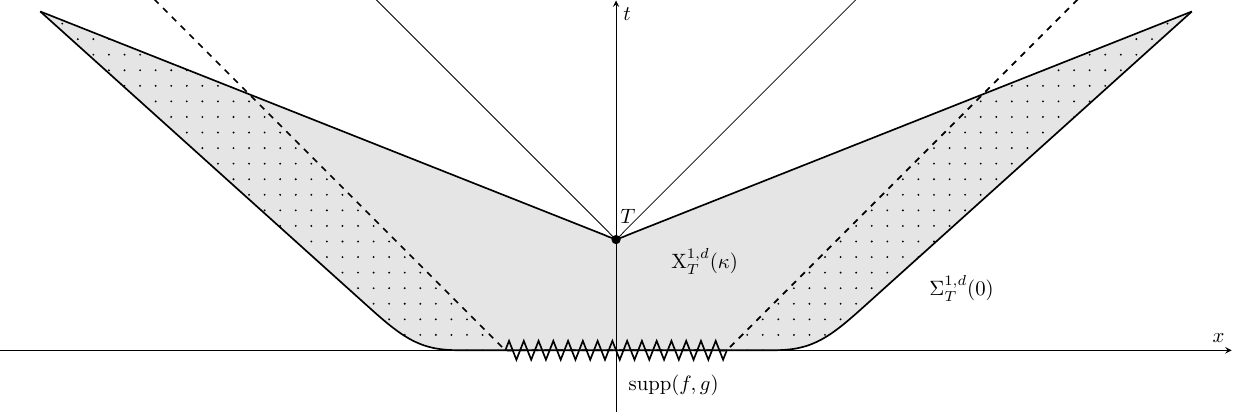}
\caption{Spacetime regions involved in the construction of the solution. The gray region illustrates the image region $\mathrm{X}^{1,d}_{T}(\kappa)$ of similarity coordinates from \Cref{SimilarityCoordinatesFlat}. This is where the stable solution $u$ has been constructed from initial data prescribed along the initial hypersurface $\Sigma^{1,d}_{T}(0)$. The zigzag line marks the support of the perturbations $f,g$ in the initial data. In the dotted region $\mathrm{X}^{1,d}_{T} \cap \Lambda^{1,d}_{T}(r)$, \Cref{FiniteSpeedOfPropagation} yields that $\psi^{\ast}_{1}$ is the unique solution.}
\label{Fig_SolutionConstruction}
\end{figure}
\subsection{Proof of Theorems \ref{THMWM} and \ref{THMYM}}
Before we conclude the stability theorem for the corotational wave maps equation, we prove a radial lemma which relates the Sobolev norm of a corotational wave map to the Sobolev norm of its profile.
\begin{lemma}
\label{CorotDataBound}
Let $n,k\in \NN$ with $n \geq 3$ and $k \geq \frac{n}{2} + 1$. Let
\begin{equation*}
U^{\ast}_{1}[0] \coloneqq \big( U^{\ast}_{1}(0,\,.\,), \pd_{0} U^{\ast}_{1}(0,\,.\,) \big) : \RR^{n} \rightarrow \mathrm{T}\mathbb{S}^{n} \subset \RR^{n+1} \times \RR^{n+1}
\end{equation*}
be the corotational initial data of the wave map \eqref{WaveMapsBlowup} with profile
\begin{equation*}
\widetilde{\psi}^{\ast}_{1}[0] \coloneqq \big( \widetilde{\psi}^{\ast}_{1}(0,|\,.\,|), \pd_{0} \widetilde{\psi}^{\ast}_{1}(0,|\,.\,|) \big) : \RR^{n} \rightarrow \RR \times \RR
\end{equation*}
given in \Cref{SelfSimilarSolutionsIntro}. Let $r > 0$ and $0 < \varepsilon_{r} \leq \frac{2(n-2)}{(n-2) + r^{2}}$. Then, there is a continuous function $C_{n,k,r,\varepsilon_{r}}: [0,\infty) \rightarrow (0,\infty)$ such that
\begin{align*}
&
\big\| \big( \widetilde{f}(|\,.\,|),\widetilde{g}(|\,.\,|) \big) - \widetilde{\psi}^{\ast}_{1}[0] \big\|_{H^{k}(\RR^{n+2}) \times H^{k-1}(\RR^{n+2})} \\&\indent\indent
\leq C_{n,k,r,\varepsilon_{r}} \Big( \big\| \big( F,G \big) - U^{\ast}_{1}[0] \big\|_{H^{k}(\RR^{n})\times H^{k-1}(\RR^{n})} \Big) \big\| \big( F,G \big) - U^{\ast}_{1}[0] \big\|_{H^{k}(\RR^{n})\times H^{k-1}(\RR^{n})}
\end{align*}
for all smooth corotational data $(F,G) : \RR^{n} \rightarrow \mathrm{T}\mathbb{S}^{n} \subset \RR^{n+1} \times \RR^{n+1}$ given by
\begin{equation*}
F(x) =
\begin{pmatrix}
|x|^{-1} \sin \big( |x| \widetilde{f}(|x|) \big) x \\
\cos \big( |x| \widetilde{f}(|x|) \big)
\end{pmatrix}
\,, \qquad
G(x) =
\begin{pmatrix}
\cos \big( |x| \widetilde{f}(|x|) \big) \widetilde{g}(|x|) x \\
-|x| \sin \big( |x| \widetilde{f}(|x|) \big) \widetilde{g}(|x|)
\end{pmatrix}
\,,
\end{equation*}
with $\widetilde{f},\widetilde{g} \in C^{\infty}_{\mathrm{ev}}([0,\infty))$ and
\begin{equation*}
\renewcommand{\arraystretch}{1.2}
\begin{array}{rcll}
(F,G) &=& U^{\ast}_{1}[0] & \text{in } \RR^{n} \setminus \BB^{n}_{r} \,, \\
F^{n+1} &\geq& -1 + \varepsilon_{r} & \text{in } \overline{\BB^{n}_{r}} \,.
\end{array}
\end{equation*}
\end{lemma}
\begin{proof}
First, note that
\begin{equation}
\label{CorotationalMapExtended}
F^{i}(x) = \operatorname{si} \big( \big| \widetilde{f}(|x|) x \big| \big) \widetilde{f}(|x|) x^{i}
\qquad\text{and}\qquad
F^{n+1}(x) = \cos \big( \big| \widetilde{f}(|x|) x \big| \big)
\end{equation}
in terms of the analytic sinus cardinalis function given by $\operatorname{si}(\rho) = \sin(\rho) / \rho$ with $\operatorname{si}(0) = 1$. Since $F^{n+1} > -1$, we have $0 \leq \big| \widetilde{f}(|x|) x \big| < \pi$ for all $|x| \leq r$ by continuity. Thus, the identity $\big| \widetilde{f}(|x|) x \big| = \arccos( F^{n+1}(x) )$ holds, implying with \Cref{CorotationalMapExtended}
\begin{equation*}
\widetilde{f}(|x|) x^{i} = \varphi( F^{n+1}(x) ) F^{i}(x)
\end{equation*}
for $x \in \RR^{d}$, where
\begin{equation*}
\varphi : (-1,1) \rightarrow \RR \,, \qquad z \mapsto \frac{\arccos(z)}{\sqrt{1-z^{2}}} \,.
\end{equation*}
For $z \in (0,1)$, we note that
\begin{equation*}
\varphi(z) = \frac{\arccos(z)}{\sqrt{1-z^{2}}} = \frac{\arcsin(\sqrt{1-z^{2}})}{\sqrt{1-z^{2}}} = \int_{0}^{1} \frac{1}{\sqrt{1 + \zeta^{2} (z^{2} - 1)}} \dd \zeta
\end{equation*}
and get that $\varphi$ has an analytic continuation beyond $z=1$. Therefore, all derivatives of $\varphi$ are bounded near $z = 1$. Moreover, the identity
\begin{equation*}
\widetilde{g}(|x|) x^{i} = F^{n+1}(x) G^{i}(x) - F^{i}(x) G^{n+1}(x)
\end{equation*}
holds. Now, we conclude with \cite[Theorem 1.3]{2022arXiv220902286O}
\begin{align*}
&
\big\| \widetilde{f}(|\,.\,|) - \widetilde{\psi}^{\ast}_{1}(0,|\,.\,|) \big\|_{H^{k}(\RR^{n+2})} \simeq
\sum_{i=1}^{n} \big\| (\,.\,)^{i} \big( \widetilde{f}(|\,.\,|) - \widetilde{\psi}^{\ast}_{1}(0,|\,.\,|) \big) \big\|_{H^{k}(\RR^{n})} \\&\indent \lesssim
\sum_{i=1}^{n} \big\| \varphi(F^{n+1}) \big\|_{H^{k}(\RR^{n})} \big\| F^{i} - {U^{\ast}_{1}}^{i}(0,\,.\,) \big\|_{H^{k}(\RR^{n})} +
\big\| \varphi(F^{n+1}) - \varphi( {U^{\ast}_{1}}^{n+1}(0,\,.\,) ) \big\|_{H^{k}(\RR^{n})}
\end{align*}
and
\begin{align*}
&
\big\| \widetilde{g}(|\,.\,|) - \pd_{0}\widetilde{\psi}^{\ast}_{1}(0,|\,.\,|) \big\|_{H^{k-1}(\RR^{n+2})} \simeq
\sum_{i=1}^{n} \big\| (\,.\,)^{i} \big( \widetilde{g}(|\,.\,|) - \pd_{0}\widetilde{\psi}^{\ast}_{1}(0,|\,.\,|) \big) \big\|_{H^{k-1}(\RR^{n})} \\&\indent\lesssim
\sum_{i=1}^{n} \Bigg( \big\| F^{n+1}\big( G^{i} - \pd_{0} {U^{\ast}_{1}}^{i}(0,\,.\,) \big) \big\|_{H^{k-1}(\RR^{n})} + \big\| F^{i}\big( G^{n+1} - \pd_{0} {U^{\ast}_{1}}^{n+1}(0,\,.\,) \big) \big\|_{H^{k-1}(\RR^{n})} \\&\indent+ \big\| \big( F^{n+1} - {U^{\ast}_{1}}^{n+1}(0,\,.\,) \big) \pd_{0} {U^{\ast}_{1}}^{i}(0,\,.\,) \big\|_{H^{k-1}(\RR^{n})} + \big\| \big( F^{i} - {U^{\ast}_{1}}^{i}(0,\,.\,) \big) \pd_{0} {U^{\ast}_{1}}^{n+1}(0,\,.\,) \big\|_{H^{k-1}(\RR^{n})} \Bigg)
\end{align*}
for all corotational $(F,G):\RR^{n} \rightarrow \mathrm{T}\mathbb{S}^{n} \subset \RR^{n+1} \times \RR^{n+1}$ as in the assumption. We remark that the assumption $F^{n+1} \geq -1 + \varepsilon_{r}$ ensures a uniform bound for terms of the form $\varphi^{(\ell)} \circ F^{n+1}$ while the upper bound for $\varepsilon_{r} > 0$ leaves enough room for $F^{n+1}$ to be arbitrarily close to
\begin{equation*}
{U_{1}^{\ast}}^{n+1}(0,x) = \frac{(n-2) - |x|^{2}}{(n-2) + |x|^{2}} = - 1 + \frac{2(n-2)}{(n-2) + |x|^{2}} \,.
\end{equation*}
From here, the bounds follow from an application of H\"{o}lder's inequality and Sobolev embedding like in the proof of \Cref{LocLip}.
\end{proof}
The proofs of the main theorems are an application of \Cref{THM}.
\begin{proof}[Proof of \Cref{THMWM}]
First, pick $\delta^{\ast} > 0$ and $C^{\ast} \geq 1$ from \Cref{THM}. There is a constant $C \geq C^{\ast}$ such that the function from \Cref{CorotDataBound} satisfies $C_{n,k,r,\varepsilon_{r}} \leq C$ on $[0, \delta^{\ast} / (C^{\ast})^{3}]$. Put $M \coloneqq C^{3}$ and let $0 < \delta \leq \delta^{\ast}$. Let $(F,G):\RR^{n} \rightarrow \mathrm{T}\mathbb{S}^{n} \subset \RR^{n+1} \times \RR^{n+1}$ with $f,g \in C^{\infty}_{\mathrm{rad}}(\RR^{n})$ be data as in the assumption of the theorem. Let $\widetilde{f}, \widetilde{g} \in C^{\infty}_{\mathrm{ev}}([0,\infty))$ be the respective radial profile of $f,g \in C^{\infty}_{\mathrm{rad}}(\RR^{n})$. We define $f_{0}, g_{0} \in C^{\infty}_{\mathrm{rad}}(\RR^{n+2})$ by
\begin{equation*}
f_{0}(x) = \widetilde{f}(|x|) - \widetilde{\psi}^{\ast}_{1}(0,|x|)
\qquad\text{and}\qquad
g_{0}(x) = \widetilde{g}(|x|) - \pd_{0} \widetilde{\psi}^{\ast}_{1}(0,|x|)
\end{equation*}
for $x \in \RR^{n+2}$. Then $\operatorname{supp}(f_{0},g_{0}) \subset \BB^{n+2}_{r}$ and
\begin{align*}
\| (f_{0},g_{0}) \|_{H^{k}(\RR^{n+2}) \times H^{k-1}(\RR^{n+2})} &\leq C \big\| \big(F,G) - \big( U^{\ast}_{1}(0,\,.\,), \pd_{0} U^{\ast}_{1}(0,\,.\,) \big) \big\|_{H^{k}(\RR^{n})\times H^{k-1}(\RR^{n})} \\&\leq
C \frac{\delta}{M} = \frac{\delta}{C^{2}}
\end{align*}
by \Cref{CorotDataBound}. We infer from \Cref{THM} a $T^{\ast} \in [1 - \frac{\delta}{C}, 1 + \frac{\delta}{C}]$ and a unique solution $\psi \in C^{\infty} \big( \Omega^{1,n+2}_{T^{\ast}}(\kappa) \big)$ to
\begin{align*}
(\Box \psi)(t,x) - \frac{n-1}{2} \frac{\sin\big( 2 |x| \psi(t,x) \big) - 2 |x| \psi(t,x)}{|x|^{3}} = 0 & \quad \text{in } \Omega^{1,n+2}_{T^{\ast}}(\kappa) \,, \\
\psi(0,\,.\,) = \psi^{\ast}_{1}(0,\,.\,) + f_{0}
\quad\text{and}\quad
\pd_{0}\psi(0,\,.\,) = \pd_{0}\psi^{\ast}_{1}(0,\,.\,) + g_{0} & \quad
\text{in } \RR^{n+2} \,,
\end{align*}
which is radially symmetric. Let $\widetilde{\psi}$ be the radial profile of $\psi$ and define in $\RR^{1,n}$ the smooth corotational map
\begin{equation*}
U : \Omega^{1,n}_{T^{\ast}}(\kappa) \rightarrow \mathbb{S}^{n} \subset \RR^{n+1} \,, \qquad
U(t,x) =
\begin{pmatrix}
|x|^{-1} \sin \big( |x| \widetilde{\psi}(t,|x|) \big) x \\
\cos \big( |x| \widetilde{\psi}(t,|x|) \big)
\end{pmatrix}
\,.
\end{equation*}
We get
\begin{align*}
&
(\pd^{\mu} \pd_{\mu} U)(t,x) + \big\langle (\pd^{\mu} U)(t,x), (\pd_{\mu} U)(t,x) \big\rangle U(t,x) \\&\indent\indent=
\Big(
-\widetilde{\psi}_{tt} + \widetilde{\psi}_{rr} + \frac{n+1}{r} \widetilde{\psi}_{r} -
\frac{n-1}{2} \frac{\sin\big( 2 r \widetilde{\psi} \big) - 2 r \widetilde{\psi}}{r^{3}}
\Big)
\begin{pmatrix}
\cos\big( |x| \widetilde{\psi}(t,|x|) \big) x \\
- |x| \sin \big( |x| \widetilde{\psi}(t,|x|) \big)
\end{pmatrix}
\\&\indent\indent=
0 \,,
\end{align*}
as well as $U(0, \,.\,) = F$ and $\pd_{0} U(0, \,.\,) = G$
in $\RR^{n}$, where we abbreviate $r = |x|$ and
\begin{equation*}
\widetilde{\psi} \equiv \widetilde{\psi}(t,r) \,, \qquad
\widetilde{\psi}_{r} \equiv \pd_{r} \widetilde{\psi}(t,r) \,, \qquad
\widetilde{\psi}_{rr} \equiv \pd_{r}^{2} \widetilde{\psi}(t,r) \,, \qquad
\widetilde{\psi}_{tt} \equiv \pd_{t}^{2} \widetilde{\psi}(t,r) \,,
\end{equation*}
for $(t,x) \in \RR^{1,n}$. Finally, to show the stability estimate for this wave map, we set
\begin{equation*}
\widetilde{u}_{1}(\tau,\rho) = \big( T^{\ast} \ee^{-\tau} \big) \big( \widetilde{\psi} - \widetilde{\psi}^{\ast}_{T^{\ast}} \big) \circ (T^{\ast}+T^{\ast}\ee^{-\tau}\widetilde{h}(\rho),T^{\ast}\ee^{-\tau}\rho)
\end{equation*}
for $(\tau,\rho) \in [0,\infty) \times [0,R]$. Note that the stability estimate in \Cref{THM} asserts
\begin{equation*}
\| \widetilde{u}_{1}(\tau, |\,.\,|) \|_{H^{k}(\BB^{n+2}_{R})} \lesssim \delta \ee^{-\omega\tau}
\end{equation*}
for all $\tau \geq 0$. In terms of this function, an application of the fundamental theorem of calculus implies the identities
\begin{align*}
\big( U^{i} - U^{\ast}_{T^{\ast}}\!\!^{i} \big) \circ {\chi\mathstrut}_{T^{\ast}}(\tau,\xi) &= \xi^{i}\widetilde{u}_{1}(\tau,|\xi|) \int_{0}^{1} \cos(|\,.\,|) \circ \big( \widetilde{\psi}^{*}(-\widetilde{h}(|\xi|),|\xi|) \xi + z \widetilde{u}_{1}(\tau,|\xi|) \xi \big) \dd z \,, \\
\big( U^{n+1} - U^{\ast}_{T^{\ast}}\!\!^{n+1} \big) \circ {\chi\mathstrut}_{T^{\ast}}(\tau,\xi) &= \xi^{i} \widetilde{u}_{1}(\tau,|\xi|) \int_{0}^{1} \pd_{i} \cos(|\,.\,|) \circ \big( \widetilde{\psi}^{*}(-\widetilde{h}(|\xi|),|\xi|) \xi + z \widetilde{u}_{1}(\tau,|\xi|) \xi \big) \dd z \,,
\end{align*}
for $(\tau,\xi) \in [0,\infty) \times \overline{\BB^{n}_{R}}$. As $H^{k}(\BB^{n}_{R})$ is a Sobolev algebra, we get with \cite[Theorem 1.3]{2022arXiv220902286O}
\begin{align*}
\big\| \big( U - U^{\ast}_{T^{\ast}} \big) \circ {\chi\mathstrut}_{T^{\ast}}(\tau,\,.\,) \big\|_{H^{k}(\BB^{n}_{R})} &= \sum_{i=1}^{n+1} \big\| \big( U^{i} - {U^{\ast}_{T^{\ast}}}\!\!^{i} \big) \circ {\chi\mathstrut}_{T^{\ast}}(\tau,\,.\,) \big\|_{H^{k}(\BB^{n}_{R})} \\&\lesssim
\sum_{i=1}^{n} \big\| (\,.\,)^{i} \widetilde{u}_{1}(\tau, |\,.\,|) \big\|_{H^{k}(\BB^{n}_{R})} \\&\simeq \| \widetilde{u}_{1}(\tau, |\,.\,|) \|_{H^{k}(\BB^{n+2}_{R})} \\&\lesssim \delta \ee^{-\omega\tau}
\end{align*}
for all $\tau \geq 0$.
\end{proof}
\begin{proof}[Proof of \Cref{THMYM}]
There are smooth even functions $\widetilde{f},\widetilde{g} \in C^{\infty}_{\mathrm{ev}}([0,\infty))$ such that the initial data in the theorem are given by
\begin{equation*}
a_{\mu}{}^{ij}(x) = \widetilde{f}(|x|)( x^{i} \delta_{\mu}{}^{j} - x^{j} \delta_{\mu}{}^{i}) \,, \qquad b_{\mu}{}^{ij}(x) = \widetilde{g}(|x|) ( x^{i} \delta_{\mu}{}^{j} - x^{j} \delta_{\mu}{}^{i}) \,,
\end{equation*}
in $\RR^{n}$. Then \cite[Theorem 1.3]{2022arXiv220902286O} yields
\begin{equation*}
\| (a,b) \|_{H^{k}(\RR^{n}) \times H^{k-1}(\RR^{n})} \simeq \big\| \big( \widetilde{f}(|\,.\,|),\widetilde{g}(|\,.\,|) \big) \big\|_{H^{k}(\RR^{n+2}) \times H^{k-1}(\RR^{n+2})} \,.
\end{equation*}
According to \Cref{THM}, let $\psi \in C^{\infty} \big( \Omega^{1,n+2}_{T^{\ast}}(\kappa) \big)$ be the unique solution to the Cauchy problem
\begin{align*}
(\Box \psi)(t,x) + \big( n-2 \big) \big( 3 \psi(t,x)^{2} - |x|^{2} \psi(t,x)^{3} \big) = 0 & \quad \text{in } \Omega^{1,n+2}_{T^{\ast}}(\kappa) \,, \\
\psi(0,\,.\,) = \psi^{\ast}_{1}(0,\,.\,) + \widetilde{f}(|\,.\,|)
\quad\text{and}\quad
\pd_{0}\psi(0,\,.\,) = \pd_{0}\psi^{\ast}_{1}(0,\,.\,) + \widetilde{g}(|\,.\,|) & \quad
\text{in } \RR^{n+2} \,.
\end{align*}
The solution $\psi$ is radially symmetric with radial profile $\widetilde{\psi}$ and we set
\begin{align*}
A_{\mu} &: \Omega^{1,n}_{T^{\ast}}(\kappa) \rightarrow \mathfrak{so}(n) \,, & {A_{\mu}}^{ij}(t,x) &= \widetilde{\psi}(t,|x|) ( x^{i} \delta_{\mu}{}^{j} - x^{j} \delta_{\mu}{}^{i}) \,, \\
F_{\mu\nu} &: \Omega^{1,n}_{T^{\ast}}(\kappa) \rightarrow \mathfrak{so}(n) \,, & F_{\mu\nu} &= \pd_{\mu} A_{\nu} - \pd_{\nu} A_{\mu} + [A_{\mu}, A_{\nu}] \,,
\end{align*}
for $\mu,\nu = 0,1,\ldots,n$. It follows that
\begin{align*}
\pd^{\mu} {F_{\mu 0}}^{ij} + [A^{\mu}, F_{\mu 0}]^{ij} &= 0 \,, \\
\pd^{\mu} {F_{\mu \ell}}^{ij} + [A^{\mu}, F_{\mu \ell}]^{ij} &= - \big( -\widetilde{\psi}_{tt} + \widetilde{\psi}_{rr} + \tfrac{n+1}{r} \widetilde{\psi}_{r} + ( n-2 ) ( 3 \widetilde{\psi}^{2} - r^{2} \widetilde{\psi}^{3} ) \big) \big( x^{i} \delta_{\ell}^{j} - x^{j} \delta_{\ell}^{i} \big) \\&= 0 \,,
\end{align*}
as well as $A_{\mu}(0, \,.\,) = {A_{1}^{\ast}}_{\mu}(0,\,.\,) + a_{\mu}$ and $\pd_{0} A_{\mu}(0, \,.\,) = \pd_{0} {A_{1}^{\ast}}_{\mu}(0,\,.\,) + b_{\mu}$. Hence $A_{\mu}$ is a smooth Yang-Mills connection. To infer the stability estimate, we compute
\begin{align*}
\big( {A_{\mu}}^{ij} - {A^{\ast}_{T^{\ast}}}_{\mu}{}^{ij} \big) \circ {\chi\mathstrut}_{T^{\ast}}(\tau,\xi) &= T^{*} \ee^{-\tau} \big( \psi - \psi^{\ast}_{T^{\ast}} \big) \circ {\chi\mathstrut}_{T^{\ast}}(\tau,\xi) ( \xi^{i} \delta_{\mu}{}^{j} - \xi^{j} \delta_{\mu}{}^{i}) \\&=
\frac{\ee^{\tau}}{T^{*}} \widetilde{u}_{1}(\tau,|\xi|) ( \xi^{i} \delta_{\mu}{}^{j} - \xi^{j} \delta_{\mu}{}^{i}) \,,
\end{align*}
where
\begin{equation*}
\widetilde{u}_{1}(\tau,\rho) = \big( T^{\ast} \ee^{-\tau} \big)^{2} \big( \widetilde{\psi} - \widetilde{\psi}^{\ast}_{T^{\ast}} \big) \circ (T^{\ast}+T^{\ast}\ee^{-\tau}\widetilde{h}(\rho),T^{\ast}\ee^{-\tau}\rho)
\end{equation*}
for $(\tau,\rho) \in [0,\infty) \times [0,R]$. This gives with the definition of corotational Sobolev norms and \cite[Theorem 1.3]{2022arXiv220902286O}
\begin{align*}
\frac{\big\| \big( A - A^{\ast}_{T^{\ast}} \big) \circ {\chi\mathstrut}_{T^{\ast}}(\tau,\,.\,) \big\|_{H^{k}(\BB^{n}_{R})}}{\big\| A^{\ast}_{T^{\ast}} \circ {\chi\mathstrut}_{T^{\ast}}(\tau,\,.\,) \big\|_{H^{k}(\BB^{n}_{R})}} &\simeq
\sum_{i=1}^{n} \big\| (\,.\,)^{i} \widetilde{u}_{1}(\tau, |\,.\,|) \big\|_{H^{k}(\BB^{n}_{R})} \\&\simeq \| \widetilde{u}_{1}(\tau, |\,.\,|) \|_{H^{k}(\BB^{n+2}_{R})} \\&\lesssim
\delta \ee^{-\omega \tau}
\end{align*}
for all $\tau \geq 0$.
\end{proof}
\appendix
\section{Similarity coordinates}
\label{AppendixSimilarityCoordinates}
In this section, we study coordinate systems on Minkowski spacetime $\RR^{1,d}$ that are adapted to self-similar solutions of nonlinear wave equations. This property can be formalized quite generally by defining that a map $\chi : \RR \times \RR^{d} \rightarrow \RR^{1,d}$ is a \emph{similarity variable} if there exists a function $\lambda : \RR \times \RR^{d} \rightarrow \RR$ such that
\begin{equation*}
\chi(\tau,\xi) = \lambda(\tau,\xi) \chi(0,\xi)
\end{equation*}
for all $(\tau, \xi) \in \RR \times \RR^{d}$. If $\tau$ is interpreted as a time variable and $\xi$ is interpreted as a space variable, this expresses that moving forward along the new time coordinate while the new space coordinate is kept fix, occurs via scaling of an initial reference point. Instead of studying similarity variables in the above generality, we will choose exponential time
\begin{equation*}
\lambda(\tau,\xi) = \ee^{-\tau}
\end{equation*}
for the scaling factor, which becomes convenient for the transformation of differential operators later on. Also, we require compatibility with the wave evolution, that is, hypersurfaces of constant coordinate time shall be spacelike. Since spacelike hypersurfaces in Minkowski spacetime are the image of a graph, we focus on similarity variables of the following form.
\begin{definition}
\label{GraphicalSimilarityVariable}
Let $h\in C^{\infty}(\RR^{d})$. We define \emph{graphical similarity variables} by
\begin{equation*}
\chi :\RR \times \RR^{d}\rightarrow\RR^{1,d} \,, \quad \chi(\tau,\xi) = (\ee^{-\tau}h(\xi),\ee^{-\tau}\xi) \,.
\end{equation*}
\end{definition}
We remark that by a spacetime translation, graphical similarity variables may be centred around any given event in spacetime. When similarity variables define a coordinate system on their image, we call them \emph{similarity coordinates}. The function $h$ is referred to as a \emph{height function}. One of the simplest examples of a height function is given by
\begin{equation}
\label{ExampleSSC}
h(\xi) = -1 \,,
\end{equation}
yielding \emph{standard similarity coordinates}. It is even possible to cover the whole complement of a future light cone with similarity coordinates. This is demonstrated in \cite{MR4338226} with \emph{hyperboloidal similarity coordinates} and a height function of the form
\begin{equation}
\label{ExampleHSC}
h(\xi) = \alpha \sqrt{1 + \alpha^{-2}|\xi|^{2}} - \alpha - \beta
\end{equation}
for some $\alpha\geq 1$, $\beta > 0$. The examples in \Cref{ExampleSSC,ExampleHSC} already bear the characterising properties of height functions in order to yield similarity coordinates on the complement of a cone, which we discuss next.
\subsection{Geometry of graphical similarity coordinates}
We impose a convexity condition on the height function in \Cref{GraphicalSimilarityVariable} and introduce an auxiliary function to determine the geometry of graphical similarity variables.
\begin{lemma}
\label{GSCMap}
Let $h\in C^{\infty}(\RR^{d})$ with $h(0)<0$ and with positive semidefinite Hessian $(\pd^{2} h)(\xi)$ for all $\xi\in\RR^{d}$. Consider the smooth map $F \in C^{\infty} \big( (0,\infty)\times\RR \times \RR^{d} \big)$ given by
\begin{equation*}
F(\alpha,\beta,\xi) = \alpha h(\alpha^{-1}\xi) - \beta \,.
\end{equation*}
Then, for each $(\beta,\xi)\in\RR \times \RR^{d}$ the map $F(\,.\,,\beta,\xi): (0,\infty) \rightarrow \RR$ is strictly monotonously decreasing and the scaling relation
\begin{equation*}
F(\alpha,\lambda\beta,\lambda\xi) = \lambda F(\lambda^{-1}\alpha,\beta,\xi)
\end{equation*}
holds for all $\lambda>0$.
\end{lemma}
\begin{proof}
First, consider the smooth function $c \in C^{\infty}(\RR^{d})$ given by
\begin{equation*}
c(\xi) = \xi^{i} (\pd_{i} h)(\xi) - h(\xi) \,.
\end{equation*}
Let $z\geq 0$. Then
\begin{equation*}
\pd_{z} c(z\xi) = z \xi^{i} \xi^{j} (\pd_{i} \pd_{j} h)(z\xi) \geq 0
\qquad\text{implies}\qquad
c(\xi) \geq c(0) = - h(0) > 0
\end{equation*}
for all $\xi \in \RR^{d}$. We find
\begin{equation}
\label{PartialF}
\pd_{\alpha} F(\alpha,\beta,\xi) = - c(\alpha^{-1}\xi) \leq h(0) < 0 \,,
\end{equation}
so $F(\,.\,,\beta,\xi)$ is strictly monotonously decreasing. The scaling relation follows from
\begin{equation*}
F(\alpha,\lambda\beta,\lambda\xi) = \alpha h(\alpha^{-1} \lambda\xi) - \lambda\beta = \lambda( (\lambda^{-1} \alpha) h((\lambda^{-1} \alpha)^{-1}\xi) - \beta) = \lambda F(\lambda^{-1}\alpha,\beta,\xi) \,.
\qedhere
\end{equation*}
\end{proof}
If the height function satisfies the assumptions of \Cref{GSCMap}, then graphical similarity variables provide a coordinate system on their image.
\begin{lemma}
\label{GSCLemma}
Let $h\in C^{\infty}(\RR^{d})$ with $h(0)<0$ and positive semidefinite Hessian $(\pd^{2} h)(\xi)$ for all $\xi\in\RR^{d}$. Then, the image $\mathrm{X}^{1,d} \subset \RR^{1,d}$ of similarity variables $\chi: \RR \times \RR^{d} \rightarrow \RR^{1,d}$ is open and
\begin{equation*}
\chi: \RR \times \RR^{d} \rightarrow \mathrm{X}^{1,d} \,, \quad \chi(\tau,\xi) = \big( \ee^{-\tau}h(\xi),\ee^{-\tau}\xi \big) \,,
\end{equation*}
is a diffeomorphism with inverse
\begin{equation*}
\chi^{-1}: \mathrm{X}^{1,d} \rightarrow \RR \times \RR^{d} \,, \quad
\chi\mathstrut^{-1}(t,x) = \big( -\log h^{+}(t,x), h^{+}(t,x)^{-1} x \big) \,,
\end{equation*}
for a positive function $h^{+} \in C^{\infty}(\mathrm{X}^{1,d})$. Moreover, if $(t,x)\in\mathrm{X}^{1,d}$, then for any $\lambda>0$ also $( \lambda t, \lambda x)\in\mathrm{X}^{1,d}$ and
\begin{equation*}
h^{+}(\lambda t, \lambda x) = \lambda h^{+}(t,x) \,.
\end{equation*}
\end{lemma}
\begin{proof}
According to the inverse function theorem, it suffices to show that $\chi$ is an injective local diffeomorphism. The assumptions on $h$ yield with Taylor's theorem
\begin{equation*}
h(\xi) - \xi^{i} (\pd_{i} h)(\xi) = h(0) - \frac{1}{2} \int_{0}^{1} \xi^{i} \xi^{j} ( \pd_{i} \pd_{j} h ) ( (1-z)\xi ) (1-z) \dd z \leq h(0) < 0
\end{equation*}
and hence the Jacobian determinant
\begin{equation*}
\det(\pd\chi)(\tau,\xi) = \ee^{-(d+1)\tau}(\xi^{i} (\pd_{i}h)(\xi) - h(\xi))
\end{equation*}
vanishes nowhere. So $\chi$ is a local diffeomorphism. To show that $\chi$ is injective, suppose that $(\tau_{1},\xi_{1}),(\tau_{2},\xi_{2})\in\RR \times \RR^{d}$ with $\chi(\tau_{1},\xi_{1}) = \chi(\tau_{2},\xi_{2})$. This gives
\begin{equation*}
\ee^{-\tau_{1}}h(\xi_{1}) = \ee^{-\tau_{2}}h(\xi_{2}) \,, \quad \ee^{-\tau_{1}}\xi_{1} = \ee^{-\tau_{2}}\xi_{2} \,,
\end{equation*}
which implies
\begin{equation*}
h(\xi_{1}) = \ee^{\tau_{1}-\tau_{2}} h(\ee^{-(\tau_{1}-\tau_{2}})\xi_{1})
\qquad\text{or equivalently}\qquad
F(1,0,\xi_{1}) = F \big( \ee^{\tau_{1}-\tau_{2}},0,\xi_{1} \big)
\end{equation*}
in terms of the function $F$ from \Cref{GSCMap}. Strict monotonicity of $F$ in the first argument yields $\tau_{1} = \tau_{2}$ and so $\chi$ is injective. Next, let $(\tau,\xi)\in\RR \times \RR^{d}$ and set $(t,x) = \chi(\tau,\xi)$. That is,
\begin{equation*}
t = \ee^{-\tau} h(\xi) \,, \qquad x = \ee^{-\tau} \xi \,,
\end{equation*}
which implies
\begin{equation*}
\ee^{-\tau} h \big( \ee^{\tau} x \big) -t = 0 \,.
\end{equation*}
In terms of the function $F$ from \Cref{GSCMap}, this condition means that $F(\,.\,,t,x)$ has a unique zero $\alpha = h^{+}(t,x) > 0$ which satisfies
\begin{equation*}
\tau = -\log h^{+}(t,x) \,, \qquad \xi = \frac{1}{h^{+}(t,x)} x \,.
\end{equation*}
Finally, note that $(t,x) = \chi(\tau,\xi)$ gives $( \lambda t, \lambda x) = \chi(\tau-\log\lambda,\xi)$. Then $h^{+}( \lambda t, \lambda x)$ is the unique zero of $F(\,.\,,\lambda t, \lambda x)$ and the scaling condition in \Cref{GSCMap} yields
\begin{equation*}
0 = F \big( h^{+}(\lambda t, \lambda x), \lambda t, \lambda x \big) =
\lambda F \big( \lambda^{-1} h^{+}(\lambda t, \lambda x), t, x \big) \,,
\end{equation*}
so by uniqueness $\lambda^{-1} h^{+}(\lambda t, \lambda x) = h^{+}(t,x)$.
\end{proof}
Some geometric quantities of Minkowski spacetime are computed in similarity coordinates in \cite[Appendix A]{MR4661000}. Next, let us also study the time slices of graphical similarity variables.
\begin{lemma}
\label{GeometryFoliation}
Let $h\in C^{\infty}(\RR^{d})$ with $|(\pd h)(\xi)|<1$ for all $\xi\in\RR^{d}$. Then, for all $\tau\in\RR$ the level set
\begin{equation*}
\Sigma^{1,d}(\tau) = \chi \big( \{\tau\} \times \RR^{d} \big)
\end{equation*}
is a smooth spacelike hypersurface in $\RR^{1,d}$. The matrix of the first fundamental form of $\Sigma^{1,d}(\tau)$ and its inverse matrix have components
\begin{equation*}
\gamma(\tau)_{ij} = \ee^{-2\tau} \big( \delta_{ij} - (\pd_{i} h) (\pd_{j} h) \big) \,, \qquad
\gamma(\tau)^{ij} = \ee^{2\tau} \Big( \delta^{ij} + \frac{(\pd^{i} h)(\pd^{j} h)}{1 - |\pd h|^{2}} \Big) \,,
\end{equation*}
respectively, and metric determinant
\begin{equation*}
\det \gamma(\tau) = \ee^{-2d\tau} \big( 1 - |\pd h|^{2} \big) \,.
\end{equation*}
Moreover, for all $\tau\in\RR$,
\begin{equation*}
\nu(\tau) = \frac{(1,\pd h)}{\sqrt{1-|\pd h|^{2}}} \in C^{\infty}(\RR^{d},\RR^{1,d})
\end{equation*}
defines the components of a future-directed timelike unit normal vector field for $\Sigma^{1,d}(\tau)$.
\end{lemma}
\begin{proof}
For each $\tau\in\RR$, the smooth map
\begin{equation*}
\chi(\tau,\,.\,):\RR^{d} \rightarrow \RR^{1,d} \,, \qquad \xi \mapsto (\ee^{-\tau}h(\xi),\ee^{-\tau}\xi) \,,
\end{equation*}
defines a chart for $\Sigma^{1,d}(\tau)$. Hence, $\Sigma^{1,d}(\tau)$ is a $d$-dimensional smooth submanifold of $\RR^{1,d}$. The components of the pullback of the Minkowski metric $\eta$ on $\RR^{1,d}$ under $\Sigma^{1,d}(\tau)$ are given by
\begin{equation*}
\gamma(\tau)_{ij} = \eta \big( \pd_{i} \chi(\tau,\,.\,), \pd_{j} \chi(\tau,\,.\,) \big) = \ee^{-2\tau} \big( \delta_{ij} - (\pd_{i} h) (\pd_{j} h) \big) \,.
\end{equation*}
If $(\pd h)(\xi) = 0$, this matrix is a multiple of the identity matrix. If $(\pd h)(\xi) \neq 0$, the matrix has the $(d-1)$-fold eigenvalue $\ee^{-2\tau}$ with eigenvectors orthogonal to $(\pd h)(\xi) \in \RR^{d}$ and the simple eigenvalue $\ee^{-2\tau} (1-|(\pd h)(\xi)|^{2})$ with eigenvector $(\pd h)(\xi) \in \RR^{d}$. Since $|(\pd h)(\xi)| < 1$, the matrix is positive definite and thus $\Sigma^{1,d}(\tau)$ is a smooth spacelike hypersurface. The product of these eigenvalues gives the formula for the metric determinant. The inverse of the matrix $\gamma(\tau)$ can be computed from a Neumann series. Finally, we have
\begin{equation*}
\eta \big( \pd_{i}\chi(\tau,\,.\,), \nu(\tau) \big) = 0 \,, \qquad \eta \big( \nu(\tau),\nu(\tau) \big) = -1 \,,
\end{equation*}
for $i=1,\ldots,d$.
\end{proof}
It follows that the image of similarity coordinates is foliated by its time slices. This foliation also covers future lightlike infinity in the following sense.
\begin{lemma}
\label{GSCFoliation}
Let $h \in C^{\infty}(\RR^{d})$ such that
\begin{equation*}
h(0) < 0 \,, \qquad
|(\pd h)(\xi)| < 1 \,, \qquad
(\pd^{2} h)(\xi) \text{ is positive semidefinite,}
\end{equation*}
for all $\xi\in\RR^{d}$. Let $\tau\neq\tau'$. Then $\Sigma^{1,d}(\tau)$ is not asymptotic to $\Sigma^{1,d}(\tau')$.
\end{lemma}
\begin{proof}
Note that for fixed $\tau\in\RR$, each level set
\begin{equation*}
\Sigma^{1,d}(\tau) = \{ (\ee^{-\tau}h(\ee^{\tau}\xi),\xi) \in \RR^{1,d} \mid \xi \in \RR^{d} \}
\end{equation*}
from \Cref{GeometryFoliation} is the graph of the function $\xi \mapsto \ee^{-\tau}h(\ee^{\tau}\xi) = F(\ee^{-\tau},0,\xi)$, where $F$ is defined in \Cref{GSCMap}. We show that none of these graphs are asymptotic to each other. As in the proof of \Cref{GSCMap}, we compute
\begin{equation*}
\pd_{\tau} F(\ee^{-\tau},0,\xi) = \xi^{i} (\pd_{i} h)(\ee^{\tau}\xi) - h(\ee^{\tau}\xi) \geq - \ee^{-\tau} h(0)
\end{equation*}
for all $(\tau,\xi) \in \RR \times \RR^{d}$ and conclude for $\tau \neq \tau'$ that
\begin{equation*}
F(\ee^{-\tau},0,\xi) - F(\ee^{-\tau'},0,\xi) \geq - h(0) \Big( \ee^{-\tau'} - \ee^{-\tau} \Big) > 0
\end{equation*}
for all $\xi \in \RR^{d}$.
\end{proof}
\subsection{Radially symmetric height functions}
So far, we have formulated conditions for graphical similarity variables to yield a coordinate system on their image. Here, we describe the image in case the height function is radially symmetric. We begin with locating the backward light cone of the origin in similarity coordinates.
\begin{lemma}
\label{GSCLightCone}
Let $h \in C^{\infty}_{\mathrm{rad}}(\RR^{d})$ with $h(0) < 0$ and $|(\pd h)(\xi)| < 1$ for all $\xi\in\RR^{d}$ and with radial profile $\widetilde{h} \in C^{\infty}_{\mathrm{ev}}([0,\infty))$. Then, the function
\begin{equation*}
[0,\infty) \rightarrow \RR \,, \qquad \rho \mapsto \widetilde{h}(\rho)^{2} - \rho^{2} \,,
\end{equation*}
is strictly monotonously decreasing and has a unique simple zero $R_{0} > 0$.
\end{lemma}
\begin{proof}
The assumptions imply $0 \leq \widetilde{h}'(\rho) < 1$ and $\widetilde{h}(\rho) < \rho$. This gives $\widetilde{h}(\rho)\widetilde{h}'(\rho) - \rho < 0$ for all $\rho > 0$ and so $\rho \mapsto \widetilde{h}(\rho)^{2} - \rho^{2}$ is strictly monotonously decreasing. Since $\widetilde{h}'(\rho) - 1 < 0$ and $\widetilde{h}(0) < 0$, we have $\widetilde{h}(\rho) - \rho < 0$ for all $\rho \geq 0$. Because $\widetilde{h}(0) + \rho < \widetilde{h}(\rho) + \rho < 2\rho$ and $\widetilde{h}'(\rho) + 1 > 1$, there is a unique $R_{0} > 0$ such that $\widetilde{h}(R_{0}) + R_{0} = 0$. This is the unique zero of $\widetilde{h}^{2}(\rho) - \rho^{2} = ( \widetilde{h}(\rho) + \rho ) ( \widetilde{h}(\rho) - \rho )$.
\end{proof}
The image of similarity coordinates with a radially symmetric height function is the complement of a cone.
\begin{lemma}
\label{GSCImage}
Let $h \in C^{\infty}_{\mathrm{rad}}(\RR^{d})$ such that
\begin{equation*}
h(0) < 0 \,, \qquad
|(\pd h)(\xi)| < 1 \,, \qquad
(\pd^{2} h)(\xi) \text{ is positive semidefinite,}
\end{equation*}
for all $\xi\in\RR^{d}$ and with radial profile $\widetilde{h} \in C^{\infty}_{\mathrm{ev}}([0,\infty))$. Let $R>0$. Then, graphical similarity coordinates $\chi:\RR \times \RR^{d} \rightarrow \RR^{1,d}$ have the image regions
\begin{align*}
\chi(\RR \times \RR^{d}) &= \big\{ (t,x) \in \RR^{1,d} \mid \kappa |x| > t \big\} \,, \\
\chi \big( \RR\times\BB^{d}_{R} \big) &= \big\{ (t,x) \in \RR^{1,d} \mid {\kappa\mathstrut}_{R} |x| > t \big\} \,, \\
\chi \big( (0,\infty)\times\BB^{d}_{R} \big) &= \big\{ (t,x) \in \RR^{1,d} \mid {\kappa\mathstrut}_{R} |x| > t > \widetilde{h}(|x|) \big\} \,,
\end{align*}
where
\begin{equation*}
\kappa = \lim_{\rho\to\infty} \widetilde{h}'(\rho)
\qquad\text{and}\qquad
{\kappa\mathstrut}_{R} = \frac{\widetilde{h}(R)}{R} \,.
\end{equation*}
\end{lemma}
\begin{proof}
\Cref{GSCLemma} yields that $\chi:\RR \times \RR^{d} \rightarrow \RR^{1,d}$ are coordinates on their image. Now, suppose $(t,x) \in \RR^{1,d}$ with $(t,x) = \chi(\tau,\xi)$ for some $(\tau,\xi) \in \RR \times \RR^{d}$. If $x=0$ then $\xi = 0$ and $t = \ee^{-\tau} \widetilde{h}(0) < 0$. If $x \neq 0$, then
\begin{equation}
\label{ImageT}
\frac{t}{|x|} = \frac{\widetilde{h}(|\xi|)}{|\xi|} \,.
\end{equation}
Since the map $(0,\infty) \ni \rho \mapsto \rho^{-1} \widetilde{h}(\rho)$ is strictly monotonously increasing by \Cref{GSCMap} with
\begin{equation*}
\lim_{\rho\to\infty} \frac{\widetilde{h}(\rho)}{\rho} = \lim_{\rho\to\infty} \widetilde{h}'(\rho) \eqqcolon \kappa \in [0,1] \,,
\end{equation*}
we get $\kappa |x| > t$. If additionally $0 < |\xi| < R$ and $\tau > 0$ we conclude from strict monotonicity that ${\kappa\mathstrut}_{R} |x| > t$. On the other hand, fix $(t,x)\in\RR^{1,d}$ with $\kappa |x| > t$ and consider the map $\tau \mapsto F\big( \ee^{-\tau},t,x\big) = \ee^{-\tau} \widetilde{h}(\ee^{\tau}|x|)-t$. By \Cref{GSCMap}, this map is strictly monotonously increasing and we have
\begin{equation*}
\lim_{\tau\to-\infty} F\big( \ee^{-\tau},t,x \big) = - \infty
\qquad\text{and}\qquad
\lim_{\tau\to\infty} F\big( \ee^{-\tau},t,x\big) = \kappa|x| - t > 0 \,,
\end{equation*}
so there exists a unique $\tau \in \RR$ with $F\big( \ee^{-\tau},t,x) = 0$. Then $\tau$ and $\xi \coloneqq \ee^{\tau} x$ satisfy $\chi(\tau,\xi) = (t,x)$. If we are given ${\kappa\mathstrut}_{R} |x| > t > \widetilde{h}(|x|)$, then \Cref{ImageT} and strict monotonicity imply $\tau > 0$ and $|\xi| < R$.
\end{proof}
The previous lemma also shows that coordinate cylinders get mapped to complements of cones. On such regions, we consider the transition diffeomorphism between radial similarity coordinates.
\begin{lemma}
\label{TransitionDiffeo}
Let $\chi,\overline{\chi} : \RR \times \RR^{d} \rightarrow \RR^{1,d}$ be graphical similarity coordinates with respective radially symmetric height functions $h,\overline{h} \in C^{\infty}_{\mathrm{rad}}(\RR^{d})$ as in \Cref{GSCImage}. Let $R,\overline{R}>0$ such that $\chi ( \RR \times \BB^{d}_{R} ) = \overline{\chi} ( \RR \times \BB^{d}_{\overline{R}} )$. Then, there is a positive function $h_{+} \in C^{\infty}_{\mathrm{rad}}(\RR^{d})$ such that the transition diffeomorphism $\chi^{-1} \circ \overline{\chi} : \RR \times \BB^{d}_{\overline{R}} \rightarrow \RR \times \BB^{d}_{R}$ is given by
\begin{equation*}
\big( \chi^{-1} \circ \overline{\chi} \big)(\overline{\tau},\overline{\xi}) = \big( \overline{\tau} - \log h_{+}(\overline{\xi}), h_{+}(\overline{\xi})^{-1} \overline{\xi} \big) \,.
\end{equation*}
\end{lemma}
\begin{proof}
By \Cref{GSCImage}, the transition diffeomorphism $\chi^{-1} \circ \overline{\chi} : \RR \times \BB^{d}_{\overline{R}} \rightarrow \RR \times \BB^{d}_{R}$ is well-defined. By \Cref{GSCLemma}, there is a positive radially symmetric function $h^{+} \in C^{\infty}_{\mathrm{rad}}(\RR^{d})$ such that
\begin{equation*}
\big(\chi^{-1}\circ\overline{\chi} \big)(\overline{\tau},\overline{\xi}) = \big( \overline{\tau} - \log h^{+} \big( \overline{h}(\overline{\xi}), \overline{\xi} \big), h^{+} \big( \overline{h}(\overline{\xi}), \overline{\xi} \big)^{-1} \overline{\xi} \big) \,.
\end{equation*}
Now set $h_{+}(\overline{\xi}) = h^{+} \big( \overline{h}(\overline{\xi}), \overline{\xi} \big) > 0$. Since $h,\overline{h} \in C^{\infty}_{\mathrm{rad}}(\RR^{d})$, also $h_{+} \in C^{\infty}_{\mathrm{rad}}(\RR^{d})$.
\end{proof}
\subsection{Construction of height functions}
Similarity coordinates with a radially symmetric height function as in \Cref{GSCImage} can be concretely implemented in various ways. To give some examples, we first translate positive semidefiniteness of a radial function to nonnegativity of the second derivative of its radial profile.
\begin{lemma}
\label{RadialHeight}
Let $h \in C^{\infty}_{\mathrm{rad}}(\RR^{d})$ with radial profile $\widetilde{h} \in C^{\infty}_{\mathrm{ev}}([0,\infty))$. Then, the Hessian matrix $(\pd^{2} h)(\xi)$ is positive semidefinite for all $\xi\in\RR^{d}$ if and only if $\widetilde{h}''(\rho) \geq 0$ for all $\rho \geq 0$.
\end{lemma}
\begin{proof}
We have for the second partial derivatives of $h$ the identity
\begin{equation*}
(\pd_{i} \pd_{j} h)(\xi) = (D\widetilde{h})(|\xi|) \delta_{ij} + (D^{2}\widetilde{h})(|\xi|) \xi_{i} \xi_{j}
\end{equation*}
in terms of the operation $(D\widetilde{h})(\rho) = \widetilde{h}'(\rho)/\rho$. Now, suppose that $(\pd^{2} h)(\xi)$ is positive semidefinite for all $\xi\in\RR^{d}$. Then
\begin{equation*}
0 \leq v^{\top} (\pd^{2} h)(\xi) v = (D\widetilde{h})(|\xi|) |v|^{2} + (D^{2}\widetilde{h})(|\xi|) (\xi \cdot v)^{2}
\end{equation*}
for all $v,\xi\in\RR^{d}$, where $v^{\top}$ denotes the transpose vector. Thus $(D\widetilde{h})(|\xi|) \geq 0$ and $(D^{2}\widetilde{h})(|\xi|) \geq 0$ for all $\xi\in\RR^{d}$, hence also
\begin{equation*}
\widetilde{h}''(\rho) = (D \widetilde{h})(\rho) + \rho^{2} (D^{2} \widetilde{h})(\rho) \geq 0
\end{equation*}
for all $\rho \geq 0$. Conversely, suppose that $\widetilde{h}''(\rho) \geq 0$ for all $\rho \geq 0$. Since $\widetilde{h}'(0) = 0$, it follows that also $\widetilde{h}'(\rho) \geq 0$ for all $\rho \geq 0$. Hence $(\pd^{2} h)(\xi)$ is positive semidefinite for all $\xi\in\RR^{d}$.
\end{proof}
We end this section with the presentation of an exemplary class of radially symmetric height functions. The associated similarity coordinates have the convenient property to coincide with standard similarity coordinates in a backward light cone, whereas beyond they can be arranged to extend to any possible image according to \Cref{GSCImage}.
\begin{lemma}
\label{GSCExemplary}
Let $0 < \overline{\kappa} < 1$, $\overline{r} > 0$ and $0 < \varepsilon < \overline{r}$. Then, there is a function $\overline{h} \in C^{\infty}_{\mathrm{rad}}(\RR^{d})$ with positive semidefinite Hessian $(\pd^{2} \overline{h})(\xi)$ and bounded gradient $|(\pd \overline{h})(\xi)| < 1$ for all $\xi \in \RR^{d}$ which satisfies
\begin{equation*}
\overline{h}(\xi) =
\renewcommand{\arraystretch}{1.2}
\left\{
\begin{array}{ll}
\overline{\kappa} ( |\xi| - \overline{r} ) - 1 & \text{if } |\xi| \geq \overline{r} + \varepsilon \,, \\
- 1 & \text{if } |\xi| \leq \overline{r} - \varepsilon \,.
\end{array}
\right.
\end{equation*}
\end{lemma}
\begin{proof}
Let $\phi,\phi_{\varepsilon} \in C^{\infty}(\RR)$ be given by
\begin{equation*}
\phi(\rho) =
\renewcommand{\arraystretch}{1.2}
\left\{
\begin{array}{ll}
\ee^{-\frac{1}{1-\rho^{2}}} \,, & |\rho|<1 \,, \\
0 \,, & |\rho| \geq 1 \,,
\end{array}
\right.
\qquad\text{and}\qquad
\phi_{\varepsilon}(\rho) = \frac{\varepsilon^{-1}}{\| \phi \|_{L^{1}(\RR)}} \phi(\varepsilon^{-1}\rho) \,.
\end{equation*}
Let $\mathsf{h} \ast {\phi\mathstrut}_{\varepsilon}$ be the convolution of the function $\mathsf{h}\in C(\RR)$ given by
\begin{equation*}
\mathsf{h}(\rho) = \frac{\overline{\kappa}}{2} \big( |\rho| - \overline{r} + \big| |\rho| - \overline{r} \big| \big) - 1
\end{equation*}
with the function $\phi_{\varepsilon} \in C^{\infty}(\RR)$. Note that $\mathsf{h} \ast {\phi\mathstrut}_{\varepsilon} \in C^{\infty}(\RR)$ is even and given by
\begin{equation*}
(\mathsf{h} \ast {\phi\mathstrut}_{\varepsilon})(\rho) =
\renewcommand{\arraystretch}{1.5}
\left\{
\begin{array}{ll}
\mathsf{h}(\rho) & \text{if } | \rho - \overline{r} | \geq \varepsilon \,, \\
\displaystyle{
\overline{\kappa} \int_{-\varepsilon}^{\rho-\overline{r}} (\rho - z - \overline{r}) {\phi\mathstrut}_{\varepsilon}(z) \dd z - 1
} & \text{if } | \rho - \overline{r} | < \varepsilon \,.
\end{array}
\right.
\end{equation*}
It follows from \Cref{RadialHeight} that $\overline{h} = (\mathsf{h} \ast {\phi\mathstrut}_{\varepsilon})(|\,.\,|) \in C^{\infty}_{\mathrm{rad}}(\RR^{d})$ has a positive semidefinite Hessian.
\end{proof}
\section{The free wave flow in similarity coordinates}
\label{AppendixFreeWave}
In this section, we develop a general functional analytic framework for the wave operator
\begin{equation*}
(\Box u)(t,x) = (-\pd_{t}^{2} + \Delta_{x} ) u(t,x)
\end{equation*}
in similarity coordinates without imposing any symmetry restrictions on the wave flow. First, we give a formulation of the wave flow as a first-order system in similarity coordinates which is encompassed by methods from semigroup theory. Once this is provided, we work out the necessary analytical foundations for an application of the Lumer-Phillips theorem \cite[p.~83, Theorem 3.15]{MR1721989}. This concerns dissipative estimates in \Cref{EnergyDissipativity} and a density property of the range in \Cref{DenseRange}. As a result, we establish in \Cref{TheGenerationTHM} the existence and sharp growth bounds for the wave flow in similarity coordinates.
\subsection{A first-order formalism}
\label{SubSec1stOrderForm}
Throughout this section, we employ the following version of similarity coordinates from \Cref{AppendixSimilarityCoordinates}.
\begin{definition}
\label{AssumptionsAppendix}
Let $d\in\NN$. Let $h \in C^{\infty}_{\mathrm{rad}}(\RR^{d})$ such that
\begin{equation*}
h(0) < 0 \,, \qquad
|(\pd h)(\xi)| < 1 \,, \qquad
(\pd^{2} h)(\xi) \text{ is positive semidefinite} \,,
\end{equation*}
for all $\xi \in \RR^{d}$. Let $R > 0$. We define \emph{graphical similarity coordinates}
\begin{equation*}
\chi :\RR \times \RR^{d}\rightarrow\RR^{1,d} \,, \qquad \chi(\tau,\xi) = (\ee^{-\tau}h(\xi),\ee^{-\tau}\xi) \,,
\end{equation*}
in the image region
\begin{equation*}
\mathrm{X}^{1,d}_{R} \coloneqq \chi \big( (0,\infty)\times\BB^{d}_{R} \big) \,.
\end{equation*}
\end{definition}
Graphical similarity coordinates are well-adapted to the wave evolution as we have seen in \Cref{GSCLemma,GeometryFoliation,GSCFoliation}. The image region is explicitly determined in \Cref{GSCImage}. We frequently consider the space
\begin{equation*}
C^{\infty} \big( \overline{\mathrm{X}}^{1,d}_{R} \big) \coloneqq \big\{ u \in C^{\infty} \big( \mathrm{X}^{1,d}_{R} \big) \mid u \circ \chi \in C^{\infty} \big( \overline{ (0,\infty) \times \BB^{d}_{R} } \big) \big\}
\end{equation*}
of smooth functions in the image region whose derivatives of all orders are bounded away from the center of the coordinate system. Now, we introduce the following differential operator.
\begin{definition}
\label{WaveFlowOperation}
Let $d\in\NN$ and $R>0$. Let $c,w \in C^{\infty}_{\mathrm{rad}}(\RR^{d})$ be given by
\begin{equation*}
c(\xi) = \xi^{i} (\pd_{i} h)(\xi) - h(\xi)
\qquad\text{and}\qquad
w(\xi) = 1-(\pd^{i} h)(\xi)(\pd_{i} h)(\xi) \,.
\end{equation*}
For $\mathbf{f} \in C^{\infty} ( \overline{\BB^{d}_{R}} )^{2}$, we define
\begin{equation*}
\renewcommand{\arraystretch}{1.5}
\mathbf{L}_{\chi} \mathbf{f} =
\begin{bmatrix}
- \xi^{i} \pd_{i} f_{1}
+ c f_{2} \\
\displaystyle{
\frac{c}{w} \pd^{i} \pd_{i} f_{1}
- \Big( 1 + \frac{c}{w} \pd^{i} \pd_{i} h \Big) f_{2}
- \Big( \xi^{i} + 2 \frac{c}{w} \pd^{i} h \Big) \pd_{i} f_{2}
}
\end{bmatrix}
\in C^{\infty} ( \overline{\BB^{d}_{R}} )^{2} \,.
\end{equation*}
\end{definition}
This definition is based on the formulation of the wave operator in similarity coordinates as a first-order flow.
\begin{lemma}
\label{TransitionRelation}
Let $u\in C^{\infty} \big( \overline{\mathrm{X}}^{1,d}_{R} \big)$ and define $\mathbf{u}
\in C^{\infty} \big( \overline{ (0,\infty) \times \BB^{d}_{R} } \big)^{2}$ with components
\begin{equation*}
u_{1}(\tau,\xi) = (u\circ\chi)(\tau,\xi) \,, \qquad
u_{2}(\tau,\xi) = \ee^{-\tau} \big( (\pd_{0} u) \circ\chi \big) (\tau,\xi) \,.
\end{equation*}
Then
\begin{equation*}
\pd_{\tau}
\mathbf{u}(\tau,\,.\,)
=
\mathbf{L}_{\chi}
\mathbf{u}(\tau,\,.\,)
-
\begin{bmatrix}
0 \\
\displaystyle{
\ee^{-2\tau} \frac{c}{w} \big( (\Box u) \circ \chi \big)(\tau,\,.\,)
}
\end{bmatrix}
\,.
\end{equation*}
\end{lemma}
\begin{proof}
The transformation for partial derivatives of $u\in C^{\infty} \big( \overline{\mathrm{X}}^{1,d}_{R} \big)$ reads
\begin{align}
\label{pd0SC}
\big( (\pd_{0} u) \circ \chi \big) (\tau,\xi) &= \ee^{\tau} \frac{1}{c(\xi)} (\pd_{\tau} + \xi^{i}\pd_{\xi^{i}}) (u\circ \chi)(\tau,\xi) \,, \\
\label{pdiSC}
\big( (\pd_{i} u) \circ \chi \big) (\tau,\xi) &= \ee^\tau \pd_{\xi^{i}} (u\circ \chi)(\tau,\xi) - (\pd_{i} h)(\xi) \big( (\pd_{0} u) \circ \chi \big) (\tau,\xi) \,.
\end{align}
From this we obtain
\begin{equation*}
\pd_{\tau} u_{1}(\tau,\xi) = - \xi^{i}\pd_{\xi^{i}} u_{1}(\tau,\xi) + c(\xi) u_{2}(\tau,\xi)
\end{equation*}
and
\begin{align*}
\big( (\Box u) \circ \chi \big)(\tau,\xi) &= \big( (- \pd_{0}^{2} u + \pd^{i} \pd_{i} u) \circ \chi \big) (\tau,\xi) \\&= \ee^{2\tau} \Big(
- \frac{w(\xi)}{c(\xi)} \pd_{\tau} u_{2}(\tau,\xi)
- \Big( \frac{w(\xi)}{c(\xi)} \xi^{i} + 2(\pd^{i} h)(\xi) \Big) \pd_{\xi^{i}} u_{2}(\tau,\xi) \\&\indent-
\Big( \frac{w(\xi)}{c(\xi)} + ( \pd^{i} \pd_{i} h)(\xi) \Big) u_{2}(\tau,\xi) +
\pd_{\xi^{i}} \pd_{\xi_{i}} u_{1}(\tau,\xi) \Big) \,,
\end{align*}
which yields the asserted relation.
\end{proof}
\subsection{Commuting differential operators}
\label{SubSecCommDiffOp}
We implement classical spacetime derivatives as differential operators in the first-order formalism.
\begin{definition}
\label{CommutingOperators}
Let $d\in \NN$ and $R>0$. Let $c_{\mu}\in C^{\infty}(\RR^{d})$ be given by
\begin{equation*}
c_{\mu} = 1 \quad \text{for } \mu = 0
\qquad\text{and}\qquad
c_{\mu} = - \pd_{\mu} h \quad \text{for } \mu = 1,\ldots,d \,.
\end{equation*}
For $\mu=0,1,\ldots,d$ and $\mathbf{f} \in C^{\infty} ( \overline{\BB^{d}_{R}} )^{2}$, we define
\begin{equation*}
\renewcommand{\arraystretch}{1.5}
\mathbf{D}_{\mu} \mathbf{f} =
\begin{bmatrix}
\delta_{\mu}{}^{i}\pd_{i} f_{1} + c_{\mu} f_{2} \\
\displaystyle{
\frac{c_{\mu}}{w} \pd^{i} \pd_{i} f_{1}
- \frac{c_{\mu}}{w} (\pd^{i} \pd_{i} h) f_{2}
+ \Big( \delta_{\mu}{}^{i} - 2\frac{c_{\mu}}{w} \pd^{i} h \Big) \pd_{i} f_{2}
}
\end{bmatrix}
\in C^{\infty} ( \overline{\BB^{d}_{R}} )^{2} \,.
\end{equation*}
\end{definition}
This relates the evolution variables in the sense of \Cref{TransitionRelation} and the spacetime derivatives of the corresponding function as follows.
\begin{lemma}
\label{PartialDRelation}
Let $u\in C^{\infty}\big( \overline{\mathrm{X}}^{1,d}_{R} \big)$ and define $\mathbf{u}, \mathbf{u}_{\mu} \in C^{\infty} \big( \overline{ (0,\infty) \times \BB^{d}_{R} } \big)^{2}$ with components
\begin{align*}
u_{1}(\tau,\xi) &= (u\circ \chi)(\tau,\xi) \,, &
u_{2}(\tau,\xi) &= \ee^{-\tau} \big( (\pd_{0} u)\circ \chi \big)(\tau,\xi) \,, \\
u{_{\mu}}_{1}(\tau,\xi) &= \big( (\pd_{\mu} u)\circ \chi\big)(\tau,\xi) \,, &
u{_{\mu}}_{2}(\tau,\xi) &= \ee^{-\tau} \big( (\pd_{0} \pd_{\mu} u)\circ \chi\big)(\tau,\xi) \,,
\end{align*}
for $\mu = 0,1,\ldots,d$, respectively. Then
\begin{equation*}
\mathbf{u}_{\mu}(\tau,\,.\,)
=\ee^{\tau}
\mathbf{D}_{\mu} \mathbf{u}(\tau,\,.\,)
-
\begin{bmatrix}
0 \\
\displaystyle{
\ee^{-\tau} \frac{c_{\mu}}{w} \big( (\Box u) \circ \chi \big)(\tau,\,.\,)
}
\end{bmatrix}
\,.
\end{equation*}
\end{lemma}
\begin{proof}
We have from the transformations \eqref{pd0SC} and \eqref{pdiSC}
\begin{align}
\notag
\big( (\pd_{\mu}u) \circ \chi\big)(\tau,\xi) &= e^{\tau} \delta_{\mu}{}^{i} \pd_{\xi^{i}} (u\circ \chi)(\tau,\xi) + c_{\mu}(\xi) \big( (\pd_{0} u) \circ \chi\big)(\tau,\xi) \\&=
\label{pdmuxitau}
e^{\tau} \Big( \delta_{\mu}{}^{i} + \frac{c_{\mu}(\xi)}{c(\xi)} \xi^{i} \Big) \pd_{\xi^{i}} (u\circ \chi)(\tau,\xi) + e^{\tau} \frac{c_{\mu}(\xi)}{c(\xi)} \pd_{\tau} (u\circ \chi)(\tau,\xi) \,.
\end{align}
This gives
\begin{equation*}
u{_{\mu}}_{1}(\tau,\xi) = e^{\tau} \delta_{\mu}{}^{i} \pd_{\xi^{i}} u_{1}(\tau,\xi) + e^{\tau} c_{\mu}(\xi) u_{2}(\tau,\xi)
\end{equation*}
and together with \Cref{TransitionRelation}
\begin{align*}
&
u{_{\mu}}_{2}(\tau,\xi) = e^{-\tau} \big( (\pd_{\mu}\pd_{0}u) \circ \chi \big)(\tau,\xi) \\&\indent=
\Big( \delta_{\mu}{}^{i} + \frac{c_{\mu}(\xi)}{c(\xi)} \xi^{i} \Big) \pd_{\xi^{i}} \big( (\pd_{0}u)\circ \chi \big)(\tau,\xi) + \frac{c_{\mu}(\xi)}{c(\xi)} \pd_{\tau} \big( (\pd_{0}u)\circ \chi \big)(\tau,\xi) \\&\indent=
e^{\tau} \frac{c_{\mu}(\xi)}{w(\xi)} \pd_{\xi^{i}} \pd_{\xi_{i}} u_{1}(\tau,\xi) - e^{\tau} \frac{c_{\mu}(\xi)}{w(\xi)} (\pd^{i} \pd_{i} h)(\xi) u_{2}(\tau,\xi) + e^{\tau} \Big( \delta_{\mu}{}^{i} - 2 \frac{c_{\mu}(\xi)}{w(\xi)} (\pd^{i} h)(\xi) \Big) \pd_{\xi^{i}} u_{2}(\tau,\xi) \\&\indent\indent-
e^{-\tau} \frac{c_{\mu}(\xi)}{w(\xi)} \big( (\Box u) \circ \chi \big)(\tau,\xi) \,,
\end{align*}
which yields the relation.
\end{proof}
The differential operators in \Cref{CommutingOperators} inherit the commutation relations for classical spacetime derivatives. In particular, they commute with the wave evolution operator from \Cref{WaveFlowOperation} up to a decay-inducing term.
\begin{lemma}
\label{DmuLdCommutationRelation}
Let $d\in\NN$ and $R>0$. Then
\begin{equation*}
\mathbf{D}_{\mu} \mathbf{L}_{\chi} \mathbf{f} - \mathbf{L}_{\chi} \mathbf{D}_{\mu} \mathbf{f} = - \mathbf{D}_{\mu} \mathbf{f}
\qquad\text{and}\qquad
\mathbf{D}_{\mu} \mathbf{D}_{\nu} \mathbf{f} = \mathbf{D}_{\nu} \mathbf{D}_{\mu} \mathbf{f}
\end{equation*}
for all $\mathbf{f}\in C^{\infty} ( \overline{\BB^{d}_{R}} )^{2}$ and all $\mu,\nu = 0,1,\ldots,d$.
\end{lemma}
\begin{proof}
Let $u\in C^{\infty} \big( \overline{\mathrm{X}}^{1,d}_{R} \big)$ and define the variables $\mathbf{u}, \mathbf{u}_{\mu} \in C^{\infty} \big( \overline{ (0,\infty) \times \BB^{d}_{R} } \big)^{2}$ as in \Cref{PartialDRelation}. \Cref{TransitionRelation} applied to $\pd_{\mu} u \in C^{\infty} \big( \overline{\mathrm{X}}^{1,d}_{R} \big)$ yields with \Cref{PartialDRelation}
\begin{equation*}
\pd_{\tau} \mathbf{u}_{\mu}(\tau,\,.\,)
=
\ee^{\tau} \mathbf{L}_{\chi} \mathbf{D}_{\mu}
\mathbf{u}(\tau,\,.\,)
-
\mathbf{L}_{\chi}
\begin{bmatrix}
0 \\
\displaystyle{
\ee^{-\tau} \frac{c_{\mu}}{w} \big( (\Box u) \circ \chi \big)(\tau,\,.\,)
}
\end{bmatrix}
-
\begin{bmatrix}
0 \\
\displaystyle{
\ee^{-2\tau} \frac{c}{w} \big( (\Box \pd_{\mu} u) \circ \chi \big)(\tau,\,.\,)
}
\end{bmatrix}
\,.
\end{equation*}
The left-hand side can be computed in a second way by first differentiating the relation in \Cref{PartialDRelation} and then using \Cref{TransitionRelation},
\begin{align*}
\pd_{\tau} \mathbf{u}_{\mu}(\tau,\,.\,)
&=
\ee^{\tau}
\mathbf{D}_{\mu} \mathbf{u}(\tau,\,.\,)
+
\ee^{\tau}\mathbf{D}_{\mu} \mathbf{L}_{\chi} \mathbf{u}(\tau,\,.\,)
-
\mathbf{D}_{\mu}
\begin{bmatrix}
0 \\
\displaystyle{
\ee^{-\tau} \frac{c}{w} \big( (\Box u) \circ \chi \big)(\tau,\,.\,)
}
\end{bmatrix}
\\&\indent-
\pd_{\tau}
\begin{bmatrix}
0 \\
\displaystyle{
\ee^{-\tau} \frac{c_{\mu}}{w} \big( (\Box u) \circ \chi \big)(\tau,\,.\,)
}
\end{bmatrix}
\,.
\end{align*}
Subtracting the latter two identities from each other and using the identities
\begin{align*}
\pd_{\xi^{i}} \frac{c(\xi)}{w(\xi)} &= \Big( \xi^{j} + 2 \frac{c(\xi)}{w(\xi)} (\pd^{j} h)(\xi) \Big) \frac{(\pd_{i}\pd_{j} h)(\xi)}{w(\xi)} \,, \\
\pd_{\xi^{i}} \frac{c_{\mu}(\xi)}{w(\xi)} &= -\Big( \delta_{\mu}{}^{j} - 2 \frac{c_{\mu}(\xi)}{w(\xi)} (\pd^{j} h)(\xi) \Big) \frac{(\pd_{i}\pd_{j} h)(\xi)}{w(\xi)} \,,
\end{align*}
as well as transformation \eqref{pdmuxitau} implies
\begin{equation}
\label{DiLdcommutationfield}
\Big(
\mathbf{D}_{\mu} \mathbf{L}_{\chi} -
\mathbf{L}_{\chi} \mathbf{D}_{\mu} +
\mathbf{D}_{\mu}
\Big) \mathbf{u}(\tau,\,.\,)
=
\mathbf{0} \,.
\end{equation}
Furthermore, consider the variable $\mathbf{u}_{\mu\nu} \in C^{\infty} \big( \overline{ (0,\infty) \times \BB^{d}_{R} } \big)^{2}$ with components
\begin{equation*}
u{_{\mu\nu}}_{1}(\tau,\,.\,) = \big( (\pd_{\mu}\pd_{\nu} u) \circ \chi \big)(\tau,\,.\,) \,, \quad
u{_{\mu\nu}}_{2}(\tau,\,.\,) = \ee^{-\tau} \big( (\pd_{0} \pd_{\mu}\pd_{\nu} u) \circ \chi \big)(\tau,\,.\,) \,.
\end{equation*}
By construction, two applications of \Cref{PartialDRelation} yield
\begin{equation*}
\mathbf{u}_{\mu\nu}(\tau,\,.\,)
=
\ee^{2\tau} \mathbf{D}_{\mu}\mathbf{D}_{\nu}
\mathbf{u}(\tau,\,.\,)
-
\mathbf{D}_{\mu}
\begin{bmatrix}
0 \\
\displaystyle{
\frac{c_{\nu}}{w} \big( (\Box u) \circ \chi \big)(\tau,\,.\,)
}
\end{bmatrix}
-
\begin{bmatrix}
0 \\
\displaystyle{
\ee^{-\tau} \frac{c_{\mu}}{w} \big( (\Box \pd_{\nu} u) \circ \chi \big)(\tau,\,.\,)
}
\end{bmatrix}
\,.
\end{equation*}
A computation shows
\begin{align*}
\mathbf{D}_{\mu}
&
\begin{bmatrix}
0 \\
\displaystyle{
\frac{c_{\nu}}{w} \big( (\Box u) \circ \chi \big)(\tau,\,.\,)
}
\end{bmatrix}
-
\mathbf{D}_{\nu}
\begin{bmatrix}
0 \\
\displaystyle{
\frac{c_{\mu}}{w} \big( (\Box u) \circ \chi \big)(\tau,\,.\,)
}
\end{bmatrix}
\\&=
-
\begin{bmatrix}
0 \\
\displaystyle{
\ee^{-\tau} \frac{c_{\mu}}{w} \big( (\pd_{\nu} \Box u)\circ \chi \big)(\tau,\,.\,) - \ee^{-\tau} \frac{c_{\nu}}{w} \big( (\pd_{\mu} \Box u)\circ \chi \big)(\tau,\,.\,)
}
\end{bmatrix}
\,.
\end{align*}
Hence, symmetry of second partial derivatives yields
\begin{equation}
\label{DmuDnucommutationfield}
\Big( \mathbf{D}_{\mu}\mathbf{D}_{\nu} - \mathbf{D}_{\nu}\mathbf{D}_{\mu} \Big) \mathbf{u}(\tau,\,.\,)
= \mathbf{0} \,.
\end{equation}
To conclude, let $\mathbf{f} \in C^{\infty} ( \overline{\BB^{d}_{R}} )^{2}$ and choose $u = v \circ\chi^{-1} \in C^{\infty} \big( \overline{\mathrm{X}}^{1,d}_{R} \big)$ with
\begin{equation*}
v(\tau,\xi) = f_{1}(\xi) + \tau \big( c(\xi) f_{2}(\xi) - \xi^{i} \pd_{\xi^{i}} f_{1}(\xi) \big) \,.
\end{equation*}
Then $\mathbf{u}(0, \,.\,) = \mathbf{f}$ and evaluating \Cref{DmuDnucommutationfield,DiLdcommutationfield} at $\tau = 0$ yields the result.
\end{proof}
\subsection{Inner products}
\label{SubSecIP}
With these preparations at hand, we are ready to furnish Hilbert space structures and prove dissipative estimates for the wave evolution operator from \Cref{WaveFlowOperation} in all space dimensions. The energy discussed in \Cref{SubSecEnTop} motivates the following choice of sesquilinear forms.
\begin{definition}
\label{1InnerProduct}
Let $d\in\NN$ and $R>0$. Fix $\varepsilon_{1}>0$. We define for $\mathbf{f},\mathbf{g} \in C^{\infty} ( \overline{\BB^{d}_{R}} )^{2}$
\begin{equation*}
\Big( \mathbf{f} \,\Big|\, \mathbf{g} \Big)_{\mathfrak{E}^{1}(\BB^{d}_{R})} = \int_{\BB^{d}_{R}} \overline{\pd^{i} f_{1}}\pd_{i} g_{1} + \int_{\BB^{d}_{R}} \overline{f_{2}} g_{2} w + \frac{2 \varepsilon_{1}}{R} \int_{\pd\BB^{d}_{R}} \overline{f_{1}} g_{1} \,.
\end{equation*}
The induced quadratic form is defined by
\begin{equation*}
\| \mathbf{f} \|_{\mathfrak{E}^{1}(\BB^{d}_{R})} = \sqrt{\Big( \mathbf{f} \,\Big|\, \mathbf{f} \Big)_{\mathfrak{E}^{1}(\BB^{d}_{R})}} \,.
\end{equation*}
\end{definition}
The boundary term ensures that this sesquilinear form is positive definite and thus an inner product.
\begin{lemma}
\label{1Norm}
Let $d\in\NN$ and $R>0$. Fix $\varepsilon_{1}>0$ in \Cref{1InnerProduct}. Then
\begin{equation*}
\| \mathbf{f} \|_{\mathfrak{E}^{1}(\BB^{d}_{R})} \simeq \| ( f_{1},f_{2} ) \|_{H^{1}(\BB^{d}_{R}) \times L^{2}(\BB^{d}_{R})}
\end{equation*}
for all $\mathbf{f} = (f_{1},f_{2}) \in C^{\infty} ( \overline{\BB^{d}_{R}} )^{2}$.
\end{lemma}
\begin{proof}
We have
\begin{equation*}
\| \mathbf{f} \|_{\mathfrak{E}^{1}(\BB^{d}_{R})} \simeq \| \pd f_{1} \|_{L^{2}(\BB^{d}_{R})} + \| f_{1} \|_{L^{2}(\pd\BB^{d}_{R})} + \| f_{2} \|_{L^{2}(\BB^{d}_{R})} \simeq \| f_{1} \|_{H^{1}(\BB^{d}_{R})} + \| f_{2} \|_{L^{2}(\BB^{d}_{R})}
\end{equation*}
for all $\mathbf{f} \in C^{\infty} ( \overline{\BB^{d}_{R}} )^{2}$, where the second equivalence follows from the trace lemma \cite[Lemma 2.1]{MR4778061}.
\end{proof}
We utilize the differential operators from \Cref{CommutingOperators} to extend the inner product in \Cref{1InnerProduct} to also include derivatives of any given nonnegative integer order.
\begin{definition}
\label{InnerProducts}
Let $d,k\in\NN$ and $R>0$. Fix $\varepsilon_{j}>0$ for $1 \leq j < \frac{d}{2} + 1$. We define for $\mathbf{f},\mathbf{g} \in C^{\infty} ( \overline{\BB^{d}_{R}} )^{2}$ recursively
\begin{equation*}
\renewcommand{\arraystretch}{1.5}
\Big( \mathbf{f} \,\Big|\, \mathbf{g} \Big)_{\mathfrak{E}^{k}(\BB^{d}_{R})} =
\left\{
\begin{array}{ll}
\displaystyle{
\int_{\BB^{d}_{R}} \overline{\pd^{i} f_{1}}\pd_{i} g_{1} + \int_{\BB^{d}_{R}} \overline{f_{2}} g_{2} w + \frac{2 \varepsilon_{1}}{R} \int_{\pd\BB^{d}_{R}} \overline{f_{1}} g_{1}
} & \text{if } k = 1 \,, \\
\displaystyle{
\sum_{\mu=0}^{d} \Big( \mathbf{D}_{\mu} \mathbf{f} \,\Big|\, \mathbf{D}_{\mu} \mathbf{g} \Big)_{\mathfrak{E}^{k-1}(\BB^{d}_{R})} +
\frac{2\varepsilon_{k}}{R} \int_{\pd\BB^{d}_{R}} \overline{f_{1}} g_{1}
} & \text{if } 2\leq k < \frac{d}{2} + 1 \,,
\\ \displaystyle{
\sum_{\mu=0}^{d} \Big( \mathbf{D}_{\mu} \mathbf{f} \,\Big|\, \mathbf{D}_{\mu} \mathbf{g} \Big)_{\mathfrak{E}^{k-1}(\BB^{d}_{R})} +
\Big( \mathbf{f} \,\Big|\, \mathbf{g} \Big)_{\mathfrak{E}^{k-1}(\BB^{d}_{R})}
} & \text{if } \frac{d}{2} + 1 \leq k \,.
\end{array}
\right.
\end{equation*}
The induced quadratic form is defined by
\begin{equation*}
\| \mathbf{f} \|_{\mathfrak{E}^{k}(\BB^{d}_{R})} = \sqrt{\Big( \mathbf{f} \,\Big|\, \mathbf{f} \Big)_{\mathfrak{E}^{k}(\BB^{d}_{R})}} \,.
\end{equation*}
\end{definition}
These sesquilinear forms are inner products whose induced norms are equivalently characterized as Sobolev norms.
\begin{proposition}
\label{kNorm}
Let $d,k\in\NN$ and $R>0$. Fix $\varepsilon_{j} > 0$ for $1\leq j < \frac{d}{2} + 1$ in \Cref{InnerProducts}. Then the equivalence
\begin{equation*}
\| \mathbf{f} \|_{\mathfrak{E}^{k}(\BB^{d}_{R})}
\simeq
\| (f_{1},f_{2}) \|_{H^{k}(\BB^{d}_{R})\times H^{k-1}(\BB^{d}_{R})}
\end{equation*}
holds for all $\mathbf{f} = (f_{1},f_{2}) \in C^{\infty} ( \overline{\BB^{d}_{R}} )^{2}$.
\end{proposition}
\begin{proof}
\begin{enumerate}[wide, itemsep=1em, topsep=1em]
\item[``$\lesssim$'':] The inequality is valid in the base case $k=1$ by \Cref{1Norm}. We proceed by induction and assume that for an arbitrary but fixed $k\in\NN$ the bound
\begin{equation*}
\| \mathbf{f} \|_{\mathfrak{E}^{k}(\BB^{d}_{R})} \lesssim \| (f_{1},f_{2}) \|_{H^{k}(\BB^{d}_{R})\times H^{k-1}(\BB^{d}_{R})}
\end{equation*}
holds for all $\mathbf{f}\in C^{\infty} ( \overline{\BB^{d}_{R}} )^{2}$. Hence,
\begin{align*}
\| \mathbf{D}_{\mu} \mathbf{f} \|_{\mathfrak{E}^{k}(\BB^{d}_{R})} &\lesssim \| \mathbf{D}_{\mu} \mathbf{f}\|_{H^{k}(\BB^{d}_{R}) \times H^{k-1}(\BB^{d}_{R})} \\&\lesssim
\| \pd f_{1} \|_{H^{k}(\BB^{d}_{R})} + \| f_{2} \|_{H^{k}(\BB^{d}_{R})} \\&\indent+
\| \pd^{2} f_{1} \|_{H^{k-1}(\BB^{d}_{R})} +
\| f_{2} \|_{H^{k-1}(\BB^{d}_{R})} +
\| \pd f_{2} \|_{H^{k-1}(\BB^{d}_{R})} \\&\lesssim
\| (f_{1},f_{2}) \|_{H^{k+1}(\BB^{d}_{R})\times H^{k}(\BB^{d}_{R})}
\end{align*}
for all $\mathbf{f} \in C^{\infty} ( \overline{\BB^{d}_{R}} )^{2}$. With this, the first inequality follows.
\item[``$\gtrsim$'':] Again, the inequality holds if $k=1$ by \Cref{1Norm}. By induction, assume that for an arbitrary but fixed $k\in\NN$ the bound
\begin{equation*}
\| \mathbf{f} \|_{\mathfrak{E}^{k}(\BB^{d}_{R})} \gtrsim \| (f_{1},f_{2}) \|_{H^{k}(\BB^{d}_{R})\times H^{k-1}(\BB^{d}_{R})}
\end{equation*}
holds for all $\mathbf{f}\in C^{\infty} ( \overline{\BB^{d}_{R}} )^{2}$.
\begin{enumerate}[wide,itemsep=1em,topsep=1em]
\item[\textit{Case $1\leq k < \frac{d}{2}$.}] We conclude with
\begin{equation}
\label{WhyweneedD0}
[\mathbf{D}_{0} \mathbf{f}]_{1} = f_{2}
\qquad\text{and}\qquad
\pd_{i} \mathbf{f} = \mathbf{D}_{i} \mathbf{f} + (\pd_{i} h) \mathbf{D}_{0} \mathbf{f} \,,
\end{equation}
and the trace lemma \cite[Lemma 2.1]{MR4778061} for the Sobolev norm the estimate
\begin{align*}
\| (f_{1},f_{2}) \|_{H^{k+1}(\BB^{d}_{R})\times H^{k}(\BB^{d}_{R})} &\simeq \sum_{i=1}^{d} \| (\pd_{i} f_{1}, \pd_{i} f_{2}) \|_{H^{k}(\BB^{d}_{R})\times H^{k-1}(\BB^{d}_{R})} + \| f_{1} \|_{L^{2}(\pd\BB^{d}_{R})} + \| f_{2} \|_{L^{2}(\BB^{d}_{R})} \\&\lesssim
\sum_{\mu=0}^{d} \| \mathbf{D}_{\mu} \mathbf{f} \|_{H^{k}(\BB^{d}_{R})\times H^{k-1}(\BB^{d}_{R})} + \| f_{1} \|_{L^{2}(\pd\BB^{d}_{R})} \\&\lesssim
\sum_{\mu=0}^{d} \| \mathbf{D}_{\mu} \mathbf{f} \|_{\mathfrak{E}^{k}(\BB^{d}_{R})} + \| f_{1} \|_{L^{2}(\pd\BB^{d}_{R})} \\&\simeq
\| \mathbf{f} \|_{\mathfrak{E}^{k+1}(\BB^{d}_{R})}
\end{align*}
for all $\mathbf{f}\in C^{\infty} ( \overline{\BB^{d}_{R}} )^{2}$. \item[\textit{Case $k \geq \frac{d}{2}$.}] We conclude with \eqref{WhyweneedD0} as well
\begin{align*}
\| (f_{1},f_{2}) \|_{H^{k+1}(\BB^{d}_{R})\times H^{k}(\BB^{d}_{R})} &\simeq \sum_{i=1}^{d} \| (\pd_{i} f_{1}, \pd_{i} f_{2}) \|_{H^{k}(\BB^{d}_{R})\times H^{k-1}(\BB^{d}_{R})} + \| (f_{1},f_{2}) \|_{H^{k}(\BB^{d}_{R})\times H^{k-1}(\BB^{d}_{R})} \\&\lesssim
\sum_{\mu=0}^{d} \| \mathbf{D}_{\mu} \mathbf{f} \|_{H^{k}(\BB^{d}_{R})\times H^{k-1}(\BB^{d}_{R})} + \| (f_{1},f_{2}) \|_{H^{k}(\BB^{d}_{R})\times H^{k-1}(\BB^{d}_{R})} \\&\lesssim
\sum_{\mu=0}^{d} \| \mathbf{D}_{\mu} \mathbf{f} \|_{\mathfrak{E}^{k}(\BB^{d}_{R})} + \| \mathbf{f} \|_{\mathfrak{E}^{k}(\BB^{d}_{R})} \\&\simeq
\| \mathbf{f} \|_{\mathfrak{E}^{k+1}(\BB^{d}_{R})}
\end{align*}
for all $\mathbf{f}\in C^{\infty} ( \overline{\BB^{d}_{R}} )^{2}$.
\end{enumerate}
This proves the other direction.
\qedhere
\end{enumerate}
\end{proof}
\subsection{Dissipative estimates}
\label{SubSecDissEst}
The energy estimate in \Cref{SubSecEnTop} for the wave equation in similarity coordinates manifests itself as a dissipative estimate for the wave evolution operator from \Cref{WaveFlowOperation}.
\begin{lemma}
\label{1Dissipativity}
Let $d\in \NN$ and $R \geq R_{0}$ with $R_{0}>0$ from \Cref{GSCLightCone}. Then
\begin{equation*}
\Re \Big( \mathbf{L}_{\chi} \mathbf{f} \,\Big|\, \mathbf{f} \Big)_{\mathfrak{E}^{1}(\BB^{d}_{R})} \leq \Big( \frac{d}{2}-1\Big) \| \mathbf{f} \|_{\mathfrak{E}^{1}(\BB^{d}_{R})}^{2} +
\Big( -\frac{d}{2} + 1 + \varepsilon_{1} \Big) \frac{2\varepsilon_{1}}{R} \int_{\pd\BB^{d}_{R}} |f_{1}|^{2}
\end{equation*}
for all $\mathbf{f} \in C^{\infty} ( \overline{\BB^{d}_{R}} )^{2}$.
\end{lemma}
\begin{proof}
We have
\begin{equation*}
\Big( \mathbf{L}_{\chi} \mathbf{f} \,\Big|\, \mathbf{f} \Big)_{\mathfrak{E}^{1}(\BB^{d}_{R})} =
\int_{\BB^{d}_{R}} \overline{(\pd^{i}[\mathbf{L}_{\chi}\mathbf{f}]_{1})} (\pd_{i} f_{1}) +
\int_{\BB^{d}_{R}} \overline{[\mathbf{L}_{\chi}\mathbf{f}]_{2}} f_{2} w +
\frac{2\varepsilon_{1}}{R} \int_{\pd\BB^{d}_{R}} \overline{[\mathbf{L}_{\chi}\mathbf{f}]_{1}} f_{1}
\end{equation*}
with integrands
\begin{align*}
\overline{(\pd^{i}[\mathbf{L}_{\chi}\mathbf{f}]_{1})} (\pd_{i} f_{1}) &=
- | \pd f_{1} |^{2}
- \overline{ \xi^{j} (\pd_{i} \pd_{j} f_{1})} (\pd^{i} f_{1})
- c (\pd^{i} \pd_{i} f_{1}) \overline{f_{2}} + \pd_{i} \big( c (\pd^{i} f_{1}) \overline{f_{2}} \big) \,, \\
\overline{[\mathbf{L}_{\chi}\mathbf{f}]_{2}} f_{2} w &=
c \overline{(\pd^{i} \pd_{i} f_{1})} f_{2} -
\big( w + c (\pd^{i} \pd_{i} h) \big) |f_{2}|^{2} -
\big( \xi^{j} w + 2c (\pd^{j} h) \big) \overline{(\pd_{j} f_{2})} f_{2} \,, \\
\overline{[\mathbf{L}_{\chi}\mathbf{f}]_{1}} f_{1} &= - \overline{\xi^{i}(\pd_{i} f_{1})} f_{1} + \overline{c f_{2}} f_{1} \,.
\end{align*}
Integration by parts yields the equation
\begin{align*}
\Re\Big( \mathbf{L}_{\chi} \mathbf{f} \,\Big|\, \mathbf{f} \Big)_{\mathfrak{E}^{1}(\BB^{d}_{R})} &=
\Big( \frac{d}{2} - 1 \Big) \int_{\BB^{d}_{R}} \Big( | \pd f_{1} |^{2} + |f_{2}|^{2} w \Big)
\\&\indent+
\frac{1}{R} \int_{\pd\BB^{d}_{R}} \Re \Big(
\xi^{i} (\pd_{i} f_{1}) \overline{c f_{2}} -
\overline{\xi^{i} (\pd_{i} f_{1})} 2\varepsilon_{1} f_{1} +
\overline{c f_{2}} 2\varepsilon_{1} f_{1}
\Big)
\\&\indent-
\frac{1}{R} \int_{\pd\BB^{d}_{R}} \frac{1}{2} \Big(
|\xi|^{2} | \pd f_{1} |^{2} + \big( w |\xi|^{2} + 2c \xi^{i} (\pd_{i} h) \big) |f_{2}|^{2} \Big) \,.
\end{align*}
The first boundary term is estimated with the elementary inequality
\begin{equation*}
\Re \Big(
\xi^{i} (\pd_{i} f_{1}) \overline{c f_{2}} -
\overline{\xi^{i} (\pd_{i} f_{1})} 2\varepsilon_{1} f_{1} +
\overline{c f_{2}} 2\varepsilon_{1} f_{1}
\Big) \leq \frac{1}{2} \Big( |\xi|^{2} | \pd f_{1} |^{2} + c^{2} |f_{2}|^{2} + 4\varepsilon_{1}^{2} |f_{1}|^{2} \Big) \,.
\end{equation*}
We then use
\begin{align*}
c(\xi)^{2} - w(\xi) |\xi|^{2} - 2 c(\xi) \xi^{i} (\pd_{i} h)(\xi) &= - (\xi^{i} (\pd_{i} h)(\xi))^{2} + |\xi|^{2} |(\pd h)(\xi)|^{2} + h(\xi)^{2} - |\xi|^{2} \\&=
h(\xi)^{2} - |\xi|^{2}
\end{align*}
with $h(\xi)^{2} - |\xi|^{2} \leq 0$ if and only if $|\xi| \geq R_{0}$ and conclude the estimate
\begin{align*}
\Re\Big( \mathbf{L}_{\chi} \mathbf{f} \,\Big|\, \mathbf{f} \Big)_{\mathfrak{E}^{1}(\BB^{d}_{R})} &\leq
\Big( \frac{d}{2} - 1 \Big) \int_{\BB^{d}_{R}} \Big( | \pd f_{1} |^{2} + |f_{2}|^{2} w \Big) +
\frac{2\varepsilon_{1}^{2}}{R} \int_{\pd\BB^{d}_{R}} |f_{1}|^{2} \\&=
\Big( \frac{d}{2} - 1 \Big) \| \mathbf{f} \|_{\mathfrak{E}^{1}(\BB^{d}_{R})}^{2} + \Big( - \frac{d}{2} + 1 + \varepsilon_{1} \Big) \frac{2\varepsilon_{1}}{R} \int_{\pd\BB^{d}_{R}} |f_{1}|^{2} \,.
\qedhere
\end{align*}
\end{proof}
The recursive construction of the inner products allows to obtain inductively the desired dissipative estimates.
\begin{proposition}
\label{EnergyDissipativity}
Let $d,k\in\NN$ and $R \geq R_{0}$ with $R_{0}>0$ from \Cref{GSCLightCone}. Fix $0<\varepsilon_{1}\leq\frac{1}{2}$ in \Cref{InnerProducts} and set
\begin{equation*}
\varepsilon_{k} = \varepsilon_{1} \prod_{i=1}^{k-1}
\frac{ \big( \frac{d}{2} - i - \varepsilon_{1} \big) \varepsilon_{1}}{R^{2}}
> 0 \qquad \text{for } 2 \leq k < \frac{d}{2} + 1 \,.
\end{equation*}
Then
\begin{align*}
&\Re \Big( \mathbf{L}_{\chi} \mathbf{f} \,\Big|\, \mathbf{f} \Big)_{\mathfrak{E}^{k}(\BB^{d}_{R})} \\&\indent\leq
\renewcommand{\arraystretch}{1.5}
\left\{
\begin{array}{ll}
\displaystyle{
\Big( \frac{d}{2}-k\Big) \| \mathbf{f} \|_{\mathfrak{E}^{k}(\BB^{d}_{R})}^{2} +
\Big(-\frac{d}{2}+k+\varepsilon_{1}\Big) \frac{2 \varepsilon_{k}}{R} \int_{\pd\BB^{d}_{R}} |f_{1}|^{2}
} & \text{ if } 1\leq k < \frac{d}{2} + 1 \,, \\
\displaystyle{
\varepsilon_{1} \| \mathbf{f} \|_{\mathfrak{E}^{k}(\BB^{d}_{R})}^{2}
} & \text{ if } \frac{d}{2} + 1 \leq k \,,
\end{array}
\right.
\end{align*}
for all $\mathbf{f}\in C^{\infty} ( \overline{\BB^{d}_{R}} )^{2}$.
\end{proposition}
\begin{proof}
The estimate for $k=1$ is proved in \Cref{1Dissipativity}. Now, let us assume by induction that the lemma holds for an arbitrary but fixed integer $k \in \NN$.
\begin{enumerate}[wide,itemsep=1em,topsep=1em]
\item[\textit{Case $1\leq k < \frac{d}{2}$.}] We find with the commutation relation from \Cref{DmuLdCommutationRelation}
\begin{align*}
\Re \Big( \mathbf{L}_{\chi} \mathbf{f} \,\Big|\, \mathbf{f} \Big)_{\mathfrak{E}^{k+1}(\BB^{d}_{R})} = \sum_{\mu=0}^{d} \Re \Big( \mathbf{D}_{\mu} \mathbf{L}_{\chi} \mathbf{f} \,\Big|\, \mathbf{D}_{\mu} \mathbf{f} \Big)_{\mathfrak{E}^{k}(\BB^{d}_{R})} &+ \frac{2\varepsilon_{k+1}}{R} \Re \int_{\pd\BB^{d}_{R}} \overline{[\mathbf{L}_{\chi}\mathbf{f}]_{1}} f_{1} \\=
\sum_{\mu=0}^{d} \Re \Big( \mathbf{L}_{\chi} \mathbf{D}_{\mu} \mathbf{f} \,\Big|\, \mathbf{D}_{\mu} \mathbf{f} \Big)_{\mathfrak{E}^{k}(\BB^{d}_{R})} - \sum_{\mu=0}^{d} \| \mathbf{D}_{\mu}\mathbf{f} \|_{\mathfrak{E}^{k}(\BB^{d}_{R})}^{2} &+ \frac{2\varepsilon_{k+1}}{R} \Re \int_{\pd\BB^{d}_{R}} \overline{[\mathbf{L}_{\chi}\mathbf{f}]_{1}} f_{1} \,.
\end{align*}
From the assumptions in \Cref{SimilarityCoordinatesFlat} it follows that $h(\xi) < |\xi|$ and so the boundary term is estimated by
\begin{align*}
\Re\int_{\pd\BB^{d}_{R}}\overline{[\mathbf{L}_{\chi}\mathbf{f}]_{1}} f_{1} &= \Re \int_{\pd\BB^{d}_{R}} \Big( - \overline{h[\mathbf{D}_{0}\mathbf{f}]_{1}} f_{1} - \overline{\xi^{i} [\mathbf{D}_{i}\mathbf{f}]_{1}} f_{1} \Big) \\&\leq
\varepsilon_{1} \int_{\pd\BB^{d}_{R}} |f_{1}|^{2} +
\frac{R^{2}}{2\varepsilon_{1}} \int_{\pd\BB^{d}_{R}} \sum_{\mu=0}^{d} \Big| [\mathbf{D}_{\mu}\mathbf{f}]_{1} \Big|^{2} \,.
\end{align*}
This implies with the induction hypothesis and the choice of $\varepsilon_{k+1} > 0$ from above
\begin{align*}
\Re \Big( \mathbf{L}_{\chi} \mathbf{f} \,\Big|\, \mathbf{f} \Big)_{\mathfrak{E}^{k+1}(\BB^{d}_{R})} &\leq
\Big( \frac{d}{2} - k - 1 \Big) \sum_{\mu=0}^{d} \| \mathbf{D}_{\mu}\mathbf{f} \|_{\mathfrak{E}^{k}(\BB^{d}_{R})}^{2} + \varepsilon_{1}\frac{2\varepsilon_{k+1}}{R} \int_{\pd\BB^{d}_{R}} |f_{1}|^{2} \\&\indent+
\Big( \Big( - \frac{d}{2} + k + \varepsilon_{1} \Big) \frac{2\varepsilon_{k}}{R} + \frac{R}{\varepsilon_{1}} \varepsilon_{k+1} \Big) \int_{\pd\BB^{d}_{R}} \sum_{\mu=0}^{d} \Big|[\mathbf{D}_{\mu}\mathbf{f}]_{1}\Big|^{2} \\&\leq
\Big( \frac{d}{2} - k - 1 \Big) \| \mathbf{f} \|_{\mathfrak{E}^{k+1}(\BB^{d}_{R})}^{2} + \Big(-\frac{d}{2} + k + 1 + \varepsilon_{1} \Big) \frac{2 \varepsilon_{k+1}}{R} \int_{\pd\BB^{d}_{R}} |f_{1}|^{2} \,.
\end{align*}
So the lemma is proved for all $1\leq k < \frac{d}{2} + 1$.
\item[\textit{Case $k\geq \frac{d}{2}$.}] Note that when $k = \Big\lceil\frac{d}{2}\Big\rceil \geq \frac{d}{2}$ the induction hypothesis gives
\begin{equation*}
\Re \Big( \mathbf{L}_{\chi} \mathbf{f} \,\Big|\, \mathbf{f} \Big)_{\mathfrak{E}^{k}(\BB^{d}_{R})} \leq \varepsilon_{1} \| \mathbf{f} \|_{\mathfrak{E}^{k}(\BB^{d}_{R})}^{2} \,.
\end{equation*}
So, by induction and the commutation relation from \Cref{DmuLdCommutationRelation}
\begin{align*}
\Re \Big( \mathbf{L}_{\chi} \mathbf{f} \,\Big|\, \mathbf{f} \Big)_{\mathfrak{E}^{k+1}(\BB^{d}_{R})} &= \sum_{\mu=0}^{d} \Re\Big( \mathbf{D}_{\mu}\mathbf{L}_{\chi} \mathbf{f} \,\Big|\, \mathbf{D}_{\mu}\mathbf{f} \Big)_{\mathfrak{E}^{k}(\BB^{d}_{R})} + \Re \Big( \mathbf{L}_{\chi} \mathbf{f} \,\Big|\, \mathbf{f} \Big)_{\mathfrak{E}^{k}(\BB^{d}_{R})} \\&\leq
\varepsilon_{1} \Big( \sum_{\mu=0}^{d} \| \mathbf{D}_{\mu}\mathbf{f} \|_{\mathfrak{E}^{k}(\BB^{d}_{R})}^{2} + \| \mathbf{f} \|_{\mathfrak{E}^{k}(\BB^{d}_{R})}^{2} \Big) - \sum_{\mu=0}^{d} \| \mathbf{D}_{\mu}\mathbf{f} \|_{\mathfrak{E}^{k}(\BB^{d}_{R})}^{2} \\&\leq
\varepsilon_{1} \| \mathbf{f} \|_{\mathfrak{E}^{k}(\BB^{d}_{R})}^{2}
\end{align*}
and the estimates are proved.
\qedhere
\end{enumerate}
\end{proof}
\subsection{Density of the range}
\label{SubSecDenseRange}
\Cref{TransitionRelation} shows that a well-posedness theory of the Cauchy problem for the wave equation in similarity coordinates can be established if the differential operator from \Cref{WaveFlowOperation} can be realized as the generator of a strongly continuous semigroup. So far, we have shown in \Cref{kNorm,EnergyDissipativity} that this operator is densely defined and dissipative on a Hilbert space. Thus, the Lumer-Phillips theorem provides a necessary and sufficient condition for a generation theorem. In order to apply it, we further need to prove a density result for the range of this operator. For this, we introduce the Laplace-Beltrami operator in similarity coordinates.
\begin{definition}
\label{GSCLaplaceBeltrami}
For a map $v \in C^{\infty} \big( \overline{ (0,\infty) \times \BB^{d}_{R} } \big)$, we define the Laplace-Beltrami operator with respect to the Minkowski metric in graphical similarity coordinates from \Cref{AssumptionsAppendix} by
\begin{equation*}
(\Box_{\chi} v) (\tau,\xi) =
\ee^{2\tau} \Big( c^{00}(\xi) \pd_{\tau}^{2} + c^{i0}(\xi)\pd_{\xi^{i}}\pd_{\tau} + c^{ij}(\xi)\pd_{\xi^{i}}\pd_{\xi^{j}} + c^0(\xi)\pd_{\tau} + c^{i}(\xi)\pd_{\xi^{i}} \Big) v(\tau,\xi) \,,
\end{equation*}
where the smooth coefficients are given by
\begin{align*}
c^{00}(\xi) &= -\frac{w(\xi)}{c(\xi)^{2}} \,, \\
c^{i0}(\xi) &= -2\Big( \frac{w(\xi)}{c(\xi)^{2}}\xi^{i} + \frac{(\pd^{i} h)(\xi)}{c(\xi)} \Big) \,, \\
c^{ij}(\xi) &= \delta^{ij} - \frac{w(\xi)}{c(\xi)^{2}} \xi^{i}\xi^{j} - \frac{\xi^{i}(\pd^{j} h)(\xi) + \xi^{j}(\pd^{i} h)(\xi)}{c(\xi)} \,, \\
c^{0}(\xi) &= -\frac{w(\xi)}{c(\xi)^{2}} - \Big( \delta^{ij} - \frac{w(\xi)}{c(\xi)^{2}} \xi^{i}\xi^{j} - \frac{\xi^{i}(\pd^{j} h)(\xi) + \xi^{j}(\pd^{i} h)(\xi)}{c(\xi)} \Big) \frac{(\pd_{i}\pd_{j} h)(\xi)}{c(\xi)} \,, \\
c^{i}(\xi) &= -2\Big( \frac{w(\xi)}{c(\xi)^{2}}\xi^{i} + \frac{(\pd^{i} h)(\xi)}{c(\xi)} \Big) - \Big( \delta^{jk} - \frac{w(\xi)}{c(\xi)^{2}} \xi^{j}\xi^{k} - \frac{\xi^{j} (\pd^{k} h)(\xi) + \xi^{k} (\pd^{j} h)(\xi)}{c(\xi)} \Big) \frac{(\pd_{j}\pd_{k} h)(\xi)}{c(\xi)} \xi^{i} \,.
\end{align*}
\end{definition}
Recall that if $u\in C^{\infty} \big( \overline{\mathrm{X}}^{1,d}_{R} \big)$ and $v \in C^{\infty} \big( \overline{ (0,\infty) \times \BB^{d}_{R} } \big)$ are related via $v = u \circ \chi$, then the classical wave operator is related to the Laplace-Beltrami operator via
\begin{equation}
\label{LaplaceBeltramiSimilarityCoordinates}
(\Box u ) \circ \chi = \Box_{\chi} v \,.
\end{equation}
The following lemma links the respective condition from the Lumer-Phillips theorem to the existence of \emph{mode solutions} of an inhomogeneous wave equation in similarity coordinates.
\begin{lemma}
\label{Range}
Let $\lambda\in\CC$. Let $\mathbf{F} \in C^{\infty} ( \overline{\BB^{d}_{R}} )^{2}$ and define $F \in C^{\infty} ( \overline{\BB^{d}_{R}} )$ by
\begin{align}
\begin{split}
\label{F}
F(\xi) &= \Big( - \Big( (\lambda+1) \frac{w(\xi)}{c(\xi)^{2}} + \Big( \delta^{ij} - \frac{w(\xi)}{c(\xi)^{2}} \xi^{i}\xi^{j} - 2\frac{(\pd^{i} h)(\xi)\xi^{j}}{c(\xi)} \Big) \frac{(\pd_{i}\pd_{j} h)(\xi)}{c(\xi)} \Big) F_{1}(\xi) \\&\indent-
\Big( \frac{w(\xi)}{c(\xi)^{2}} \xi^{i} + \frac{(\pd^{i} h)(\xi)}{c(\xi)} \Big) (\pd_{i} F_{1})(\xi) - \frac{w(\xi)}{c(\xi)} F_{2}(\xi) \Big) \,.
\end{split}
\end{align}
Let $f\in C^{\infty} ( \overline{\BB^{d}_{R}} )$ and define $\mathbf{f} \in C^{\infty} ( \overline{\BB^{d}_{R}} )^{2}$ by
\begin{equation*}
\mathbf{f} =
\begin{bmatrix}
f \\
\displaystyle{
\frac{\lambda}{c} f + \frac{1}{c} \xi^{i} \pd_{i} f - \frac{1}{c} F_{1}
}
\end{bmatrix}
\,.
\end{equation*}
Then
\begin{equation*}
\lambda\mathbf{f} - \mathbf{L}_{\chi}\mathbf{f} - \mathbf{F} =-
\left.
\begin{bmatrix}
0 \\
\displaystyle{
\frac{c}{w} \Big( \Box_{\chi} v(\tau,\,.\,) - G(\tau,\,.\,) \Big)
}
\end{bmatrix}
\right|_{\tau = 0}
\,,
\end{equation*}
where $v, G \in C^{\infty}(\RR \times \overline{\BB^{d}_{R}})$ are given by
\begin{align}
\label{vlambda}
v_{\lambda}(\tau,\xi) &= \ee^{\lambda\tau} f(\xi) \,, \\
\label{Glambda}
G_{\lambda}(\tau,\xi) &= \ee^{(\lambda+2)\tau} F(\xi) \,.
\end{align}
\end{lemma}
\begin{proof}
Let $\mathbf{v} \in C^{\infty}(\RR\times\overline{\BB^{d}_{R}})^{2}$ be given by
\begin{equation*}
\mathbf{v}(\tau,\xi) \coloneqq \ee^{\lambda\tau} \mathbf{f}(\xi) =
\begin{bmatrix}
v_{\lambda}(\tau,\xi) \\
\displaystyle{
\frac{1}{c(\xi)} \big( \pd_{\tau} + \xi^{i} \pd_{\xi^{i}} \big) v_{\lambda}(\tau,\xi) - \frac{1}{c(\xi)} \ee^{\lambda\tau} F_{1}(\xi)
}
\end{bmatrix}
\,.
\end{equation*}
Using \Cref{TransitionRelation} and relation \eqref{LaplaceBeltramiSimilarityCoordinates}, we find
\begin{align*}
\lambda\mathbf{f} - \mathbf{L}_{\chi}\mathbf{f} - \mathbf{F} &=
\left.\Big( \pd_{\tau} \mathbf{v}(\tau,\,.\,) - \mathbf{L}_{\chi} \mathbf{v}(\tau,\,.\,) \Big) \right|_{\tau = 0} - \mathbf{F} \\&=-
\left.
\begin{bmatrix}
0 \\
\displaystyle{
\frac{c}{w} \Big( \Box_{\chi} v_{\lambda}(\tau,\,.\,) - G_{\lambda}(\tau,\,.\,) \Big)
}
\end{bmatrix}
\right|_{\tau = 0}
\,.
\qedhere
\end{align*}
\end{proof}
Solutions to the wave equation of the form \eqref{vlambda} are called \emph{mode solutions} and they are invariant under transition diffeomorphisms between graphical similarity coordinates.
\begin{lemma}
\label{ModeInvariance}
Let $\chi,\overline{\chi}:\RR \times \RR^{d} \rightarrow \RR^{1,d}$ be graphical similarity coordinates and $R,\overline{R} > 0$ such that $\chi(\RR\times\BB^{d}_{R}) = \overline{\chi}(\RR\times\BB^{d}_{\overline{R}})$. Let $\lambda\in\CC$ and let $f,F \in C^{\infty} ( \overline{\BB^{d}_{R}} )$ and $\overline{f},\overline{F} \in C^{\infty} ( \overline{\BB^{d}_{R}} )$ be related via the transition diffeomorphism $\chi^{-1} \circ \overline{\chi}$ from \Cref{TransitionDiffeo}, that is
\begin{equation}
\label{barfbarF}
\overline{f}(\overline{\xi}) = h_{+}(\overline{\xi})^{-\lambda} f \big( h_{+}(\overline{\xi})^{-1} \overline{\xi} \big) \,, \qquad \overline{F}(\overline{\xi}) = h_{+}(\overline{\xi})^{-\lambda+2} F \big( h_{+}(\overline{\xi})^{-1} \overline{\xi} \big) \,.
\end{equation}
Set
\begin{equation}
\label{GbarG}
G(\tau,\xi) = \ee^{(\lambda+2)\tau} F(\xi) \,, \qquad \overline{G}(\overline{\tau},\overline{\xi}) = \ee^{(\lambda+2)\overline{\tau}} \overline{F}(\overline{\xi}) \,.
\end{equation}
Then
\begin{equation*}
\renewcommand{\arraystretch}{1.5}
\left\{
\begin{array}{rcl}
\Box_{\chi} v &=& G \,, \\
v(\tau,\xi) &=& \displaystyle{
\ee^{\lambda\tau} f(\xi) \,,
}
\end{array}
\right.
\qquad\text{if and only if}\qquad
\renewcommand{\arraystretch}{1.5}
\left\{
\begin{array}{rcl}
\Box_{\overline{\chi}} \overline{v} &=& \overline{G} \,, \\
\overline{v}(\overline{\tau},\overline{\xi}) &=& \displaystyle{
\ee^{\lambda\overline{\tau}} \overline{f}(\overline{\xi}) \,.
}
\end{array}
\right.
\end{equation*}
\end{lemma}
\begin{proof}
Note that $v, G \in C^{\infty} \big( \overline{ (0,\infty) \times \BB^{d}_{R} } \big)$ and $\overline{v}, \overline{G} \in C^{\infty} \big( \overline{ (0,\infty) \times \BB^{d}_{\overline{R}} } \big)$ are related via
\begin{equation*}
v \circ \big( \chi^{-1} \circ \overline{\chi} \big) = \overline{v}
\qquad\text{and}\qquad
G \circ \chi^{-1} \circ \overline{\chi} = \overline{G}\,.
\end{equation*}
Hence
\begin{equation*}
\big( \Box_{\chi} v - G \big) \circ \big( \chi^{-1} \circ \overline{\chi} \big) = \Box_{\overline{\chi}} \overline{v} -\overline{G} \,.
\qedhere
\end{equation*}
\end{proof}
Our final result for the operator from \Cref{WaveFlowOperation} verifies the assumptions in the Lumer-Phillips theorem.
\begin{proposition}
\label{DenseRange}
Let $d\in\NN$ and $R \geq R_{0}$ with $R_{0}>0$ from \Cref{GSCLightCone}. For all $\mathbf{F} \in C^{\infty} ( \overline{\BB^{d}_{R}} )^{2}$ there is a unique $\mathbf{f} \in C^{\infty} ( \overline{\BB^{d}_{R}} )^{2}$ such that
\begin{equation*}
\big( \frac{d}{2} \mathbf{I} - \mathbf{L}_{\chi} \big) \mathbf{f} = \mathbf{F} \,.
\end{equation*}
If $\mathbf{F} \in C^{\infty}_{\mathrm{rad}}(\overline{\BB^{d}_{R}})^{2}$, then also $\mathbf{f} \in C^{\infty}_{\mathrm{rad}}(\overline{\BB^{d}_{R}})^{2}$.
\end{proposition}
\begin{proof}
Throughout this proof, set $\lambda = \frac{d}{2}$. Let $\mathbf{F} \in C^{\infty} ( \overline{\BB^{d}_{R}} )^{2}$ be arbitrary but fixed and define $F \in C^{\infty} ( \overline{\BB^{d}_{R}} )$ as in \Cref{F}. We define for $n\in\NN$ the spherical harmonics approximation
\begin{equation*}
F_{n}(\xi) = \sum_{\ell=0}^{n} \sum_{m\in\Omega_{d,\ell}} F_{\ell,m}(|\xi|) Y_{\ell,m}(\tfrac{\xi}{|\xi|}) \,, \qquad F_{\ell,m}(\rho) = \int_{\pd\BB^{d}} \overline{Y_{\ell,m}} F(\rho\,.\,) \,.
\end{equation*}
We recall from \cite[Lemma A.2]{MR3537340} that $F_{n} \in C^{\infty} ( \overline{\BB^{d}_{R}} )$ and
\begin{equation}
\label{ApproximationAllk}
\lim_{n\to\infty} \Big\| \frac{c}{w} ( F_{n} - F ) \Big\|_{H^{k-1}(\BB^{d}_{R})} = 0
\end{equation}
for any $k \in \NN$. Next, fix similarity coordinates $\overline{\chi}: \RR \times \BB^{d}_{\overline{R}} \rightarrow \RR^{1,d}$ with a height function $\overline{h}(|\,.\,|) \in C^{\infty}_{\mathrm{rad}}(\RR^{d})$ with $\overline{h}(\rho) = -1$ for $\rho \in [0,1]$ and with the same image as $\chi : \RR \times \BB^{d}_{R} \rightarrow \RR^{1,d}$, see \Cref{GSCImage,GSCExemplary}. As a matter of transparency, we use overbared variables for the analysis in this coordinate system. Let $\overline{F}_{n} \in C^{\infty}(\overline{\BB^{d}_{\overline{R}}})$ be related to $F_{n} \in C^{\infty} ( \overline{\BB^{d}_{R}} )$ via \Cref{barfbarF}, i.e.
\begin{align*}
\overline{F}_{n}(\overline{\xi}) = \sum_{\ell=0}^{n} \sum_{m\in\Omega_{d,\ell}} \overline{F}_{\ell,m}(|\overline{\xi}|) Y_{\ell,m}(\tfrac{\overline{\xi}}{|\overline{\xi}|}) \,, \qquad
\overline{F}_{\ell,m}(\overline{\rho}) = \widetilde{h}_{+}(\overline{\rho})^{-\lambda+2} F_{\ell,m} \big( \widetilde{h}_{+}(\overline{\rho})^{-1} \overline{\rho} \big) \,.
\end{align*}
Furthermore, define $G_{n} \in C^{\infty}(\RR \times \overline{\BB^{d}_{R}})$, $\overline{G}_{n} \in C^{\infty}(\RR \times \overline{\BB^{d}_{\overline{R}}})$ according to \Cref{GbarG} and consider the problem
\begin{equation}
\label{ModeProblem}
\renewcommand{\arraystretch}{1.5}
\left\{
\begin{array}{rcl}
\Box_{\overline{\chi}} \overline{v}_{n} &=& \overline{G}_{n} \,, \\
\overline{v}_{n}(\overline{\tau},\overline{\xi}) &=& \displaystyle{
\ee^{\lambda\overline{\tau}} \overline{f}_{n}(\overline{\xi}) \,,
}
\end{array}
\right.
\end{equation}
for solutions of the form
\begin{equation}
\label{SphericalHarmonicsAnsatz}
\overline{f}_{n}(\overline{\xi}) = \sum_{\ell=0}^{n} \sum_{m\in\Omega_{d,\ell}} f_{\ell,m}(|\overline{\xi}|) Y_{\ell,m}(\tfrac{\overline{\xi}}{|\overline{\xi}|}) \,.
\end{equation}
Inserting the expansion \eqref{SphericalHarmonicsAnsatz} in problem \eqref{ModeProblem} yields for $(\overline{\tau},\overline{\rho}) \in \RR \times (0,\overline{R})$ and $\omega \in \pd \BB^{d}$ the identity
\begin{align*}
& \Big( \Box_{\overline{\chi}} \overline{v}_{n} -\overline{G}_{n} \Big) (\overline{\tau},\overline{\rho}\omega) = \ee^{(\lambda+2)\overline{\tau}} \Big( \Big( \overline{c}^{ij} \pd_{i}\pd_{j} + \big( \lambda \overline{c}^{i0} + \overline{c}^{i} \big) \pd_{i} + \lambda \big( \lambda \overline{c}^{00} + \overline{c}^{0} \big) \Big) \overline{f}_{n} - \overline{F}_{n} \Big)(\overline{\rho}\omega) \\&=
\sum_{\ell=0}^{n} \sum_{m\in\Omega_{d,\ell}} \ee^{(\lambda+2)\overline{\tau}} \Big( \overline{a}_{2}(\overline{\rho}) \overline{f}_{\ell,m}''(\overline{\rho}) + \overline{a}_{1}(\overline{\rho};d,\lambda) \overline{f}_{\ell,m}'(\overline{\rho}) + \overline{a}_{0}(\overline{\rho};d,\lambda,\ell) \overline{f}_{\ell,m}(\overline{\rho}) - \overline{F}_{\ell,m}(\overline{\rho}) \Big) Y_{\ell,m}(\omega)
\end{align*}
with coefficients given by
\begin{align*}
\overline{a}_{2}(\overline{\rho}) &= - \frac{\overline{\rho}^{2} - \overline{h}(\overline{\rho})^{2}}{\big( \overline{\rho} \overline{h}'(\overline{\rho}) - \overline{h}(\overline{\rho}) \big)^{2}} \,, \\
\overline{a}_{1}(\overline{\rho};d,\lambda) &= 2\Big( \frac{d-3}{2} - \lambda \Big) \frac{\overline{\rho} - \overline{h}(\overline{\rho})\overline{h}'(\overline{\rho})}{\big( \overline{\rho} \overline{h}'(\overline{\rho}) - \overline{h}(\overline{\rho}) \big)^{2}} - \Big( \frac{d-1}{\overline{\rho}} - \frac{\overline{\rho} \overline{h}''(\overline{\rho})}{\overline{\rho} \overline{h}'(\overline{\rho}) - \overline{h}(\overline{\rho})} \Big) \frac{\overline{\rho}^{2} - \overline{h}(\overline{\rho})^{2}}{\big( \overline{\rho} \overline{h}'(\overline{\rho}) - \overline{h}(\overline{\rho}) \big)^{2}} \,, \\
\overline{a}_{0}(\overline{\rho};d,\lambda,\ell) &= - \lambda (\lambda+1) \frac{1-\overline{h}'(\overline{\rho})^{2}}{\big( \overline{\rho} \overline{h}'(\overline{\rho}) - \overline{h}(\overline{\rho}) \big)^{2}} - \frac{\ell(\ell+d-2)}{\overline{\rho}^{2}} \\&\indent - \lambda \Big( \frac{d-1}{\overline{\rho} \overline{h}'(\overline{\rho}) - \overline{h}(\overline{\rho})}\frac{\overline{h}'(\overline{\rho})}{\overline{\rho}} - \frac{\overline{h}''(\overline{\rho})}{\overline{\rho} \overline{h}'(\overline{\rho})-\overline{h}(\overline{\rho})} \frac{\overline{\rho}^{2} - \overline{h}(\overline{\rho})^{2}}{\big( \overline{\rho} \overline{h}'(\overline{\rho}) - \overline{h}(\overline{\rho}) \big)^{2}} \Big) \,.
\end{align*}
It follows that \eqref{ModeProblem} holds if and only if for $\ell\in\{0,\ldots,n\}$, $m\in\Omega_{d,\ell}$ there are functions $\overline{f}_{\ell,m} \in C^{\infty}([0,\overline{R}])$ such that
\begin{equation}
\label{DecoupledODE}
\overline{a}_{2}(\overline{\rho}) \overline{f}_{\ell,m}''(\overline{\rho}) + \overline{a}_{1}(\overline{\rho};d,\lambda) \overline{f}_{\ell,m}'(\overline{\rho}) + \overline{a}_{0}(\overline{\rho};d,\lambda,\ell) \overline{f}_{\ell,m}(\overline{\rho}) = \overline{F}_{\ell,m}(\overline{\rho}) \quad \text{in } (0,\overline{R}) \,.
\end{equation}
Now, the analysis of the differential equation is reduced to the one carried out in \cite[Proof of Lemma 2.6]{MR4778061}. Indeed, since $\overline{h} \equiv -1$ in $[0,1]$, \Cref{DecoupledODE} reduces in $(0,1)$ precisely to \cite[Eq. (2.8)]{MR4778061}, which has for $\lambda=\frac{d}{2}$ a smooth solution given in \cite[Eq. (2.13)]{MR4778061}. Since \Cref{DecoupledODE} has no real-valued singular points when $\overline{\rho}>1$, each solution in $[0,1]$ can be extended to a smooth solution $\overline{f}_{\ell,m} \in C^{\infty}([0,\overline{R}])$. Solutions in case $\ell = 0$ are smooth and even, so if $\mathbf{F}$ is radially symmetric, also the mode solution to problem \eqref{ModeProblem} is radially symmetric. By \Cref{ModeInvariance}, we have thus obtained for each $n\in\NN$ a solution $f_{n} \in C^{\infty} ( \overline{\BB^{d}_{R}} )$ to
\begin{equation*}
\renewcommand{\arraystretch}{1.5}
\left\{
\begin{array}{rcl}
\Box_{\chi} v_{n} &=& G_{n} \,, \\
v_{n}(\tau,\xi) &=& \displaystyle{
\ee^{\lambda\tau} f_{n}(\xi) \,.
}
\end{array}
\right.
\end{equation*}
With this, put
\begin{equation*}
\mathbf{f}_{n} =
\begin{bmatrix}
f_{n} \\
\displaystyle{
\frac{\lambda}{c} f_{n} + \frac{1}{c} \xi^{i} \pd_{i} f_{n} - \frac{1}{c} [\mathbf{F}]_{1}
}
\end{bmatrix}
\in C^{\infty} ( \overline{\BB^{d}_{R}} )^{2} \,,
\qquad
\mathbf{F}_{n} \coloneqq \big( \frac{d}{2} \mathbf{I} - \mathbf{L}_{\chi} \big) \mathbf{f}_{n} \in C^{\infty} ( \overline{\BB^{d}_{R}} )^{2} \,.
\end{equation*}
By construction and \Cref{Range},
\begin{equation*}
\mathbf{F}_{n} - \mathbf{F} =
-
\left.
\begin{bmatrix}
0 \\
\displaystyle{
\frac{c}{w} \Big( \Box_{\chi} v_{n}(\tau,\,.\,) - G(\tau,\,.\,) \Big)
}
\end{bmatrix}
\right|_{\tau = 0}
=
\begin{bmatrix}
0 \\
\displaystyle{
\frac{c}{w} ( F_{n} - F ) 
}
\end{bmatrix}
\,,
\end{equation*}
so \Cref{ApproximationAllk} yields
\begin{equation}
\label{SmoothConvergenceInhomogeneity}
\lim_{n\to\infty} \big\| \mathbf{F}_{n} - \mathbf{F} \|_{H^{k}(\BB^{d}_{R}) \times H^{k-1}(\BB^{d}_{R})} = 0 \quad \text{for any } k \in \NN \,.
\end{equation}
Now, consider for any $k > \frac{d}{2}$ and for fixed $\varepsilon_{1} = \varepsilon > 0$ the inner product $(\,.\,| \,..\,)_{\mathfrak{E}^{k}(\BB^{d}_{R})}$ from \Cref{InnerProducts}. \Cref{EnergyDissipativity} yields the dissipativity estimate
\begin{equation}
\label{NormDissipativity}
\big\| \big( \alpha \mathbf{I} - \mathbf{L}_{\chi} \big) \mathbf{f} \big\|_{\mathfrak{E}^{k}(\BB^{d}_{R})} \geq ( \alpha - \varepsilon ) \| \mathbf{f} \|_{\mathfrak{E}^{k}(\BB^{d}_{R})}
\end{equation}
for all $\alpha > \varepsilon$ and all $\mathbf{f} \in C^{\infty} ( \overline{\BB^{d}_{R}} )^{2}$. This implies for $\alpha = \lambda = \frac{d}{2}$ the inequality
\begin{equation*}
\| \mathbf{f}_{n} - \mathbf{f}_{m} \|_{\mathfrak{E}^{k}(\BB^{d}_{R})} \leq \frac{1}{\lambda - \varepsilon} \| \mathbf{F}_{n} - \mathbf{F}_{m} \|_{\mathfrak{E}^{k}(\BB^{d}_{R})}
\end{equation*}
for any $m,n \in \NN$. Since $\| \,.\, \|_{\mathfrak{E}^{k}(\BB^{d}_{R})} \simeq \| \,.\, \|_{H^{k}(\BB^{d}_{R})\times H^{k-1}(\BB^{d}_{R})}$ by \Cref{kNorm}, we get from this, \Cref{SmoothConvergenceInhomogeneity} and Sobolev embedding that there exists an $\mathbf{f} \in C^{\infty} ( \overline{\BB^{d}_{R}} )^{2}$ with
\begin{equation}
\label{SmoothConvergenceSolution}
\lim_{n\to\infty} \big\| \mathbf{f}_{n} - \mathbf{f} \|_{H^{k}(\BB^{d}_{R}) \times H^{k-1}(\BB^{d}_{R})} = 0 \quad \text{for any } k \in \NN \,.
\end{equation}
Finally, we follow the proof of \cite[p.~82 f., Proposition 3.14]{MR1721989} to show that $\mathbf{g} \coloneqq \lambda \mathbf{f} - \mathbf{L}_{\chi} \mathbf{f} - \mathbf{F}$ equals $\mathbf{0}$. \Cref{SmoothConvergenceInhomogeneity,NormDissipativity,SmoothConvergenceSolution} yield
\begin{align*}
\| \mathbf{g} \|_{\mathfrak{E}^{k}(\BB^{d}_{R})} &= \lim_{n\to\infty} \| (\alpha - \varepsilon) ( \mathbf{f}_{n} - \mathbf{f} ) + \mathbf{g} \|_{\mathfrak{E}^{k}(\BB^{d}_{R})} \\&\leq
\frac{1}{\alpha - \varepsilon} \lim_{n\to\infty} \big\| \alpha (\alpha - \varepsilon) ( \mathbf{f}_{n} - \mathbf{f} ) + \alpha \mathbf{g} - (\alpha - \varepsilon) \mathbf{L}_{\chi} ( \mathbf{f}_{n} - \mathbf{f} ) - \mathbf{L}_{\chi} \mathbf{g} \big\|_{\mathfrak{E}^{k}(\BB^{d}_{R})} \\&=
\frac{1}{\alpha - \varepsilon} \big\| \varepsilon \mathbf{g} - \mathbf{L}_{\chi} \mathbf{g} \big\|_{\mathfrak{E}^{k}(\BB^{d}_{R})} 
\end{align*}
for all $\alpha > \varepsilon$, so we conclude $\mathbf{g} = \mathbf{0}$ as $\alpha \to \infty$. Uniqueness of the smooth solution follows from \Cref{NormDissipativity}.
\end{proof}
\subsection{The generation theorem}
We collect the results from the previous subsections and prove well-posedness of the free wave flow in similarity coordinates for all space dimensions. We emphasize once more that no symmetry assumptions are imposed on the function spaces and data. Nevertheless, the flow in radial symmetry is readily concluded from restriction properties of the general case.
\begin{theorem}
\label{TheGenerationTHM}
Let $d,k\in\NN$. Let $h \in C^{\infty}_{\mathrm{rad}}(\RR^{d})$ such that
\begin{equation*}
h(0) < 0 \,, \qquad
|(\pd h)(\xi)| < 1 \,, \qquad
(\pd^{2} h)(\xi) \text{ is positive semidefinite} \,,
\end{equation*}
for all $\xi \in \RR^{d}$. Let $R_{0} > 0$ be determined as in \Cref{GSCLightCone} and fix $R \geq R_{0}$. We define on
\begin{equation*}
\mathfrak{H}^{k}(\BB^{d}_{R}) \coloneqq H^{k}(\BB^{d}_{R}) \times H^{k-1}(\BB^{d}_{R})
\end{equation*}
the operator $\mathbf{L}_{\chi} : \mathfrak{D}(\mathbf{L}_{\chi}) \subset \mathfrak{H}^{k}(\BB^{d}_{R}) \rightarrow \mathfrak{H}^{k}(\BB^{d}_{R})$ by
\begin{equation*}
\renewcommand{\arraystretch}{1.5}
\mathbf{L}_{\chi} \mathbf{f} =
\begin{bmatrix}
- \xi^{i} \pd_{i} f_{1}
+ c f_{2} \\
\displaystyle{
\frac{c}{w} \pd^{i} \pd_{i} f_{1}
- \Big( 1 + \frac{c}{w} \pd^{i} \pd_{i} h \Big) f_{2}
- \Big( \xi^{i} + 2 \frac{c}{w} \pd^{i} h \Big) \pd_{i} f_{2}
}
\end{bmatrix}
\,, \qquad
\mathfrak{D}(\mathbf{L}_{\chi}) = C^{\infty}(\overline{\BB^{d}_{R}})^{2} \,,
\end{equation*}
where the coefficients $c,w \in C^{\infty}_{\mathrm{rad}}(\RR^{d})$ are given by
\begin{equation*}
c(\xi) = \xi^{i} (\pd_{i} h)(\xi) - h(\xi) \,, \qquad
w(\xi) = 1-(\pd^{i} h)(\xi)(\pd_{i} h)(\xi) \,.
\end{equation*}
Then, this operator is closable and its closure $\overline{\mathbf{L}_{\chi}} : \mathfrak{D}(\overline{\mathbf{L}_{\chi}}) \subset \mathfrak{H}^{k}(\BB^{d}_{R}) \rightarrow \mathfrak{H}^{k}(\BB^{d}_{R})$ is the generator of a strongly continuous semigroup $\mathbf{S}_{\chi} : [0,\infty) \rightarrow \mathfrak{L}(\mathfrak{H}^{k}(\BB^{d}_{R}))$ with the following properties.
\begin{enumerate}[itemsep=1em,topsep=1em]
\item For any $0 < \varepsilon < \frac{1}{2}$ there is a constant $M_{d,k,R,h,\varepsilon} \geq 1$ such that the growth estimate
\begin{equation*}
\| \mathbf{S}_{\chi}(\tau) \mathbf{f} \|_{\mathfrak{H}^{k}(\BB^{d}_{R})} \leq M_{d,k,R,h,\varepsilon} \ee^{ \omega_{d,k,\varepsilon} \tau} \| \mathbf{f} \|_{\mathfrak{H}^{k}(\BB^{d}_{R})} \,, \qquad \omega_{d,k,\varepsilon} = \max \{ \tfrac{d}{2} - k, \varepsilon \} \,,
\end{equation*}
holds for all $\mathbf{f} \in \mathfrak{H}^{k}(\BB^{d}_{R})$ and all $\tau \geq 0$.
\item\label{itemrad} The radial subspace $\mathfrak{H}^{k}_{\mathrm{rad}}(\BB^{d}_{R}) \coloneqq H^{k}_{\mathrm{rad}}(\BB^{d}_{R}) \times H^{k-1}_{\mathrm{rad}}(\BB^{d}_{R})$ is an invariant closed subspace of $\mathfrak{H}^{k}(\BB^{d}_{R})$ for the semigroup $\mathbf{S}_{\chi}$, that is
\begin{equation*}
\mathbf{S}_{\chi}(\tau) \mathbf{f} \in \mathfrak{H}^{k}_{\mathrm{rad}}(\BB^{d}_{R})
\end{equation*}
for all $\mathbf{f} \in \mathfrak{H}^{k}_{\mathrm{rad}}(\BB^{d}_{R})$ and all $\tau \geq 0$. Furthermore, the operator $\mathbf{L}_{\chi} {\mathord{\restriction}} : \mathfrak{D}(\mathbf{L}_{\chi} {\mathord{\restriction}}) \subset \mathfrak{H}^{k}_{\mathrm{rad}}(\BB^{d}_{R}) \rightarrow \mathfrak{H}^{k}_{\mathrm{rad}}(\BB^{d}_{R})$ given by
\begin{equation*}
\mathbf{L}_{\chi} {\mathord{\restriction}} \mathbf{f} = \mathbf{L}_{\chi} \mathbf{f} \,, \qquad
\mathfrak{D}(\mathbf{L}_{\chi} {\mathord{\restriction}})
=
C^{\infty}_{\mathrm{rad}}(\overline{\BB^{d}_{R}})^{2} \,,
\end{equation*}
is well-defined and closable with closure $\overline{\mathbf{L}_{\chi} {\mathord{\restriction}}} : \mathfrak{D}(\overline{\mathbf{L}_{\chi} {\mathord{\restriction}}}) \subset \mathfrak{H}^{k}_{\mathrm{rad}}(\BB^{d}_{R}) \rightarrow \mathfrak{H}^{k}_{\mathrm{rad}}(\BB^{d}_{R})$ given by
\begin{equation*}
\overline{\mathbf{L}_{\chi} {\mathord{\restriction}}} \mathbf{f} = \overline{\mathbf{L}_{\chi}} \mathbf{f} \,, \qquad
\mathfrak{D}(\overline{\mathbf{L}_{\chi} {\mathord{\restriction}}}) = \mathfrak{D}(\overline{\mathbf{L}_{\chi}}) \cap 
\mathfrak{H}^{k}_{\mathrm{rad}}(\BB^{d}_{R}) \,,
\end{equation*}
which is the generator of the subspace semigroup $\mathbf{S}_{\chi}{\mathord{\restriction}} : [0,\infty) \rightarrow \mathfrak{L}(\mathfrak{H}^{k}_{\mathrm{rad}}(\BB^{d}_{R}))$ given by
\begin{equation*}
\mathbf{S}_{\chi}{\mathord{\restriction}}(\tau) \mathbf{f} = \mathbf{S}_{\chi}(\tau) \mathbf{f}
\end{equation*}
for all $\mathbf{f} \in \mathfrak{H}^{k}_{\mathrm{rad}}(\BB^{d}_{R})$ and all $\tau \geq 0$.
\item Consider the spacelike hypersurface $\Sigma^{1,d}_{R}(0) \coloneqq \chi\big( \{0\} \times \BB^{d}_{R} \big)$. Let $f,g \in C^{\infty}\big(\Sigma^{1,d}_{R}(0)\big)$ such that $\mathbf{f} \coloneqq (f, g) \circ \chi(0,\,.\,) \in C^{\infty}(\overline{\BB^{d}_{R}})^{2}$. Then $\mathbf{u} \coloneqq \mathbf{S}_{\chi}(\,.\,)\mathbf{f} \in C^{\infty} \big( \overline{(0,\infty) \times \BB^{d}_{R}} \big)^{2}$ and $u \coloneqq u_{1} \circ \chi^{-1} \in C^{\infty} \big( \overline{\mathrm{X}}^{1,d}_{R} \big)$ is the unique solution to the Cauchy problem
\begin{equation*}
\renewcommand{\arraystretch}{1.5}
\left\{
\begin{array}{rcll}
\Box u &=& 0 & \text{in } \mathrm{X}^{1,d}_{R} \,, \\
u|_{\Sigma^{1,d}_{R}(0)} &=& f \,, & \\
(\pd_{0} u)|_{\Sigma^{1,d}_{R}(0)} &=& g \,. &
\end{array}
\right.
\end{equation*}
\end{enumerate}
\end{theorem}
\begin{proof}
\begin{enumerate}[wide,itemsep=1em,topsep=1em]
\item By \Cref{kNorm}, the sesquilinear form $( \,.\, \,|\, \,..\, )_{\mathfrak{E}^{k}(\BB^{d}_{R})}$ from \Cref{InnerProducts} uniquely extends to an inner product on $\mathfrak{H}^{k}(\BB^{d}_{R})$ whose induced norm is equivalent to the classical Sobolev norm. \Cref{EnergyDissipativity} yields for $\varepsilon_{1} = \varepsilon > 0$ the dissipative estimate
\begin{equation*}
\Big( \mathbf{L}_{\chi} \mathbf{f} \,\Big|\, \mathbf{f} \Big)_{\mathfrak{E}^{k}(\BB^{d}_{R})} \leq \omega_{d,k,\varepsilon} \| \mathbf{f} \|_{\mathfrak{E}^{k}(\BB^{d}_{R})}^{2}
\end{equation*}
for all $\mathbf{f} \in \mathfrak{D}(\mathbf{L}_{\chi})$. Consequently, the operator $\mathbf{L}_{\chi}$ is linear, densely defined and dissipative, hence closable in $\mathfrak{H}^{k}(\BB^{d}_{R})$. \Cref{DenseRange} shows for $\lambda = \frac{d}{2}$ that $\lambda \mathbf{I} - \mathbf{L}_{\chi}$ has range $C^{\infty}(\overline{\BB^{d}_{R}})^{2}$, which is dense in $\mathfrak{H}^{k}(\BB^{d}_{R})$. Now, we conclude from the Lumer-Phillips theorem \cite[p.~83, Theorem 3.15]{MR1721989} that the closure $\overline{\mathbf{L}_{\chi}}$ generates a contraction semigroup with the asserted properties.
\item By \Cref{DenseRange}, we have for $\lambda = \frac{d}{2}$ that $\mathbf{R}_{\overline{\mathbf{L}_{\chi}}}(\lambda) \mathbf{f} \in C^{\infty}_{\mathrm{rad}}(\overline{\BB^{d}_{R}})^{2}$ for all $\mathbf{f} \in C^{\infty}_{\mathrm{rad}}(\overline{\BB^{d}_{R}})^{2}$. This statement extends to all $\lambda > \frac{d}{2}$ via a Neumann series argument, see \cite[p.~82, 3.14 Proposition]{MR1721989}. Now, The Post-Widder inversion formula \cite[p.~223, 5.5 Corollary]{MR1721989} gives
\begin{equation*}
\mathbf{S}(\tau) \mathbf{f} = \lim_{n\to\infty} \Big( \tfrac{n}{\tau} \mathbf{R}_{\overline{\mathbf{L}_{\chi}}}(\tfrac{n}{\tau}) \Big)^{n} \mathbf{f} \in \mathfrak{H}^{k}_{\mathrm{rad}}(\BB^{d}_{R})
\end{equation*}
for all $\mathbf{f} \in C^{\infty}_{\mathrm{rad}}(\overline{\BB^{d}_{R}})^{2}$. By density and boundedness this extends to $\mathfrak{H}^{k}_{\mathrm{rad}}(\BB^{d}_{R})$.
\item The semigroup is established in the spaces $\mathfrak{H}^{k}(\BB^{d}_{R})$ for all $k\in \NN$. So \cite[Theorem C.1]{2022arXiv220706952G} gives $\mathbf{u}(\tau) \in \mathfrak{H}^{k}(\BB^{d}_{R})$ for all $k\in \NN$. Hence $\mathbf{u}(\tau) \in C^{\infty} ( \overline{\BB^{d}_{R}} )^{2}$ for all $\tau \geq 0$ by Sobolev embedding. Using \cite[p.~50, 1.3 Lemma]{MR1721989}, we have
\begin{equation*}
\pd_{\tau} \mathbf{u}(\tau) = \overline{\mathbf{L}_{\chi}} \mathbf{u}(\tau) = \mathbf{L}_{\chi} \mathbf{u}(\tau) \in C^{\infty} ( \overline{\BB^{d}_{R}} )^{2}
\end{equation*}
for all $\tau \geq 0$. By induction, $\pd_{\tau}^{\ell} \mathbf{u}(\tau) = \mathbf{L}_{\chi}^{\ell} \mathbf{u}(\tau) \in C^{\infty} ( \overline{\BB^{d}_{R}} )^{2}$ for all $\ell \in \NN$. Schwarz's theorem \cite[p.~235, Theorem 9.41]{MR385023} implies $\mathbf{u} \in C^{\infty} \big( \overline{ (0,\infty) \times \BB^{d}_{R}} \big)^{2}$. Finally, the transition lemma \ref{TransitionRelation} yields that $u$ solves the Cauchy problem.
\qedhere
\end{enumerate}
\end{proof}


\begin{thebibliography}{100}

\bibitem{MR3636308}
Pawe{\l} Biernat, Piotr Bizo\'{n}, and Maciej Maliborski.
\newblock Threshold for blowup for equivariant wave maps in higher dimensions.
\newblock {\em Nonlinearity}, 30(4):1513--1522, 2017.

\bibitem{MR4338226}
Pawe{\l} Biernat, Roland Donninger, and Birgit Sch\"{o}rkhuber.
\newblock Hyperboloidal similarity coordinates and a globally stable blowup
  profile for supercritical wave maps.
\newblock {\em Int. Math. Res. Not. IMRN}, 2021(21):16530--16591, 2021.

\bibitem{MR2069700}
P.~Bizo\'{n}, Yu.~N. Ovchinnikov, and I.~M. Sigal.
\newblock Collapse of an instanton.
\newblock {\em Nonlinearity}, 17(4):1179--1191, 2004.

\bibitem{MR1877844}
P.~Bizo\'{n} and Z.~Tabor.
\newblock On blowup of {Y}ang-{M}ills fields.
\newblock {\em Phys. Rev. D (3)}, 64(12):121701, 4, 2001.

\bibitem{MR1799875}
Piotr Bizo\'{n}.
\newblock Equivariant self-similar wave maps from {M}inkowski spacetime into
  3-sphere.
\newblock {\em Comm. Math. Phys.}, 215(1):45--56, 2000.

\bibitem{MR1923684}
Piotr Bizo\'{n}.
\newblock Formation of singularities in {Y}ang-{M}ills equations.
\newblock {\em Acta Phys. Polon. B}, 33(7):1893--1922, 2002.

\bibitem{MR3355819}
Piotr Bizo\'{n} and Pawe{\l} Biernat.
\newblock Generic self-similar blowup for equivariant wave maps and
  {Y}ang-{M}ills fields in higher dimensions.
\newblock {\em Comm. Math. Phys.}, 338(3):1443--1450, 2015.

\bibitem{Bizo__2005}
Piotr Bizo\'{n} and Tadeusz Chmaj.
\newblock Convergence towards a self-similar solution for a nonlinear wave
  equation: A case study.
\newblock {\em Physical Review D}, 72(4), 8 2005.

\bibitem{MR1767966}
Piotr Bizo\'{n}, Tadeusz Chmaj, and Zbis{\l}aw Tabor.
\newblock Dispersion and collpse of wave maps.
\newblock {\em Nonlinearity}, 13(4):1411--1423, 2000.

\bibitem{MR1862811}
Piotr Bizo\'{n}, Tadeusz Chmaj, and Zbis{\l}aw Tabor.
\newblock Formation of singularities for equivariant {$(2+1)$}-dimensional wave
  maps into the 2-sphere.
\newblock {\em Nonlinearity}, 14(5):1041--1053, 2001.

\bibitem{MR3883339}
Timothy Candy and Sebastian Herr.
\newblock On the division problem for the wave maps equation.
\newblock {\em Ann. PDE}, 4(2):Paper No. 17, 61, 2018.

\bibitem{MR1622539}
Thierry Cazenave, Jalal Shatah, and A.~Shadi Tahvildar-Zadeh.
\newblock Harmonic maps of the hyperbolic space and development of
  singularities in wave maps and {Y}ang-{M}ills fields.
\newblock {\em Ann. Inst. H. Poincar\'{e} Phys. Th\'{e}or.}, 68(3):315--349,
  1998.

\bibitem{MR3680948}
Athanasios Chatzikaleas, Roland Donninger, and Irfan Glogi\'{c}.
\newblock On blowup of co-rotational wave maps in odd space dimensions.
\newblock {\em J. Differential Equations}, 263(8):5090--5119, 2017.

\bibitem{MR4700297}
Po-Ning Chen, Roland Donninger, Irfan Glogi\'{c}, Michael McNulty, and Birgit
  Sch\"{o}rkhuber.
\newblock Co-dimension one stable blowup for the quadratic wave equation beyond
  the light cone.
\newblock {\em Comm. Math. Phys.}, 405(2):Paper No. 34, 46, 2024.

\bibitem{2023arXiv231007042C}
Po-Ning {Chen}, Michael {McNulty}, and Birgit {Sch{\"o}rkhuber}.
\newblock {Singularity formation for the higher dimensional Skyrme model in the
  strong field limit}.
\newblock {\em arXiv e-prints}, page arXiv:2310.07042, October 2023.

\bibitem{MR3592527}
Elisabetta Chiodaroli and Joachim Krieger.
\newblock A class of large global solutions for the wave-map equation.
\newblock {\em Trans. Amer. Math. Soc.}, 369(4):2747--2773, 2017.

\bibitem{MR1230285}
Demetrios Christodoulou and A.~Shadi Tahvildar-Zadeh.
\newblock On the asymptotic behavior of spherically symmetric wave maps.
\newblock {\em Duke Math. J.}, 71(1):31--69, 1993.

\bibitem{MR1223662}
Demetrios Christodoulou and A.~Shadi Tahvildar-Zadeh.
\newblock On the regularity of spherically symmetric wave maps.
\newblock {\em Comm. Pure Appl. Math.}, 46(7):1041--1091, 1993.

\bibitem{MR4182428}
Piotr~T. Chru\'{s}ciel.
\newblock {\em Elements of general relativity}.
\newblock Compact Textbooks in Mathematics. Birkh\"{a}user/Springer, Cham,
  2019.

\bibitem{MR3538419}
O.~Costin, R.~Donninger, and X.~Xia.
\newblock A proof for the mode stability of a self-similar wave map.
\newblock {\em Nonlinearity}, 29(8):2451--2473, 2016.

\bibitem{MR3623242}
Ovidiu Costin, Roland Donninger, and Irfan Glogi\'{c}.
\newblock Mode stability of self-similar wave maps in higher dimensions.
\newblock {\em Comm. Math. Phys.}, 351(3):959--972, 2017.

\bibitem{MR3475668}
Ovidiu Costin, Roland Donninger, Irfan Glogi\'{c}, and Min Huang.
\newblock On the stability of self-similar solutions to nonlinear wave
  equations.
\newblock {\em Comm. Math. Phys.}, 343(1):299--310, 2016.

\bibitem{MR3403756}
R.~C\^{o}te.
\newblock On the soliton resolution for equivariant wave maps to the sphere.
\newblock {\em Comm. Pure Appl. Math.}, 68(11):1946--2004, 2015.

\bibitem{MR3318089}
R.~C\^{o}te, C.~E. Kenig, A.~Lawrie, and W.~Schlag.
\newblock Characterization of large energy solutions of the equivariant wave
  map problem: {I}.
\newblock {\em Amer. J. Math.}, 137(1):139--207, 2015.

\bibitem{MR3318090}
R.~C\^{o}te, C.~E. Kenig, A.~Lawrie, and W.~Schlag.
\newblock Characterization of large energy solutions of the equivariant wave
  map problem: {II}.
\newblock {\em Amer. J. Math.}, 137(1):209--250, 2015.

\bibitem{MR2443303}
Rapha\"{e}l C\^{o}te, Carlos~E. Kenig, and Frank Merle.
\newblock Scattering below critical energy for the radial 4{D} {Y}ang-{M}ills
  equation and for the 2{D} corotational wave map system.
\newblock {\em Comm. Math. Phys.}, 284(1):203--225, 2008.

\bibitem{MR2020108}
Piero D'Ancona and Vladimir Georgiev.
\newblock On the continuity of the solution operator to the wave map system.
\newblock {\em Comm. Pure Appl. Math.}, 57(3):357--383, 2004.

\bibitem{MR3401013}
Benjamin Dodson and Andrew Lawrie.
\newblock Scattering for radial, semi-linear, super-critical wave equations
  with bounded critical norm.
\newblock {\em Arch. Ration. Mech. Anal.}, 218(3):1459--1529, 2015.

\bibitem{MR2839272}
Roland Donninger.
\newblock On stable self-similar blowup for equivariant wave maps.
\newblock {\em Comm. Pure Appl. Math.}, 64(8):1095--1147, 2011.

\bibitem{MR3278903}
Roland Donninger.
\newblock Stable self-similar blowup in energy supercritical {Y}ang-{M}ills
  theory.
\newblock {\em Math. Z.}, 278(3-4):1005--1032, 2014.

\bibitem{2023arXiv231012016D}
Roland {Donninger}.
\newblock {Spectral theory and self-similar blowup in wave equations}.
\newblock {\em arXiv e-prints}, page arXiv:2310.12016, October 2023.

\bibitem{MR2412310}
Roland Donninger and Peter~C. Aichelburg.
\newblock On the mode stability of a self-similar wave map.
\newblock {\em J. Math. Phys.}, 49(4):043515, 9, 2008.

\bibitem{MR3861895}
Roland Donninger and Irfan Glogi\'{c}.
\newblock On the existence and stability of blowup for wave maps into a
  negatively curved target.
\newblock {\em Anal. PDE}, 12(2):389--416, 2019.

\bibitem{MR4661000}
Roland Donninger and Matthias Ostermann.
\newblock A globally stable self-similar blowup profile in energy supercritical
  {Y}ang-{M}ills theory.
\newblock {\em Comm. Partial Differential Equations}, 48(9):1148--1213, 2023.

\bibitem{MR3537340}
Roland Donninger and Birgit Sch\"{o}rkhuber.
\newblock On blowup in supercritical wave equations.
\newblock {\em Comm. Math. Phys.}, 346(3):907--943, 2016.

\bibitem{MR2881965}
Roland Donninger, Birgit Sch\"{o}rkhuber, and Peter~C. Aichelburg.
\newblock On stable self-similar blow up for equivariant wave maps: the
  linearized problem.
\newblock {\em Ann. Henri Poincar\'{e}}, 13(1):103--144, 2012.

\bibitem{2022arXiv221208374D}
Roland {Donninger} and David {Wallauch}.
\newblock {Optimal blowup stability for three-dimensional wave maps}.
\newblock {\em arXiv e-prints}, page arXiv:2212.08374, December 2022.

\bibitem{MR4640202}
Roland Donninger and David Wallauch.
\newblock Optimal blowup stability for supercritical wave maps.
\newblock {\em Adv. Math.}, 433:Paper No. 109291, 86, 2023.

\bibitem{MR3878592}
Thomas Duyckaerts, Hao Jia, Carlos Kenig, and Frank Merle.
\newblock Universality of blow up profile for small blow up solutions to the
  energy critical wave map equation.
\newblock {\em Int. Math. Res. Not. IMRN}, (22):6961--7025, 2018.

\bibitem{MR4397184}
Thomas Duyckaerts, Carlos Kenig, Yvan Martel, and Frank Merle.
\newblock Soliton resolution for critical co-rotational wave maps and radial
  cubic wave equation.
\newblock {\em Comm. Math. Phys.}, 391(2):779--871, 2022.

\bibitem{MR649158}
Douglas~M. Eardley and Vincent Moncrief.
\newblock The global existence of {Y}ang-{M}ills-{H}iggs fields in
  {$4$}-dimensional {M}inkowski space. {I}. {L}ocal existence and smoothness
  properties.
\newblock {\em Comm. Math. Phys.}, 83(2):171--191, 1982.

\bibitem{MR649159}
Douglas~M. Eardley and Vincent Moncrief.
\newblock The global existence of {Y}ang-{M}ills-{H}iggs fields in
  {$4$}-dimensional {M}inkowski space. {II}. {C}ompletion of proof.
\newblock {\em Comm. Math. Phys.}, 83(2):193--212, 1982.

\bibitem{MR1721989}
Klaus-Jochen Engel and Rainer Nagel.
\newblock {\em One-parameter semigroups for linear evolution equations}, volume
  194 of {\em Graduate Texts in Mathematics}.
\newblock Springer-Verlag, New York, 2000.
\newblock With contributions by S. Brendle, M. Campiti, T. Hahn, G. Metafune,
  G. Nickel, D. Pallara, C. Perazzoli, A. Rhandi, S. Romanelli and R.
  Schnaubelt.

\bibitem{MR3359541}
Can Gao and Joachim Krieger.
\newblock Optimal polynomial blow up range for critical wave maps.
\newblock {\em Commun. Pure Appl. Anal.}, 14(5):1705--1741, 2015.

\bibitem{2022arXiv220801503G}
Cristian {Gavrus}.
\newblock {The gauge-invariant I-method for Yang-Mills}.
\newblock {\em arXiv e-prints}, page arXiv:2208.01503, August 2022.

\bibitem{MR3585834}
Dan-Andrei Geba and Manoussos~G. Grillakis.
\newblock {\em An introduction to the theory of wave maps and related geometric
  problems}.
\newblock World Scientific Publishing Co. Pte. Ltd., Hackensack, NJ, 2017.

\bibitem{MR140316}
M.~Gell-Mann and M.~L\'{e}vy.
\newblock The axial vector current in beta decay.
\newblock {\em Nuovo Cimento (10)}, 16:705--726, 1960.

\bibitem{MR2450171}
Pierre Germain.
\newblock Besov spaces and self-similar solutions for the wave-map equation.
\newblock {\em Comm. Partial Differential Equations}, 33(7-9):1571--1596, 2008.

\bibitem{MR2494812}
Pierre Germain.
\newblock On the existence of smooth self-similar blowup profiles for the wave
  map equation.
\newblock {\em Comm. Pure Appl. Math.}, 62(5):706--728, 2009.

\bibitem{MR3812220}
T.~Ghoul, S.~Ibrahim, and V.~T. Nguyen.
\newblock Construction of type {II} blowup solutions for the 1-corotational
  energy supercritical wave maps.
\newblock {\em J. Differential Equations}, 265(7):2968--3047, 2018.

\bibitem{MR638511}
J.~Ginibre and G.~Velo.
\newblock The {C}auchy problem for coupled {Y}ang-{M}ills and scalar fields in
  the temporal gauge.
\newblock {\em Comm. Math. Phys.}, 82(1):1--28, 1981/82.

\bibitem{MR653018}
J.~Ginibre and G.~Velo.
\newblock The {C}auchy problem for coupled {Y}ang-{M}ills and scalar fields in
  the {L}orentz gauge.
\newblock {\em Ann. Inst. H. Poincar\'{e} Sect. A (N.S.)}, 36(1):59--78, 1982.

\bibitem{MR678488}
J.~Ginibre and G.~Velo.
\newblock The {C}auchy problem for the
  {${\mathrm{O}}(N),\,{\mathbf{C}}{\mathrm{P}}(N-1),$} and
  {$G_{{\mathbf{C}}}(N,\,p)$} models.
\newblock {\em Ann. Physics}, 142(2):393--415, 1982.

\bibitem{2022arXiv220706952G}
Irfan {Glogi{\'c}}.
\newblock {Globally stable blowup profile for supercritical wave maps in all
  dimensions}.
\newblock {\em arXiv e-prints}, page arXiv:2207.06952, July 2022.

\bibitem{MR4469070}
Irfan Glogi\'{c}.
\newblock Stable blowup for the supercritical hyperbolic {Y}ang-{M}ills
  equations.
\newblock {\em Adv. Math.}, 408:Paper No. 108633, 2022.

\bibitem{2023arXiv230510312G}
Irfan {Glogi{\'c}}.
\newblock {Global-in-space stability of singularity formation for Yang-Mills
  fields in higher dimensions}.
\newblock {\em arXiv e-prints}, page arXiv:2305.10312, May 2023.

\bibitem{MR3627409}
Roland Grinis.
\newblock Quantization of time-like energy for wave maps into spheres.
\newblock {\em Comm. Math. Phys.}, 352(2):641--702, 2017.

\bibitem{MR596432}
Chao~Hao Gu.
\newblock On the {C}auchy problem for harmonic maps defined on two-dimensional
  {M}inkowski space.
\newblock {\em Comm. Pure Appl. Math.}, 33(6):727--737, 1980.

\bibitem{MR3837560}
Mark J.~D. Hamilton.
\newblock {\em Mathematical gauge theory}.
\newblock Universitext. Springer, Cham, 2017.
\newblock With applications to the standard model of particle physics.

\bibitem{MR3842064}
Jacek Jendrej and Andrew Lawrie.
\newblock Two-bubble dynamics for threshold solutions to the wave maps
  equation.
\newblock {\em Invent. Math.}, 213(3):1249--1325, 2018.

\bibitem{2021arXiv210610738J}
Jacek {Jendrej} and Andrew {Lawrie}.
\newblock {Soliton resolution for energy-critical wave maps in the equivariant
  case}.
\newblock {\em arXiv e-prints}, page arXiv:2106.10738, June 2021.

\bibitem{MR4409881}
Jacek Jendrej and Andrew Lawrie.
\newblock An asymptotic expansion of two-bubble wave maps in high equivariance
  classes.
\newblock {\em Anal. PDE}, 15(2):327--403, 2022.

\bibitem{MR4589375}
Jacek Jendrej and Andrew Lawrie.
\newblock Continuous time soliton resolution for two-bubble equivariant wave
  maps.
\newblock {\em Math. Res. Lett.}, 29(6):1745--1766, 2022.

\bibitem{MR4599919}
Jacek Jendrej and Andrew Lawrie.
\newblock Uniqueness of two-bubble wave maps in high equivariance classes.
\newblock {\em Comm. Pure Appl. Math.}, 76(8):1608--1656, 2023.

\bibitem{MR3730929}
Hao Jia and Carlos Kenig.
\newblock Asymptotic decomposition for semilinear wave and equivariant wave map
  equations.
\newblock {\em Amer. J. Math.}, 139(6):1521--1603, 2017.

\bibitem{MR2314330}
Mathias Jungen.
\newblock On equivariant self-similar wave maps.
\newblock {\em NoDEA Nonlinear Differential Equations Appl.}, 13(4):469--483,
  2006.

\bibitem{MR1335452}
Tosio Kato.
\newblock {\em Perturbation theory for linear operators}.
\newblock Classics in Mathematics. Springer-Verlag, Berlin, 1995.
\newblock Reprint of the 1980 edition.

\bibitem{MR1663216}
Markus Keel and Terence Tao.
\newblock Local and global well-posedness of wave maps on {$\mathbf{R}^{1+1}$}
  for rough data.
\newblock {\em Internat. Math. Res. Notices}, (21):1117--1156, 1998.

\bibitem{MR1231427}
S.~Klainerman and M.~Machedon.
\newblock Space-time estimates for null forms and the local existence theorem.
\newblock {\em Comm. Pure Appl. Math.}, 46(9):1221--1268, 1993.

\bibitem{MR1338675}
S.~Klainerman and M.~Machedon.
\newblock Finite energy solutions of the {Y}ang-{M}ills equations in
  {$\mathbf{R}^{3+1}$}.
\newblock {\em Ann. of Math. (2)}, 142(1):39--119, 1995.

\bibitem{MR1381973}
S.~Klainerman and M.~Machedon.
\newblock Smoothing estimates for null forms and applications.
\newblock volume~81, pages 99--133 (1996). 1995.
\newblock A celebration of John F. Nash, Jr.

\bibitem{MR1446618}
Sergiu Klainerman and Matei Machedon.
\newblock On the regularity properties of a model problem related to wave maps.
\newblock {\em Duke Math. J.}, 87(3):553--589, 1997.

\bibitem{MR1843256}
Sergiu Klainerman and Igor Rodnianski.
\newblock On the global regularity of wave maps in the critical {S}obolev norm.
\newblock {\em Internat. Math. Res. Notices}, (13):655--677, 2001.

\bibitem{MR1452172}
Sergiu Klainerman and Sigmund Selberg.
\newblock Remark on the optimal regularity for equations of wave maps type.
\newblock {\em Comm. Partial Differential Equations}, 22(5-6):901--918, 1997.

\bibitem{MR1901147}
Sergiu Klainerman and Sigmund Selberg.
\newblock Bilinear estimates and applications to nonlinear wave equations.
\newblock {\em Commun. Contemp. Math.}, 4(2):223--295, 2002.

\bibitem{MR1626261}
Sergiu Klainerman and Daniel Tataru.
\newblock On the optimal local regularity for {Y}ang-{M}ills equations in
  {${\bf R}^{4+1}$}.
\newblock {\em J. Amer. Math. Soc.}, 12(1):93--116, 1999.

\bibitem{MR2488946}
J.~Krieger.
\newblock Global regularity and singularity development for wave maps.
\newblock In {\em Surveys in differential geometry. {V}ol. {XII}. {G}eometric
  flows}, volume~12 of {\em Surv. Differ. Geom.}, pages 167--201. Int. Press,
  Somerville, MA, 2008.

\bibitem{MR2372807}
J.~Krieger, W.~Schlag, and D.~Tataru.
\newblock Renormalization and blow up for charge one equivariant critical wave
  maps.
\newblock {\em Invent. Math.}, 171(3):543--615, 2008.

\bibitem{MR2522426}
J.~Krieger, W.~Schlag, and D.~Tataru.
\newblock Renormalization and blow up for the critical {Y}ang-{M}ills problem.
\newblock {\em Adv. Math.}, 221(5):1445--1521, 2009.

\bibitem{MR1990880}
Joachim Krieger.
\newblock Global regularity of wave maps from {${\bf R}^{3+1}$} to surfaces.
\newblock {\em Comm. Math. Phys.}, 238(1-2):333--366, 2003.

\bibitem{MR2094472}
Joachim Krieger.
\newblock Global regularity of wave maps from {$\mathbf{R}^{2+1}$} to {$H^2$}.
  {S}mall energy.
\newblock {\em Comm. Math. Phys.}, 250(3):507--580, 2004.

\bibitem{MR4065147}
Joachim Krieger and Shuang Miao.
\newblock On the stability of blowup solutions for the critical corotational
  wave-map problem.
\newblock {\em Duke Math. J.}, 169(3):435--532, 2020.

\bibitem{2020arXiv200908843K}
Joachim {Krieger}, Shuang {Miao}, and Wilhelm {Schlag}.
\newblock {A stability theory beyond the co-rotational setting for critical
  Wave Maps blow up}.
\newblock {\em arXiv e-prints}, page arXiv:2009.08843, September 2020.

\bibitem{MR2895939}
Joachim Krieger and Wilhelm Schlag.
\newblock {\em Concentration compactness for critical wave maps}.
\newblock EMS Monographs in Mathematics. European Mathematical Society (EMS),
  Z\"{u}rich, 2012.

\bibitem{MR3087010}
Joachim Krieger and Jacob Sterbenz.
\newblock Global regularity for the {Y}ang-{M}ills equations on high
  dimensional {M}inkowski space.
\newblock {\em Mem. Amer. Math. Soc.}, 223(1047):vi+99, 2013.

\bibitem{MR3664812}
Joachim Krieger and Daniel Tataru.
\newblock Global well-posedness for the {Y}ang-{M}ills equation in {$4+1$}
  dimensions. {S}mall energy.
\newblock {\em Ann. of Math. (2)}, 185(3):831--893, 2017.

\bibitem{MR3465437}
Andrew Lawrie and Sung-Jin Oh.
\newblock A refined threshold theorem for {$(1+2)$}-dimensional wave maps into
  surfaces.
\newblock {\em Comm. Math. Phys.}, 342(3):989--999, 2016.

\bibitem{2024arXiv240815345M}
Michael {McNulty}.
\newblock {Singularity formation for the higher dimensional Skyrme model}.
\newblock {\em arXiv e-prints}, page arXiv:2408.15345, August 2024.

\bibitem{MR2016196}
Andrea Nahmod, Atanas Stefanov, and Karen Uhlenbeck.
\newblock On the well-posedness of the wave map problem in high dimensions.
\newblock {\em Comm. Anal. Geom.}, 11(1):49--83, 2003.

\bibitem{MR3190112}
Sung-Jin Oh.
\newblock Gauge choice for the {Y}ang-{M}ills equations using the
  {Y}ang-{M}ills heat flow and local well-posedness in {$H^1$}.
\newblock {\em J. Hyperbolic Differ. Equ.}, 11(1):1--108, 2014.

\bibitem{MR3357182}
Sung-Jin Oh.
\newblock Finite energy global well-posedness of the {Y}ang-{M}ills equations
  on {$\mathbb{R}^{1+3}$}: an approach using the {Y}ang-{M}ills heat flow.
\newblock {\em Duke Math. J.}, 164(9):1669--1732, 2015.

\bibitem{MR3907955}
Sung-Jin Oh and Daniel Tataru.
\newblock The hyperbolic {Y}ang-{M}ills equation for connections in an
  arbitrary topological class.
\newblock {\em Comm. Math. Phys.}, 365(2):685--739, 2019.

\bibitem{MR3923343}
Sung-Jin Oh and Daniel Tataru.
\newblock The threshold theorem for the {$(4+1)$}-dimensional {Y}ang-{M}ills
  equation: an overview of the proof.
\newblock {\em Bull. Amer. Math. Soc. (N.S.)}, 56(2):171--210, 2019.

\bibitem{MR4113787}
Sung-Jin Oh and Daniel Tataru.
\newblock The hyperbolic {Y}ang-{M}ills equation in the caloric gauge: local
  well-posedness and control of energy-dispersed solutions.
\newblock {\em Pure Appl. Anal.}, 2(2):233--384, 2020.

\bibitem{MR4298746}
Sung-Jin Oh and Daniel Tataru.
\newblock The threshold conjecture for the energy critical hyperbolic
  {Y}ang-{M}ills equation.
\newblock {\em Ann. of Math. (2)}, 194(2):393--473, 2021.

\bibitem{MR4518477}
Sung-Jin Oh and Daniel Tataru.
\newblock The {Y}ang-{M}ills heat flow and the caloric gauge.
\newblock {\em Ast\'{e}risque}, (436):viii+128, 2022.

\bibitem{2022arXiv220902286O}
Matthias {Ostermann}.
\newblock {A characterisation of the subspace of radially symmetric functions
  in Sobolev spaces}.
\newblock {\em arXiv e-prints}, page arXiv:2209.02286, September 2022.

\bibitem{MR4778061}
Matthias Ostermann.
\newblock Stable blowup for focusing semilinear wave equations in all
  dimensions.
\newblock {\em Trans. Amer. Math. Soc.}, 377(7):4727--4778, 2024.

\bibitem{MR710486}
A.~Pazy.
\newblock {\em Semigroups of linear operators and applications to partial
  differential equations}, volume~44 of {\em Applied Mathematical Sciences}.
\newblock Springer-Verlag, New York, 1983.

\bibitem{MR2929728}
Pierre Rapha\"{e}l and Igor Rodnianski.
\newblock Stable blow up dynamics for the critical co-rotational wave maps and
  equivariant {Y}ang-{M}ills problems.
\newblock {\em Publ. Math. Inst. Hautes \'{E}tudes Sci.}, 115:1--122, 2012.

\bibitem{MR2918544}
Jeffrey Rauch.
\newblock {\em Hyperbolic partial differential equations and geometric optics},
  volume 133 of {\em Graduate Studies in Mathematics}.
\newblock American Mathematical Society, Providence, RI, 2012.

\bibitem{MR2680419}
Igor Rodnianski and Jacob Sterbenz.
\newblock On the formation of singularities in the critical {${\rm O}(3)$}
  {$\sigma$}-model.
\newblock {\em Ann. of Math. (2)}, 172(1):187--242, 2010.

\bibitem{MR4353567}
Casey Rodriguez.
\newblock Threshold dynamics for corotational wave maps.
\newblock {\em Anal. PDE}, 14(7):2123--2161, 2021.

\bibitem{MR385023}
Walter Rudin.
\newblock {\em Principles of mathematical analysis}.
\newblock International Series in Pure and Applied Mathematics. McGraw-Hill
  Book Co., New York-Auckland-D\"{u}sseldorf, third edition, 1976.

\bibitem{MR546505}
Irving Segal.
\newblock The {C}auchy problem for the {Y}ang-{M}ills equations.
\newblock {\em J. Functional Analysis}, 33(2):175--194, 1979.

\bibitem{MR3519539}
Sigmund Selberg and Achenef Tesfahun.
\newblock Null structure and local well-posedness in the energy class for the
  {Y}ang-{M}ills equations in {L}orenz gauge.
\newblock {\em J. Eur. Math. Soc. (JEMS)}, 18(8):1729--1752, 2016.

\bibitem{MR1168115}
J.~Shatah and A.~Tahvildar-Zadeh.
\newblock Regularity of harmonic maps from the {M}inkowski space into
  rotationally symmetric manifolds.
\newblock {\em Comm. Pure Appl. Math.}, 45(8):947--971, 1992.

\bibitem{MR933231}
Jalal Shatah.
\newblock Weak solutions and development of singularities of the {${\rm
  SU}(2)$} {$\sigma$}-model.
\newblock {\em Comm. Pure Appl. Math.}, 41(4):459--469, 1988.

\bibitem{MR1674843}
Jalal Shatah and Michael Struwe.
\newblock {\em Geometric wave equations}, volume~2 of {\em Courant Lecture
  Notes in Mathematics}.
\newblock New York University, Courant Institute of Mathematical Sciences, New
  York; American Mathematical Society, Providence, RI, 1998.

\bibitem{MR1890048}
Jalal Shatah and Michael Struwe.
\newblock The {C}auchy problem for wave maps.
\newblock {\em Int. Math. Res. Not.}, (11):555--571, 2002.

\bibitem{MR1278351}
Jalal Shatah and A.~Shadi Tahvildar-Zadeh.
\newblock On the {C}auchy problem for equivariant wave maps.
\newblock {\em Comm. Pure Appl. Math.}, 47(5):719--754, 1994.

\bibitem{MR2325100}
Jacob Sterbenz.
\newblock Global regularity and scattering for general non-linear wave
  equations. {II}. {$(4+1)$} dimensional {Y}ang-{M}ills equations in the
  {L}orentz gauge.
\newblock {\em Amer. J. Math.}, 129(3):611--664, 2007.

\bibitem{MR2657817}
Jacob Sterbenz and Daniel Tataru.
\newblock Energy dispersed large data wave maps in {$2+1$} dimensions.
\newblock {\em Comm. Math. Phys.}, 298(1):139--230, 2010.

\bibitem{MR2657818}
Jacob Sterbenz and Daniel Tataru.
\newblock Regularity of wave-maps in dimension {$2+1$}.
\newblock {\em Comm. Math. Phys.}, 298(1):231--264, 2010.

\bibitem{MR1985457}
Michael Struwe.
\newblock Radially symmetric wave maps from {$(1+2)$}-dimensional {M}inkowski
  space to the sphere.
\newblock {\em Math. Z.}, 242(3):407--414, 2002.

\bibitem{MR1990477}
Michael Struwe.
\newblock Equivariant wave maps in two space dimensions.
\newblock volume~56, pages 815--823. 2003.
\newblock Dedicated to the memory of J\"{u}rgen K. Moser.

\bibitem{MR1971037}
Michael Struwe.
\newblock Radially symmetric wave maps from {$(1+2)$}-dimensional {M}inkowski
  space to general targets.
\newblock {\em Calc. Var. Partial Differential Equations}, 16(4):431--437,
  2003.

\bibitem{MR1759884}
Terence Tao.
\newblock Ill-posedness for one-dimensional wave maps at the critical
  regularity.
\newblock {\em Amer. J. Math.}, 122(3):451--463, 2000.

\bibitem{MR1820329}
Terence Tao.
\newblock Global regularity of wave maps. {I}. {S}mall critical {S}obolev norm
  in high dimension.
\newblock {\em Internat. Math. Res. Notices}, (6):299--328, 2001.

\bibitem{MR1869874}
Terence Tao.
\newblock Global regularity of wave maps. {II}. {S}mall energy in two
  dimensions.
\newblock {\em Comm. Math. Phys.}, 224(2):443--544, 2001.

\bibitem{MR1964470}
Terence Tao.
\newblock Local well-posedness of the {Y}ang-{M}ills equation in the temporal
  gauge below the energy norm.
\newblock {\em J. Differential Equations}, 189(2):366--382, 2003.

\bibitem{MR2459308}
Terence Tao.
\newblock Global behaviour of nonlinear dispersive and wave equations.
\newblock In {\em Current developments in mathematics, 2006}, pages 255--340.
  Int. Press, Somerville, MA, 2008.

\bibitem{2008arXiv0805.4666T}
Terence {Tao}.
\newblock {Global regularity of wave maps III. Large energy from
  $\mathbb{R}^{1+2}$ to hyperbolic spaces}.
\newblock {\em arXiv e-prints}, page arXiv:0805.4666, May 2008.

\bibitem{2008arXiv0806.3592T}
Terence {Tao}.
\newblock {Global regularity of wave maps IV. Absence of stationary or
  self-similar solutions in the energy class}.
\newblock {\em arXiv e-prints}, page arXiv:0806.3592, June 2008.

\bibitem{2008arXiv0808.0368T}
Terence {Tao}.
\newblock {Global regularity of wave maps V. Large data local wellposedness and
  perturbation theory in the energy class}.
\newblock {\em arXiv e-prints}, page arXiv:0808.0368, August 2008.

\bibitem{2009arXiv0906.2833T}
Terence {Tao}.
\newblock {Global regularity of wave maps VI. Abstract theory of minimal-energy
  blowup solutions}.
\newblock {\em arXiv e-prints}, page arXiv:0906.2833, June 2009.

\bibitem{2009arXiv0908.0776T}
Terence {Tao}.
\newblock {Global regularity of wave maps VII. Control of delocalised or
  dispersed solutions}.
\newblock {\em arXiv e-prints}, page arXiv:0908.0776, August 2009.

\bibitem{MR1641721}
Daniel Tataru.
\newblock Local and global results for wave maps. {I}.
\newblock {\em Comm. Partial Differential Equations}, 23(9-10):1781--1793,
  1998.

\bibitem{MR1827277}
Daniel Tataru.
\newblock On global existence and scattering for the wave maps equation.
\newblock {\em Amer. J. Math.}, 123(1):37--77, 2001.

\bibitem{MR2043751}
Daniel Tataru.
\newblock The wave maps equation.
\newblock {\em Bull. Amer. Math. Soc. (N.S.)}, 41(2):185--204, 2004.

\bibitem{MR2130618}
Daniel Tataru.
\newblock Rough solutions for the wave maps equation.
\newblock {\em Amer. J. Math.}, 127(2):293--377, 2005.

\bibitem{MR2961944}
Gerald Teschl.
\newblock {\em Ordinary differential equations and dynamical systems}, volume
  140 of {\em Graduate Studies in Mathematics}.
\newblock American Mathematical Society, Providence, RI, 2012.

\bibitem{MR3385623}
Achenef Tesfahun.
\newblock Local well-posedness of {Y}ang-{M}ills equations in {L}orenz gauge
  below the energy norm.
\newblock {\em NoDEA Nonlinear Differential Equations Appl.}, 22(4):849--875,
  2015.

\bibitem{turok1990global}
Neil Turok and David Spergel.
\newblock Global texture and the microwave background.
\newblock {\em Physical Review Letters}, 64(23):2736, 1990.

\bibitem{MR65437}
C.~N. Yang and R.~L. Mills.
\newblock Conservation of isotopic spin and isotopic gauge invariance.
\newblock {\em Phys. Rev. (2)}, 96:191--195, 1954.

\end{thebibliography}
\end{document}